\documentclass[10pt]{article}
\usepackage{amsmath,amssymb,amsthm,amsfonts,amscd,euscript,verbatim, t1enc, newlfont, xypic, graphicx} 
\usepackage{hyperref}
\hypersetup{breaklinks}
\usepackage{cancel} 
\providecommand{\keywords}[1]{\textbf{\textit{Key words and phrases }} #1}
\providecommand{\subjclass}[1]{\textbf{\textit{2010 Mathematics Subject Classification.}} #1}
\hypersetup{breaklinks}
\theoremstyle{definition}
\newtheorem{theo}{Theorem}[subsection]

\newtheorem{theore}{Theorem}[section]
\newtheorem{rrema}[theore]{Remark}
\newtheorem{pr}[theo]{Proposition}

\newtheorem{prop}[theore]{Proposition} 
 \newtheorem{lem}[theo]{Lemma}
 \newtheorem{coro}[theo]{Corollary}
  
\theoremstyle{remark}
\newtheorem{rema}[theo]{Remark}

\newtheorem{defi}[theo]{Definition}

 \newcommand\lan{\langle}
\newcommand\ra{\rangle}
 
\newcommand\ob{^{-1}}

\newcommand\proo{\operatorname{Pro}}
\newcommand\proj{\operatorname{Proj}}
\newcommand\obj{\operatorname{Obj}}
\newcommand\mo{\operatorname{Mor}}
\newcommand\id{\operatorname{id}}

\newcommand\cu{\underline{C}}
\newcommand\du{\underline{D}}
\newcommand\au{\underline{A}}
\newcommand\bu{\underline{B}}
\newcommand\eu{\underline{E}}
\newcommand\cupr{\underline{C}'}

\newcommand\z{{\mathbb{Z}}}

\newcommand\q{{\mathbb{Q}}}

\newcommand\p{\mathbb{P}}
\newcommand\ab{\underline{\operatorname{Ab}}}

\newcommand\ca{{\mathcal{A}}}

\newcommand\lam{\Lambda}
\newcommand\al{\alpha}
\newcommand\be{\beta}

\newcommand\modd{\operatorname{Mod}}
\newcommand\mmodd{\operatorname{mod}}

\newcommand\ns{\{0\}}

\DeclareMathOperator\prli{\varprojlim}
\DeclareMathOperator\inli{\varinjlim}
 \DeclareMathOperator\ke{\operatorname{Ker}}
 \DeclareMathOperator\cok{\operatorname{Coker}}
\DeclareMathOperator\imm{\operatorname{Im}}
\DeclareMathOperator\co{\operatorname{Cone}}

\DeclareMathOperator\kar{\operatorname{Kar}}
\DeclareMathOperator\adfu{\operatorname{AddFun}}
\DeclareMathOperator\adfur{\operatorname{AddFun}_R}

\newcommand\hrt{{\underline{Ht}}}
\newcommand\injh{\operatorname{Inj}\hrt}
\newcommand\hw{{\underline{Hw}}}
\newcommand\ec{\mathcal{S}}
\newcommand\bc{\mathcal{BCD}}

\newcommand\pwcu{\operatorname{Post}_w(\cu)}
\newcommand\cuw{\cu_w}

\newcommand\wstu{w^{st}}
\newcommand\cuz{\underline{C}_0}
\newcommand\duz{{\underline{D}}_0}
\newcommand\hdu{\widehat{H}}
\newcommand\wz{w_0}
\newcommand\kw{K_{\mathfrak{w}}}
\newcommand\psvr{\operatorname{PShv}^R}
\newcommand\psvcuz{\operatorname{PShv}(\cu_0)}
\newcommand\opp{{^{op}}}
\newcommand\perpp{{}^{\perp}}

\newcommand\psv{\operatorname{PShv}}

\newcommand\lo{\mathcal{LO}}
\newcommand\ro{\mathcal{RO}}
\newcommand\cp{\mathcal{P}}
\newcommand\cocp{\underline{\coprod}\cp} 
\newcommand\cq{\mathcal{Q}}
\newcommand\cpt{{\tilde{\mathcal{P}}}}
\newcommand\lscp{L_s\mathcal{P}}

\newcommand\ho{\operatorname{Ho}}
\newcommand\gm{\mathcal{M}}
\newcommand\gmp{\operatorname{Pro}-\mathcal{M}}
\newcommand\alz
{{\aleph_0}}
\newcommand\oo{{\pmb{1}}}
\newcommand\hatc{\operatorname{Coh}_{\cp}}
\newcommand\fil{{\operatorname{Fil}}}
\newcommand\spe{{\operatorname{Spec}\,}}
\newcommand\dmgm{\operatorname{DM_{gm}}}
\newcommand\chow{\operatorname{Chow}}
\newcommand\chowe{\operatorname{Chow^{eff}}}

\newcommand\ql{{\mathbb{Q}_l}}
\newcommand\dbm{D^b_m}

\numberwithin{equation}{subsection}

\begin{document}

 \title{On  torsion pairs,  (well generated) weight structures,  adjacent $t$-structures, and related    (co)homological functors} 
 \author{Mikhail V. Bondarko
   \thanks{
	The author's work on sections 1--2 was supported by the RFBR grant No. 15-01-03034-a and by  Dmitry Zimin's Foundation "Dynasty", whereas the work on sections 3--5  was supported by the Russian Science Foundation grant no. 16-11-10073.}}\maketitle
\begin{abstract}
This paper contains a rich collection of results related to weight structures and $t$-structures. 
 For any weight structure $w$ we study {\it pure} (co)homological  functors; these "ignore all weights except weight zero" and have already found several applications (in particular, to Picard groups of triangulated categories). 
Moreover, 
we study {\it virtual $t$-truncations} of cohomological functors. 
  The resulting functors are defined in terms of $w$ but are closely related to $t$-structures; so we prove in several cases that a weight structure $w$ "gives" a $t$-structure (that is {\it adjacent} or {\it $\Phi$-orthogonal} to it). 
	
	We also study in detail {\it well generated} weight structures  (and prove that any {\it 
	perfect} set of objects generates a certain weight structure). 
	 We prove the existence of weight structures right adjacent 
	 to {\it compactly generated}  $t$-structures (using Brown-Comenetz duality); this implies that the hearts of the latter have 
injective cogenerators and  satisfy the AB3* axiom. 
 Actually, "most of" these hearts are Grothendieck abelian (due to the existence of "bicontinuously orthogonal" weight structures). 

It is convenient for us to use the notion of {\it torsion pairs}; these essentially generalize both weight structures and $t$-structures. We prove several new properties of torsion pairs; 
 in particular, we generalize a theorem of D. Pospisil and J. \v{S}\v{t}ov\'\i\v{c}ek to obtain a complete classification of  
 compactly generated torsion pairs. 

\end{abstract}



\tableofcontents

 \section*{Introduction}

The main goal of the current paper is to demonstrate the utility of  weight structures to the construction and study of $t$-structures and of
 (co)homological functors from triangulated categories (into abelian ones). In particular, for a weight structure $w$ we study {\it $w$-pure functors} (i.e., those that "only see $w$-weight zero").\footnote{The relation of pure functors to Deligne's purity of (singular and \'etale) cohomology is recalled in Remark \ref{rwrange}(5).}\ 
Functors of this type have already found interesting applications in several papers (note in particular that the results of  our \S\ref{sdetect} are important for the study of Picard groups of triangulated categories in \cite{bontabu}). 
So the author believes that the reader not interested in the construction weight structures and $t$-structures (that we will start discussing very soon) may still benefit from \S\ref{sws} of the paper where a rich collection of properties of pure functors and  {\it virtual $t$-truncations} of (co)homological functors (with respect to $w$) is proved.

 Now, virtual $t$-truncations are defined in terms of weight structures; still they are closely related to $t$-structures (whence the name). 
 Respectively, our results yield the existence of some vast new families of $t$-structures. To describe one of the main results of this sort here we recall that a $t$-structure $t=(\cu^{t\le 0}, \cu^{t\ge 0})$ (for a triangulated category $\cu$) is said to be (right) adjacent to $w$ if $\cu^{t\le 0}=\cu_{w\ge 0}$.\footnote{In \cite{bws} $t$ was said to be left adjacent to $w$ in this case; we discuss this distinction in "conventions" in \S\ref{rwsts} below.} 
For a triangulated category $\cu$ that is closed with respect to (small) coproducts and a weight structure $w$ on it we will say that $w$ is {\it smashing} whenever $\cu_{w\ge 0}$ is closed with respect to $\cu$-coproducts (note that $\cu_{w\le 0}$ is $\coprod$-closed automatically).

\begin{theore}\label{tadjti}
Let $\cu$ be a  triangulated category  that is closed with respect to coproducts and satisfies the following Brown representability property: any functor $\cu\opp\to \ab$ that respects ($\cu\opp$)-products is representable (in $\cu$).\footnote{Thanks to the foundational results of A. Neeman and others, this property is known to hold for several important classes of triangulated categories; in particular, it suffices to assume that either $\cu$ or $\cu\opp$ is compactly generated.}

Then for  a weight structure $w$  on $\cu$ there exists a $t$-structure right adjacent to it if and only if $w$ is smashing. Moreover, the heart of $t$ (if $t$ exists) is equivalent to the category of all those additive functors $\hw\opp\to \ab$ that respect products.\footnote{Here $\hw$ is the heart of $w$; note also that $G: \hw\opp\to \ab$ respects products whenever it converts $\hw$-coproducts into products of groups.}\end{theore}

Note here that  triangulated categories closed with respect to coproducts have recently become very popular in homological algebra and found applications in various areas of mathematics; so a significant part of this paper is dedicated to categories of this sort. Still we prove certain alternative versions of Theorem \ref{tadjti}; some of them can be applied to "quite small" triangulated categories.
 
So, if instead of the Brown representability condition for $\cu$ we demand it to satisfy the {\it $R$-saturatedness one} (see Definition \ref{dsatur} below; this is an  "$R$-linear finite" version of the Brown representability) then for any {\it bounded} $w$ on $\cu$ there will exist a $t$-structure right adjacent to it. 
Note that this version of the result can be applied to the bounded derived category $D^b(X)$  of coherent sheaves on regular separated finite-dimensional scheme that is proper over the spectrum of a Noetherian ring $R$ 
   (see Proposition \ref{psatur}(II) and Remark \ref{rsatur}(1)). 
We also prove  two generalizations of this existence result (see Proposition \ref{psaturdu} and Remark \ref{roq}); they "produce" certain $t$-structures from weight structures on the bounded derived category of coherent sheaves on any 
 proper $R$-scheme $X$ and also on the triangulated category of perfect complexes on $X$.\footnote{The author suspects that the content of the paper will not be evenly interesting to the readers. Most of its content be found in more recent texts of the author; see Remark \ref{rnew} below.  
 Note moreover that the readers not (much) interested in "large" categories may ignore all matters related to infinite coproducts on the first reading (this includes compact objects and smashing torsion pairs). 
On the other hand, \S\ref{sswc} (and so, weight complexes and Postnikov towers) are mentioned explicitly in \S\ref{sws} only. So, the reader is encouraged to look for his personal trajectory through this paper and also look at the papers cited in Remark \ref{rnew}.} 

So, one may (roughly) say that any "reasonable" weight structure (on a triangulated category satisfying some Brown representability condition) can be used to construct certain $t$-structures. This 
 demonstrates the importance of constructing weight structures.

However, the $t$-structures constructed using Theorem \ref{tadjti} appear to be somewhat "exotic" (yet cf. Theorem 1 of \cite{zvon}) and possibly the $t$-structures constructed via the aforementioned "$R$-saturated" versions of the theorem are "more useful". Still we also prove  (using adjacent and {\it $\Phi$-orthogonal} weight structures in a crucial way)   several properties of {compactly generated} $t$-structures.\footnote{$t$-structures of these type appear to be originally introduced in \cite{talosa}. They have become a popular object of study recently, with plenty of examples important to various areas of mathematics.} 

To formulate the following theorem we need some definitions.

For a class of objects $S$ of a triangulated category $\cu$ we will write $S\perpp$ (resp. $\perpp S$) for the class of those $M\in \obj \cu$ such that the morphism group $\cu(N,M)$ (resp. $\cu(M,N)$) is zero   for all $N\in S$.

We will say that a $t$-structure $t=(\cu^{t\le 0},\cu^{t\ge 0})$ on $\cu$ is {\it generated} by a class $\cp\subset \obj \cu$ whenever $\cu^{t\ge 0}=\cap_{i\ge 1}(\cp\perpp[i])$. 


 \begin{theore}\label{tgroth}
Let $\cu$ be a  triangulated category  that is closed with respect to coproducts; 
let $t$ be 
 a $t$-structure on it generated by a set of compact objects (we will say that $t$ is compactly generated if such a generating set $\cp$ exists).\footnote{Recall that $P\in \obj \cu$ is said to be compact if the corepresentable functor $\cu(P,-)$ respects coproducts. Now, any set of compact objects generates a $t$-structure according to Theorem A.1 of \cite{talosa}; the corresponding class $\cu^{t\le 0}$ is the smallest subclass of $\obj \cu$ that is closed with respect to $[1]$, extensions, and coproducts, and contains $\cp$.}

1. Then the heart $\hrt$ of  $t$ has an injective cogenerator  and satisfies the AB3* axiom. 

2. Assume in addition either that $\cu$ is the homotopy category of a proper simplicial stable model category or that $t$ is non-degenerate. Then $\hrt$ is a Grothendieck abelian category.

\end{theore}


Now we describe a result that was important for the proof 
Theorem \ref{tgroth}(1).

One of the main topics of this paper is the study of the 
 relation of  (smashing) weight structures  to 
 perfect sets of objects (that we will now define; recall that these are also closely related to the Brown representability condition).\footnote{More generally, perfect classes are closely related to smashing torsion pairs; 
 see Proposition \ref{psym}(\ref{iperftp}).}


\begin{theore}\label{textw}
Assume that $\cu$ is closed with respect to small coproducts. Let $\cp$ be a {\it perfect} set of objects of $\cu$ (i.e., assume that the class {\it $\cp$-null} of those morphisms that are annihilated by corepresentable functors 
of the type $\cu(P,-)$ for $P\in \cp$ is closed with respect to  coproducts).\footnote{See Remark \ref{requivdef} below for a comparison of this definition with  other ones in the literature.} 

Then 
the couple $w=(L,R)$ is a smashing weight structure, where $R=\cap_{i<0}(\cp\perpp[i])$ and $L=(\perpp R)[1]$. 

Moreover,  the class $L$ may be described "more explicitly" in terms of $\cp$; cf. 
Corollary \ref{cwftw} below. 

\end{theore}

Thus if 
 we assume in addition that the Brown representability condition is fulfilled for $\cu$\footnote{Recall that this is always the case if the class $\cp$ {\it Hom-generates}
$\cu$, i.e., $\cap_{i\in \z}(\cp[i]\perpp)=\ns$. 
 Since all elements of $\cap_i(\cp[i]\perpp)$ are {\it left degenerate} with respect to $w$, if the Brown representability condition is not fulfilled for $\cu$ then $w$ is necessarily "somewhat pathological".}then we also obtain the (right) adjacent $t$-structure $t=(R, (R\perpp)[1])$. 
This $t$-structure is {\it cogenerated} by any class $\cp'$ that is {\it  weakly 
 symmetric} to $\cp$ (see Definition \ref{dsym}(\ref{iwsym}) below; 
 if the elements of $\cp$ are compact then we can construct $\cp'$ using {\it Brown-Comenetz duality}). Somewhat surprisingly to the author, a similar chain of arguments  gives the existence of a weight structure $w$ that is {\it right adjacent} to a given  compactly generated $t$-structure (i.e., $\cu^{t\ge 0}=\cu_{w\le 0}$); this yields the proof of Theorem \ref{tgroth}(1).\footnote{Note here that our proof of Theorem \ref{textw}  (that is somewhat similar to the corresponding proof from \cite{kraucoh}) does not allow constructing $t$-structures "directly from perfect sets"  since it relies crucially on $\cu_{w\le 0}[-1]\subset \cu_{w\le 0}$ (and does not work for general torsion pairs that  will be discussed soon; yet cf. Remark \ref{rigid}(1)). Still we are able to prove the existence of $t$-structure generated by $\cp$ whenever $\cp$ is {\it symmetric} to some $\cp'\subset \obj \cu$ (see Theorem \ref{tsymt}; note that any such $\cp$ is certainly perfect).} 
 Moreover, the "opposite" weight structure $w\opp$ in the category $\cu\opp$ is perfectly generated but not compactly generated (so, there exist plenty of perfectly generated weight structures that are not compactly generated; recall that the latter class of weight structures was introduced in \cite{paucomp}). 

We also prove the following "well generatedness" result for weight structures  (obtaining in particular that all smashing weight structures on well generated categories can be obtained from Theorem \ref{textw}).

\begin{prop}\label{pwgwstr}
Assume that $\cu$ is a well generated triangulated category (i.e., there exists a {\it regular} cardinal $\al$ 
and a perfect set  $S$ of {\it $\al$-small} objects such that $S\perpp=\ns$; see Definition \ref{dwg}(\ref{idpg})). 

Then for any smashing weight structure $w$  
  on $\cu$ there exists a  cardinal $\al'$ 
	 such that for any regular $\be\ge \al'$ the weight structure $w$ is strongly $\be$-well generated in the following sense: the couple $(\cu_{w\le 0}\cap \obj \cu^\be, \cu_{w\ge 0}\cap \obj \cu^\be)$ is a weight structure on the triangulated subcategory $\cu^\be$ of $\cu$ consisting of $\be$-compact objects (see Definition \ref{dbecomp}(\ref{idcomp})), the class $\cp=\cu_{w\le 0}\cap \obj \cu^\be$ is essentially small and perfect, and $w=(L,R)$  for $L$ and $R$ being the classes described in Theorem \ref{textw} (if these conditions are fulfilled for $w$ and  some $\cp$ then we also say that $w$ is perfectly generated by $\cp$). 
\end{prop}


A significant part of the 
"easier" results of the current paper is stated in terms of {\it torsion pairs} (as defined in \cite{aiya}; cf. Remark \ref{rgen}(4) below); these essentially generalize both weight structures and $t$-structures.\footnote{Note that in \cite{postov} torsion pairs were called complete Hom-orthogonal pairs.}\  This certainly makes the corresponding results more general; note also that (the main subject of) \cite{bkw}  yields an interesting family of examples of torsion pairs that do not come either from weight structures or from $t$-structures.  
Probably the most interesting result about general torsion pairs proved in this paper is the classification of compactly generated ones (in Theorem \ref{tclass}); 
 we drop the assumption that $\cu$ is a "stable derivator" category that was necessary for the proof of the closely related Corollary 3.8 of 
 \cite{postov}. We also relate {\it adjacent} torsion pairs to "Brown-Comenetz-type symmetry". 


We also study {\it dualities} between triangulated categories and their relation to torsion pairs. 
 We 
 demonstrate that for a $t$-structure $t$ on $\cupr$ that does not possess a left adjacent weight structure there still may exist 
$w$ on a category $\cu$ that is {$\Phi$-orthogonal } to $t$, where $\Phi:\cu\opp\times \cu'\to \ab$ is a ("bicontinuous") 
  duality of triangulated categories. 
Moreover, for an interesting sort of dualities that we study in \S\ref{sdual2} and for any compactly generated $t$ on $\cupr$ the objects of the heart of the corresponding orthogonal $w$ give faithful exact 
 "stalk" functors $\hrt\to \ab$ that respect coproducts; taking the functor $\Phi(P,-)$ for $P$ being a {\it cogenerator} of $\hw$ we easily obtain that $\hrt$ is an AB5 abelian category (cf. Theorem \ref{tgroth}). 
So we demonstrate once again that weight structures "shed some light on $t$-structures".

Moreover, dualities  are important for \cite{bgn} where they are  applied to the study of coniveau spectral sequences and the homotopy $t$-structures on various motivic stable homotopy categories.  
 This matter is also related to the study of the stable homotopy category in \cite{prospect}. 

\begin{rrema}\label{rnew} As we have already explained, in this paper several distinct interesting topics are considered. For this reason the author decided to split this text into several new ones, each of which will be mostly concentrated on a single question and treat it accurately and in more detail (respectively, these papers will contain more results than the current one; however, the 
proofs may be really different, and the arguments used in this text may be of certain interest as well).

So, some basics on general and "concrete" torsion pairs (that corresponds to our sections \ref{shop}--\ref{ssws}) and symmetry between them (see \S\ref{scomp})  can be found in \cite{bvt} (where torsion pairs are called torsion theories). The theory of weight complexes, pure functors and detecting weights by them, and smashing weight structures (cf. our \S\S\ref{sswc}--\S\ref{sdetect}) are considered in detail in \cite{bwcp}. Weight range, its relation to virtual $t$-truncations, and several existence of adjacent and orthogonal $t$-structures and weight structures statements (cf. \S\S\ref{svtt}--\ref{sadjw} and \S\ref{sdual1}) are treated in \cite{bvtr}; moreover, examples coming from various categories of quasicoherent sheaves are treated there more carefully than in Remark  \ref{roq} below.  Perfect and well generated weight structures, their applications to adjacent $t$-structures, compactly generated torsion pairs,  and certain auxiliary results (corresponding to  \S\ref{spgtp} 
and \S\ref{sgengroth})  are treated and essentially  improved in \cite{bpws}.
Note however that some of the terminology and notation in these papers differs from the one we use here.

The reader is also recommended to look at the recent papers \cite{bsnew} and  \cite{bkwn}. 
\end{rrema}

Let us  now describe the contents  of the paper. Some more information of this sort may be found in the beginnings of sections.

In \S\ref{sold} 
 we 
we define torsion pairs and prove several of their properties; these are mostly simple but new. 
We also relate torsion pairs to $t$-structures and study $t$-projective 
objects   (essentially following \cite{zvon}); they are related to adjacent weight structures that we will study later.


We start \S\ref{sws}  from recalling some  basics on weight structures (among those are some properties of weight complexes; though these are not really new, we treat this subject more accurately than in \cite{bws} where 
 weight complexes were originally defined). Next we 
 introduce pure functors (in \S\ref{sdetect}; their "construction" in Proposition \ref{ppure} appears to be a very useful statement). An "intrinsic" definition of pure functors is given by Proposition \ref{pwrange}.
We relate the weight range of functors to virtual $t$-truncations. We also study the properties of virtual $t$-truncations and weight complexes
 under the assumption that $w$ is smashing. As an application, we treat the representability for virtual $t$-truncations of representable functors; this result has important applications to the construction of (adjacent) $t$-structures below.

In \S\ref{sadjbrown} we 
 investigate the question when weight structures and $t$-structures admit (right or left) adjacent $t$-structures or weight structures (respectively); we also study adjacent "structures" if they exist. To prove the existence of these adjacent structures on a triangulated category $\cu$ we usually assume a certain Brown representability-type condition for $\cu$. 
 We also recall the definition of perfect and symmetric classes and relate them to Brown-Comenetz duality and torsion pairs. 
 This gives a funny (general) criterion for the existence of an adjacent torsion pair in terms of "symmetry" along with a "new" description of a $t$-structure that is right adjacent to a given compactly generated weight structure. 

 In \S\ref{spgtp} we study compactly generated torsion pairs and  perfectly generated weight structures; we prove several new and interesting results about them (and some of these statements were formulated above). In particular, we prove that for any "symmetric" $\cp,\cp'\subset \obj \cu$ there exist adjacent $t$ and $w$ (co)generated by them. This implies the existence of a weight structure right adjacent to a given compactly generated $t$; hence  the category $\hrt$ has an injective cogenerator and satisfies the AB3* axiom. Combining this fact with the results of \cite{humavit} we deduce that $\hrt$ is Grothendieck abelian whenever $t$ is non-degenerate.  
 The   results and arguments of  
the section are closely related to the properties of localizing subcategories of triangulated categories as studied by A. Neeman, H. Krause and other authors; see Remark \ref{rtst2}(\ref{it6}) below 
 for an "explanation" of this  similarity. 

In \S\ref{skan} we study dualities between triangulated categories and torsion pairs orthogonal with respect to them. Considering a "very simple" duality  we prove that for certain weight structures inside the bounded derived category of coherent sheaves on a scheme $X$ that is proper over $\spe R$ (where $R$ is a noetherian ring) there necessarily exist (right or left) orthogonal $t$-structures.

 Our main tools for constructing "more complicated" dualities  are Kan extensions of (co)homological functors from a triangulated subcategory  $\cuz$ to $\cu$; their properties are rather  interesting for themselves. We describe a duality between the homotopy category of filtered pro-objects for a stable proper Quillen model category $\gm$ with $\ho(\gm)$. 
	The properties of this duality imply that for any compactly generated $t$-structure on $\ho(\gm)$ its heart is  a Grothendieck abelian category; 
	 they 
	 are also applied in \cite{bgn} to the study of {\it motivic pro-spectra } and {\it generalized coniveau spectral sequences}.
	
	For the convenience of the reader we also make a list of the main definitions and notation used in this paper. Regular cardinals, Karoubi-closures and related matters, $D\perp E$, $\perpp D$, and $D\perpp$,  suspended, cosuspended, extension-closed, and strict classes of objects, extension-closures, envelopes, subcategories generated by classes of objects of triangulated categories,  homological and cohomological functors, compact objects, localizing and colocalizing subcategories, cogenerators, Hom-generators, compactly generated categories, cc, cp, and pp functors, and the Brown representability condition (along with its dual) are defined in \S\ref{snotata}; 
	(smashing, cosmashing, countably smashing,  adjacent, and compactly generated) torsion pairs (often denoted by $s=(\lo,\ro)$; sometimes we also use the notation $(L_sM,R_sM)$), generators for them, and $\cp$-null morphisms are introduced in \S\ref{shop}; $t$-structures ($t=(\cu^{t\le 0},\cu^{t\ge 0})$) and several 
	 of their "types" (including smashing and cosmashing $t$-structures), their hearts ($\hrt$), associated torsion pairs, $t$-homology ($H_0^t$), and $t$-projective objects are defined in \S\ref{sts}; weight structures ($w=(\cu_{w\le 0},\cu_{w\ge 0})$; we also define $\cu_{w=i}$ and $\cu_{[m,n]}$), their "types", hearts ($\hw$), and associated  (weighty) torsion pairs, adjacent weight and $t$-structures, $m$-weight decompositions,  negative subcategories, and the weight structure $\wstu$ are  introduced in   \S\ref{ssws}; (weight) Postnikov towers along with the corresponding filtrations and (weight) complexes, weakly homotopic  morphisms of complexes, and the category $\kw(\hw)$ are defined in \S\ref{sswc}; pure homological functors ($H^\ca$)  and purely $R$-representable homology (with values in $\psvr(\bu)$) is introduced in \S\ref{sdetect}, virtual $t$-truncations ($\tau^{\ge m}H$ and $\tau^{\le m}H$), weight range of functors, and pure cohomological functors ($H_\ca$) are defined in \S\ref{svtt};    functors of $R$-finite type and $R$-saturated categories are introduced in \S\ref{sadjt}; $\al$-small objects, countably perfect and perfect classes of objects, perfectly generated and well generated categories,  weakly symmetric and symmetric classes, and Brown-Comenetz duals of functors and objects are defined in \S\ref{scomp}; countable homotopy colimits $\inli Y_i$, strongly extension-closed classes and strong extension-closures,  (naive) big hulls, and zero classes of (collections of) functors are introduced in \S\ref{scoulim}; 
perfectly generated weight structures,  $\cp$-approximations, and contravariantly finite classes are defined  in \S\ref{sperfws}; weakly and strongly $\be$-well  generated  weight structures (and torsion pairs) are studied in \S\ref{swgws}; coextended functors and coextensions  are defined in \S\ref{scoext}; (nice) dualities and orthogonal structures are introduced in \S\ref{sdual1}; biextensions are defined in  \S\ref{sdual1}; stalk functors are introduced in \S\ref{sgengroth}; certain model structures on pro-objects are recalled in \S\ref{sprospectra}.

	The author is deeply grateful to prof. F. D\'eglise, prof. G.C. Modoi,  prof. Salorio M.J. Souto, and prof. J. \v{S}\v{t}ov\'\i\v{c}ek for their very useful comments.


\section{
On torsion pairs and   $t$-structures ("simple properties")}\label{sold}

This section is dedicated to the basics on torsion pairs and $t$-structures. 


In \S\ref{snotata} we introduce some 
notation 
and recall several important properties of triangulated categories (mostly from \cite{neebook}). 

In \S\ref{shop} we define and study torsion pairs (in the terminology of \cite{aiya}). Our results are rather easy; yet the author does not know any references for most of them. 

In \S\ref{sts} we recall some basics on $t$-structures and relate them to torsion pairs. We also study the notions of $t$-projective 
 objects for the purpose of using them in \S\ref{sadjw} and later.

  \subsection{Some categorical preliminaries}
	\label{snotata}
  
	When we will write $i\ge c$ or $i\le c$ (for some $c\in \z$) we will mean that $i$ is an integer satisfying this inequality.
	
	A cardinal $\al$ is said to be {\it regular} if it cannot be presented as a sum of less then $\al$ cardinals that are less than $\al$.

Most of the categories of this paper will be locally small. When considering a category that is not locally small we will usually say that it is (possibly) big; we will not need much of these categories. 
	
For categories $C,D$ we write 
$D\subset C$ if $D$ is a full 
subcategory of $C$.

Given a category $C$ and  $X,Y\in\obj C$ we  will write $C(X,Y)$ for  the set of morphisms from $X$ to $Y$ in $C$.
We will say that $X$ is  a {\it
retract} of $Y$ if $\id_X$ can be factored through $Y$.\footnote{Certainly,  if $C$ is triangulated 
then $X$ is a retract of $Y$ if and only if $X$ is its direct summand.}\

For a category $C$ the symbol $C^{op}$ will denote its opposite category.

For a subcategory $D\subset C$
we will say that $D$ is {\it Karoubi-closed} in $C$ if it
contains all retracts of its objects in $C$. We call the
smallest Karoubi-closed subcategory $\kar_C(D)$ of $C$ containing $D$  the {\it
Karoubi-closure} of $D$ in $C$. 

The {\it Karoubi envelope} $\kar(\bu)$ (no lower index) of an additive
category $\bu$ is the category of "formal images" of idempotents in $\bu$ (so $\bu$ is embedded into an idempotent complete category). 

$\ab$ is the category of abelian groups.


$\cu$, $\cu'$,  $\cuz$, $\du$, $\duz$,  and $\eu$ will always denote certain triangulated categories.
$\cu$ will often be endowed with a weight structure $w$; we  always assume that this is the case in those formulations where $w$ is mentioned without any explanations. 


For $f\in\cu (X,Y)$, where $X,Y\in\obj\cu$, we  call the third vertex
of (any) distinguished triangle $X\stackrel{f}{\to}Y\to Z$ a cone of
$f$.

We will often consider some representable and corepresentable functors and their restrictions. So for $\du'$ being a full triangulated subcategory of a triangulated category $\du$  and $M\in \obj \du$ we will often write $H^M: \du'\to \ab$ for the restriction of the corepresentable (homological) functor $\du(M,-)$ to $\du'$ (yet $H^P$ in Proposition \ref{porthop} will denote a certain {\it coextension});  $H_M: \du'\opp \to \ab$ is the restriction of the  functor $\du(-,M)$ to $\du'$. 
We will often be interested in the case $\du'=\du$ in this notation; we  assume that the domain of the functors $H^M$ and $H_M$ is the category $\cu$ if not specified otherwise.

For $X,Y\in \obj \cu$ we will write $X\perp Y$ if $\cu(X,Y)=\ns$.
For $D,E\subset \obj \cu$ we will write $D\perp E$ if $X\perp Y$
 for all $X\in D,\ Y\in E$.
For $D\subset\obj \cu$ we will write  $D^\perp$ for the class
$$\{Y\in \obj \cu:\ X\perp Y\ \forall X\in D\};$$
sometimes we will write $\perp_{\cu}$ instead to indicate the category that we are considering. Dually, ${}^\perp{}D$ is the class
$\{Y\in \obj \cu:\ Y\perp X\ \forall X\in D\}$. 

In this paper all complexes will be cohomological, i.e., the degree of all differentials is $+1$; respectively, we  use cohomological notation for their terms. 
 
 We will use the term {\it exact functor} for a functor of triangulated categories (i.e., for a  functor that preserves the structures of triangulated categories).

We will say that a class $\cp\subset \obj \cu$ is {\it suspended} if $\cp[1]\subset \cp$; $\cp$ is {\it cosuspended} if $\cp[-1]\subset \cp$.

A class $\cp\subset \obj \cu$ is said to be {\it extension-closed} if $0\in \cp$ and for any distinguished triangle $A\to B\to C$  in $\cu$ we have the implication $A,C\in
\cp\implies B\in \cp$. In particular, an extension-closed $\cp$ is {\it strict} (i.e., contains all objects of $\cu$ isomorphic to its elements).

The smallest extension-closed class $\cp'$ containing a given $\cp\subset \obj \cu$ will be called the {\it extension-closure} of $\cp$. The  smallest extension-closed Karoubi-closed  $\cp'\subset \obj \cu$ containing $\cp$ will be called the {\it envelope} of $\cp$.

We  call the smallest Karoubi-closed triangulated subcategory $\du$ of $\cu$ such that $\obj \du$ contains $\cp$ the {\it triangulated subcategory densely generated by} $\cp$; we will write  $\du=\lan \cp\ra_{\cu}$.

We will say that  $\cu$ {\it has coproducts} (resp. products, resp. countable coproducts) whenever it contains arbitrary small coproducts (resp. products, resp. countable coproducts) of families of its objects.

$\au$ will usually denote some abelian category; 
 the case $\au=\ab$ is the most important one for the purposes of this paper.

We will call a covariant additive functor $\cu\to \au$ for an abelian $\au$ {\it homological} if it converts distinguished
triangles into long exact sequences; homological functors $\cu^{op}\to \au$ will be called {\it cohomological} when considered
as contravariant functors from $\cu$ into $ \au$.

For additive categories $C,D$ 
the symbol $\adfu(C,D)$ will denote the (possibly, big)
category of additive functors from $C$ to $D$. Certainly, if $D$ is abelian then an $\adfu(C,D)$-complex  $X\to Y\to Z$ is exact in $Y$ whenever $X(P)\to Y(P)\to Y(P)$ is exact (in $D$) for any $P\in \obj C$. 


Below we will sometimes need some properties of the Bousfield localization setting (cf. \S9.1 of \cite{neebook}; most of these statements are contained in Propositions 1.5 and 1.6 of \cite{bondkaprserr}).

\begin{pr}\label{pbouloc}

 Let $F:\eu\to \cu$ be a full embedding of triangulated categories; assume that $\eu$ is  Karoubi-closed in $\cu$.  Denote by $\du$ the full subcategory of $\cu$ whose object class is $\obj \eu\perpp$; denote the embedding $\du\to \cu$ by $i$.

Then the following statements are valid.

I. A (left or right) adjoint to an exact functor is exact.

II.  $\du$ is a Karoubi-closed 
 triangulated subcategory of $\cu$.

III. 
Assume that  $F$  possesses a right adjoint $G$. 

\begin{enumerate}
\item\label{ibou1} Then for any $N\in \obj \cu$ there exists a distinguished triangle 
\begin{equation}\label{ebou}
N'\stackrel{f}{\to} N\stackrel{g}{\to} N''\to N'[1]
\end{equation}  with $N'\in   \obj \eu$ and $ N'' \in\obj \du$ (here we consider $\eu$ as a subcategory of $\cu$ via $F$), 
 and the triangle (\ref{ebou}) is unique up to a canonical isomorphism. 

\item\label{ibouort} $\perpp\obj \du=\obj \eu$.

\item\label{ibou2}  
 The functor $i$  possesses an (exact) 
 left adjoint $A$ and 
  the morphism $g$ in  (\ref{ebou})   is given by the unit of this adjunction. Moreover, this unit transformation yields an equivalence of  the Verdier localization of $\cu$ by $\eu$ (that is locally small in this case) to $\du$. 

\item\label{ibou3}
The morphism $f$ in  (\ref{ebou})  is given by the counit of the adjunction $F\dashv G$, and this counit gives an  
 equivalence $\eu\cong \cu/\du$.


\end{enumerate}

IV. A full embedding $F$ as above possesses a right adjoint if and only if  
for any $N\in \obj \cu$ there exists a distinguished triangle (\ref{ebou}).


\end{pr}
\begin{proof}
I. This is  Lemma 5.3.6 of \cite{neebook}. 

II. Obvious.

III.\ref{ibou1}. 
  The existence of (\ref{ebou}) is immediate from  Proposition 9.1.18 (that says that the Verdier localization $\cu\to \cu/\eu$ is a {\it Bousfield localization} functor in the sense of  Definition 9.1.1 of ibid.)  and Proposition 9.1.8 of ibid. (cf. also Proposition 1.5 of \cite{bondkaprserr}). Lastly, the essential uniqueness of  (\ref{ebou}) easily follows from  
\cite[Proposition 1.1.9]{bbd} (cf. also Proposition \ref{pbw}(\ref{icompl}) below). 

\ref{ibouort}. According to the aforementioned  Proposition 9.1.18 of  \cite{neebook}, we can deduce the assertion from   Corollary 9.1.14 of  ibid. 

 \ref{ibou2}.  The "calculation" of $g$ is given by Proposition 9.1.8 of  
 ibid. also. It remains to apply  Theorem 9.1.16 of  ibid.


 \ref{ibou3}. Corollary 9.1.14 of ibid. allows to deduce the assertion from the previous one.  

IV. The "only if" part of the assertion is given by assertion III.\ref{ibou1}. The converse implication is given by 
 Proposition 9.1.18 of ibid.
\end{proof}

Now we recall some terminology, notation, and  statements related to infinite coproducts and products in triangulated categories. 
 Some of these definitions and results may 
will be generalized in \S\ref{scomp} below.

All the coproducts and products in this paper will be small.
We will say that a subclass $\cp$ of $\obj \cu$ is {\it coproductive} (resp. {\it productive}) if it is closed with respect to all (small) coproducts (resp. products) that exist in $\cu$.

For  triangulated categories closed with respect to (all) coproducts or products 
we will just say that these categories have coproducts (resp. products).
\footnote{These are   axioms [TR5] and  [TR5*] 
 of \cite{neebook}, respectively.} 

We recall the following very useful statement.

\begin{pr}\label{pcoprtriang}
Assume that $\cu$ has  coproducts (resp. products, resp. countable coproducts). Then $\cu$ is Karoubian and all (small) coproducts  (resp. products, resp. countable coproducts) of distinguished triangles in $\cu$ are distinguished.
\end{pr}
\begin{proof}
The first  of the assertions is given by Proposition 1.6.8 of \cite{neebook}, and the second one   is 
is given by Remark 1.2.2  and Proposition 1.2.1 of ibid.
\end{proof}


\begin{defi}\label{dcomp}
Assume that $\cu$  has coproducts; $\cp\subset \obj \cu$.
\begin{enumerate}
\item\label{idcompa}
An object $M$ of $\cu$ is said to be {\it compact} if 
 the functor $H^M=\cu(M,-):\cu\to \ab$ respects coproducts.

\item\label{idloc}
 For $\du\subset \cu$ ($\du$ is a triangulated category that may be equal to $\cu$)
  one says that 
	$\cp$ generates $\du$ {\it as a localizing subcategory} of $\cu$ if  
$\du$ is the smallest full strict triangulated subcategory of $\cu$ that contains $\cp$ 
and is closed with respect to  $\cu$-coproducts.

If this is the case then we will also say that $\cp$ {\it cogenerates} $\du\opp$ as a {\it colocalizing} subcategory of $\cu\opp$.

\item\label{idhg} We will say that $\cp$ {\it Hom-generates} a full triangulated category $\du$ of $\cu$ containing $\cp$ 
if $\obj \du\cap(\cup_{i\in \z}\cp[i])\perpp=\ns$. 

\item\label{idcg} We will say that $\cu$ is {\it compactly generated} if it is Hom-generated by a set of compact objects.

\item\label{idcc} 
It will be convenient for us to use the following somewhat clumsy  terminology: a homological functor $H:\cu\to \au$ (where $\au$ is an abelian category) will be called a {\it cc} functor if it respects coproducts (i.e., the image of any coproduct in $\cu$ is the corresponding coproduct in $\au$); $H$ will be called a {\it wcc} functor if it respects countable coproducts.

 A cohomological functor  $H$ from $\cu$ into $\au$ will be called a {\it cp} functor if it converts all (small) coproducts into $\au$-products.

Dually, for a triangulated category $\du$ that 
has products we will call a homological functor $H:\du\to \au$   a {\it pp} functor if its respects products.

\item\label{idbrown} We will say that $\cu$  satisfies the  {\it Brown representability} property whenever any 
cp functor from $\cu$ into $\ab$ is representable. 

Dually, we will say that a triangulated category $\du$ satisfies the  {\it dual Brown representability} property if it has products and any pp functor  from $\du$ into $\ab$ is corepresentable (i.e., if $\du\opp$ satisfies the   Brown representability property).

\end{enumerate}
\end{defi}

\begin{pr}\label{pcomp}
Assume  that $\cu$ has coproducts.

 I. Let  $\cp\subset \obj \cu$.

 1. If $\cp$ 
 generates $\cu$  as its own localizing subcategory 
 then  it also Hom-generates $\cu$. 

Conversely, if $\cp$ Hom-generates $\cu$ and the embedding into $\cu$ of its localizing subcategory $\cu'$ generated by $\cp$ possesses a right adjoint then $\cu'=\cu$.

2. More generally, denote by $C$ the smallest coproductive extension-closed subclass  of $\obj \cu$ containing $\cp$. Then $\cp\perpp \cap C=\ns$.

 II.1. Assume that  $\cu$ is compactly generated.  
Then both $\cu$ and $\cu\opp$ satisfy the Brown representability  property.

2. Assume that  $\cu$ satisfies the Brown representability  property. Then it has products and any exact functor $F$ from $\cu$ (into a triangulated category $\du$) that respects coproducts possesses an exact right adjoint $G$. 


\end{pr}
\begin{proof}

I.1.  
The first part of the assertion is essentially a part of \cite[Proposition 8.4.1]{neebook} (note that the simple argument used for the proof of this implication does not require $\cp$ to be a set); it is also a particular case of assertion I.2.

Now assume that the embedding $\cu'\to \cu$ possesses a right adjoint. Then Proposition \ref{pbouloc}(III.\ref{ibou1}) implies that $\cu'=\cu$ whenever $\obj \cu'\perpp=\ns$ (note here that $\cu'$ is a strict subcategory of $\cu$). Lastly, $\obj \cu'\perpp$ is certainly zero if $\cp$ Hom-generates $\cu$ (since $\cp\subset \obj \cu'$). 

2. Similarly to the proof of loc. cit., if $N$ belongs to $\cp^{\perp}\cap C$ 
then the class $C\cap {}^{\perp}N$ contains $\cp$, coproductive and extension-closed. Hence $N\perp N$ and we obtain $N=0$.

II.1. The Brown representability property for $\cu$ is given by  Proposition 8.4.2 of  \cite{neebook}. The Brown representability for $\cu\opp$ is immediate from the combination of Theorem 8.6.1 with  Remark 6.4.5 of ibid.

2.  
The first part of the assertion is given by Proposition 8.4.6 of \cite{neebook}.
The second part is immediate from Theorem 8.4.4 of ibid. (combined with  Proposition \ref{pbouloc}(I)). 

\end{proof}

\subsection{On torsion pairs}\label{shop}

As we have already said, 
this paper is mostly dedicated to the study of weight structures and $t$-structures. Now, these notions have much in common; so we start from recalling an (essentially) more general definition of a torsion pair (in the terminology of \cite[Definition 1.4]{aiya};   in \cite[Definition 3.2]{postov} torsion pairs were called complete Hom-orthogonal pairs).\footnote{Another reason of passing to this more general notion is that some of our results are valid for arbitrary torsion pairs.}  

\begin{defi}\label{dhop}

 A couple $s$ of classes $\lo,\ro\subset\obj \cu$ (of $s$-left 
 and $s$-right  objects, respectively)  
will be said to be a {\it torsion pair} (for $\cu$) if $\lo^{\perp}=\ro$,  $\lo={}^{\perp}\ro$, and 
for any $M\in\obj \cu$ there
exists a distinguished triangle
\begin{equation}\label{swd}
L_sM\stackrel{a_M}{\to} M\stackrel{n_M}{\to} R_sM
{\to} L_sM[1]\end{equation} 
such that $L_sM\in \lo $ and $ R_sM\in \ro$. We will call any triangle of this form an {\it $s$-decomposition} of $M$; $a_M$ will be called an {\it $s$-decomposition morphism}.
\end{defi}

We will also need the following auxiliary definitions. 

\begin{defi}\label{dhopo}
Let $s=(\lo,\ro)$ be a torsion pair.

 1. We will say that   $s$ is {\it coproductive} (resp. {\it productive}) if  $\ro $ is coproductive (resp. $\lo$ is productive). We will also say that $s$ is  smashing (resp.  cosmashing) if $\cu$  in addition has coproducts (resp.  products).

We will also use the following modification of the smashing condition: we will say that $s$ is {\it countably smashing} whenever 
$\cu$ has 
 coproducts and $\ro$ is 
 closed with respect to countable $\cu$-coproducts.\footnote{In the current paper we are not interested in triangulated categories that only have countable coproducts; yet in \cite{bsnew} we  demonstrate this (weaker) assumption is rather useful also.}

2.  For another torsion pair $s'=(\lo',\ro')$ for $\cu$ we will say that $s$ is {\it left adjacent} to $s'$ or that $s'$ is {\it right adjacent} to $s$ if $\ro=\lo'$.

3. We will say (following \cite[Definition 3.1]{postov}) that $s$ is {\it generated by $\cp\subset \lo$} if $\cp^\perp=\ro$.\footnote{It suffices to assume that $\cp\subset \obj \cu$ and $\cp^\perp=\ro$ since then $\cp$ certainly lies in $\lo$. } 

We will say that $s$ is {\it compactly generated} if it is generated by some set of compact objects.



4. For $\cupr$ being a full triangulated subcategory of $\cu$ we will say that $s$ {\it restricts} to it whenever $(\lo\cap \obj \cupr,\ro\cap \obj \cupr)$ is a torsion pair for $\cu'$.

5. For a class $\cp\subset \obj \cu$ we will say that a $\cu$-morphism $h$ is {\it $\cp$-null} whenever for all $M\in \cp$ we have $H^M(h)=0$ (where  $H^M=\cu(M,-):\cu\to \ab$).\footnote{Note that the class of $\cp$-null morphisms is not necessarily shift-stable 
in contrast to the main examples of the paper  \cite{christ} where this notion was introduced.} 
\end{defi}

\begin{rema}\label{rgen}
1. If $\cp$ generates a  torsion pair $s$ then $\ro= \cp^\perp$ and $\lo ={}^{\perp}\ro$; thus $\cp$ determines $s$ uniquely. So we will say that $s$ is {\bf the}  torsion pair generated by $\cp$.

2. On the other hand,  
for a class $\cp\subset \obj \cu$   the corresponding couple $s=(\lo,\ro)$  (where $\ro= \cp^\perp$ and $\lo ={}^{\perp}\ro$)  certainly satisfies the orthogonality properties prescribed by Definition \ref{dhop}. Yet simple examples demonstrate that  the existence of $s$-decompositions can fail in general. 

In \cite[Definition 3.1]{postov} a couple that satisfies only the orthogonality  properties in  Definition \ref{dhop} was called  a {\it Hom-orthogonal pair} (in contrast to complete Hom-orthogonal pairs). The reader may easily note that several of our arguments below work for "general"  Hom-orthogonal pairs; yet the author has chosen not to treat this more general definition in the current paper. 

3. The object $M$ "rarely" determines its $s$-decomposition triangle (\ref{swd}) canonically (cf. Remark \ref{rstws}(1) below). Yet we will often need some choices of its ingredients; so we will use the notation of (\ref{swd}). 

4. Our definition of torsion pair actually follows \cite[Definition 3.2]{postov} and differs from Definition 1.4(i) of \cite{aiya}. However, Proposition \ref{phop}(\ref{itp9}) below yields immediately that these two definitions are equivalent.

As noted in \cite{postov}, some other authors use the term "torsion pair" to denote the couple $s$ associated with a $t$-structure (see Remark \ref{rtst1}(\ref{it1}) below). So, the term "complete Hom-orthogonal pair" would be less ambiguous; yet it does not fit well (linguistically) with the notion of $\Phi$-orthogonality that we will introduce below.  
 \end{rema}

We make some simple observations. 

\begin{pr}\label{phop}

Let $s=(\lo,\ro)$ be a torsion pair for $\cu$, $i, j\in \z$. Then the following statements are valid.

\begin{enumerate}
\item\label{itp1}
Both $\lo$ and $\ro$ are Karoubi-closed and extension-closed in $\cu$.

\item\label{itp2}
$\lo $ is coproductive and $\ro$ is productive.

\item\label{itp3}
If $s$ is coproductive (resp. productive) then the class $\ro[i]\cap \lo[j]$ is coproductive (resp. productive) also.

\item\label{itp4}
Assume that $\cu$ has coproducts. Then  $s$ is (countably) smashing if and only if the coproduct  of  any $s$-decompositions of $M_i\in \obj \cu$ gives an $s$-decomposition of $\coprod M_i$; here $i$ runs through any  (countable) index set.

Dually, if $\cu$ has products then  $s$ is  cosmashing if and only if the  product  of  any $s$-decompositions of  any small family of $M_i\in \obj \cu$ gives an $s$-decomposition of $\prod M_i$.

\item\label{itp5}
If $s$ is left adjacent to a torsion pair $s'$ (for $\cu$) then 
$s$ is coproductive and  $s'$ is productive.

\item\label{itp6}
$s^{op}=(\ro,\lo)$ is a  torsion pair for $\cu\opp$.

\item\label{itp7}
$s$-decompositions are "weakly functorial" in the following sense:  any $\cu$-morphism $g:M\to M'$ can be completed to a morphism between any choices of $s$-decompositions of $M$ and $M'$, respectively. 

\item\label{itp7p}
 If $M\in \lo$ and $M$ is a retract of $M'\in \obj \cu$ then  is also a retract of any choice of $L_sM'$ (see  Remark \ref{rgen}(3)).

\item\label{itp8}
A morphism $h\in \cu(M,N)$ is $\lo$-null if and only if factors through  an element of $\ro$. 

Moreover, the couple $(\lo,\lo-{\text{null}})$ (the latter is the class of $\lo$-null morphisms) is a projective class in the sense of \cite{christ} (see Remark \ref{rwsts}(3) below).

\item\label{itp9}
For $L,R\subset \obj \cu$ assume that $L\perp R$ and that for any $M\in \obj \cu$ there exists a distinguished triangle $l\to M\to r\to l[1]$  for $l\in L$ and $r\in R$. Then $(\kar_{\cu}(L),\kar_{\cu}(R))$ is a  torsion pair for $\cu$. 

\item\label{itp10}
 Assume that $(\lo',\ro')$ is a torsion pair 
 in a triangulated category $\cu'$; assume that $F:\cu\to \cu'$ is an exact functor that is essentially surjective on objects, and   such that $F(\lo)\subset \lo'$ and $F(\ro)\subset \ro'$. 
Then $(\lo',\ro')= (\kar_{\cu'}(F(\lo)),\kar_{\cu'}(F(\ro)))$.
\end{enumerate}

\end{pr}
\begin{proof}
\ref{itp1}--\ref{itp3}, \ref{itp5}, \ref{itp6}. Obvious.

\ref{itp4}. The "only if" implication 
follows immediately from assertion \ref{itp2} combined with  Proposition \ref{pcoprtriang} (along with its dual).

Conversely, assume that (countable) coproducts of $s$-decompositions are $s$-decompositions also. Since for any (countable) set of $R_i\in \ro$ the distinguished triangles $0\to R_i\to R_i\to 0$ are $s$-decompositions of $R_i$, we obtain $\coprod R_i\in\ro$; hence $s$ is (countably) smashing. To prove the remaining ("cosmashing") part of the assertion one can apply duality (along with assertion \ref{itp6}).

\ref{itp7}. According to \cite[Proposition 1.1.9]{bbd}, 
 it suffices to verify  the following: for any $s$-decomposition triangles (\ref{swd}) and  $L_sM'\to M'\to R_sM' {\to} L_sM'[1]$  the composition $L_sM\to M\to M'\to R_sM'$ vanishes. This certainly follows from $\lo\perp\ro$.


\ref{itp7p}. If $M$ is a retract of $M'$ then $\id_M$ factors through $M'$. Moreover, if $M\in \lo$ then the distinguished triangle $M\to M\to 0\to M[1]$ is an $s$-decomposition of $M$. Thus applying assertion \ref{itp7} twice we obtain a commutative diagram $$\begin{CD}
 M@>{g}>>L_sM'@>{h}>>M\\
@VV{\id_M}V@VV{a_{M'}}V@VV{\id_M}V \\
M@>{i}>>M'@>{p}>>M
\end{CD}$$ 
which yields that $h\circ g=\id_M$, i.e., that $M$ is a retract of $L_sM'$ indeed.

\ref{itp8}. Since $\lo\perp\ro$, any morphism that factors through $\ro$ 
is $\lo$-null.  Conversely, if $h:M\to N$ is $\lo$-null then for any choice of an $s$-decomposition of $M$ the composition $h\circ a_M$ is zero (see (\ref{swd}) for the notation). It certainly follows that $h$ factors through  $R_sM\in \ro$. 

Next, 
 for any $L\in \lo$ and $h\in \lo$-null we have $H^L(h)=0$ by definition (see \S\ref{snotata} or Definition \ref{dhopo}(5) for this notation). 
Arguing as in the proof \cite[Lemma 3.2]{christ}, 
 we easily obtain that (to prove the second part of the assertion) it remains to construct for $X\in \obj \cu$ a morphism $a:L\to X$ such that the cone morphism for it is a $\lo$-null one. According to the first part of  the assertion, for this purpose we can take $a$ to be a choice of $a_X$ (in the notation of (\ref{swd})).

\ref{itp9}.  Certainly, $\kar_{\cu}(L)\perp \kar_{\cu}(R)$.

 Assume that $M\in \perpp R$. Then in the corresponding distinguished triangle $l\to M\stackrel{f}{\to} r\to l[1]$ we have $f=0$; hence $l\in \kar_{\cu}(L)$ and $\perpp R \subset \kar_{\cu}(L)$. Dually (cf. assertion \ref{itp6}) if $M  \in \kar_{\cu}(R)$ then it is a retract of the corresponding $r$; thus $L^\perp\subset \kar_{\cu}(R)$. This concludes the proof.

\ref{itp10}. For any object $M'$ of $\cu'$ choose $M\in \obj \cu$ such that $F(M)\cong M'$; choose an $s$-decomposition $L_sM\to M \to R_sM\to L_sM[1]$ of $M$. Then 
we obtain a distinguished triangle $F(L_sM)\to M' \to F(R_sM)\to F(L_sM)[1]$. Next, the relation of $s$ to $s'$ implies that $F(\lo)\perp F(\ro)$. Hence  $(\kar_{\cu'}(F(\lo)),\kar_{\cu'}(F(\ro))$ is a torsion pair in $\cu'$.

Next, $\lo'$ and $\ro'$ are Karoubi-closed in $\cu'$; hence  $\kar_{\cu'}(F(\lo))\subset \lo'$ and $\kar_{\cu'}(F(\ro))\subset \ro'$. On the other hand, $$\kar_{\cu'}(F(\lo))=\perpp (\kar_{\cu'}(F(\ro)))\supset \perpp\ro'=\lo';$$ dualizing we obtain the remaining inclusion $\kar_{\cu'}(F(\ro))\supset \ro'$.
\end{proof}

\begin{rema}\label{rwsts} 1. We have to pay a certain price for uniting weight structures and $t$-structures in a single definition. The problem is that we have $\lo[1]\subset \lo$ for $t$-structures, whereas for weight structures the opposite inclusion is fulfilled (and we will see below that these inclusions actually characterize $t$-structures and weight structures). So, the left 
  class for a weight structure is actually "a right one with respect to shifts". Also, the definition of left and right adjacent (weight and $t$-) structures in \cite{bws} was "symmetric", i.e., $w$ being left adjacent to $t$ and $t$ being left adjacent to $w$ were synonyms; in contrast, our current convention follows Definition 3.10 of \cite{postov}. 
So, $w$ and $t$ being left adjacent in the sense introduced in the previous papers is equivalent to the torsion pair associated with $w$ (see Remark \ref{rwhop}(1) below) being left adjacent for the torsion pair  associated with $t$ "up to a shift"  (see Remark \ref{rstws}(4)).\footnote{Another way to deal with this discrepancy is to modify the definitions of $t$-structures and weight structures by the corresponding shifts; see Definition 2.1 of \cite{humavit}.}     

Lastly, we will study the {\it hearts} both for weight structures and $t$-structures. The corresponding definitions are very much similar; yet we are not able to give a single definition in terms of torsion pairs.


2. Certainly, 
part \ref{itp6}  of the proposition essentially (see part 1 of this remark) generalizes Proposition \ref{pbw}(\ref{idual}) below, whereas Proposition \ref{phop}(\ref{itp7}) is closely related to  Proposition \ref{pbw}(\ref{icompl}).

In particular, Proposition \ref{phop}(\ref{itp7}) easily implies that for any functor $H:\cu \to \au$ (for any abelian category $\au$) the correspondence $M\mapsto \imm(H(L_sM)\to H(M))$ for $M\in \obj \cu$ gives a well-defined subfunctor of $H$; cf. Proposition 2.1.2(1) of \cite{bws}. 
 Thus 
 $s$ yields a certain (two-step) filtration on any (co)homology theory defined on $\cu$. This is a certain generalization of the weight filtration defined in ibid. (cf. Remark \ref{rwhop} below). 
 The main distinction is that we don't know (in general) the relation between $s$ and the "shifted" torsion pairs $s[i]=(\lo[i],\ro[i])$; thus there appears to be no reasonable way to obtain  "longer" filtrations using this observation.

3. Recall (see Proposition 2.6 of \cite{christ}; cf. also \cite{modoi}) that a projective class in a triangulated category $\cu$ is a couple $(\cp,I)$ for $\cp\subset \obj \cu$ and $I$ being the class of $\cp$-null morphisms that satisfy the following additional conditions: $\cp$ is the largest class such that all elements of $I$ are $\cp$-null and for any $M\in \obj \cu$ there exists a distinguished triangle $L\to M\stackrel{n}{\to} R\to L[1]$ such that $L\in \cp$ and $n\in I$.

The author has proved the relation of torsion pairs to projective classes for the purposes of applying it in \S\ref{swgws} 
  below. He was not able to get anything useful from this relation yet; so the reader may ignore projective classes in 
 this text. Note however that knowing the notion of a projective class is necessary to trace the (close) relation of our proof of Theorem \ref{tpgws} to Lemma 2.2 of \cite{modoi}. 

\end{rema}

We also prove some simple statements on torsion pairs 
 in categories that have coproducts. 

\begin{pr}\label{phopft}
Let $s=(\lo,\ro)$ be a torsion pair  generated by some $\cp\subset \obj \cu$ in a triangulated category $\cu$.
Then the following statements are valid.


I.  Assume that $\cu$   has coproducts and $\cp$ is a set of compact objects.  Denote by $\du$ the localizing  subcategory of $\cu$ generated by $\cp$; $E=(\cup_{i\in \z}\cp[i])^\perp$.

1. $\lo\subset \obj \du$, whereas $\ro$ is precisely the class of extensions of  elements of $E$ by that of $D=\obj \du\cap \ro$.

2.  $(\lo,D)$ is a  torsion pair for $\du$.

II. Let $F:\cu\to \cu'$ be 
an exact functor that possesses a right adjoint $G$. 

1. Assume in addition that $F$ is a full embedding. Then $s'=(\lo', \ro')$ is the torsion pair generated by $F(\cp)$ in $\cu'$, where $\ro'$ is the class of extensions of elements of $F(\obj \cu)^{\perp_{\cu'}}$ 
by that of $F(\ro)$, and $\lo'$ is the closure of $F(\lo)$ with respect to $\cu'$-isomorphisms.  

2. Let  $s'=(\lo', \ro')$ be an arbitrary torsion pair in $\cu'$. Then the following conditions are equivalent:

a). $F(\lo)\subset \lo'$; 

b).  $F(\cp)\subset \lo'$; 
 
c).  $G(\ro')\subset \ro$. 

3. For $s'$ as in the previous assertion assume in addition that $F(\lo)\subset \lo'$ and $F(\ro)\subset \ro'$, and $F$ is essentially surjective on objects. Then $s'$ is generated by $F(\cp)$.

III. Assume that  $\cu$   has coproducts and  all elements of $\cp$ are compact in it. Then $s$ is smashing. 
\end{pr}
\begin{proof}

I.1. Certainly, any extension of an element of $E$ by an object of  $\du$ belongs to $\cp^\perp=\ro$.

Next,  Proposition \ref{pcomp}(II.2) gives the existence of  an exact right adjoint $G$ to the embedding of $\cupr$ into $\cu$.
 Moreover, any object of $\cu$ is an extension of an 
element of $\cu$ by an object of $\du$, and $G$ is  equivalent to the localization of $\cu$ by the full triangulated subcategory $\eu$ whose object class is $E$ according to Proposition \ref{pbouloc}(III.\ref{ibou1},\ref{ibou2}). Hence $\lo\subset \obj \du$ and we obtain that any element of $\ro$ can be presented as an extension of the aforementioned form.

2. Obviously, $D=\lo^{\perp_{\cupr}}$. Next, the presentation of elements of $\ro$ by extensions as above yields that $\lo={}^{\perp_{\cu'}} D$.

It remains to verify the existence of the corresponding decompositions. For $M\in \obj \cupr$ we apply $G$ to (any) its $s$-decomposition; this is easily seen to yield a decomposition of $G(M)\cong M$ with respect to the couple $(\lo,D)$.

II.1. We can certainly assume that $\cu$ is a (full) strict subcategory of $\cu'$. We apply Proposition \ref{phop}(\ref{itp9}).

Obviously, $\lo'\perp\ro'$. Since $\cu$ is Karoubi-closed in $\cu'$,
  $\lo'$ is Karoubi-closed in $\cu'$. Next, for any $N\in \obj \cu'$ there exists a essentially unique distinguished triangle 
\begin{equation}\label{eglud}
N'\to N \to N''\to N'[1]
\end{equation}
with $N'=G(N)\in\obj \cu$ and $N''\in  \cu^{\perp_{\cu'}}$; see Proposition \ref{pbouloc}(III,\ref{ibou1}).
Since $\lo\perp N''$, we obtain $\lo^\perp=\ro'$.

So it remains to verify that any $N$ (as above) possesses an $s'$-decomposition. We choose an $s$-decomposition $L_sN'\to N'\to R_sN'\to L_sN'[1]$ of $N'$. Applying the octahedral axiom to this distinguished triangle along with (\ref{eglud}) we obtain a presentation of $N$ as an extension of $R$ by $L_sN'$, where $R$ is some extension of $N''$ by $R_sN'$. Thus we obtain an  $s'$-decomposition of $N$.

2. Certainly, 
c) is equivalent both to $\cp\perp G(\ro')$ and to $\lo \perp G(\ro')$. Applying the adjunction $F \dashv G$ we obtain the equivalences in question.

3. 
For $M'\in \obj \cu$ we obviously have the following chain of equivalences:
$$F(\cp)\perp M'\iff  \cp\perp G(M') \iff G(M')\in \ro\iff \lo \perp G(M')\iff F(\lo')\perp M'.$$
Now, Proposition \ref{phop}(\ref{itp10}) implies immediately that $F(\lo)\perpp =\ro'$; hence $F(\cp)\perpp=\ro'$ indeed.

III. Obvious; cf. also Proposition \ref{psym}(\ref{isymcomp}) below.

\end{proof}

\subsection{$t$-structures: recollection, relation to torsion pairs, and $t$-projectives}\label{sts}

Now we pass to $t$-structures. 
Certainly, one can easily define them in terms of torsion pairs; still we give the "classical" definition of a $t$-structure here  for fixing the notation and for recalling its relation to Definition \ref{dhop} (explicitly).

\begin{defi}\label{dtstr}


A couple of subclasses  $\cu^{t\ge 0},\cu^{t\le 0}\subset\obj \cu$ will be said to be a
$t$-structure $t$ on $\cu$  if 
they satisfy the following conditions:

(i) $\cu^{t\ge 0},\cu^{t\le 0}$ are strict, i.e., contain all
objects of $\cu$ isomorphic to their elements.

(ii) $\cu^{t\ge 0}\subset \cu^{t\ge 0}[1]$, $\cu^{t\le
0}[1]\subset \cu^{t\le 0}$.

(iii)  $\cu^{t\le 0}[1]\perp \cu^{t\ge 0}$.

(iv) For any $M\in\obj \cu$ there exists a  {\it $t$-decomposition} distinguished triangle
\begin{equation}\label{tdec}
L_tM\to M\to R_tM{\to} L_tM[1]
\end{equation} such that $L_tM\in \cu^{t\le 0}, R_tM\in \cu^{t\ge 0}[-1]$.
\end{defi}

We also need the following auxiliary definitions.

\begin{defi}\label{dtsto}
Let 
$n\in \z$; let $t$ be a $t$-structure on $\cu$. 

\begin{enumerate}
\item\label{ito1}
$\cu^{t\ge n}$ (resp. $\cu^{t\le n}$) will denote $\cu^{t\ge
0}[-n]$ (resp. $\cu^{t\le 0}[-n]$);  $\cu^{t=0}=\cu^{t\ge 0}\cap \cu^{t\le 0}$. 

\item\label{ito2} $\hrt$ will be the full subcategory of $\cu$ whose object class is $\cu^{t=0}$.

\item\label{ito3} We will say that $t$ is 
{\it right non-degenerate} if $\cap_{i\in \z}\cu^{t\ge i} =\ns$.

We will say that $t$ is (just) {\it non-degenerate} if we also have $\cap_{i\in \z}\cu^{t\le i} =\ns$.

\item\label{ito4} We will say that $t$ is {\it bounded above} if $\cu=\cup_{i\in \z}\cu^{t\le i}$.

\item\label{itoe} Let   $\cu'$ be a triangulated category endowed with a $t$-structure $t'$. We will say that an exact functor $F:\cu\to \cu'$ is 
{\it $t$-exact} whenever $F(\cu^{t\le 0})\subset \cu'^{t'\le 0}$  and $F(\cu^{t\ge 0})\subset \cu'^{t'\ge 0}$.
\end{enumerate}

\end{defi}

\begin{rema}\label{rtst1}
\begin{enumerate}
\item\label{it1}
Recall  that $\cu^{t\le n}={}^{\perp}(\cu^{t\ge n+1})$ and $\cu^{t\ge n}=\cu^{t\le n-1}{}^{\perp}$ (for $t$ and $n$ as above). 
   Thus for $\lo=\cu^{t\le 0}$ and $\ro=\cu^{t\ge 1}$ the couple $s=(\lo,\ro)$ is a torsion pair for $\cu$ that we will call the torsion pair associated with $t$.  Conversely, if for a torsion pair $s$ we have $\lo[1]\subset \lo$ then $(\lo,\ro[1])$ is a $t$-structure (that we will say to be associated with $s$).
 For a class $\cp\subset \obj \cu$ we will say that $\cp$ generates $t$ whenever it generates the associated $s$ (certainly, then we have $\cp\subset \lo =\cu^{t\le 0}$); we will say that $t$ is compactly generated whenever $s$ is. 

We will say that $t$ is coproductive (resp., productive; resp., smashing or cosmashing) whenever $s$ is; we will say that $t$ restricts to $\cupr$ (see Definition \ref{dhopo}(4)) whenever $s$ does.

\item\label{itd} It is well known (and also follows from the previous part of this remark along with Proposition \ref{phop}(\ref{itp6})) that $(\lo,\ro[1])$ is a $t$-structure on $\cu$ if and only if $(\ro, \lo[-1])$ gives a $t$-structure on $\cu\opp$.

\item\label{it3}
 Recall that the triangle (\ref{tdec}) is canonically (and functorially) determined by $M$. So for $M'=M[n]$ and  $L_tM'\to M'\to R_tM'{\to} L_tM'[1]$ being its $t$-decomposition we will write $t^{\le n}M$ for $L_tM'[-n]$  (and this notation is $\cu$-functorial). 
Moreover, the functor $t^{\le n}: \cu\to \cu^{t\le n}$ (considered as a full subcategory of $\cu$) is  right adjoint to the embedding $ \cu^{t\le n}\to \cu$. Dually,  we obtain a functor $t^{\ge n+1}:M\mapsto R_tM'[-n-1]$, and $t^{\ge n+1}$ is left adjoint to  the embedding $ \cu^{t\ge n+1}\to \cu$.


It certainly follows that any $t$-exact functor "commutes with these adjoints". In particular, this fact may be applied in the case where $\cu^{t\le 0}$ is the class of objects of a full triangulated subcategory $\eu$ of $\cu$ (cf. Proposition \ref{prtst}(\ref{itp4}) below); in this case the triangle (\ref{tdec}) is just the 
  triangle (\ref{ebou}).


\item\label{it4}
Recall that $\hrt$ is necessarily an abelian category with short exact sequences corresponding to distinguished triangles in $\cu$.

Moreover, have a canonical isomorphism of functors $t^{\le 0}\circ t^{\ge 0}=t^{\ge 0}\circ t^{\le 0}$ (if we consider these functors as endofunctors of $\cu$). This composite functor $H_0^t$ actually takes values in $\hrt\subset \cu$, and it is homological if considered this way. 

\end{enumerate}
\end{rema}

We prove a few 
properties of $t$-structures related to these observations.

\begin{pr}\label{prtst}
Let $t$ be a $t$-structure on $\cu$.
Then the following statements are valid.

\begin{enumerate}
\item\label{it2}
Assume that $t$ is smashing. Then  $\hrt$ is an AB4 category that is closed with respect to $\cu$-coproducts, and
$H_0^t$ is a cc functor (see  Definition \ref{dcomp}(\ref{idcc})). 

\item\label{it2d}
Dually, if $t$ is cosmashing then  $\hrt$ is an AB4* category, the embedding $\hrt\to \cu$ respects products, and $H_0^t$ is a pp functor.

\item\label{itt4} Those strict triangulated subcategories $L\subset \cu$ that possess a right adjoint $G$ to the embedding functor $L\to \cu$ are in one-to-one correspondence with torsion pairs $(\lo,\ro)$ such that $\lo[1]=\lo$; the correspondence sends $L$ into $(\obj L,\obj L\perpp)$.

Moreover, any $(\lo,\ro)$ of this sort is a $t$-structure.\footnote{So, we don't have to "shift" $\ro$ as we did in Remark \ref{rtst1}(\ref{it1}).}


\item\label{it4sm} Assume that $\cu$ has coproducts. Then  smashing torsion pairs  $(\lo,\ro)$ such that $\lo[1]=\lo$ are  in one-to-one correspondence with those exact embeddings $L\to \cu$ such that the corresponding $G$ (exists and) respects coproducts.

\item\label{ittadj} Assume that $\cu$ is compactly generated; let $F:\cu\to \cu'$ be a full exact embedding respecting coproducts. 
Then the couple $(\cu'^{t'\le 0}, \cu'^{t'\ge 0})$ 
coincides with the $t$-structure generated by $F(\cu^{t\le 0})$ in $\cu'$, where $\cu'^{t'\ge 0}$  is the class of extensions of elements of $F(\obj \cu)^{\perp_{\cu'}}$ 
by that of $F(\cu^{t\ge 0})$, and $\cu'^{t'\le 0}$ is the closure of $F(\cu^{t\le 0})$ with respect to $\cu'$-isomorphisms.  

\end{enumerate}
\end{pr}
\begin{proof}
\ref{it2}.  $\hrt$ is closed with respect to $\cu$-coproducts according to Proposition \ref{phop}(\ref{itp3}).  Thus  $\hrt$ is an AB4 category according to Proposition \ref{pcoprtriang} combined with Remark \ref{rtst1}(\ref{it4}). Next,  the endofunctors $t^{\le 0} $ and $ t^{\ge 0}$ of $\cu$  respect coproducts according to Proposition \ref{phop}(\ref{itp4}); hence their composition also does. 

\ref{it2d}. This is just the categorical dual to the previous assertion (see Remark \ref{rtst1}(\ref{itd}).

\ref{itt4}. If  $G$ exists then  $(\obj L,\obj L\perpp)$ is easily seen to be a torsion pair; combine 
Proposition \ref{pbouloc}(III.\ref{ibou1}, II) with Proposition \ref{phop}(\ref{itp9}). Moreover, we certainly have $\obj L[1]=\obj L$. 

Conversely, if $\lo=\lo[1]$ then the corresponding $L$ is triangulated since it is extension-closed (see Proposition \ref{phop}(\ref{itp1})). Next, $(\lo,\ro)$ is a $t$-structure according to Remark \ref{rtst1}(\ref{it1}). Hence $G$ exists according to part \ref{it3} of the remark.


\ref{it4sm}. We should check which torsion pairs $s=(\lo,\ro)$ with $\lo=\lo[1]$ are smashing. 
According to Proposition \ref{phop}(\ref{itp4}) we should check whether coproducts  of $s$-decompositions are $s$-decompositions.

Since $s$ is a $t$-structure,  $t$-decompositions are canonical (see Remark \ref{rtst1}(\ref{it1}) hence we should check when the endofunctor $t^{\le 0}$ respects coproducts.
Now, the embedding $i:L\to \cu$ respects coproducts since it possesses a right adjoint; thus $ G$ respects coproducts if and only if $i\circ G=-^{t\le 0}$ does. 


\ref{ittadj}. Since $\cu$ is compactly generated and $F$ respects coproducts, it possesses a right adjoint (see Proposition 
\ref{pcomp}(II)). Hence it remains to apply Proposition \ref{phopft}(II.1) (we take $\cp=\cu^{t\le 0}$ in it and 
 invoke Remark \ref{rtst1}(\ref{itd})).
\end{proof}

\begin{rema}\label{rtst2}
\begin{enumerate}
\item\label{ismashs}
If  $\cu$ has coproducts and the embedding $i:L\to \cu$  
 possesses a right adjoint respecting coproducts then $L$ is  called a  {\it smashing subcategory} of $\cu$; see \cite{kellerema}.
Moreover,   these conditions are equivalent to the perfectness of the class $\obj L$ in $\cu$ (see Definition \ref{dwg}(\ref{idpc}) below) according to Proposition \ref{psym}(\ref{iperftp}).

\item\label{it5s1}
 Both $t$-structures and weight structures are essentially particular cases of torsion pairs corresponding to the cases
$\lo[1]\subset \lo$ and $\lo\subset \lo[1]$, respectively (see Remark  \ref{rtst1}(\ref{it1}) and Remark \ref{rwhop}(1) below). 
So, the "shift-stable" torsion pairs described in Proposition  \ref{prtst}(\ref{itt4})  yield the "intersection" of these cases. 

\item\label{it6} So it is no wonder that the results and arguments of \S\ref{sperfws} below are closely related to the properties of localizing subcategories of triangulated categories as studied by A. Neeman, H. Krause and others.

\item\label{ialsmash} Proposition \ref{prtst}(\ref{it4sm}) can easily be generalized as follows: if $\al$ is a regular cardinal (see \S\ref{snotata}) and $\cu$ is closed with respect to coproducts of less than $\al$ objects then for an exact embedding $L\to \cu$ possessing a right adjoint $G$ the functor $G$ respects coproducts of less than $\al$ objects if and only if the class $\obj L^{\perp_{\cu}}$ is  closed with respect to $\cu$-coproducts of less than $\al$ elements.



\end{enumerate}

\end{rema}

Now we introduce $t$-projective 
 objects (essentially following \cite{zvon}).

\begin{defi}\label{dpt}
Let $t$ be a $t$-structure on $\cu$. 

Then we will write $P_t$ for the class ${}^\perp (\du^{t\ge 1}\cup \du^{t\le -1})$; we will say that its elements are {\it $t$-projective}.\footnote{The class $P_t$ was called the {\it coheart} of $t$ in \S3 of \cite{zvon}.} 

\end{defi}

We prove some simple statements relating  $t$-projectives 
 to exact functors from $\hrt$ into abelian groups.

\begin{pr}\label{pgen}
Let $t$ be a $t$-structure  for $\cu$. 

Then the following statements are valid.

\begin{enumerate}
\item\label{ipgen1} Let $N\in \cu^{t\le 0}$. Then we have $t^{\ge 0}N\cong H_0^t(N)$, and the object $H_0^t(N)$ corepresents the restriction of the functor $H^N=\cu(N,-)$ to $\hrt$.

\item\label{ipgen2} $P_t\subset \cu^{t\le 0}$, and   for any $P\in P_t$ we have  natural isomorphisms of functors \begin{equation}\label{ept} 
\cu(P,-)\cong \cu(P,H_0^t(-))\cong \hrt (H_0^t(P), H_0^t(-));\end{equation}
the first of them is induced by the transformations $ \id_{\cu}\to t^{\ge 0}$ and $H_0^t\to t^{\ge 0}$.

 \item\label{ipgen25} For any $P\in P_t$ the object $H_0^t(P)$ is projective in $\hrt$. 


\item\label{ipgen4} Assume that $\cu$ (has products and) satisfies the dual Brown representability condition; assume that $t$ is cosmashing. Then $H_0^t$ gives an equivalence of (the full subcategory of $\cu$ given by) $P_t$ with the subcategory of projective objects of $\hrt$.

\end{enumerate}

\end{pr}
\begin{proof}

\ref{ipgen1}. The first part of the assertion is just a particular case of the definition of $H_0^t(-)$ (see Remark \ref{rtst1}(\ref{it4})). 

The second part is very easy also (cf. the proof of  \cite[Lemma 2(1)]{zvon}); just apply the fact that the functor $t^{\ge 0}$ is left adjoint to  the embedding $ \cu^{t\ge 0}\to \cu$ (see Remark \ref{rtst1}(\ref{it3})).

\ref{ipgen1}. The first part of the assertion is immediate from Remark \ref{rtst1}(\ref{it1}) (since it gives $\cu^{t\le 0}= \perpp\cu^{t\ge 1}$). The first isomorphism in (\ref{ept}) follows easily from the definitions of $P_t$ and $H^t_0$, whereas the second one is given by assertion  \ref{ipgen1}.

 \ref{ipgen25}. 
The statement is given by Lemma 2(1) of \cite{zvon} (and also easily follows from assertion \ref{ipgen1}).

\ref{ipgen4}. We should prove that any projective object $P_0$ of $\hrt$ "lifts" to $P_t$. Consider the functor $H^{P_0}=\hrt(P_0,-)\circ H_0^t:\cu\to \ab$. 
Since  $H^t_0$ is a (homological) pp functor according to (the dual to) Proposition \ref{prtst}(\ref{it2}), $H^{P_0}$  is a pp functor also. Hence it is corepresentable by some $P\in \obj \cu$ that obviously belongs to $P_t$. It remains to apply assertions \ref{ipgen1} and \ref{ipgen2}  to prove that $H_0^t(P)\cong P_0$.

\end{proof}

\begin{rema}\label{rcompgen}
Now assume that $t$ is generated by the class $\cup_{i\ge 0}\cp[i]$ for some $\cp\subset \obj\cu$  (i.e.,  $\cu^{t\ge 0}=\cap_{i\ge 1} \cp\perpp[i]$; it certainly follows that $\cp\subset \cu^{t\le 0}$). Then part \ref{ipgen1} of our proposition implies that $\cu^{t=0}\cap (\{H_0^t(P):\ P\in \cp\}\perpp)=\ns$, i.e., this class Hom-generates $\hrt$. Indeed,  if $M$ is a non-zero element of $\cu^{t=0}$ then $M\notin \cu^{t\ge 1}$. Hence there exists $P\in \cp$ such that $P\not\perp M$; thus $H_0^t(P) \not\perp M$ also.

 Moreover, if  $\cp$ is a set and $t$ is smashing then the $\cu$-coproduct $\coprod_{P\in \cp} H_0^t(P)$ Hom-generates $\hrt$ also (see Proposition \ref{phop}(\ref{itp3})). 

\end{rema}

\section{Weight structures: reminder and pure functors}\label{sws}
In \S\ref{ssws} we  recall some basics on weight structures and relate them to torsion pairs (and to the results of \S\ref{shop}).

In \S\ref{sswc} we recall some 
properties of the weight complex functors. 
Our treatment of this subject is "more accurate" than the original one in \cite{bws}.


In \S\ref{sdetect} we construct {\it pure} (cf. Remark \ref{rwrange}(5))  homological functors from $\cu$ starting from additive functors from $\hw$ into  abelian categories. We also study conditions ensuring that a functor of this sort "detects weights of objects". The results of this section are important for the study of Picard groups of triangulated categories (endowed with weight structures) carried over in \cite{bontabu}.

In \S\ref{svtt} we recall the notion of virtual $t$-truncations of (co)homological functors from $\cu$ and relate them to functors of limited {\it weight range}. 

In \S\ref{sprcoprod} we prove that  weight decompositions and weight complexes "respect coproducts" whenever $w$ is smashing; it follows that virtual $t$-truncations of cc and cp functors are cc and cp functors, respectively (in this case). 
We also study generators (see Remark \ref{rgenw}) for $\hw$. 

\subsection{
Weight structures: basics }\label{ssws}

Let us recall the definition of one of the main notions of this paper.

\begin{defi}\label{dwstr}

 A couple of subclasses $\cu_{w\le 0},\cu_{w\ge 0}\subset\obj \cu$ 
will be said to define a {\it weight structure} $w$ for a triangulated category  $\cu$ if 
they  satisfy the following conditions.

(i) $\cu_{w\le 0}$ and $\cu_{w\ge 0}$ are 
Karoubi-closed in $\cu$
(i.e., contain all $\cu$-retracts of their elements).\footnote{In \cite{bws} the axioms of a weight structure also 
required $\cu_{w\le 0}$ and $\cu_{w\ge 0}$ to be additive. Yet this  is not necessary; see Remark 1.2.3(4) of \cite{bonspkar}.}\

(ii) {\bf Semi-invariance with respect to translations.}

$\cu_{w\le 0}\subset \cu_{w\le 0}[1]$ and $\cu_{w\ge 0}[1]\subset
\cu_{w\ge 0}$.

(iii) {\bf Orthogonality.}

$\cu_{w\le 0}\perp \cu_{w\ge 0}[1]$.

(iv) {\bf Weight decompositions}.

 For any $M\in\obj \cu$ there
exists a distinguished triangle
$LM\to M\to RM 
{\to} LM[1]$
such that $LM\in \cu_{w\le 0} $ and $ RM\in \cu_{w\ge 0}[1]$.
\end{defi}

We will also need the following definitions.

\begin{defi}\label{dwso}

Let $i,j\in \z$.

\begin{enumerate}
\item\label{id1} The full subcategory  $\hw\subset \cu$ whose object class is
$\cu_{w=0}=\cu_{w\ge 0}\cap \cu_{w\le 0}$ 
 is called the {\it heart} of 
$w$.

\item\label{id2} $\cu_{w\ge i}$ (resp. $\cu_{w\le i}$, 
$\cu_{w= i}$) will denote $\cu_{w\ge
0}[i]$ (resp. $\cu_{w\le 0}[i]$,  $\cu_{w= 0}[i]$).

\item\label{id3} $\cu_{[i,j]}$  denotes $\cu_{w\ge i}\cap \cu_{w\le j}$; so, this class  equals $\ns$ if $i>j$.

$\cu^b\subset \cu$ will be the category whose object class is $\cup_{i,j\in \z}\cu_{[i,j]}$.

\item\label{id4}  We will  say that $(\cu,w)$ is {\it  bounded}  if $\cu^b=\cu$ (i.e., if
$\cup_{i\in \z} \cu_{w\le i}=\obj \cu=\cup_{i\in \z} \cu_{w\ge i}$).

Respectively, we will call $\cup_{i\in \z} \cu_{w\le i}$ (resp. $\cup_{i\in \z}\cu_{w\ge
i}$) the class of $w$-{\it bounded above} (resp. $w$-{\it bounded below}) objects; we will say that $w$ is bounded above (resp. bounded below) if all the objects of $\cu$ satisfy this property.

\item\label{ideg} 
We will call the elements of $\cap_{i\in \z}\cu_{w\le i}$ (resp. of  $\cap_{i\in \z}\cu_{w\ge i}$) {\it  right degenerate} 
(resp. {\it left degenerate}).

Respectively, we will say that $w$ is {\it non-degenerate} if $\cap_{i\in \z}\cu_{w\le i}=\cap_{i\in \z}\cu_{w\ge i}=\ns$ (i.e., if all degenerate objects of $\cu$ are trivial). We will say that $w$ is {\it  right non-degenerate} (resp. {\it  left non-degenerate}) if $\bigcap\limits_{i\in \z} \cu_{w \le i} = \{0\}$ (resp.  $\bigcap\limits_{i\in \z} \cu_{w \ge i} = \{0\}$). 

\item\label{idadj} If $t$ is a $t$-structure on $\cu$ then we will say that $w$ is left adjacent to $t$ or that $t$ is right adjacent to $w$ if $\cu_{w\ge 0}=\cu^{t\le 0}$.
Dually, $w$ is right adjacent to $t$ whenever $\cu_{w\le 0}=\cu^{t\ge 0}$.

\item\label{ineg}
An additive 
subcategory $D\subset \obj \cu$ will be called {\it negative} if for any $i>0$ we have $\obj D\perp \obj D[i]$.

\item\label{idwe}
Let  $\cu'$ be a triangulated category endowed with a weight structure $w'$; 
let $F:\cu\to \cu'$ be an exact functor.

We will say that $F$ is {\it left weight-exact} (with respect to $w,w'$) if it maps $\cu_{w\le 0}$ to $\cu'_{w'\le 0}$; it will be called {\it right weight-exact} if it maps $\cu_{w\ge 0}$ to $\cu'_{w'\ge 0}$. $F$ is called {\it weight-exact} if it is both left and right weight-exact.

\end{enumerate}
\end{defi}

\begin{rema}\label{rstws}

1. 
Similarly to Remark \ref{rgen}(3), we will sometimes need a choice of a weight decomposition of $M[-m]$ shifted by $[m]$. 
 So we take a distinguished triangle \begin{equation}\label{ewd} w_{\le m}M\to M\to w_{\ge m+1}M \end{equation} 
with some $ w_{\ge m+1}M\in \cu_{w\ge m+1}$, $ w_{\le m}M\in \cu_{w\le m}$; we will call it an {\it $m$-weight decomposition} of $M$, and call  arbitrary choices of $w_{\ge m+1}M$ and $ w_{\le m}M$ {\it weight truncations of $M$} (for all $m\in \z$).  

 We will 
  use this notation below (though $w_{\ge m+1}M$ and $ w_{\le m}M$ are not canonically determined by $M$). Moreover, when we will write arrows of the type $w_{\le m}M\to M$ or $M\to w_{\ge m+1}M$ we will always assume that they come from some $m$-weight decomposition of $M$. 

2. A  simple (and still  useful) example of a weight structure comes from the stupid
filtration on the homotopy categories of cohomological complexes
$K(B)$ for an arbitrary additive  $B$. 
In this case
$K(B)_{\wstu\le 0}$ (resp. $K(B)_{\wstu\ge 0}$) will be the class of complexes that are
homotopy equivalent to complexes
 concentrated in degrees $\ge 0$ (resp. $\le 0$); see Remark 1.2.3(1) of \cite{bonspkar} for more detail.  The heart of this weight structure 
is the Karoubi-closure  of $B$ in $K(B)$ (that is actually equivalent to $\kar(B)$). 

3. In the current paper we use the "homological convention" for weight structures; 
it was previously used in  
 \cite{wildcons}, \cite{wildab}, \cite{brelmot}, and succeeding papers of the author, 
  whereas in 
\cite{bws}, \cite{bger}, \cite{bach}, and \cite{bontabu}  the "cohomological convention" was used.\footnote{Recall also that 
D. Pauksztello has introduced weight structures independently (see \cite{paucomp}); he called them
co-t-structures. }\ In the latter convention 
the roles of $\cu_{w\le 0}$ and $\cu_{w\ge 0}$ are interchanged, i.e., one considers   $\cu^{w\le 0}=\cu_{w\ge 0}$ and $\cu^{w\ge 0}=\cu_{w\le 0}$. 
  
 
 Note also that in this paper we will (following \cite{bbd} and coherently with all the papers of the author cited here) use the "cohomological" convention for $t$-structures. 
 This "discrepancy between conventions" will force us to put somewhat weird "$-$" signs in some of the formulas (cf. also Definition \ref{dwso}(\ref{idadj})); however,  it is coherent with Definition \ref{dhopo}(2) (of adjacent torsion pairs).

Lastly, in \cite{bws} both "halves" of $w$ were required to be  additive. Yet this additional restriction is easily seen to follow from the remaining axioms; see Remark 1.2.3(4) of \cite{bonspkar}.

4. As we had already noted in Remark \ref{rwsts}, the current definition of right and left adjacent "structures" is somewhat different from the one used in previous papers of the author. Also, if $w$  is left or right adjacent to $t$ then the associated torsion pairs are only "adjacent up to a shift"; yet this is easily seen to make no difference in the proofs. 
 
\end{rema}

Now  we recall some basic 
properties of weight structures. 

\begin{pr} \label{pbw}
Let $\cu$ be a triangulated category endowed with a weight structure $w$, $M, M',M''\in \obj \cu$, $i,m,n\in \z$. Then the following statements are valid.

\begin{enumerate}

\item \label{idual}
The axiomatics of weight structures is self-dual, i.e., for $\du=\cu^{op}$
(so $\obj\cu=\obj\du$) there exists the (opposite)  weight
structure $w'$ for which $\du_{w'\le 0}=\cu_{w\ge 0}$ and
$\du_{w'\ge 0}=\cu_{w\le 0}$.

\item\label{iextw} 
 $\cu_{w\le i}$, $\cu_{w\ge i}$, and $\cu_{w=i}$ are Karoubi-closed and extension-closed in $\cu$ (and so, additive). 

 \item\label{iort}
 $\cu_{w\ge i}=(\cu_{w\le i-1})^{\perp}$ and $\cu_{w\le i}={}^{\perp} (\cu_{w\ge i+1})$.
 
\item\label{igenlm}
The class $\cu_{[m,n]}$ is the 
extension-closure of $\cup_{m\le j\le n}\cu_{w=j}$.

 \item\label{iwd0} 
 If $M\in \cu_{w\ge m}$  then $w_{\le n}M\in \cu_{[m,n]}$ (for any $n$-weight decomposition of $M$). 
Dually, if  $M\in \cu_{w\le n}$  then $w_{\ge m}M\in \cu_{[m,n]}$.

\item\label{icompl} Assume that $ m\le n$.  
				The for any (fixed) $m$-weight decomposition of $M$ and an $n$-weight decomposition of $M'$  (see Remark \ref{rstws}(1)) 	any morphism
 $g\in \cu(M,M')$ can be extended 
to a 
morphism of the corresponding distinguished triangles:
 \begin{equation}\label{ecompl} \begin{CD} w_{\le m} M@>{c}>>
M@>{}>> w_{\ge m+1}M\\
@VV{h}V@VV{g}V@ VV{j}V \\
w_{\le n} M'@>{}>>
M'@>{}>> w_{\ge n+1}M' \end{CD}
\end{equation}

Moreover, if $m<n$ then this extension is unique (provided that the rows are fixed).

\item\label{iwdext} For any distinguished triangle $M\to M'\to M''\to M[1]$  and any 
weight decompositions $LM\stackrel{a_{M}}{\to} M\stackrel{n_{M}}{\to} R_M\to LM[1]$ and $LM''\stackrel{a_{M''}}{\to} M''\stackrel{n_{M''}}{\to} R_M''\to LM''[1]$ there exists a commutative diagram 
$$\begin{CD}
LM @>{}>>LM'@>f>> LM''@>{}>>LM[1]\\
 @VV{a_M}V@VV{a_{M'}}V @VV{a_{M''}}V@VV{a_{M}[1]}V\\
M@>{}>>M'@>{}>>M''@>{}>>M[1]\\
 @VV{n_M}V@VV{n_{M'}}V @VV{n_{M''}}V@VV{n_{M}[1]}V\\
RM@>{}>>RM'@>{}>>RM''@>{}>>M[1]\end{CD}
$$
in $\cu$ whose rows are distinguished triangles and the second column is a weight decomposition (along with the first and the third one).

\item\label{isplit} If an $\hw$-morphism $f:A\to B$ is split surjective then there exists an object $C\in \cu_{w=0}$ such that $f$ is isomorphic to the canonical epimorphism $B \bigoplus C\to C$.

\item\label{iwdmod} If $M$ belongs to $ \cu_{w\le 0}$ (resp. to $\cu_{w\ge 0}$) then it is a retract of any choice of $w_{\le 0}M$ (resp. of $w_{\ge 0}M$).

\end{enumerate}
\end{pr}

\begin{proof} 
All of the assertions except the 
two 
last ones were essentially established in   \cite{bws} (pay attention to Remark \ref{rstws}(3)!). 

Assertion \ref{isplit} is immediate from assertion \ref{iextw} (that says that $\cu_{w=0}$ is Karoubi-closed in $\cu$).


\ref{iwdmod}. The case $M\in \cu_{w\le 0}$ of the assertion follows immediately from  Proposition \ref{phop}(\ref{itp7p}) (see Remark \ref{rwhop}(1) below).
The case  $M\in \cu_{w\ge 0}$ 
is just the categorical dual of the first case; see  assertion \ref{idual}. 

\end{proof}

\begin{rema}\label{rwhop}

1. Similarly to Remark \ref{rtst1}(\ref{it1}), part \ref{iort} of our proposition yields that for $\lo=\cu_{w\le 0}$ and $\ro=\cu_{w\ge 1}$ the couple $s=(\lo,\ro)$ is a torsion pair; we will call it the torsion pair associated with $w$. Conversely, if for a torsion pair $s$ we have $\lo\subset \lo[1]$ then $(\lo,\ro[-1])$ is a weight structure that we will say to be associated with $s$. In this case we will say that $s$ is {\it weighty}; for a class $\cp\subset \obj \cu$ we will say that $\cp$ generates $w$ whenever it generates $s$ (certainly, then we have $\cp\subset \lo =\cu_{w\le 0}$, $\cp\perpp=\cu_{w\ge 1}$, and $w$ is determined by $\cp$); $w$ is compactly generated whenever $s$ is. 

Respectively, we will say that $w$ is 
smashing (resp. cosmashing,   countably smashing; resp. generated by $\cp\subset \obj \cu$) whenever the corresponding $s$ is. 


Lastly, we will say that $w$ restricts to a triangulated subcategory $\cupr$ of $\cu$ whenever the associated $s$ does;   cf. Definition \ref{dhopo}(4). 

2. Certainly, the "most interesting" weight and $t$-structures are the non-degenerate ones. However, this non-degeneracy condition can be quite difficult to check (and it actually fails for certain "important" weight structures; see 
Remark 4.2.2(4) of \cite{bkwn}). So we prefer not to avoid degenerate weight and $t$-structures in this paper; 
 this makes Proposition \ref{phopft}(II.1) an important tool.

Now we describe (some) consequences of this proposition for a torsion pair associated 
 with a weight structure   $w$ (resp. with a $t$-structure $t$; see Remark \ref{rtst1}(\ref{it1})).

So, assume that $\cu$ is a full strict triangulated subcategory of a triangulated category $\cu'$, and assume that the embedding $\cu\to \cu'$ possesses a right adjoint $G$. Then there exists a  weight structure $w'$  (resp. a $t$-structure $t'$) on $\cu'$ such that 
$\cu'_{w'\le 0}=\cu_{w\le 0}$  (resp.  
$\cu'^{t'\le 0}=\cu^{t\le 0}$). It certainly follows that $\cu_{w'\ge 0}\cap \obj \cu=\cu_{w\ge 0}$ (resp. $\cu^{t'\ge 0}\cap \obj \cu=\cu^{t\ge 0}$); hence $\hw'=\hw$ (resp. $\hrt'=\hrt$). Moreover, if $w$ (resp. $t$) is generated by a class $\cp\subset \obj \cu$ in $\cu$ then  $w'$ (resp. $t'$) is generated by $\cp$ in $\cu'$.

3. Now we apply to weight structures 
Proposition \ref{phopft}(II.2). Assume that  $F:\cu\to \cu'$ is an exact functor, $G$ is its (exact) right adjoint, and $w$ is generated by some $\cp\subset \obj \cu$.   
  Then Proposition \ref{phopft}(II.2)  (combined with the first part of this remark) yields the following: $F$ is left weight-exact (with respect to $w$ and some weight structure $w'$ on $\cu'$) if and only if $F(\cp)\subset \cu'_{w'\le 0}$; these conditions are equivalent to the right weight-exactness of $G$. 

Note that this statement can certainly be applied for $\cp=\cu_{w\le 0}$; hence $F$ is left weight-exact if and only if $G$ is right weight-exact.

4. Recall also that 
for a compactly generated triangulated category $\cu$  and $F:\cu\to \cu'$ being an exact functor respecting coproducts  Proposition \ref{pcomp}(II) gives the existence of $G$. Hence 
these assumptions imply that  $F$ is left weight-exact if and only if $F(\cp)\subset \cu'_{w'\le 0}$. 

5. We will also need the categorical dual of the latter observation. Certainly, it is formulated is follows: if $\cu$ is cocompactly cogenerated (i.e., there exists a set of objects of $\cu$ that compactly generates $\cu\opp$), $F$ is an exact functor $\cu\to \cu'$ that respects products, 
$w$ and $w'$ are are weight structures on  $\cu$ and $\cu'$, respectively, and  $w$ is cogenerated by some $\cp\subset \obj \cu$ (i.e., $\cu_{w\le -1}=\perpp\cp$) then $F$ is  right weight-exact if and only if $F(\cp)\subset \cu'_{w'\ge 0}$.

6. Applying Proposition \ref{phop}(\ref{itp10}) to weighty torsion pairs we certainly obtain the following: if $F:\cu\to \cu'$ is a weight-exact functor 
  that is surjective on objects then $\cu'_{w'\ge 0} = \kar_{\cu'}(F(\cu_{w\ge 0}))$ and $\cu'_{w'\le 0} = \kar_{\cu'}(F(\cu_{w\le 0}))$.

7. Lastly, 
assume that $F$  as above possesses a right adjoint. 
Then Proposition \ref{phopft}(II.3) implies that the weight structure $w'$ is generated by $F(\cp)$.

Once again, the existence of 
right adjoint 
 is automatic whenever $\cu$ is compactly generated and $F$  respects coproducts. 

8. Assume that $w$ is generated by a class $\cp$. Then $\cap_{i\in \z}\cu_{w\ge i}=\cap_{i\in \z}\cp\perpp$. Thus $\cp$ is left non-degenerate if and only if $\cp$ Hom-generates $\cu$. 

\end{rema}

\subsection{On weight Postnikov towers and weight complexes}\label{sswc}

To define the weight complex functor we will need the following definitions.

\begin{defi}\label{dfilt}

Let $M\in \obj \cu$.

1. A datum consisting of  $M_{\le i}\in \obj \cu$, $h_i\in \cu(M_{\le i},M)$, $j_i\in  \cu(M_{\le i},M_{\le i+1})$ for $i$ running through integers
will be called a {\it filtration on $M$} 
if we have $h_{i+1}\circ j_i=h_i$  for all $i\in \z$; we will write $\fil_*M$ for this filtration.

A filtration will be called {\it bounded} if there exist $l\le m\in \z$ such that $M_{\le i}=0$ for all $i<l$ and $h_i$ are isomorphisms for all $i\ge m$.

2. A filtration as above equipped with distinguished triangles $M_{\le i-1}\stackrel{j_{i-1}}{\to}M_{\le i}\to M_i$ for all $i\in \z$ will be called a {\it  Postnikov tower} for $M$ or for  $\fil_*M$; 
this tower will be denoted by $Po_\fil$.

We will use the symbol $M^p$ to denote  $M_{-p}[p]$; we will call $M^p$ the {\it factors} of $Po_\fil$. 

3. If $\fil_*M'=(M'_{\le i}, h'_i, j_i)$ is a filtration of $M'\in \obj \cu$ and $g\in \cu(M,M')$ then we will call $g$ along with a collection of $g_{\le i}\in \cu(M_{\le i}, M'_{\le i})$  a {\it morphism of filtrations compatible with $g$} if  $g\circ h_i=h'_i\circ g_{\le i}$ and  $j'_i\circ g_{\le i} =g_{\le i+1}\circ j_i$ for all $i\in \z$. 

\end{defi}

\begin{rema}\label{rwcomp}
1. Composing (and shifting) arrows from   triangles in 
$Po_\fil$ for all pairs of two   subsequent $i$s one can construct a complex whose $i$th term equals
$M^i$ (it is easily seen that this is a complex indeed; cf. Proposition 2.2.2 of \cite{bws}). 
We will call it a complex {\it associated with} $Po_\fil$.

2. Certainly, any filtration yields a Postnikov tower (uniquely up to a non-unique isomorphism). 
Furthermore, it is easily seen that any morphism of filtrations extends to a morphism of the corresponding Postnikov towers (defined in the obvious way). 
Besides, any morphism of Postnikov towers yields a morphism of the associated complexes.

Lastly, note that morphisms of filtrations and Postnikov towers can certainly be added and composed.

3. The triangles in $Po_\fil$ also 
 give the following statement immediately: if a filtration of $M$ is bounded then $M$ belongs to the extension-closure of $\{M_i\}$. 
\end{rema}


\begin{defi}\label{dwpt}
Assume that $\cu$ is endowed with a weight structure $w$.

1. We will call a filtration (see Definition \ref{dfilt}) $\fil_*M$ of $M\in \obj \cu$ a {\it weight filtration} (of $M$) if the morphisms $h_i:M_{\le i}\to M$ yield $i$-weight decompositions for all $i\in \z$ (in particular, 
  $M_{\le i}=w_{\le i}M$). 

We will call  
the corresponding $Po_\fil$ a {\it weight Postnikov tower} for $M$.

2. $\pwcu$ will denote the category whose objects  are objects of $\cu$ endowed with arbitrary weight Postnikov towers and whose morphisms are morphisms of Postnikov towers.

$\cuw$ will be the category whose objects are the same as for $\pwcu$ and such that  $\cuw(Po_{\fil_M},Po_{\fil_{M'}})=\imm (\pwcu(Po_{\fil_M},Po_{\fil_{M'}})\to \cu(M,M'))$ (i.e., we kill those morphisms of towers that are zero on the underlying objects).

3. For an additive category $\bu$, complexes $A,B\in \obj K(\bu)$, and morphisms $m_1,m_2\in C(\hw)(A,B)$  we will write $m_1\backsim m_2$ if $m_1-m_2=d_Bh+jd_A$ for some collections of arrows $j^*,h^*:A^*\to B^{*-1}$.

We will call this relation the {\it weak homotopy one}.\footnote{This relation was earlier introduced in \cite{barrabs}; $m_1$ is {\it absolutely homologous} to $m_2$ in the terminology of that paper. Respectively, some of the results below concerning this equivalence relation were proved in ibid.}\ 
\end{defi}

The following statements were essentially proved
in \cite{bws} (in particular, see the proof of \cite[Theorem 3.3.1(I)]{bws} for assertion \ref{iwcex} below). 
Moreover, the first two of them easily follow from Proposition \ref{pbw}(\ref{icompl}) (along with the corresponding definitions).

\begin{pr}\label{pwt}
In addition to the notation introduced above assume that $\bu$ is an additive category.
\begin{enumerate}
\item\label{iwpt1}
Any choice of $i$-weight decompositions of $M$ for  $i$ running through integers naturally yields a canonical weight filtration for $M$ (with $M_{\le i}=w_{\le i}M$).

Moreover,  we have   $\co(Y_i\to M)\in \cu_{w\ge i+1}$ and     $M^i\in \cu_{w=0}$.

\item\label{iwpt2} Any $g\in \cu(M,M')$ can be extended to a morphism of (any choice of) weight filtrations for $M$ and $M'$, respectively; hence it also extends to a morphism of weight Postnikov towers.

\item\label{iwpt3} The natural functor $\cuw\to \cu$ is an equivalence of categories.

\item\label{iwhecat}
Factoring morphisms in $K(\bu)$ by the weak homotopy relation yields an additive category $\kw(\bu)$. Moreover, the corresponding full functor $K(\bu)\to \kw(\bu)$ is (additive and) conservative.

\item\label{iwhefu}
Let $\ca:\bu\to \au$ be an additive functor, where $\au$ is any abelian category. Then for any $B,B'\in \obj K(\bu)$ any pair of weakly homotopic morphisms $m_1,m_2\in C(\hw)(B,B')$  induce equal morphisms of the homology $H_*(\ca(B^i))\to H_*(\ca(B'^i))$.

\item\label{iwhefun}
Sending an object of $\cuw$ into the complex described in Remark \ref{rwcomp}(1) yields a well-defined additive functor 
$t=t_w:\cuw\to \kw(\hw)$.

We will call this functor the {\it weight  complex} one.\footnote{The term comes from \cite{gs}; yet the domain of the weight complex functor in that paper was not triangulated, whereas the target was ("the ordinary") $K^b(\chowe)$.}\ We
will often write $t(M)$ for $M\in \obj \cu$ assuming that some weight Postnikov tower for $M$ is chosen; we will say that $t(M)$ is {\it a choice of a weight complex} for $M$.

\item\label{iwcex} 
If $M_0\stackrel{f}{\to} M_1 \stackrel{g}{\to} M_2$ is a distinguished triangle in $\cu$ then for any choice of $t(M_0)$ and $t(M_2)$
 there exists a compatible choice of $(t(f),t(g))$ (so, the domain of this $t(f)$ is the chosen $t(M_0)$ and the target of  $(t(g)$ is $t(M_2)$) that can be completed to a distinguished triangle in $K(\hw)$.


\end{enumerate}
\end{pr}

\begin{rema}\label{rwc}
So, for an object $M$ of $\cu$ its weight complex $t(M)$ is well-defined up to a $K(\hw)$-endomorphism that is weakly homotopic to zero; thus it is defined in $K(\hw)$ up to a (not necessarily unique) isomorphism.

In particular, if $M\in \cu_{w\ge -n}$ for some $n\in \z$ then we can take $M_{\le i}=w_{\le i}M=0$ for all $i<-n$; hence any choice of $t(M)$ is homotopy equivalent to a complex concentrated in degrees at most $n$. Similarly, if $M\in \cu_{w\le  -n}$ then we can take $M_{\le i}=w_{\le i}M=M$ for all $i\le -n$; thus $t(M)$ is homotopy equivalent to a complex concentrated in degrees at  least $n$.  Hence $t(M)\cong 0$ whenever $M$ is left or right degenerate. 
\end{rema}


\subsection{On pure functors and detecting weights}\label{sdetect}

\begin{pr}\label{ppure} 
Assume that $\cu$ is endowed with a weight structure $w$.

Let $\ca:\hw\to \au$ be an additive functor, where $\au$ is any abelian category. 
Choose a weight complex $t(M)=(M^j)$ for each $M\in \obj \cu$, and denote by 
$H(M)=H^{\ca}(M)$ the zeroth homology of the complex $\ca(M^{j})$. Then $H(-)$ yields a homological functor 
that does not depend on the choices of weight complexes.  Moreover, the assignment $\ca\mapsto H^\ca$ is natural in $\ca$. 

\end{pr}
\begin{proof}
Immediate from Proposition \ref{pwt}(\ref{iwpt3}, \ref{iwhefu},\ref{iwcex}). 
\end{proof}

\begin{rema}\label{rpure}
1. 
(Co)homological functors of this type have already found interesting applications in  \cite{kellyweighomol},  \cite{bach}, \cite{bscwh},   \cite{bontabu}, and \cite{bgn}. We will prove some statements relevant for the latter paper just now. We will not apply the remaining 
 results of this subsection elsewhere in the paper. 

2. We will call a functor $H:\cu\to \au$ {\it pure} (or $w$-pure) if it 
equals $H^\ca$ for a certain $\ca:\hw\to \au$. In the next subsection we will prove that this definition of purity is equivalent to another ("intrinsic") one.

\end{rema}

Now we prove that pure functors can be used to  "detect weights"; these results are crucial for \cite{bontabu}. 
This notion of detecting weights is closely related to the one of {\it weight-conservativity} that was introduced in \cite{bach}.

To prove the most general case of 
  we have to recall 
a result from \cite{bkwn}.\footnote{Note however that 
 the bounded cases of  Proposition \ref{pdetect} and Proposition \ref{pdetectsse} 
 can also be easily deduced from Theorem 3.3.1(IV) of \cite{bws}.}\ 

\begin{lem}\label{lbkw}
Let $m\in \z$, $M\in \obj \cu$, where $\cu$ is endowed with a weight structure $w$.

Then  $t(M)$ belongs to  $K(\hw)_{\wstu\ge -m }$ 
(resp. to $K(\hw)_{\wstu\le -m }$; see Remark \ref{rstws}(2))  if and only if 
$M\bigoplus M[1]$ is an extension of an element of $\cu_{w\ge -m}$ by a  right degenerate  object (resp. $M\bigoplus M[-1]$  is an extension of a  left degenerate  object by an element of $\cu_{w\le -m}$).

\end{lem}
\begin{proof}
This statement is contained in 
Theorem 3.1.6 of ibid.
\end{proof}

For a homological functor $H$ the symbol $H_i$ will be used to denote the composite functor $H\circ [i]$.


\begin{pr}\label{pdetect}
Adopt the notation of Proposition \ref{ppure} and assume that the following conditions are fulfilled:

(i) the image of $\ca$ consists of  $\au$-projective objects only;

(ii)  if an $\hw$-morphism $h$ does not split (i.e., it is not a retraction) then $\ca(h)$ does not split also.

Then for any $M\in \obj \cu$, $m\in \z$, and $H=H^\ca$ we have the following:  $M$ is $w$-bounded below and $H_i(M)=0$ for all $i> m$ if and only if  $M\in \cu_{w\ge -m}$.

\end{pr}
\begin{proof}

The "if" part of the statement is very easy (for any $\ca$); just combine the definition of $H^\ca$ with Remark \ref{rwc}. 

Now we prove the converse application.
If $M$ is $w$-bounded below then the object $M'=M\bigoplus M[1]$ is $w$-bounded below also.
We also have $H_i(M')=0$ for all $i> m$, and it suffices to prove that $M'\in \cu_{w\ge -m}$ (by the axiom (i) of weight structures).
  
Now chose the minimal integer $n\ge m$ such that  $t(M')\in K(B)_{\wstu\ge -n}$ 
(see Remark \ref{rstws}(2)). 
Then Lemma \ref{lbkw} yields the existence of a distinguished triangle $M_1'\stackrel{g}{\to} M'{\to} M'_2\to M'_1[1]$ such that $M'_1$ is right $w$-degenerate and $M'_2\in  \cu_{w\ge -n}$. Since $M'$ is bounded, $g=0$; hence $M'$ is a retract of $M'_2$ and so belongs to $\cu_{w\ge -n}$ itself.

It remains to prove that $n=m$. Assume the converse. Then  $t(M')$ is homotopy equivalent to an $\hw$-complex $(N^i)$   concentrated in degrees $\le n$, and Proposition \ref{pbw}(\ref{isplit}) yields that the boundary  morphism $d^{n-1}_N:N^{n-1}\to N^n$ does not split. 
Thus  $\ca(d^{n-1}_N)$ does not split also. Since $\ca(N^n)$ is projective, this non-splitting implies that $H_n(M')\neq 0$. Thus $n\le m$.

\end{proof}

Now we formulate a simple corollary  from the proposition that will be applied in \cite{bontabu}.

\begin{coro}\label{cbontabu}
Let $\ca:\hw\to \au$ be a full additive conservative functor whose target is semi-simple. Then for a $w$-bounded object $M$ of $\cu$ we have $M\in \cu_{w=0}$ if and only if  $H^\ca_i(M)=0$ for all $i\neq 0$.
\end{coro}
\begin{proof}
Once again, the "if" part of the statement is simple (and can easily be deduced from the previous proposition).
So we prove the converse implication.

Now, Lemma \ref{lbkw} allows us to assume that $\cu=K(\hw)$ and $w=\wstu$  (since it enables us to treat $t(M)$ instead of $M$). 
Hence it  suffices to prove that  $M\in \cu_{w\ge 0}$ since then we will also have $M\in \cu_{w\le 0}$ by duality (see Proposition \ref{pbw}(\ref{idual}; note that our assumptions on $\ca$ and $\au$ are self-dual and the construction of $H^\ca$ is also so in this case).\footnote{Alternatively, one can apply Proposition \ref{pwrange}(\ref{iwrpure}) below to obtain this duality assertion.}
 Hence it suffices to verify that the functor $\ca$ satisfies the assumptions of Proposition \ref{pdetect}.  The latter is very easy: all elements of $\ca(\hw)$ are projective in $\au$ since all objects of $\au$ are (recall that $\au$ is semi-simple), and a full conservative functor obviously does not send non-split morphisms into split ones. 
\end{proof}

\begin{rema}\label{rdetect}
1. As we have (essentially) just noted,  condition (ii) of  Proposition \ref{pdetect} is obviously  fulfilled both for $\ca$ and for  the opposite functor $\ca^{op}:\hw^{op}\to \au^{op}$ whenever $\ca$ is a full 
conservative functor. In particular, it suffices to assume that $\ca$ is a full embedding.

Hence it may be useful to assume (in addition to assumption (i) of the proposition) that the image of $\ca$ consists of injective objects only. 

2.  Now we describe a general method for constructing a full embedding $\ca$ whose image consists of  $\au$-projective objects.  

Assume that $\hw$ is an essentially small (additive) $R$-linear category, where $R$ is a commutative unital ring (certainly, one may take $R=\z$ here). Denote some small skeleton of $\hw$ by $\bu$ (to avoid set-theoretical difficulties).

Consider the abelian category $\psvr(\bu)$ of $R$-linear contravariant functors from $\bu$ into 
the category of $R$-modules (cf. \S\ref{snotata}). Then  $\bu$ (and so, also $\hw$) 
embeds into the full subcategory of projective objects of  $\psvr(\bu)$ (by the Yoneda lemma; see  Lemma 5.1.2 of \cite{neebook}).  

Hence this functor "detects weights" (in the sense of  Proposition \ref{pdetect}).

3. The objects in the essential image of this functor may be called {\it purely $R$-representable homology}. Since they are usually not injective in  $\psvr(\bu)$,
a dual construction may be useful for checking whether $M\in \cu_{w\le -m}$.

4. It is easily seen that in the proof of our corollary (and so, also of the bounded case of Proposition \ref{pdetect}) one can replace the usage of Lemma \ref{lbkw} by that of \cite[Theorem 3.3.1(IV)]{bws}.

5. Condition (ii)  of Proposition \ref{pdetect}  is certainly necessary. 
Indeed, if $h\in \mo (\hw)$ does not split whereas $\ca(h)$ does, then one can easily check that $\co(h)\in \cu_{w\ge 0}\setminus \cu_{w\ge 1}$ and $H_i(\co(h))=0$ for $i\neq -1$.

\end{rema}

Now we prove a certain unbounded version of Proposition \ref{pdetect}. Note 
 that its proof  may  be (slightly) simplified  if we assume that $\hw$ is Karoubian (cf. \cite[\S3.1]{bkwn} that  demonstrates that the general case can be "reduced" to this one).

\begin{pr}\label{pdetectsse}
Assume that $\ca:\hw \to \au$ 
 (as in the previous proposition) is a full 
functor, for any $N\in \cu_{w=0}$ the ideal $\ke(\hw(N,N )\to \au(\ca(N),\ca(N))$ is a nilpotent ideal of the endomorphism ring $\hw(N,N)$, and $\au$ is abelian semi-simple.

1. Then  $H_i(M)=0$ for all $i> m$ (resp. for all  $i<m$) if and only if $M\bigoplus M[1]$ is an extension 
of an element of $\cu_{w\ge -m}$ by a  right degenerate  object (resp. $M\bigoplus M[-1]$  is an extension 
of an element of $\cu_{w\le -m}$ by a  left degenerate  object).

2. Assume in addition that  $w$  is non-degenerate. 
 Then 
  these two conditions are equivalent to $M\in \cu_{w\ge -m}$ (resp. $M\in \cu_{w\le -m}$). In particular, if $H_i(M)=0$ for all $i
\neq m$ then $M\in  \cu_{w= -m}$.
\end{pr}
\begin{proof}
1. It suffices to verify that $t(M)$ belongs to  $K(\hw)_{\wstu\le -m }$ 
(resp. to $K(\hw)_{\wstu\ge -m }$). 
Indeed, then  applying Lemma \ref{lbkw} once again we will certainly obtain the result.
Note also that it suffices to verify the "main" version of this assertion since the "resp." one is its dual. 

Now, by Proposition \ref{pbw}(\ref{iort}) 
 it suffices to check that $t(M)$ is $K(\kar(\hw))$-isomorphic to a complex concentrated in degrees $\le m$.
To construct the latter  we extend $\ca$ to an additive functor $\ca': \kar(\hw)\to \au$ and find a complex $T=(t^i)\in \obj C(\kar(\hw))$ that is  $K(\kar(\hw))$-isomorphic to $t(M)$  such that the boundary $d_t^{m}:t^{m}\to t^{m+1}$ is  killed by $\ca'$.
This is easily seen to be possible according to Theorem 1.3 of \cite{wildab} (cf. also Theorem 2.2 of ibid.). Indeed, our assumptions on $\ca$ imply immediately that $\hw$ is {\it semi-primary} in the sense of  \cite[Definition 2.3.1] {andkahn}. Then $\kar(\hw)$ is semi-primary also according to Proposition 2.3.4(c) of ibid.;  
hence $\ca': \kar(\hw)\to \au$ satisfies the  "kernel nilpotence assumption" similar to that for $\ca$.  Thus Theorem 1.3 of \cite{wildab} says that the cone of any $\kar(\hw)$-morphism is  $K(\kar(\hw))$-isomorphic to a cone of a morphism killed by $\ca'$, and so we can "replace" the 
$m$th boundary of $t(M)$ by a morphism satisfying this condition.

For this complex $T$   the result of the applying $\ca'$ termwisely to its stupid truncation complex  $\wstu_{\le -m-1}T$ is certainly zero. Hence the identity of  $\wstu_{\le -m-1}T$ is homotopy equivalent to an endomorphism killed by the termwise application of $\ca'$. Applying our nilpotence assumption to a sufficiently high power of this ($K(\kar(\hw))$-invertible) endomorphism we obtain that  for any $j\in \z$ the complex $\wstu_{\le -m-1}T$ is homotopy equivalent to a complex concentrated in degrees $\ge j$. Hence $\wstu_{\le -m-1}T$ is contractible 
and we obtain $t(M)$ is ${K(\kar(\hw))}$-isomorphic to its stupid truncation  $\wstu_{\ge -m}T$.

2. Assertion 1 implies that $M\bigoplus M[1]$  belongs to $\cu_{w\ge -m}$ (resp. $M\bigoplus M[-1]$  belongs to $\cu_{w\le -m}$) itself. Hence $M$  belongs to $\cu_{w\ge -m}$ (resp. to $\cu_{w\le -m}$) also.

The "in particular" part of the assertion follows immediately.
\end{proof}

\begin{rema}\label{rpsh}
1. Both Propositions \ref{pdetect} and \ref{pdetectsse} are easily seen to imply  certain generalizations of  \cite[Theorem 1.5]{wildcons}.

2. Proposition \ref{pdetectsse} can probably generalized. In particular, it appears to be sufficient to assume for any $N\in \cu_{w=0}$ that all the endomorphisms in $\ke(\hw(N,N )\to \au(\ca(N),\ca(N)))$ are nilpotent.

However, it is demonstrated in \cite{bontabu} that the case where $\ca$ is a full embedding (into a semi-simple category) is quite interesting (already). In this case the proof can be simplified since then $t(M)$ is obviously a retract of $\wstu_{\ge -m}t(M)$ (resp. of $\wstu_{\le -m}t(M)$). 
\end{rema}

\subsection{On virtual $t$-truncations and cohomology of bounded weight range}\label{svtt}

We  recall the   notion of virtual $t$-truncations for a cohomological functor $H:\cu\to \au$ (as defined in \S2.5 of \cite{bws} and studied in more detail in \S2 of \cite{bger}). These truncations allow us to "slice" $H$ into $w$-pure pieces.  These truncations behave as if they were given by truncations of $H$ in some triangulated "category of functors" $\du$ with respect to some $t$-structure (whence the name). 
Moreover,  this is often actually the case (and we will discuss this matter below); yet the definition does not require the existence of $\du$ (and so, does not depend on its choice). Our choice of  the numbering for them is motivated by the cohomological convention for $t$-structures; 
this convention combined with the homological numbering for weight structures causes  the (somewhat weird) "$-$" signs in the definitions and formulations of this section.

\begin{defi}\label{dvtt}
Let $H$ be a cohomological functor from $\cu$ into an abelian category $\au$; assume that $\cu$ is endowed with a weight structure $w$ and $ n\in \z$.

 We define the {\it virtual $t$-truncation} functors $\tau^{\ge -n }(H)$ (resp. $\tau^{\le -n }(H)$)  by the correspondence $$M\mapsto\imm (H(w_{\le n+1}M)\to H(w_{\le n}M)) ;$$ 
(resp. $M\mapsto\imm (H(w_{\ge n}M)\to H(w_{\ge n-1}M)) $); here we take arbitrary choices of  
the corresponding weight truncations of $M$ and connect them using Proposition \ref{pbw}(\ref{icompl}).
\end{defi}

We recall the main properties of these constructions that were established in \cite[\S2.5 and Theorem 4.4.2(7,8)]{bws}.

\begin{pr}\label{pwfil}
In the notation of the previous definition the  following statements are valid.

1. The objects $\tau^{\ge -n}(H)(M)$ and $\tau^{\le -n}(H)(M)$ are $\cu$-functorial  in $M$  (and so,
 the virtual $t$-truncations of $H$ are well-defined functors).

2. The functors $\tau^{\ge -n }(H)$ and $\tau^{\le -n }(H)$ are cohomological.

3. There exist natural transformations that  yield a long exact sequence 
\begin{equation}\label{evtt}
\begin{gathered} 
\dots \to \tau^{\ge -n +1}(H)\circ [-1] \to \tau^{\le -n }(H)\to H \\ \to \tau^{\ge -n +1}(H)\to \tau^{\le -n }(H)\circ [1]\to \dots\end{gathered} 
 \end{equation}  (i.e.,  the result of applying this sequence to any object of $\cu$ is a long exact sequence); the shift of this exact sequence by $3$ positions is given by composing the functors with $-[1]$. 

4. Assume that there exists a $t$-structure $t$ that is right adjacent to $w$. 
Then for any $M\in \obj \cu$ and $H_M=\cu(-,M):\cu\opp\to \ab $ the functors $ \tau^{\le -n }(H_M)$ and $\tau^{\ge -n}(H_M)$ are represented by $t^{\le -n}M$ and $t^{\ge -n}M$, respectively.

\end{pr}

 Moreover, a stronger and more general statement  than part 4 of this proposition 
 is given by \cite[Proposition 2.5.4(1)]{bger}; it will be applied in the proof of Proposition \ref{psaturdu}(2) below.


Now we define weight range 
 and introduce notation for pure cohomological functors; some of these statements will be applied below. 

\begin{defi}\label{drange}
1. Let $m,n\in \z$; let  $H$ be as above. 

Then we will say that $H$ is {\it of weight range} $\ge m$ (resp. $\le n$, resp.  $[m,n]$)  if it annihilates $\cu_{w\le m-1}$ (resp. $\cu_{w\ge n+1}$, resp. both of these classes). 

2. Let $\ca:\hw^{op}\to \au$ be an additive functor.  Then for  $\ca^{op}$ being the opposite functor $\hw\to \au^{op}$ we will write $H_\ca$ for the cohomological functor from $\cu$ into $\au$ obtained from  $H^{\ca^{op}}$  (see Proposition \ref{ppure}) by means of reversion of arrows.

We will call functors obtained using this construction {\it pure cohomological} ones.


\end{defi}

\begin{pr}\label{pwrange}
In the notation of the previous definition the  following statements are valid.

\begin{enumerate}

\item\label{iwrvt} 
The functor $\tau^{\ge -n}(H)$ is of weight range $\le n$, and $\tau^{\le -m}(H)$ is of weight range $\ge m$.

\item\label{iwrcrit} Assume that $w$ is bounded. Then $H$ is of weight range  $\le n$ (resp. of weight range $\ge m$) if and only if it kills $\cu_{w=i}$ for all $i>n$ (resp. for $i>m$).


\item\label{iwridemp} We have $\tau^{\ge -n}(H)\cong H$ (resp. $\tau^{\le -m}(H)\cong H$) if and only if $H$ is of weight range $\le n$ (resp. of weight range $\ge m$).

\item\label{iwrcomm} We have $\tau^{\ge -n}(\tau^{\le -m})(H)\cong \tau^{\le -m}(\tau^{\ge -n})(H)$.

\item\label{iwfil4} If $H$ is 
of weight range $\ge m$  then  $\tau^{\ge -n}(H)$ is 
  of weight range $[m,n]$.

\item\label{iwrd}
Dually, if $H$ is 
of weight range $\le n$  then  $\tau^{\le -m}(H)$ is 
  of weight range $[m,n]$.

\item\label{iwrvan} Assume that $m>n$. Then the only functors of weight range  $[m,n]$ are zero ones; thus if $H$ is of weight range $\le n$ (resp. $\ge m$) then $\tau^{\le -m}(H)=0$ (resp. $\tau^{\ge -n}(H)=0$).

\item\label{iwrpure} 
The functors of weight range $[m,m]$ 
are exactly those of the form $H_\ca\circ [-m]$ (see Definition \ref{drange}(2)), where
$\ca:\hw\to \au^{op}$ is an additive functor.

\item\label{iwfil3} The (representable) functor 
 $H_M:\cu\to \ab$  if of weight range $\ge m$ if and only if $M\in \cu_{w\ge m}$.

\item\label{iwfil5} If $H$ is of weight range $[m,n]$ then the morphism  $H(w_{\ge m}M)\to H(M)$ is surjective and the morphism   $H(M)\to H(w_{\le n}M)$ is injective (here we take arbitrary choice of the corresponding weight decompositions of $M$ and apply $H$ to their connecting morphisms). 
\end{enumerate}
\end{pr}
\begin{proof}
Let $M\in \cu_{w\ge n+1}$. Then we can take $w_{\le n}(M)=0$. Thus $\tau^{\ge -n}(H)(M)=0$, and we obtain the first part of assertion \ref{iwrvt}. It second part is easily seen to be dual to the first part.  

The "only if" part of assertion \ref{iwrcrit} is immediate from the definition of weight range. The converse implication easily follows from Proposition \ref{pbw}(\ref{igenlm}).
 
Assertion \ref{iwridemp}
is precisely Theorem 2.3.1(III2,3) of \cite{bger}; assertion \ref{iwrcomm} is given by part II.3 of that theorem.

Now let $H$ be of weight range $\ge m$.  Then $\tau^{\ge -n}(H)\cong \tau^{\ge -n}(\tau^{\le -m})(H)\cong \tau^{\le -m}(\tau^{\ge -n})(H)$ (according to the two previous assertions). It remains to apply assertion \ref{iwrvt} to obtain assertion \ref{iwfil4}.

Assertion \ref{iwrd} can be proved similarly; it is also easily seen to be dual to assertion \ref{iwfil4}.

Next, for any $l\in \z$ and any cohomological $H$ any choice of an $l$-weight decomposition triangle (cf. (\ref{ewd})) for $M$ gives the long exact sequence
\begin{equation}\label{eles}
\begin{gathered}
\dots \to H((w_{\le l}M)[1])\to  H(w_{\ge l+1}M)\to H(M)\\
\to H(w_{\le l}M)\to H((w_{\ge l+1}M)[-1])\to\dots
\end{gathered}
\end{equation}
The exactness of this sequence in $H(M)$ for $l=n$ immediately gives the first part of assertion  \ref{iwrvan}. Next, the second
part is immediate from the first one combined with assertions \ref{iwfil4} and \ref{iwrd}.

\ref{iwrpure}. It certainly suffices to verify the statement for $m=0$. Now, the functor $H^{\ca,op}$ is easily seen to be of weight range $[0,0]$ for any additive functor $\ca:\cu\to \au^{op}$; see Remark \ref{rwc}.  

Conversely, let $H$ be of weight range $[0,0]$. We take the functor $\ca$ to be the restriction of $H^{op}:\cu\to \au^{op}$ to $\hw$. 
Then 
we should check that $H$ sends $M\in \obj \cu$ into the homology in $\ca(M^0)$ of the complex $\ca(M^{-*})$ (where $t(M)=(M^*)$). This is an immediate consequence of the properties of the {\it weight spectral sequence} converging to $H^*(M)$; see Theorem 2.4.2 of \cite{bws}. Indeed, this spectral sequence converges according to part II(ii) of this theorem, and it remains to apply the vanishing for $H(M^i[j])$ for $j\neq 0$. 

Assertion \ref{iwfil3} 
is immediate from  Proposition \ref{pbw}(\ref{iort}).

Assertion \ref{iwfil5} is an immediate consequence of assertion 
\ref{iwrvan}; just apply (\ref{eles}) for $l=m$ and for $l=m-1$, respectively.
\end{proof}

\begin{rema}\label{rwrange}
1. So, we call cohomological functors of weight range $[0,0]$ and their opposite homological functors 
$w$-pure ones; this terminology is compatible with Remark \ref{rpure}(2) according to part \ref{iwrpure} of our proposition.

2. Actually, the arguments used in the proof of this statement are easily seen to be functorial enough to  
yield an equivalence of the  (possibly) big category 
 of pure (cohomological) functors $\cu\to \au$ with the one of additive contravariant functors $\hw\to \au$.

3. Sending $H$ into the pure functor $\tau^{\ge -m}(\tau^{\le -m})(H\circ [m])\cong \tau^{\le -m}(\tau^{\ge -m})(H\circ [m])$ (for $m\in \z$) 
yields a sort of "pure homology" for $H$. It will correspond to the homology of an object representing (or, more generally, {\it $\Phi$-representing}) 
$H$ in the settings that we will consider 
below.

4. Certainly, a homological functor $H$ from $\cu$ into $\au$ may be considered as a cohomological functor from $\cu$ into $\au^{op}$. Thus one can easily dualize the aforementioned results;  we give some more detail  for this here to refer to them later. 


Obviously, 
 the virtual $t$-truncation functor $\tau^{\ge -n }(H)$ (resp. $\tau^{\le -n }(H)$)  will be defined  by the correspondence $M\mapsto\imm (H(w_{\le n}M)\to  H(w_{\le n+1}M)) $ 
(resp. $M\mapsto\imm (H(w_{\ge n}M) \to H(w_{\ge n-1}M)$), whereas the arrows in 
(\ref{evtt}) should be reversed.

5. The author is using the term "pure" due to the relation of pure functors to Deligne's purity of cohomology. 

To explain it we recall that various categories of Voevodsky motives are endowed with so-called Chow weight structures; the first of these weight structures was constructed in  \cite{bws} where it was proved that the category $\dmgm(k)$ of geometric motives over a characteristic zero field $k$ is endowed with a weight structure $w_{\chow}(k)$ whose heart is the category of Chow motives over $k$.\footnote{This result was extended to the case where $k$ is a perfect field of characteristic $p>0$ in \cite{bzp}; note however that one is forced to invert $p$ in the coefficient ring in this setting.}\ Now, for any $r\in \z$ the $r$th level of the Deligne's weight filtration of either of singular of \'etale cohomology of motives certainly kills $\chow[i]$ for all values of $i$ except one (and the remaining value of $i$ is either $r$ or $-r$ 
depending on the choice of the convention for Deligne's weights).\footnote{Certainly, singular (co)homology (of motives) is only defined if $k$ is a subfield of complex numbers; then it is endowed with Deligne's weight filtration that can also be computed using $w_{\chow}(k)$ (see Remark 2.4.3 of \cite{bws}). On the other hand, Deligne's weight filtration for \'etale (co)homology can be defined (at least) for $k$ being any finitely generated field; the comparison of the corresponding weight factors  with the ones computed in terms of $w_{\chow(k)}$ is carried over in Proposition 4.3.1 of \cite{bkl}. Note also that in ibid. and in \cite[\S3.4,3.6]{brelmot} certain "relative perverse" versions of these weight calculations were discussed.}\  Thus  
(the corresponding shifts of) Deligne's pure factors of (singular and \'etale) cohomology are pure with respect to $w_{\chow(k)}$.
\end{rema}

We also formulate a simple statement for the purpose of applying it in \cite{bsnew}.

\begin{pr}\label{pcrivtt}
For $M\in \obj \cu$ the following conditions are equivalent.

(i) $M\in \cu_{w\ge 0}$. 

(ii) $H(M)=0$ for any cohomological $H$ from $\cu$ into (an abelian category) $\au$ that is of weight range $\le -1$.

(iii) $(\tau^{\ge  1}H_N)(M) =\ns$ for any $N\in \obj \cu$.

(iv) $(\tau^{\ge 1}H_M)(M) =\ns$.

\end{pr}
\begin{proof}
Condition (i) implies condition  (ii) by definition; certainly, (iii)   $\implies$ (iv). Next, condition (ii) implies condition (iii) according to Proposition  \ref{pwrange}(\ref{iwrvt}).

Now we prove that condition (iv) implies that $M\in  \cu_{w\ge 0}$. We consider the commutative diagram
\begin{equation}\label{epr} \begin{CD} w_{\le -1} M@>{c_{-1}}>> M@>{}>> w_{\ge  0}M\\
@VV{h}V@VV{\id_M}V@ VV{}V \\
w_{\le 0} M@>{c_0}>> M@>{}>> w_{\ge 1}M \end{CD}
\end{equation} 
given by Proposition \ref{pbw}(\ref{icompl})) (we take $g=\id_M$ in the proposition).
Since $(\tau^{\ge 1}H_M)(M) =\ns$, $c_0\circ h=0$; hence $c_{-1}=0$ also.  Hence the upper distinguished triangle in (\ref{epr}) yields that $M$ is a retract of 
$w_{\ge  0}M\in \cu_{w\ge 0}$.

\end{proof}

\subsection{On the relation of  smashing weight structures to cc and cp functors} 
\label{sprcoprod}


We prove a collection of properties of smashing weight structures (see Definition \ref{dhop}).

\begin{pr}\label{ppcoprws}
Let $w$ be a  smashing weight structure (on $\cu$), $i,j\in \z$; let $\au$ be an  AB4 
abelian category, and $\au'$ be an AB4* abelian category.
Then the following statements are valid. 

\begin{enumerate}
\item\label{icopr1} The classes $\cu_{w\le j}$, $\cu_{w\ge i}$, and $\cu_{[i,j]}$ are closed with respect to $\cu$-coproducts. 

\item\label{icoprhw} In particular, the category $\hw$ is closed with respect to $\cu$-coproducts, and the embedding $\hw\to \cu$ respects coproducts. 

\item\label{icopr2} 
Coproducts of 
weight decompositions are weight decompositions.

\item\label{icopr3}  Coproducts of weight Postnikov towers are weight Postnikov towers.

\item\label{icopr4} The categories $\cuw$ and $\kw(\hw)$ are closed with respect to coproducts, and the functor $t$ respects coproducts.

\item\label{icopr5} Pure functors $\cu\to \au$ respecting coproducts are exactly the functors of the form  $H^\ca$ (see Proposition \ref{ppure}), where $\ca:\hw\to \au$ is an additive functor respecting coproducts. Moreover, this correspondence is an equivalence of  (possibly, big) 
 categories. 


\item\label{icopr6} If $H:\cu\to \au$ is 
a cc functor (see Definition \ref{dcomp}(\ref{idcc})) then   $\tau^{\ge i }(H)$  and $\tau^{\le i }(H)$  are cc functors also.

\item\label{icopr6p} 
$H':\cu^{op}\to \au'$ is a cp functor then   $\tau^{\ge  i }(H')$  and $\tau^{\le i }(H')$  are cp functors also.

\item\label{icopr5p} Pure cohomological  (see Definition \ref{drange}(2)) cp functors  from $\cu$ into $\au'$ are exactly those of  the form $H_\ca$ for $\ca:\hw^{op}\to \au'$ being an additive functor 
that sends $\hw$-coproducts into products.

\item\label{icopr7} Let $\du$ be the localizing subcategory of $\cu$ generated by a class of objects $\{D_l\}$, and assume that for any of the $D_l$ a choice of (the terms of)   its weight complex $t(D_l)=(D_l^k)$  is fixed. Then any element of $\cu_{w=0}\cap \obj \du$
is a retract of a coproduct of a family  of $D_l^k$.

\item\label{icopr7p} 
For  $\{D_l\}$ and $\du$  as in the previous assertions assume that   for any of the $D_l$   a choice of  $w_{\le k}D_l$ and of $w_{\ge k}D_l$ for $k\in \z$  is fixed (we do not assume any relation between these choices). Then  for any $D\in\obj \du$ and any $m\in \z$ there exists a choice of $(w_{\le m}D)[-m]$ (resp. of $(w_{\ge m}D)[-m]$) belonging to the smallest coproductive extension-closed subclass $D_1$ (resp. $D_2$)  of $\obj \cu$ containing   $(w_{\le k}D_l)[-k]$ (resp.  $(w_{\ge k}D_l)[-k]$) for all $l$ and all $k\in\z$. Moreover,   $\cu_{w\le 0}\cap \obj \du$ and $\cu_{w\ge 0}\cap \obj \du$ lie in $D_1$ and $D_2$, respectively.
		
		Furthermore, if $\al$ is a regular cardinal and $D$ belongs to  the smallest triangulated category of $\cu$ that contains $D_l$ and closed under $\cu$-coproducts of less than $\al$ objects then all $(w_{\le m}D)[-m]$ can be chosen to belong to the smallest  extension-closed subclass of $\obj \cu$ containing   $(w_{\le k}D_l)[-k]$  and closed under coproducts of less than $\al$ objects.

\end{enumerate}
\end{pr}
\begin{proof}

\begin{enumerate}
\item This is essentially a particular case of Proposition \ref{phop}(\ref{itp2},\ref{itp3}); see Remark \ref{rwhop}(1). 

\item Immediate from the previous assertion.

\item Recalling Remark \ref{rwhop}(1) once again, we reduce the statement to Proposition \ref{phop}(\ref{itp4}).

\item Immediate from the previous assertion.

\item Immediate from assertions  \ref{icopr1} and \ref{icopr3}.

\item It certainly follows from Proposition \ref{pwrange}(\ref{iwrpure}) that pure functors are exactly those of the type $H^\ca$.
Since $\ca$ is the restriction of  $H^\ca$ to $\hw$, this (restriction) correspondence is functorial. 
Next,  
 if $H$ respects coproducts then its restriction to $\hw$ also does according to  assertion  \ref{icoprhw}. Conversely, if $\ca$ respects coproducts then $H^\ca$ also does according to assertion \ref{icopr4}.

\ref{icopr6}, \ref{icopr6p}. Immediate from assertion \ref{icopr2}.

\ref{icopr5p}. Similarly to assertion \ref{icopr5}, $\ca$ is the restriction of  $H_\ca$ to $\hw$; hence it sends coproducts into products according to  assertion  \ref{icoprhw}. Conversely, if $\ca$  sends coproducts into products then $H^\ca$ also does according to assertion \ref{icopr4}.

\ref{icopr7}. Combining assertion \ref{icopr4} with Proposition \ref{pwt}(\ref{iwcex}) we obtain that for any $M\in \obj \du$ there exists a choice of $t(M)$ all of whose terms are coproducts of $D_l^k$. 

Next, assume that   $M$ also belongs to $\cu_{w=0}$ and choose $t(M)$ so that $M^0$ is some coproduct of $D_l^k$.  Since $M^0$ is a choice of $w_{\ge 0}(w_{\le 0}M)$ (essentially by the definition of weight complexes), we obtain that $M$ is a retract of $M^0$  according to  Proposition \ref{pbw}(\ref{iwdmod}) (applied twice). 


\ref{icopr7p}.  Combining assertion \ref{icopr2} with Proposition \ref{pbw}(\ref{iwdext}) we obtain that the class  $C$ of those $D\in \obj \cu$ such that for any $m\in \z$ there exist a choice of $(w_{\le m}D)[-m]$ and of $(w_{\ge m}D)[-m]$ belonging to  $D_1$ (resp. to $D_2$)  is a coproductive  class of objects of a full triangulated subcategory of $\cu$ (i.e., there exists a triangulated $\cupr\subset \cu$ such that $\obj \cupr=C$ and $C$ is coproductive). Thus $C$ contains $\obj \du$.

 Next, if $M$ belongs to $ \obj \du \cap \cu_{w\le 0}$ (resp. to $ \obj \du \cap \cu_{w\ge 0}$) then the existence of $w_{\le 0}M$ belonging to $D_1$ (resp. of $w_{\ge 0}M$ belonging to $D_2$) implies that $M$ belongs to the Karoubi-closure of $D_1$ (resp. of $D_2$) according to  Proposition \ref{pbw}(\ref{iwdmod}). Thus to prove the "moreover" part of the assertion  it remains to note that $D_1$ and $D_2$ are Karoubian according to Remark \ref{rcoulim}(4) below. 

The proof of the "furthermore" part of the assertion is similar.

\end{enumerate}
\end{proof}

\begin{rema}\label{ral}
In all the part of our proposition 
 one can replace arbitrary small coproducts by coproducts of less than $\al$ objects in all occurrences (where $\al$ is any regular infinite cardinal). 
  In particular, in assertions \ref{icopr4} and \ref{icopr6p}  one can assume that $\cu$ and $\cu_{w\le 0}$ are closed with respect to $\cu$-coproducts of less than $\al$ of their objects; then $t$ respects these coproducts also, and virtual $t$-truncations of cohomological functors that convert coproducts of less than $\al$ objects into the corresponding  products fulfil this condition as well.\footnote{Certainly, the AB4* condition for the target category $\au'$ can also be weakened respectively.} 
\end{rema}

Part \ref{icopr6p}  of our proposition immediately implies the following corollary that will be important for us below. 

\begin{coro}\label{cvttbrown}
Assume that  $w$ is smashing.

 I.1. If $\cu$ satisfies the Brown representability condition (see Definition \ref{dcomp}(\ref{idbrown})) then 
virtual $t$-truncations of representable functors are representable.

2. If  $\cu$ is generated by a set of objects as its own localizing subcategory\footnote{Certainly, 
in this case $\cu$ is also generated by the coproduct of these objects (as its own localizing subcategory).} then $\hw$ has a generator, i.e., there exists $P\in \cu_{w=0}$ such that any object of $\hw$ is a retract of a coproduct of (copies of) $P$.

II. Let $F:\cu\to \cu'$ be an exact functor respecting coproducts. Adopt the notation of Proposition \ref{ppcoprws}(\ref{icopr7p}) and let $w'$ be a weight structure for $\cu'$; 

1. Then $F$ is left (resp. right) weight-exact if and only if  $F(w_{\le k}D_l)[-k]\in \cu'_{w'\le 0}$ (resp. $F(w_{\ge k}D_l)[-k]\in \cu'_{w'\ge 0}$) for all $l$ and all $k\in \z$.

2. Assume that all $D_l$ are $w$-bounded; let $\cp\subset \cu_{w=0}$ be a generating class for $\hw$ (i.e.,  any    object of $\hw$ is a retract of a coproduct of elements of $\cp$). Then $F$ is left (resp. right) weight-exact if and only if $F(\cp)\subset  \cu'_{w'\le 0}$ (resp.  $F(\cp)\subset  \cu'_{w'\ge 0}$). 


\end{coro}
\begin{proof}
I.1. The Brown representability condition says that $\cu$-representable functors are precisely all cp functors from $\cu$ into $\ab$. Hence the statement follows from Proposition \ref{ppcoprws}(\ref{icopr6p})  indeed.

2. Just take $\du=\cu$ in part \ref{icopr7} of the proposition; then we can take $P$ to be the coproduct of the corresponding $D_l^k$ (recall that $\hw$ has coproducts!). 

II. If $F$ is left (resp. right) weight-exact then the images of all elements of $\cu_{w\le 0}$ (resp. of $\cu_{w\ge 0}$) belong to $\cu'_{w'\le 0}$ (resp. to $\cu'_{w'\ge 0}$), and we obtain "one half" of the implications in question. 

The converse implication for assertion II.1 follows from Proposition \ref{ppcoprws}(\ref{icopr7p}) immediately.

 Now we check the converse implication for assertion II.2. Since $F$ respects coproducts and $F(\cp)\subset  \cu'_{w'\le 0}$ (resp.  $F(\cp)\subset  \cu'_{w'\ge 0}$), we obtain $F(\cu_{w=0})\subset  \cu'_{w'\le 0}$ (resp.  $F(\cu_{w=0})\subset  \cu'_{w'\ge 0}$). Thus Proposition \ref{pbw}(\ref{igenlm}) yields the result easily.

\end{proof}

\begin{rema}\label{rgenw}
1. 
Note that our definition of a generator for $\hw$ is much more "restrictive" than the assumption  that $\{P\}$ Hom-generates $\hw$ (cf. Remark \ref{rcompgen}) and even than  the (usual) "abelian version" of the definition of generators (cf. Theorem \ref{tab5}(1)). 
Moreover, it is much easier to specify generators in $\hw$ (if $w$ is smashing) than in $\hrt$.

2. A certain "finite dimensional" analogue of part I.1 of this corollary is given by Proposition \ref{psatur}(1) below.
\end{rema}

\section{On adjacent weight and  $t$-structures 
and Brown representability-type conditions}\label{sadjbrown}

In this section we study some general conditions ensuring that a torsion pair  admits a (right or left) torsion pair.

In \S\ref{sadjt} we prove that a 
 weight structure admits a right adjacent $t$-structure if and only if the virtual $t$-truncations of representable functors are representable. It easily follows that in a triangulated categories satisfying the Brown representability condition a weight structure admits a right adjacent $t$-structure if and only if it is smashing. Certainly, the  dual to this statement is also true. Moreover, a similar argument demonstrates that if the representable functors from an $R$-linear category $\cu$ are precisely the {\it $R$-finite type} ones (i.e., if $\cu$ is {\it $R$-saturated}) then all bounded weight structures on $\cu$ admit right adjacent $t$-ones. Note here that if $R$ is Noetherian then for $X$ being a 
  regular separated finite-dimensional scheme that is proper over $\spe R$  its bounded derived category of coherent sheaves (as well as its dual) is $R$-saturated according to a  recent result of Neeman. 

In \S\ref{sadjw} we study when  a $t$-structure $t$ admits a (left or right) adjacent weight structure $w$; 
 however, the results of this section are "not as nice" as their "mirror" ones in \S\ref{sadjt} (at least, in the case where $\cu$ has coproducts; cf. Remark \ref{rnondeg}). 

In \S\ref{scomp} we recall the notions of perfectly generated and well generated triangulated categories along with their (Brown representability) properties; we also relate perfectness to smashing torsion pairs. 
 Next we define symmetric classes and study their relation to perfect classes, Brown-Comenetz duality,  and adjacent torsion pairs (obtaining a new criterion for the existence of the latter). 
 This gives  one more "description" of a $t$-structure that is right adjacent to a given compactly generated weight structure. 

\subsection{On  the existence of adjacent  $t$-structures}\label{sadjt}


\begin{pr}\label{phadj}
Let $w$ be left adjacent to a $t$-structure $t$ on $\cu$. Then  $\hrt$ is a full exact subcategory of the (possibly, big) abelian category $\adfu(\hw^{op},\ab)$ (see \S\ref{snotata}).  

Moreover, $\cu_{w=0}=P_t$ and the functor $H^P=\cu(P,-)$  is isomorphic to $\hrt (H_0^t(P), H_0^t(-))$ for any $P\in \cu_{w=0}$.

\end{pr}
\begin{proof}
The first part of the assertion is given by part 4 of   \cite[Theorem 4.4.2]{bws}. 

The equality $\cu_{w=0}=P_t$   is immediate from  Remark \ref{rtst1}(\ref{it1}) (recall Definition \ref{dwso}(\ref{idadj})). 

The last of the assertions is given by Proposition \ref{pgen}(\ref{ipgen2}).
\end{proof}

As we have essentially already noted (see Proposition \ref{phop}(\ref{itp5})), if a weight structure possesses a left adjacent $t$-structure then it is coproductive (i.e., the associated torsion pair is coproductive). Now we prove that 
the converse implication is valid also if $\cu$ satisfies the Brown representability property (in particular, if it is compactly generated or {\it perfectly generated} in the sense of   Definition \ref{dwg}(\ref{idpc}) below).

\begin{theo}\label{tadjt} Assume that $\cu$ 
satisfies the Brown representability condition. 

1. Then for a weight structure $w$ on $\cu$ there exists a $t$-structure right adjacent to it if and only if $w$ is smashing. 

2. If a right adjacent $t$ exists then $\hrt$ is equivalent to the full subcategory of  $\adfu(\hw^{op},\ab)$ consisting of those functors that sends $\hw$-coproducts into products.
\end{theo}
\begin{proof}

1. The "only if" assertion is immediate from Proposition \ref{phop}(\ref{itp5}) (and very easy for itself).

Conversely,  assume that $w$ is smashing.  According to Proposition \ref{padjt}(\ref{ile4}) below it suffices to verify that for any $M\in \obj \cu$ the functor $\tau^{\le 0}H_M$ is representable (in $\cu$; recall that $H_M$ denotes the functor $\cu(-,M)$). The latter statement is given by Corollary \ref{cvttbrown}(I.1).

2. Applying Proposition \ref{phadj} we obtain that is suffices to find out which functors $\hw\opp\to \ab$ are represented by objects of $\hrt$. Since the embedding $\hw\to \cu$ respects coproducts (see Proposition \ref{ppcoprws}(\ref{icopr1})), all these functors send $\hw$-coproducts into $\ab$-products.

Conversely, let  $\ca:\hw\opp\to \ab$ be an additive functor converting coproducts into products. Then the corresponding $H_\ca$ is a cp functor (see Definition \ref{dcomp}(\ref{idcc})) according to Proposition \ref{ppcoprws}(\ref{icopr5p}). Hence it is representable by some $M\in \obj \cu$. Since $H_\ca$ is also of weight range $[0,0]$ (see Proposition \ref{pwrange}(\ref{iwrpure})), we have $M\in \cu^{t=0}$. 

\end{proof}
 
\begin{rema}\label{rstable}

1. Proposition \ref{pwsym}(\ref{iwsymcgwt}) below gives some more information on the adjacent weight structure $t$ whenever $w$ is  compactly generated (see Remark \ref{rwhop}(1)). 

 2. Recall from Proposition \ref{prtst}(\ref{itt4}, \ref{it4sm}) 
 that "shift-stable" weight structures are in one-to-one correspondence with exact embeddings $i:L\to \cu$ possessing  
  right adjoints.  Hence applying our theorem  in this case we obtain the following: if 
 $i$ possesses a right adjoint respecting coproducts and $\cu$ satisfies 
the Brown representability condition then for the full triangulated subcategory $R$ of $\cu$ with $\obj R=L\perpp$ the embedding $R\to \cu$ possesses a right adjoint also. 
Thus $R$ is {\it admissible} in $\cu$ in the sense of  \cite{bondkaprserr} and the embedding $R\to\cu$ may be completed  to a {\it gluing datum} (cf. \cite[\S1.4]{bbd} or \cite[\S9.2]{neebook}). 
 So we re-prove Corollary 2.4 of \cite{nisao}.
\end{rema}

\begin{coro}\label{cdualt}
Let $\cu$ be a category satisfying the dual Brown representability property (recall that this is the case if $\cu$ is compactly generated; see  Proposition \ref{pcomp}(II.1)). 

1. Then for a weight structure $w$ on $\cu$ there exists a $t$-structure left adjacent to it if and only if $w$ is cosmashing. 

2. If  
 these equivalent conditions are fulfilled then $\hrt$  is anti-equivalent to the subcategory of  $\adfu(\hw,\ab)$ consisting of those functors that respect products.

\end{coro}
\begin{proof}
This 
 is just the categorical dual to Theorem \ref{tadjt}.
\end{proof}

Now we re-formulate the existence of a $t$-structure right adjacent to $w$ in terms of virtual $t$-truncations; this finishes the proof of  Theorem \ref{tadjt}.

\begin{pr}\label{padjt}
Let $w$ be a weight structure for $\cu$, $M\in \obj \cu$, and assume that for the functor $H_M=\cu(-,M)$ its virtual $t$-truncation $\tau^{\le 0}H_M$ is represented by some object $M^{\le 0}$ of $\cu$. 

Then the following statements are valid.
\begin{enumerate}
\item\label{ile1} $M^{\le 0}$ belongs to $\cu_{w\ge 0}$.

\item\label{ile2}
The natural transformation $\tau^{\le 0 }(H_M)\to H_M$ mentioned in (\ref{evtt}) is induced by some $f\in \cu(M^{\le 0},M)$. 

\item\label{ile3} The object $\co(f)$ belongs to $\cu_{w\ge 0}^{\perp}$.

\item\label{ile4} There exists a $t$-structure $t$ (on $\cu$) right adjacent to $w$ if and only if  the functor  $\tau^{\le 0}H_{M'}$ is $\cu$-representable for any object $M'$ of $\cu$. 

\item\label{ile5} 
For $M'\in \obj \cu$ the representability of the  functor  $\tau^{\le 0}H_{M'}$ is equivalent to that of  $\tau^{\ge 1}H_{M'}$.

\end{enumerate}

\end{pr}
\begin{proof}
1.  $\tau^{\le 0}H_M$ is of weight range $\ge 0$ (see Proposition \ref{pwrange}(\ref{iwrvt})). Hence the assertion follows from part \ref{iwfil3} of the same proposition.

2. Immediate from the Yoneda lemma.

3. For any $N\in \obj \cu$ applying the functor $H^N=\cu(N,-)$ to the distinguished triangle $M^{\le 0}\to M \to M^{\ge 1} \to M^{\le 0}[1]$ one obtains a long exact sequence that yields the following short one: \begin{equation} \label{eshort}
\begin{gathered} 0\to \cok(H^N(M^{\le 0})\stackrel{h^1_N}{\to} H^N(M))\to H^N(\co(f)) \\
 \to \ke (H^N(M^{\le 0}[1])\stackrel{h^2_N}{\to} H^N(M[1]))\to 0. 
\end{gathered}
\end{equation} 
So, for any $N\in \cu_{w\ge 0}$ we should check that $h^1_N$ is surjective and $h^2_N$ is injective.

Applying  (\ref{evtt}) to the functor $H_M$ (in the case $n=0$) 
we obtain a long exact sequence of functors 
\begin{equation} \label{evttp} 
\dots\to \tau^{\le 0 }(H_M)\to H_M\to \tau^{\ge 1}(H_M)\to \tau^{\le 0}(H_M)\circ [1]\to H_M\to \dots
\end{equation}
 Applying this sequence of functors to $N$ we obtain that the surjectivity of  $h^1_N$ along with the injectivity of $h^2_N$ is equivalent to  $\tau^{\ge 1}(H_M)(N)=\ns$. So, recalling Proposition \ref{pwrange}(\ref{iwrvt}) once again (to obtain that   $\tau^{\ge 1}(H_M)$ is of weight range $\le -1$) we conclude the proof.

4. The   "only if" part of the assertion is immediate from Proposition \ref{pwfil}(4).

To prove the converse implication we should check that the couple $(\cu_{w\ge 0}, \cu_{w\ge 0}^{\perp}[1])$ is a $t$-structure if our representability assumption is fulfilled. It is easily seen that the only non-trivial axiom check here is the existence of $t$-decompositions (see Definition \ref{dtstr}), which is given by the previous assertion.

5. If  $\tau^{\le 0}H_{M'}$ is representable then the previous assertions imply the existence of a distinguished triangle $M'^{\le 0}\to M' \to M'^{\ge 1} \to M'^{\le 0}[1]$ with $M'\in \cu_{w\ge 0}^{\perp}$. Then the object $M'^{\ge 1} $ represents the functor  $\tau^{\ge 1}H_{M'}$ according to Theorem 2.3.1(III.4) of \cite{bger} (and so, $\tau^{\ge 1}H_{M'}$ is representable). The proof of the converse implication is similar.
\end{proof}

Now we describe one more application of Proposition \ref{padjt}. It relies on a modified version of the Brown representability property that we will now define.

\begin{defi}\label{dsatur} 
Let $R$ be an associative commutative unital ring, and assume that $\cu$ is $R$-linear.

1. We will say that an $R$-linear cohomological functor $H$ from $\cu$ into $R-\modd$ is  {\it of $R$-finite type} whenever for any $M\in \obj\cu$ the $R$-module $H(M)$ is finitely generated and $H(M[i])=\ns$ for almost all $i\in \z$.  

2. We will say that $\cu$ is {\it $R$-saturated} if the  representable functors from $\cu$ are exactly all the  $R$-finite type ones.

3. The symbol $\adfur(C,D)$ will denote the (possibly, big) category of $R$-linear (additive) functors from $C$ into $D$ whenever $C$ and $D$ are $R$-linear categories.   

4. We will write $R-\mmodd$ for the category of finitely generated $R$-modules.

\end{defi}

\begin{pr}\label{psatur}
Assume that $\cu$ is $R$-linear and endowed with a bounded weight structure $w$.

I. Then all virtual $t$-truncations of functors of $R$-finite type are of $R$-finite type also.

II. Assume in addition that $\cu$ is $R$-saturated.
Then the following statements are valid.

1. For any $i\in \z$ and $M\in \obj \cu$ the functors $\tau^{\le -i }(H_M)$  and   $\tau^{\ge - i }(H_M)$   are representable.

2. There exists a $t$-structure right adjacent to $w$.

3. Its heart $\hrt$  naturally embeds  into the category 
$\adfur(\hw\opp,R-\mmodd)$.
This embedding is essentially surjective 
 whenever $R$ is noetherian. 
\end{pr}
\begin{proof}
I.  Recall that virtual $t$-truncations (of cohomological functors) are cohomological. Moreover, virtual $t$-truncations of $R$-linear functors are obviously $R$-linear.

Now, let $H$ be a functor of $R$-finite type. It obviously follows  that the values of  $\tau^{\ge  0}(H)$ are finitely generated $R$-modules.
Next, for any $N\in \obj \cu$ we can take $w_{\ge 0}(N[j])$ to be $0$ for $j$ small enough and to be equal to $N[j]$ for $j$ large enough (recall that $w$ is bounded);  
hence the functor $\tau^{\ge  0}(H)$ is of $R$-finite type. Applying this argument to $H\circ [-i]$ (for $i\in \z$) we obtain that the functor $\tau^{\ge - i}(H)$ is of $R$-finite type also. The proof for $\tau^{\le - i }(H)$ is similar.

II.1. Immediate from assertion I combined with the definition of saturatedness. 

2. According to  Proposition \ref{padjt}(\ref{ile4}), the assertion follows from the previous one.

3. Certainly, restricting functors 
of  $R$-finite type from 
 $\cu$ to $\hw$ gives functors of the type described. This restriction gives an embedding of $\hrt$ into 
 $\adfur(\hw\opp,R-\mmodd)$ according to Proposition \ref{phadj}. Lastly, if a functor $A$ 
belongs to $\adfur(\hw\opp,R-\mmodd)$  and $R$ is noetherian then the pure functor $H_\ca$ (see Definition \ref{drange}(2)) is easily seen to be of $R$-finite type (as a cohomological functor from $\cu$ into $R-\mmodd$), whereas the object representing it belongs to $\cu^{t=0}$ according to Proposition \ref{pwrange}(\ref{iwrpure}).
\end{proof}

\begin{rema}\label{rsatur}
1. Note  now that Corollary 0.5 of \cite{neesat} easily implies that $\cu$ is $R$-saturated whenever $R$ is a noetherian ring, $X$ is a  regular 
 scheme that is proper over $\spe R$, and $\cu=D^b(X)$ (the bounded derived category of coherent sheaves). Indeed, the 
 assumptions on $X$ imply that in this case the derived category of perfect complexes of sheaves equals $\cu$, and then loc. cit. gives the result immediately (cf. Remark \ref{roq}(1) below). 

Moreover,  in this case $\cu\cong \cu\opp$ (since there exists a dualizing complex on $X$). 
Thus $t$-structures left adjacent to bounded weight structures on $D^b(X)$ exist also. 

We recall also that in the case where $R$ is a field the saturatedness is question is given by Corollary 3.1.5 of \cite{bvdb}. Moreover, Theorem 4.3.4 of 
 ibid. is a certain a "non-commutative geometric" analogue of this statement. 

2. 
Now we discuss possible weight structures on the category $\cu=D^b(X)$ (as above).

We recall that bounded weight structures for $\cu$ are determined by their hearts (see Proposition \ref{pbw}(\ref{igenlm})), whereas the latter are precisely all the  negative additive Karoubi-closed subcategories of $\cu$ that densely generate it (see 
Corollary 2.1.2 of \cite{bonspkar}).

Moreover, the author suspects that in the "geometric" examples mentioned above there exists a single "dense" generator $P$ of $\hw$, i.e., all objects of $\hw$ are retracts of (finite) powers of $P$.

One can construct a rich family of bounded weight structures on $D^b(X)$ at least in the case where $X=\p^n$  (for some $n>0$; the author does not know whether it is necessary to assume that $R$ is a field here) since one can use gluing for constructing weight structures (cf. \cite[\S8.2]{bws}).  More generally, it suffices to assume that $D^b(X)$ possesses a {\it full exceptional collection} of objects.  

Note however that degenerate weight structures are certainly possible in triangulated categories of this type (cf. Proposition \ref{prtst}(\ref{itt4})); so, these weight structures are not bounded. Still the author  conjectures that the boundedness restriction is not actually necessary (the proof may rely on the existence of a {\it strong generator} in the sense of \cite{bvdb} for $\cu$).

3. The only examples of  $R$-saturated triangulated categories 
for a non-noetherian $R$  known to the author are direct sums of $R/J_i$-saturated categories, where $\{J_i\}$ 
 is a finite collection of ideals of $R$ and all $R/J_i$ are noetherian. So, a general $R$ was mentioned in the proposition just for the sake of generality. 
Our definition of 
 $R$-saturatedness  may be "not optimal"; 
 it could make sense to put  finitely presented modules instead of finitely generated ones into the definition. Note however that any length two $\hw$-complex is a weight complex of some object of $\cu$. This allows to describe $\hrt$  completely (cf. Proposition \ref{psatur}(II.3)) for any possible definition of $R$-saturatedness (and can possibly help choosing among the definitions).

4. Actually, all 
of the statements of this paper may easily be proved in the $R$-linear context (cf. \cite{zvon}). 
This is formally a generalization (since one can take $R=\z$); yet Propositions  \ref{psatur} and   \ref{psaturdu} (along with the examples to them) appear to be the only statements for which this setting is really actual. 

Note also that below we define Brown-Comenetz duals (of objects and functors) "using" the group $\q/\z$. However, the only property of this group that we will actually apply is that it is an injective cogenerator of the category $\ab$. 
Hence for an $R$-linear category $\cu$ one can replace our Definition \ref{dsym}(\ref{ibcomf},\ref{ibcomo}) below by any its "$R$-linear analogue"; this 
  may make sense if $R$ is a field.
\end{rema}

\subsection{On adjacent  weight structures}\label{sadjw}

In this subsection we will assume that $\cu$ is endowed with a $t$-structure $t$.
For the construction of certain weight structures we will need the following statement.

\begin{lem}\label{lconstws}
Assume that certain extension-closed 
classes $\cu_-$ and $\cu_+$ of objects of $\cu$  satisfy 
 the axioms (i)---(iii) of Definition \ref{dwstr} (for $\cu_{w\le 0}$ and $ \cu_{w\ge 0}$, respectively).
Let us call a $\cu$-distinguished triangle $X\to M\to Y[1]$ a {\it pre-weight decomposition} of $M$ if $X$ belongs to $\cu_-$ and $Y$ belongs to   $\cu'_+$.

Then the following statements are valid.

1. The class $C$ of objects possessing pre-weight decompositions is extension-closed (in $\cu$).

2. Assume that $C$ contains a subclass $\cp$ such that $\cp=\cp[1]$. Then $C$ also contains the object class of the smallest strict triangulated  subcategory of $\cu$ containing $\cp$. 

3. Assume that $\cu$ has coproducts, and  $\cu_-$ and $\cu_+$ are coproductive. Then $C$ is coproductive also.
\end{lem}
\begin{proof}
1. Immediate from Theorem 2.1.1(I.1)  of \cite{bonspkar}  (cf. also Remark 1.5.5(1) of \cite{bws}).

2. Immediate from assertion 1.

3. Once again, it suffices to recall Proposition \ref{pcoprtriang}. 
\end{proof}

Now we  prove a simple statement on the existence of $w$ that is left adjacent to $t$.

\begin{pr}\label{pconstrwfromt}

I. Assume that there exists a weight structure $w$ left adjacent to   $t$. Then for any $M\in  \cu^{t=0}$ there exists an $\hrt$-epimorphism from the $\hrt$-projective object $H_0^t(P)$ into $ M$ for some $P\in P_t$,   and the functor $H_0^t$  induces  an equivalence of $\kar(\hw)$ with the category of projective objects of $\hrt$.  

II. The converse implication is valid under any of the following additional assumptions.

1.  $t$ is bounded above (see Definition \ref{dtstr}).  
 
2. There exists 
an integer $n$ such that $\cu^{t\ge n}\perp \cu^{t\le 0}$. 
 
\end{pr}
\begin{proof}

I. 
Fix $M\in \cu^{t=0}$ and consider its 
weight decomposition $P\stackrel{p}{\to} M\to M'\to P[1]$. Since $M\in \cu^{t\le 0}=\cu_{w\ge 0}$, we have $P\in \cu_{w=0}$ according to Proposition \ref{pbw}(\ref{iwd0})). Next, since $P\in \cu_{w\ge 0}=\cu^{t\le 0}$, the object  $P_0=t^{\ge 0}P$  equals  $H_0^t(P)$; hence $P_0$ is projective in $\hrt$ according to
Proposition \ref{pgen}(\ref{ipgen25}).

The adjunction property for the functor $t^{\ge 0}$ (see Remark \ref{rtst1}(\ref{it3})) implies that $p$ factors through the $t$-decomposition 
morphism $P\to P_0$.  Now we check that the corresponding morphism $P_0\to M$ is an $\hrt$-epimorphism. This is certainly equivalent to its cone $C$ belonging to $\cu^{t\le -1}$. The octahedral axiom of triangulated categories gives a distinguished triangle $(t^{\le -1}P)[1]\to M'\to C\to (t^{\le -1}P)[2]$; it yields the assertion in question since $M'\in \cu_{w\ge 1}=\cu^{t\le -1}$  and the class $\cu^{t\le -1}$ is extension-closed.

Next, the category of projective objects of $\hrt$ is certainly Karoubian.
According to Proposition \ref{phadj} it remains to verify that  for any projective object $Q$ of $\hrt$ there exists $R\in P_t$ such that $Q $ is a retract of $ H_0^t(R)$. Now, the first part of the assertion implies the existence of an $\hrt$-epimorphism $H_0^t(R')\to Q$ for some $R'\in P_t$. Since $H_0^t(R')$ is projective in $\hrt$, this epimorphism splits, i.e.,   $Q$ equals the image of some idempotent isomorphism of $H_0^t(R')$.  
Lifting this endomorphism to $\hw$ we obtain the result. 

II.1.  We set $\cu_{w\ge 0}=\cu^{t\le 0}$ and take $\cu_{w\le 0}$ to be the 
envelope of $\cup_{i<0} P_t[i]$.\footnote{Actually, this envelope also equals the extension-closure of $\cup_{i<0} P_t[i]$; see Proposition \ref{pbw}(\ref{igenlm}).}We should prove that this couple yields a weight structure for $\cu$, since this weight structure would certainly be adjacent to $t$. Now, this "candidate weight structure" obviously satisfies axioms (i) and (ii) in Definition \ref{dwstr}. Next, since $P_t[i]\perp \cu^{t\le 0}$ for any $i<0$; hence the orthogonality axiom (iii) is fulfilled also.

It remains to verify the existence of a weight decomposition for any $M\in \cu^{t\le i}$ by induction on $i$. The statement is obvious for $i< 0$ since $M\in \cu_{w\ge 1}=\cu^{t\le -1}$ and we can take a "trivial" weight decomposition $0\to M\to M\to 0$. 

Now assume that existence of $w$-decompositions is known for any $M\in \cu^{t\le j}$ for some $j\in \z$. We should verify the existence of weight decomposition of an element $N$ of $\cu^{t\le j+1}$. Certainly, $N$ is an extension of $N'[-j-1]=H_0^t(N[j+1])[-j-1]$   by $t^{\le j}N$ (see Remark \ref{rtst1}(\ref{it3}) for the notation). Since the latter object possesses a weight decomposition, Lemma \ref{lconstws}(1) allows us to verify the existence of a weight decomposition of $N'[-j-1]$ (instead of $N$). Now we choose a surjection $t^{=0}P\to N'$ whose existence is given by our assumptions. Then a cone $C $ of the corresponding composed morphism  $P\to N'$ is easily seen to belong to $t^{\le -1}$. Since both $P$ and $C$ possess weight decompositions, applying Lemma \ref{lconstws}(1)  once again we obtain the assertion in questions.

2. We take $\cu_{w\ge 0}=\cu^{t\le 0}$ and $\cu_{w\le 0}={}^{\perp}\cu^{t\le -1}$. Once again it suffices to verify the existence of a $w$-decomposition for an object $M$ of $\cu$.

We consider the (full) triangulated subcategory $\cu'$ consisting of $t$-bounded below objects, i.e., $\obj \cu'=\cup_{i\in \z}\cu^{t\le i}$. According to the previous assertion, any object of $\cu'$ possesses a weight decomposition with respect to the corresponding weight structure; thus it also possesses a $w$-decomposition.

Now, the $t$-decomposition of the object $M[n-2]$ yields a presentation of $M$ as extension of an element $M'$ of $\cu^{t\ge n-1}$ by an element $M''$ of $\cu^{t\le n-2}$. Since $M''\in \obj \cu'$, it possesses a $w$-decomposition. Next, our "extra" orthogonality assumption on $t$ yields that $M''\in \cu_{w\le 0}$; hence one take the triangle $M''\to M''\to 0\to M''[1]$ as a $w$-decomposition of $M''$. Lastly, applying 
 Lemma \ref{lconstws}(1)  once again we obtain that $M$ possesses a $w$-decomposition also.
\end{proof}

\begin{rema}\label{rexenproj}
1. Assume  that the category $\cu\opp$ is $R$-saturated (see Definition \ref{dsatur}; in particular, $\cu$ may equal the category $D^b(X)$ or $D^b(X)\opp$ for $X$ being  a regular separated finite-dimensional scheme that is proper over $\spe R$ for a Noetherian $R$)
	 and $t$ is a bounded above $t$-structure on $\cu$.
 Then our proposition  (combined with Proposition \ref{phadj}) easily implies that there exists a weight structure left adjacent to $t$ if and only if $\hrt$ has enough projectives (since the corresponding pure functors are corepresented by elements of $P_t$). 
 Moreover, in this case $\hw$ is equivalent to $\proj \au$.

2. 
The assumption of the existence of an $\hrt$-epimorphism $t^{=0}P\to M$ with $P\in P_t$ for any $M\in \cu^{t=0}$ naturally generalizes the condition of the existence of enough projectives 
that allows to relate the derived category of $\au$ to $K(\proj \au)$.
Note however that in this setting we have $\hw\subset \hrt$; this is not the case in general. 

Moreover, the condition $\cu^{t\ge n}\perp \cu^{t\le 0}$ for $n\gg 0$ is a natural generalization of the finiteness of the cohomological dimension condition (for an abelian category).

3. One can easily see that $P_t$ is {\it negative} 
 for any $t$. So our existence of  $w$ results are closely related to the statements on "constructing $w$ from a negative subcategory"; 
see \S2.2 of \cite{bsnew}, 
 \cite[\S4.3,4.5]{bws},  \cite[Corollary 2.1.2]{bonspkar}, and Remark \ref{rsatur}(2) above.

\end{rema}

\begin{theo}\label{tadjw}
Assume that $\cu$ has coproducts and is endowed with a $t$-structure $t$ such that its localizing subcategory $\cu'$ generated by $\cu^{t\le 0}$ satisfies the dual Brown representability condition.

1. Then there exists a weight structure left adjacent to $t$ if and only if $t$ is productive 
and the category $\hrt$ has enough projectives. 

2. If such a left adjacent $w$ exists then $\hw$ is equivalent to the subcategory of  projective objects of $\hrt$. 

\end{theo}
\begin{proof}
1. If $w$ exists then $t$ is productive according to  Proposition \ref{phop}(\ref{itp5}).  Next, the existence of enough projectives in $\hrt$ follows from 
  Proposition \ref{pconstrwfromt}(I). 

Conversely,  assume that $t$ is productive and the category $\hrt$ has enough projectives. Once again, for the "candidates" $\cu_{w\ge 0}=\cu^{t\le 0}$ and  $\cu_{w\le 0}=\perpp (\cu_{w\ge 0}[1])$ it suffices to verify the existence of a weight decomposition for any $Y\in \obj \cu$.

 For each projective object $P_0$ of $\hrt$ 
 Proposition \ref{pgen}(\ref{ipgen4})  gives the existence of $P\in P_t$ such that $H_0^t(P)\cong P_0$ (here $H^t_0$ is the $t$-homology on $\cu$; see Remark \ref{rtst1}(\ref{it4})).  Thus the existence of enough projectives in $\hrt$ is equivalent to the fact that for any  $M\in  \cu^{t=0}$ there exists an $\hrt$-epimorphism $H_0^t(P)\to M$ for $P\in P_t$.

We take $\cuz\subset \cu$ to be the triangulated category of $t$-bounded below objects (i.e., $\obj \cuz=\cup_{i\in\z}\cu^{t\le i}$). According to Proposition \ref{pconstrwfromt}(II.1), there exists a weight structure $w_0$ for $\cuz$ with $\cu_{\wz\ge 0}=\cu^{t\le 0}$.
Now we study the class $C$ of objects   possessing pre-weight decompositions with respect to $w$ in the terms of Lemma \ref{lconstws}.
The existence of $\wz$ certainly implies that  $C$ contains $\obj \cuz$. Applying parts 2 and 3 of the lemma we obtain that $C$ actually contains $\obj \cupr$.
On the other hand, since the class $C'=\perpp \obj \cupr$ is contained  in $\cu_{w\le 0}$, $C$ also contains $C'$. 

According to 
Lemma \ref{lconstws}(1), it remains to verify that any object of $\cu$ can be presented as an extension of an object of $\cu'$ by an element of $C'$. The latter is immediate from 
 Proposition \ref{pcomp}(II) combined with Proposition \ref{pbouloc}(III.\ref{ibou1}).

2. Similarly to the proof of Theorem \ref{tadjt}(II), Proposition \ref{phadj} gives an embedding of $\hw$ into the category of projective objects of $\hrt$. Hence the arguments used in the proof of assertion 1 (when $P$ was constructed from $P_0$) allow us to conclude the proof. 

\end{proof}

\begin{rema}\label{restrt}

1. It appears that one is "usually" interested in the case where $\cupr=\cu$. 

2. It can be easily seen that in Theorem \ref{tadjt} we could have replaced the Brown representability assumption for $\cu$ by that for the category $\cu'$ being its localizing category generated by $\cu_{w\ge 0}$ (and so also by $\cuz=\cup_{i\in \z} \cu_{w\le i}$). 

3. It is actually not necessary to assume that the whole $\cu$ has coproducts when defining $\cu'$ (in both of these settings). Indeed, it suffices to assume that $\cu^{t\le 0}$ is contained in some triangulated category $\cu''\subset \cu$ that has coproducts such that the embedding $\cu''\to \cu$ respects them. 

 \end{rema}

Certainly the dual to Theorem \ref{tadjw} is also valid; it is formulated as follows.

\begin{coro}\label{cadjw}
Assume that $\cu$ has products and is endowed with a $t$-structure $t$ such that its colocalizing subcategory $\cu'$ cogenerated by $\cu^{t\ge 0}$ (see Definition \ref{dcomp}(\ref{idloc})) satisfies the  Brown representability condition.

1. Then there exists a weight structure right adjacent to $t$ if and only if $t$ is coproductive 
and the category $\hrt$ has enough injectives. 

2. If such a left adjacent $w$ exists then $\hw$ is equivalent to the subcategory of  injective objects of $\hrt$. 
\end{coro}

\begin{rema}\label{rnondeg}
1. As it often happens when dealing with "large" triangulated categories, the "roles" of  Theorem \ref{tadjw} and Corollary \ref{cadjw} seem to be somewhat different. 
This "asymmetry" occurs when one tries to apply these statements to a compactly generated (or more generally, {\it well generated}; see Definition \ref{dwg}(\ref{idpg}) below)  $\cu$; recall that these condition are far from being self-dual (see Corollary E.1.3 and  Remark 6.4.5 of \cite{neebook}).  

Now, if $\cu$ is well generated then it appears to be "quite reasonable" to consider smashing  $t$-structures (only); moreover, the existence of enough injectives seems to be a  rather "reasonable" restriction on $t$ (see Corollary \ref{csymt} below; note however that is not clear how to prove it in general without constructing a right adjacent weight structure {\bf first}). 
 Yet it may be difficult to check the   Brown representability condition for $\cu'$ (unless 
  $\cupr=\cu$;\footnote{This is a rather "natural"  additional assumption 
for the corollary (as well as  for Theorem \ref{tadjw}) since otherwise the corresponding  weight structure $w$ would be "rather degenerate" in the following sense: the class $\cap_i\cu_{w\ge i}$ 
would contain the non-zero class $ {}^{\perp_{\cu}} \cupr$
} still checking the latter could be difficult also).

On the other hand, it appears that for all "known" cosmashing $t$-structures the left adjacent weight structures 
 can be easily described without using the results of this subsection.

So, neither Theorem \ref{tadjw} nor Corollary \ref{cadjw}  appear to be really "practical".

2. One 
 results on the existence of adjacent weight and $t$-structures are certainly "not symmetric": constructing a (right or left) adjacent $t$-structure is "much easier". On the other hand, Theorem \ref{tpgws} below is a tool of constructing weight structures that appears not to possess a $t$-structure analogue (see Remark \ref{rigid}(1)). Rather funnily, these two "asymmetries" appear to "compensate" each other. So, the "main general" sorts of weight structures for compactly generated categories that are 
"easy to construct" are the compactly generated ones (that are smashing) and the (cosmashing) weight structures right adjacent to compactly generated $t$-structures, whereas the main sorts of $t$-structures  are the   compactly generated ones and the  ones  right adjacent to compactly generated weight structures (these are smashing and cosmashing, respectively). 

3. Note also that there do exist smashing examples that are not compactly generated (inside a compactly generated $\cu$). 
In \cite{kellerema}  
for a ring $R$ along with 
 its two-sided ideal $I$ satisfying certain conditions 
 possesses a right adjoint respecting coproducts  but $L$ is not  generated by compact objects of $\cu$; a family of couples $(R,I)$ satisfying the conditions 
 the following was proved:  the embedding 
 $i$ of the localizing subcategory $L$ generated by $I$ into the  derived category $\cu$ of right $R$-modules 
 proved in Proposition \ref{prtst}(\ref{it4sm}), the corresponding couple $s_L=(\obj  L,\obj L^{\perp_{\cu}})$ is a "shift-stable" torsion pair for $\cu$ (cf. Remark \ref{rtst2}(\ref{it5s1}); so, it is a weight structure and a $t$-structure simultaneously); still it is certainly not compactly generated. 

One can probably obtain other examples of non-compactly generated weight structures for  $\cu$ of this type "starting from" 
 $s_L$ since one can (considering it as a weight structure and) "join" it with any compactly generated weight structure for $\cu$ (see Remark \ref{revenmorews}(1) below).

Moreover, in Remark \ref{rsymt}(\ref{irsymt1},\ref{irsymt4}) below a rich family of smashing weight structures on categories dual to compactly generated ones is described; these weight structures are not compactly generated. 

\end{rema}

\subsection{On (countably) perfect and symmetric classes; their relation to Brown-Comenetz duality and adjacent torsion pairs} 
\label{scomp}

Now we recall the notion of perfectly 
generated triangulated categories; categories of this type satisfy the Brown representability property according to results of A. Neeman and H. Krause. 

\begin{defi}\label{dwg}
Let $\al$ be a regular infinite  cardinal.\footnote{Recall that $\al$ is said to be regular if it cannot be presented as a sum of less then $\al$ cardinals that are less than $\al$.} 
\begin{enumerate}
\item\label{idsmall}
 An object $M$ of $\cu$ is said to be {\it $\al$-small} if for any set of $N_i\in \obj \cu$ any morphism $M\to \coprod N_i$ factors through the coproduct of a subset of $\{N_i\}$ of cardinality less than $\al$.

\item\label{idpc} 
 We will say that a  class $\cp\subset \obj \cu$ is {\it countably perfect} if the class of $\cp$-null morphisms 
 is closed with respect to countable coproducts (recall that $h$ is $\cp$-null if for all $M\in \cp$ we have $H^M(h)=0$; see Definition \ref{dhopo}(5)). We will say that $\cp$ is {\it  perfect} if the class of $\cp$-null morphisms is closed with respect to arbitrary coproducts. 

\item\label{idpg} 
 We will say that $\cu$ is {\it perfectly generated} if there exists a countably perfect {\bf set} $\cp_0\subset \obj \cu$ 
that Hom-generates it.

We will also say that $\cu$ is $\al$-{\it well generated} (or just well generated) if  (in addition) all elements of $\cp_0$ are $\al$-small and $\cp_0$ is  perfect.
\end{enumerate}
\end{defi}

Now we prove a few properties of these definitions. 

\begin{pr}\label{psym}
Let $\cp$ be a class of objects of a triangulated category $\cu$ that has coproducts. Denote by  $\cu'$ the localizing subcategory  of $\cu$ generated by $\cp$ and denote by $\du$ the full subcategory of $\cu$ whose object class equals $\obj \cupr\perpp$; denote the embedding $\cu'\to \cu$ by $i$.
Then the following statements are valid.

\begin{enumerate}
\item\label{isymcomp} If a set $\cq$  Hom-generates $\cu$ then it also $\alz$-well generates it if and only if all elements of $\cq$ are compact.

\item\label{isymeu} $\du$ is triangulated and $\obj \du=\cap_{i\in \z}(\cp\perpp[i])$. Moreover, if $\cp$ is (countably) perfect then $\du$ is closed with respect to (countable) $\cu$-coproducts.

\item\label{isymbr}  Assume that $\cp$ is a (countably) perfect set. 
Then 
$\cu'$ is perfectly generated by $\cp$ 
and has the  Brown representability property. Moreover,  the embedding 
 $i$ possesses an exact right adjoint $G$ that 
  gives an equivalence $\cu/\du\to \cu'$ and respects (countable) coproducts.\footnote{In particular, the localization $\cu/\du$ is a category, i.e., its Hom-classes are sets.} 

\item\label{isymuni} 
If $\cp_i$ is a collection of (countably) perfect 
 subclasses of $\obj \cu$ then $\cup  \cp_i$ is (countably) perfect  also.

\item\label{iperftp} Assume that $s=(\lo,\ro)$ is a torsion pair for $\cu$. Then $s$ is (countably) smashing if and only if $\lo$ is (countably) perfect.

\end{enumerate}
\end{pr}
\begin{proof}
\ref{isymcomp}. Certainly, $\alz$-small objects are precisely the compact ones. Hence any  $\alz$-well generating set consists of compact objects. To get the converse implication it suffices  to note that any class of compact objects is perfect (see assertion \ref{isymuni}).

\ref{isymeu}. The first part of the assertion is obvious;  the second one follows immediately from Proposition \ref{pwsym}(\ref{iwsmor}).

\ref{isymbr}. The set $\cp$ Hom-generates the category $\cupr$ according to Proposition \ref{pcomp}(I.1); certainly it is perfect in it. Hence Lemma \ref{lperf}(\ref{ipereq}) below (see Remark \ref{requivdef}) implies that  
 $\cp$ perfectly generates $\cupr$ (also) in the sense of \cite[Definition 1]{kraucoh}. Hence the Brown representability condition for $\cupr$ is given by Theorem A of \cite{kraucoh}. Given this condition the existence of $G$  follows from Proposition \ref{pcomp}(II.2). Lastly, 
Proposition \ref{prtst}(\ref{it4sm}) says  that $G$ respects (countable) coproducts.

\ref{isymuni}. It suffices to note that uniting  $\cp_i$ corresponds to intersecting the corresponding classes of null morphisms.

\ref{iperftp}. Immediate from Proposition \ref{phop}(\ref{itp8}).

\end{proof}

\begin{rema}\label{rwg}
1. Recall also that any well generated triangulated category possessing a combinatorial model satisfies the dual Brown representability property; see \S0 of \cite{neefrosi} (the statement is given by the combination of Theorems 0.17 and 0.14 of ibid.).

2.  It is well known that 
the class of well generated triangulated categories is "much bigger" than that of compactly generated ones; the class of perfectly generated categories is even bigger.\footnote{Recall that for every combinatorial stable 
Quillen model category $K$ its homotopy category is well generated; see  \cite[Proposition 6.10]{rosibr}.}\ Moreover, if $\cu$ is well generated then any its subcategory $\cu'$ generated by a set of objects as a localizing subcategory is well generated also; the Verdier localization $\cu/\cupr$ exists and is well generated (see 
Theorem 4.4.9 of \cite{neebook}). Note that the obvious analogue of this result for compactly generated categories is wrong (so, if $\cu$ is compactly generated then  its set-generated localizing subcategory $\cu'$ along with the quotient $\cu/\cu'$ is only well generated in general). In particular,  in \cite{neeshman} it was proved that the derived category of sheaves on a non-compact manifold is well generated but not compactly generated, whereas it is a localization of the compactly generated derived category of presheaves. 
 
However, the reader may assume that all the perfectly generated triangulated categories we study are actually 
 compactly generated since most of 
 our results are quite interesting in this case also. 

3. Certainly, a set $\cp$ is (countably) perfect 
  if and only if the coproduct of its elements forms a (countably) perfect set. 
 \end{rema}

Now, constructing perfect sets (and classes) "without using compact objects" is rather difficult. The main source of "non-compact" perfect classes in our paper is a certain Brown-Comenetz-type "symmetry"; the idea originates from \cite{kraucoh}. We start 
 with some definitions.

\begin{defi}\label{dsym}

Let $\cp$ and $\cp'$ be subclasses of $\obj \cu$, $P\in \obj \cu$.

\begin{enumerate}

\item\label{iwsym}
We will say that $\cp$ is {\it weakly symmetric} to 
$\cp'$ if $\cp\perpp=\perpp\cp'$.

\item\label{iconull}
We will say that a $\cu$-morphism $h$ is {\it $\cp$-conull} whenever for all $M\in \cp$ we have $H_M(h)=0$ (where  $H_M=\cu(-,M):\cu\opp\to \ab$).

\item\label{isym}
We will say that $\cp$ is {\it symmetric} to $\cp'$ if the class $\cp$-null (see Definition \ref{dhopo}(5)) coincides with  the class of $\cp'$-conull morphisms. 

\item\label{ibcomf} Let $H:\cu\to \ab$ be a homological functor. Then we will call the functor $\hdu:M\mapsto \ab(H(M),\q/\z)$ from $\cu$ into $ \ab $  the {\it Brown-Comenetz dual} of $H$. 

\item\label{ibcomo}
We will call an object of $\cu$ the {\it Brown-Comenetz dual} of $P$ and denote it by $\hat{P}$ if it represents the Brown-Comenetz dual of the functor $H^M=\cu(M,-):\cu\to \ab$.  
\end{enumerate}
\end{defi}

\begin{pr}\label{psymb}
Assume that $\cp$, $\cp'$, along with certain $\cp_i$ and $\cp'_i$ for  $i$ running through some $I$, are subclasses of $\obj \cu$; let $P\in \obj \cu$ and  $h$ be a $\cu$-morphism. Assume that $H:\cu\to \ab$ is a cc functor.  

I. Then the following statements are valid. \begin{enumerate}
\item\label{iwsmor} 
$P$ belongs  to $\cp\perpp$ if and only if the morphism $\id_P$ is $\cp$-null; dually, $P\in \perpp\cp$  if and only if
$\id_M$ is $\cp$-conull.

\item\label{iwsym1} If $\cp$ is symmetric to $\cp'$ then it is also weakly symmetric to $\cp'$.

If we assume in addition that $\cu$ has coproducts then $\cp$ is perfect.

\item\label{iws2}  $\cp$ is (weakly) symmetric to  $\cp'$ if and only if $\cp'$ is (weakly) symmetric to  $\cp$ in the category $\cu\opp$.\footnote{This is why we use the word "symmetric".}

\item\label{iws1} If 
 $\cp_i$ is (weakly) symmetric to $\cp'_i$ for any $i\in I$ then $\cup\cp_i$ is (weakly) symmetric to $\cup\cp'_i$. 

\item\label{isbcd} The Brown-Comenetz dual functor $\hdu$ is a cohomological  functor that converts $\cu$-coproducts into  $\ab$-products. Moreover, if $\hdu$ is represented by some $N\in \obj \cu$  then  $h$ is $\{N\}$-conull if and only if $H(h)=0$. 

\item\label{isym4} Assume 
 that  $\cu$ has coproducts, $\cp$ is symmetric to $\cp'$, and both of them are sets.
Then  the localizing subcategory $\cupr$ generated by $\cp$  satisfies both the Brown representability condition and 
 its dual, the embedding $i:\cu'\to \cu$ has an (exact) right adjoint $G$, and $\cp$ is symmetric to the class $G(\cp')$ in $\cu'$. 

Moreover, for $\du$ being the full subcategory of $\cu$ whose object class equals $\obj \cupr\perpp$, the functor
 $G$ gives an equivalence to $\cu'$ of  the full subcategory  $\du'$ of $\cu$ whose object class equals $\obj \du\perpp$. Furthermore, the embedding $\du'\to \cu$ respects products and possesses a left adjoint, and $\du'^{op}$  (has coproducts and) is perfectly generated by $\cp'$.

\end{enumerate}

II. Assume in addition that $\cu$ (has coproducts) and satisfies the Brown representability condition. 

\begin{enumerate} 
\item\label{iws21} Then $\hdu$ is representable by some $N\in \obj \cu$.

\item\label{iws22} Assume that $P$ is compact in $\cu$. Then its Brown-Comenetz dual  object $\hat{P}$ exists.

\item\label{iws23} Assume that all objects of $\cp$ are compact. Then $\cp$ is symmetric to the set of the Brown-Comenetz duals of elements of $\cu$. 


\end{enumerate}
\end{pr}
\begin{proof}
I.\ref{iwsmor}. Obvious.

\ref{iwsym1}. The first part of the assertion is immediate from the previous one; the second one is obvious.

\ref{iws2}, \ref{iws1}. Obvious.

\ref{isbcd}. Certainly, $\hdu$ converts $\cu$-coproducts into products of abelian groups. It is cohomological since $\q/\z$ is an injective object of $\ab$. Since it also cogenerates $\ab$, we obtain that $H(h)=0$ if $\hdu(h)=0$, whereas the converse implication is automatic. 

\ref{isym4}. The set $\cp$ is perfect according to assertion I.\ref{iwsym1}; hence $\cu'$ is perfectly generated by $\cp$. Thus $\cupr$ satisfies  Brown representability  and $i$ possesses an exact right adjoint $G$ that respects coproducts and gives an equivalence $\cu/\du\to \cu'$ (see Proposition \ref{psym}(\ref{isymbr})). Thus $G$ 
  restricts to a fully faithful functor $j$ from $\du'$ into $\cupr$.

Next, 
 the adjunction of $i$ to $G$ immediately yields that $\cp$ is symmetric to the class $G(\cp')$ in 
 $\cu'$ indeed.  Hence $\cupr$ has the Brown representability property; thus it has products according to  Proposition \ref{pcomp}(II.2). 
Thus $G(\cp')$ perfectly generates the category $\cu'^{op}$; therefore 
$\cupr$ satisfies
 the dual Brown representability condition.\footnote{Alternatively, Remark \ref{requivdef} below allows to deduce this fact from \cite[Theorem B]{kraucoh}.}  

Now, $\du'$ is certainly a triangulated subcategory of $\cu$ that is closed with respect to $\cu$-products. We have  $\obj \du\perp \cp'$ since $\cp\perp \obj \du$; hence $\cp'$ perfectly generates $\du'^{op}$ (since this statement becomes true after we apply $j$) and $j$ is an equivalence. It remains to apply (the dual to) Proposition \ref{pcomp}(II.2).

II.\ref{iws21},\ref{iws22}. By the definition of Brown representability,  it suffices to note that both $\hdu$ and $\widehat{H^P}$ are cp functors.

\ref{iws23}. Easy; combine assertions II.\ref{iws22}, I.\ref{isbcd}, and  I.\ref{iws1}.
\end{proof}

Now we establish a (somewhat funny) general criterion for a torsion pair $s'$ to be adjacent to $s$. 

\begin{pr}\label{pwsym}
 Let $\cp$, $\cp'$ be subclasses of $\obj \cu$; 
for torsion pairs  $s=(\lo,\ro)$  and $s'=(\lo,\ro)$ assume that a class $\cp_s\subset \obj \cu$ generates $s$ and some $\cp_{s'}$ {\it cogenerates} $s'$ (i.e., $\lo=\perpp \cp_{s'}$ and so $s'\subset \ro'$).
 Then the following statements are valid.

 \begin{enumerate}
\item\label{iws3}  $\cp_s$ is weakly symmetric to  $\cp'$ 
 if and only if  
$\lo$ is weakly symmetric to  $\cp'$. 

\item\label{iws3p} $\cp$ is weakly symmetric to  $\cp_{s'}$ 
if and only if $\cp$ is weakly symmetric to  $\ro'$.

\item\label{iws4} 
The following conditions are equivalent: 

(i) $\cp_s$ is weakly symmetric to  $\cp_{s'}$.

(ii) 
$\lo$ is  symmetric to  $\ro'$.

(iii) There exist a class $\cq$ that generates $s$ and $\cq'$ that cogenerates $s'$ such that 
$\cq$ is symmetric to $\cq'$.

  (iv) $s$ is left adjacent to $s'$. 

 \item\label{iwsymcgatp} Assume that $\cu$ (has coproducts) and satisfies the Brown representability property, $\cp_s$ is a class of compact objects,\footnote{Recall that any  set $\cp$ of compact objects  generates a torsion pair according to Theorem 4.1 of \cite{aiya}; cf. Theorem \ref{tclass}(\ref{iclass1}) below.}\ 
 and  $s'$ is right adjacent to $s$.  Then $s$ is smashing, $s'$ is cosmashing, and $s'$ is cogenerated by the Brown-Comenetz duals of elements of $\cp_s$ (i.e.,  $\lo'=\perpp\widehat{\cp_s}$ where $\widehat{\cp_s}=\{\hat{P}:\ P\in \cp_s\}$).

 \item\label{iwsymcgwt} For $\cu$ as in the previous assertion assume that $\cp$ is a class of compact objects and it generates a weight structure $w$.  Then $w$ is smashing, there exists a cosmashing $t$-structure $t$ right  adjacent to $w$, and $t$ is cogenerated by the Brown-Comenetz duals of elements of $\cp$ (i.e.,  $\cu^{t\le -1}=\perpp\widehat{\cp}$, where $\widehat{\cp}=\{\hat{P}:\ P\in \cp\}$).
\end{enumerate}

\end{pr}
\begin{proof}
\ref{iws3}, \ref{iws3p}: obvious (recall the corresponding definitions).


\ref{iws4}. By definition, $s$ is left adjacent to $s'$  if and only if $\lo\perpp=\perpp \ro'$; 
 hence (i) is equivalent to (iv). 
Applying 
Proposition \ref{pwsym}(\ref{iwsmor})  we also obtain that (iii) implies (i).  Next, the equivalence of (ii) to (iv) is immediate from  Proposition \ref{phop}(\ref{itp8}). Since (ii) implies (iii), we obtain the result.

\ref{iwsymcgatp}. $s$ is smashing and $s'$ is cosmashing according to Proposition \ref{phop}(\ref{itp5}) (recall that $\cu$ has products according to Proposition \ref{pcomp}(II.2)). Next, $\cp_s$ is symmetric to  $\widehat{\cp_s}$ according to Proposition \ref{psymb}(II.\ref{iws23}). Hence it remains to apply the previous assertion.

\ref{iwsymcgwt}. Certainly, $w$ is smashing (see Proposition \ref{phopft}(III)). 
Hence $t$ exists according to Theorem \ref{tadjt}, and it remains to apply the previous assertion. 

\end{proof}

\begin{rema}\label{rwsym}

1. Part \ref{iwsymcgwt} of corollary is the first application of part \ref{iwsymcgatp}; another application is (essentially) Corollary \ref{csymt} below. 
Note also that Theorem 3.11 of \cite{postov} states that in a compactly generated {\it algebraic} triangulated category $\cu$ for any compactly generated torsion pair $s$ (i.e., we assume $\cp_s$ to be a set in  assertion \ref{iwsymcgatp}) 
  there exists a  torsion pair right adjacent to it. Hence this torsion pair $s'$
is cogenerated by the corresponding $\widehat{\cp_s}$ (also).

2. In theory, (part \ref{iws4} of) our proposition gives a complete description of all couples of adjacent torsion pairs: one can start with $s$, take a generating class $\cp_s$ for it (that may be equal to $\lo$), find a class $\cp_{s'}$ that is weakly symmetric to $\cp_s$ (if any), and "cogenerate" $s'$ (if $\cp_{s'}$ does cogenerate some torsion pair). 

Still constructing (weakly) symmetric classes and cogenerating torsion pairs by them appears to be rather difficult in general; 
 the author can only "do" this by applying the Brown-Comenetz duality.

So, one may say that we construct (weakly) symmetric classes "elementwisely"; 
the author wonders whether a "more involved" method exists.

However, our proposition demonstrates the relation of adjacent torsion pairs to "Brown-Comenetz-type symmetry"; this point of view appears to be new.

3. The problem with the symmetry condition is that the class $\cp_s$-null is not determined by $s$; so even if $s$ and $s'$ are adjacent, it may be difficult to find "small" symmetric $\cq$ and $\cq'$ as in 
 condition \ref{iws4}(iii). So, we only have two (rather) "extreme" ("basic") types of symmetric classes: the ones coming from classes of compact objects (that "usually" 
 have bounded cardinality) via Brown-Comenetz duality and 
  the "big" ones of the "type" $(\lo,\ro')$ (see  
 condition \ref{iws4}(ii))).\footnote{One can also "unite" symmetric classes (see Proposition \ref{psymb}(I.\ref{iws1})); 
 yet this does not give "really new" essentially small symmetric classes (if $\cu$ has coproducts).}

Note also that in the proof of Theorem \ref{tsymt} we do not actually need $\cp$ to be symmetric; it  suffices to assume that $\cp$ is weakly symmetric to $\cp'$ and $\cp'$ is a set that is perfect in the category $\cu\opp$.

\end{rema}

\section{On perfectly generated weight structures and torsion pairs}\label{spgtp}

In this section we will always assume that $\cu$ has coproducts and $\cp$ is a subclass of $\obj \cu$. Our goal is to study the case when $\cp$ is a (countably) perfect class and it generates a torsion pair $s$ for $\cu$. Our results are more satisfactory when $s$ is weighty; 
 in particular, we prove that all weight structures on well generated triangulated categories are ({\it perfectly generated} and) {\it strongly well generated}.

In \S\ref{scoulim} we recall the notion of countable homotopy colimit in $\cu$ (that is one of the main tools of this section) and introduce several related notions and facts.

In \S\ref{scghop} we 
 prove that compactly generated torsion pairs  are in one-to-one correspondence with extension-closed Karoubi-closed essentially small classes of compact objects (in $\cu$). This result (slightly) generalizes Theorem 3.7 of \cite{postov} (along with Corollary 3.8 of ibid.). 

In \S\ref{sperfws} we study 
  the (naturally defined) perfectly generated weight structures. The existence Theorem \ref{tpgws} is absolutely new; still its proof has some "predecessors".  It appears to be rather difficult to construct examples of  perfectly generated weight structures that are not compactly generated; still we construct a curious family of those using suspended symmetric sets $(\cp,\cp')$ (see Definition \ref{dsym}(\ref{isym})). 
	 For any couple of this sort  the set $\cp'$ is  perfect in the category $\cu\opp$; 
	 so we obtain a certain weight structure $w$ on $\cu$ whereas its left adjacent $t$-structure (whose existence is essentially given by Corollary \ref{cdualt}) is generated by $\cp$. The author does not know how to construct "new" $t$-structures using this result; it implies however that for any compactly generated $t$-structure there exists a right adjacent weight structure.\footnote{One has to assume in addition that $\cu$ satisfies the Brown representability condition; however this is "almost automatic".}\ Now, the opposite to this weight structure (in $\cu\opp$) is perfectly generated but ("almost never") compactly generated, whereas the existence of $w$ implies that  the category $\hrt$ has an injective cogenerator and satisfies the AB3* axiom. Moreover, we deduce (using the results of \cite{humavit}) that $\hrt$ is Grothendieck abelian whenever $t$ is non-degenerate.
	 

In \S\ref{swgws} we 
 develop a certain theory of well generated torsion pairs. We prove several  relations between torsion pairs and (countably) perfect classes. Probably, the most interesting result in this section is the fact that a smashing weight structure on a well generated triangulated category is {\it strongly well generated} (i.e., it restricts to the subcategory of {\it $\be$-compact} objects for any large enough regular $\be$ and it can be "recovered" from this restriction); in particular, it is perfectly generated (i.e., it may be constructed using Theorem \ref{tpgws}).

\subsection{On countable homotopy colimits}
\label{scoulim}


We recall the basics of the theory  of countable (filtered)
homotopy colimits in triangulated categories (as introduced in \cite{bokne}; some more detail can be found in \cite{neebook}; cf. also \S4.2 of \cite{bws}). 
We will only apply the results of this subsection to triangulated categories that have coproducts; so we will not mention this restriction below.

\begin{defi}\label{dcoulim}

For a sequence of objects $Y_i$ 
 of $\cu$ for $i\ge 0$ and maps $\phi_i:Y_{i}\to Y_{i+1}$  
we consider the morphism $a:\oplus \id_{Y_i}\bigoplus \oplus (-\phi_i): D\to D$ (we can define it since its $i$-th component  can be easily factorized  as the composition $Y_i\to Y_i\bigoplus Y_{i+1}\to D$).  Denote  a cone of $a$ by $Y$. We will write $Y=\inli Y_i$ and call $Y$ a {\it homotopy colimit} of $Y_i$ (we will not consider any other homotopy colimits in this paper).

 Moreover, $\co(\phi_i)$ will be denoted by $Z_{i+1}$ and we set $Z_0=Y_0$.

\end{defi}

\begin{rema}\label{rcoulim}
1. Note that these homotopy colimits are not really canonical and functorial in $Y_i$ since the choice of a cone is not canonical. They are only defined up to non-canonical isomorphisms; still this is satisfactory for our purposes.

Note also that the definition of $Y$ gives a canonical morphism $D\to Y$.

2. By Lemma 1.7.1 of \cite{neebook},  a homotopy colimit of $Y_{i_j}$ is the same (up to an isomorphism) for any subsequence of $Y_i$. 
In particular, we can discard any (finite) number of first terms in $(Y_i)$.

3. 
By Lemma 1.6.6 of \cite{neebook}, $M$ is a homotopy colimit of $M\stackrel{\id_M}{\to}M\stackrel{\id_M}{\to}
M\stackrel{\id_M}{\to} M\stackrel{\id_M}{\to}\dots$. 

4. More generally, if $p$ is an idempotent endomorphism of $M$  then $p$ is isomorphic to a retraction of $M$ onto  a homotopy colimit $N$ of $M\stackrel{p}{\to}M\stackrel{p}{\to} M\stackrel{p}{\to} M\stackrel{p}{\to}\dots$ (see the proof of Proposition 1.6.8 of ibid). It easily follows that any extension-closed 
 coproductive subclass of $\obj \cu$ that is closed either with respect   $[1]$ or $[-1]$ is Karoubian; see  also Corollary 2.1.3(2) of \cite{bsnew} for an alternative proof of this fact.

5. Below we will often want to say something on some (co)homology of $Y$ along with morphisms from it.

We start from 
treating representable cohomology.

Let $T$ be an object of $\cu$ and consider the cp functor $H_T=\cu(-,T)$ from $\cu$ into $\ab$. 
Then the distinguished triangle defining $Y$ certainly yields a long exact sequence
$$\begin{aligned} \dots\to H_T(D[1])\stackrel{H_T(a[1])} \to H_T(D[1])(\cong \prod H_T(Y_i[1]))\to
 H_T(Y) \\
\to H_T(D)(\cong \prod H_T(Y_i)) \stackrel{H_T(a)}{\to} H_T(D)\to\dots \end{aligned} $$
According to Remark A.3.6 of \cite{neebook} this  gives a short exact sequence $$0\to \prli^1 H_T(Y_i[1])\to H_T(Y)\to \prli H_T(Y_i) \to 0.$$
 Thus any morphism $f\in \cu(Y,T)$ gives a "coherent" system of morphisms $Y_i\to T$. Conversely, for any coherent system $(f_i)$ of this sort there exists its "lift" to some $f\in \cu(Y,T)$ that we will say to be {\it compatible} with $(f_i)$. Certainly, the compatibility of $f$ with $(f_i)$ is 
fulfilled if and only if the composition \begin{equation}\label{ecompat}
Y_i\to \coprod Y_i\to Y\stackrel{f}{\to} T\text{ equals } f_i\end{equation} for any $i\ge 0$. Hence for any functor $F$ from $\cu$ the composition $F(Y_i)\to F(\coprod Y_i)\to F(Y)\stackrel{F(f)}{\to} F(T)$ equals $F(f_i)$.
\end{rema}

We study the behaviour of homotopy colimits  under 
cp and wcc functors.

\begin{lem}\label{lcoulim} 
Assume that $Y=\inli Y_i$ (in $\cu$), $\au$ is an abelian category; let $H$ (resp. $H'$) be a cp (resp. a wcc; see Definition \ref{dcomp}(\ref{idcc})) functor from $\cu$ into $\au$. Then the following statements are valid.

1.  The obvious connecting morphisms $Y_i\to Y$ give an epimorphism  $H(Y)\to \prli H(Y_i)$.


2. This epimorphism $H(Y)\to \prli H(Y_i)$ is an isomorphism whenever $\au$ is an AB4* category and all the morphisms $H(\phi_i[1])$ are 
epimorphic for 
 $i\gg 0$.

3.  $ \inli H'(Y_i)$ naturally 
 injects into $H'(Y) $.

This monomorphism is an isomorphism  if either 
(i) 
for  $i\gg 0$ there exist objects $A$ and $A_i\in \obj \au$,  along with compatible isomorphisms $H'(Y_i[1])\cong A_i\bigoplus A$ and $H(\phi_i[1])\cong (0:A_i\to A_{i+1}) \bigoplus \id_A$, 

or

(ii) $\au$ is an AB5 category.
 
In particular, if $C$ is compact then $\cu(C,Y)\cong \inli \cu(C,Y_i)$.

Furthermore, in case (i) we have $ \inli H'(Y_i[1])\cong A$. 

\end{lem}
\begin{proof}
1. We argue as in Remark \ref{rcoulim}(5).

We  have a long exact sequence
  $$\dots\to H(D[1])\stackrel{H(a[1])} \to H(D[1])\to
 H(Y)\to H(D)\stackrel{H(a)}{\to} H(D)\to\dots.$$

	 Since $H(D)\cong \prod H(Y_i)$, the kernel of $H(a)$ equals $\prli H(Y_i)$ (and this inverse limit exists in $\au$). 
	This yields the result.

2. Remark A.3.6 of \cite{neebook} yields that (in this case) the cokernel of $H(a[1])$ equals the $1$-limit of the objects $H(Y_i[1])$.
Next, by Remark \ref{rcoulim}(2)  we can assume that the homomorphisms $\phi[1]^*$ are surjective for all $i$. Hence the statement is given by Lemma A.3.9  of ibid. 

3. Similarly to the proof of assertion 1, we consider the long exact sequence
$$\dots\to H'(D)\stackrel{H'(a)}{\to} H'(D)\to H'(Y)\to
 H'(D[1]) \stackrel{H'(a[1])} \to H'(D[1])  \to\dots.$$

Since $H'(D)\cong \coprod H'(Y_i)$, it easily follows that the cokernel
of $H'(a)$ is $\inli H'(Y_i)$, this gives the first part of the assertion.

To prove its second part 
 we should verify that $H'(a[1])$ is monomorphic (if either of the two additional assumptions is fulfilled).
We will write  $B_i$ and $f_i$ for $H'(Y_i[1])$ and $H'(\phi_i[1])$, respectively, whereas the morphism $H'(a[1])$ (that can certainly be expressed in terms of $\id_{B_i}$ and $f_i$) will be denoted by $h$.

If (i) is valid then 
Remark \ref{rcoulim}(2) enables us to assume that $B_i\cong A\bigoplus A_i$ and $f_i\cong \id_A\bigoplus 0$ for all $i\ge 0$. Moreover, the additivity of the object $\ke(h)$ with respect to direct sums of $(B_i,f_i)$ reduces its calculation to the following two cases: (1) $(A=0;\ f_i=0)$ and (2) $(A_i=0;\ f_i\cong \id_A)$.
In 
 case (1) $h$ is isomorphic to $\id_{\coprod B_i}$; so it is monomorphic. In 
 case (2) $h$ is monomorphic as well since 
  the morphism matrix $$\begin{pmatrix}\id_A &\id_A &\id_A &\dots \\
0 & \id_A &\id_A &\dots\\
0 & 0 &\id_A &\dots\\
0 & 0 &0 &\dots\\
\dots & \dots &\dots &\dots \end{pmatrix}$$ gives the inverse morphism (cf. the proof of \cite[Lemma 1.6.6]{neebook}). 

Moreover, the additivity of direct limits in abelian categories implies that  $\inli(H'(Y_i[1]))\cong A\bigoplus A'$, where $A'$ is the direct limit of $A_0\stackrel{0}{\to} A_1\stackrel{0}{\to}A_2\stackrel{0}{\to}\dots$; certainly $A'=0$.

To prove version (ii) of the assertion note that  the composition of $H'(a[1])$ with the obvious monomorphism $\coprod_{i\le j} H'(Y_i[1])\to \coprod_{i\ge 0} H'(Y_i[1])$ is easily seen to be monomorphic for  each $j\ge 0$. If  $\au$ is an AB5 category then  it follows that the morphism  $H'(a[1])$ is monomorphic itself.
\end{proof}

We will also need the following definitions.\footnote{The terminology we introduce is new; yet 
big hulls were essentially considered 
in (Theorem 3.7 of) \cite{postov}.} 

\begin{defi}\label{dses}

 1. A class $\cpt\subset \obj \cu$ will be called {\it strongly extension-closed} if it 
 contains $0$ and for any $\phi_i:Y_{i}\to Y_{i+1}$ such that $Y_0\in \cpt$ and $\co(\phi_i)\in \cpt$ for all $i\ge 0$ we have $\inli_{i\ge 0} Y_i\in \cpt$ (i.e. $\cpt$ contains all possible cones of the corresponding distinguished triangle; note that these are isomorphic).

2. The smallest strongly extension-closed Karoubi-closed class of objects of $\cu$ that  contains a class $\cp\subset \obj\cu$ and is closed with respect to arbitrary (small) coproducts will be called the {\it strong extension-closure} of $\cp$.

3. We will write $\cocp$ for the closure of $\cp$ with respect to $\cu$-coproducts (and in \S\ref{sperfws} below we will also use this notation for the full subcategory of $\cu$ formed by these objects). 

Also, we will call  the class of the objects of $\cu$ that may be presented as homotopy limits of  $Y_i$ with $Y_0$ and $\co(\phi_i)\in \cocp$ 
the {\it naive big hull} of $\cp$.  We will call the Karoubi-closure of the naive big hull of  the class $\cp$  its {\it big hull}.

\end{defi}

Now we prove a few simple properties of these notions. 

\begin{lem}\label{lbes} 
Let $\cp$ be a class of objects of $\cu$; denote its strong extension-closure by $\cpt$.

\begin{enumerate}
\item\label{iseses}
$\cpt$ is extension-closed in $\cu$; it contains the big hull of $\cp$.
	 
\item\label{isesperp}
Let $H$ be a cp functor from $\cu$ into a AB4*-category $\au$, and assume that  the restriction of $H$ to $\cp$ is zero. Then $H$  kills $\cpt$ also.

In particular, if for some $D\subset \obj \cu$ we have $ \cp\perp D$ then $\cpt\perp D$ also.

\item\label{isescperp}
Let $H'$ be a cc functor from $\cu$ into a AB5-category. Then $H'$ kills $\cpt$ whenever it kills $\cp$.

In particular, if  $D$ is a class of compact objects in $\cu$ and $D\perp \cp$ then $D\perp \cpt$ also.

\item\label{izs}  {\it Zero classes} 
 of arbitrary families of cp and cc functors (into AB4* and AB5 categories, respectively) are strongly extension-closed (i.e. for any cp functors $H_i$ and cc functors $H'_i$ of this sort the classes $\{M\in \obj \cu:\ H_i(M)=0\ \forall i\}$ and $\{M\in \obj \cu:\ H'_i(M)=0\ \forall i\}$ are strongly extension-closed).

In particular,  if $(\lo,\ro)$ is a 
torsion pair in $\cu$ then $\lo$ is strongly extension-closed.

\item\label{ilwd} 
Adopt the notation of  Definition \ref{dcoulim}; let $w$ be a countably smashing weight structure on $\cu$.  Choose some $w$-decompositions $LZ_i\to Z_i\to RZ_i\to LZ_i[1]$ of $Z_i$ (see Definition \ref{dcoulim}) for $i\ge 0$. 

Then there exists some $w$-decompositions $LY_i\to Y_i\to RY_i\to LY_i[1]$ for $i\ge 0$ and connecting morphisms $l_i:LY_i\to LY_{i+1}$ for $i\ge 0$ such that the corresponding squares commute, 
$LY_0=LZ_0$,  $\co(l_i)\cong  LZ_{i+1}$, and there exists a weight decomposition $\inli LY_i\to Y\to RY\to (\inli LY_i)[1]$ (for some $RY\in \cu_{w\ge 1}$). 

\end{enumerate}
\end{lem}
\begin{proof}

\ref{iseses}. For any  distinguished triangle $X\to Y\to Z$ for $X,Z\in \cpt$ the object $Y$ is the colimit of $X\stackrel{f}{\to} Y\stackrel{\id_Y}{\to} Y \stackrel{\id_Y}{\to} Y \stackrel{\id_Y}{\to} Y\to \dots$ (see Remark \ref{rcoulim}(3)). Since
   a cone of $f$ is $Z$, whereas a cone of $\id_Y$ is $0$, 
	$\cpt$ is extension-closed indeed. It contains the big hull of $\cp$ by definition.

\ref{isesperp}. Since for any $d\in D$ the functor 
$H_d:\cu\to \ab$ converts arbitrary coproducts into products, it suffices to verify the first part of the statement.

Thus it suffices to verify that $H(Y)=0$ if $Y=\inli Y_i$ and $H$ kills  cones of the connecting morphisms $\phi_i$.

Now, $H(Y_j)=\ns$ for any $j\ge 0$ (by obvious induction).  Next,  the  long exact sequence $$\dots  \to H(Y_{i+1}[1]) \stackrel{H(\phi_i[1])}{\to} H(Y_i[1]) \to H(\co(\phi_i)) (=0) \to H(Y_{i+1}) \to H(Y_i)\to \dots  $$
gives the surjectivity of $H(\phi_i[1])$. Hence $H(Y)\cong \prli H(Y_i)=0$ according to Lemma \ref{lcoulim}(1,2).

\ref{isescperp}.  Once again, it suffices to verify the first part of the assertion. Similarly to the previous argument the result easily follows from  Lemma \ref{lcoulim}(3).

  \ref{izs}. The first part of the assertion is immediate from the previous assertions. To deduce the "in particular" part we note that $\lo$ is precisely the zero class of the    cp functors $\{H_N\}$ for $N$ running through  $\ro$ and $H_N=\cu(-,N)$.

\ref{ilwd}. We can construct $LY_i$ and $l_i$ satisfying the conditions in question expect 
 the last ("colimit") one inductively using  Proposition \ref{pbw}(\ref{iwdext}).

Now we consider the commutative square $$ \begin{CD}
 \coprod LY_i@>{La}>>\coprod LY_i\\
@VV{\coprod a_{Y_i}}V@VV{\coprod a_{Y_i}}V \\
D@>{a}>>D
\end{CD}$$
where $La=\oplus \id_{LY_i}\bigoplus \oplus (-l_i): \coprod LY_i\to \coprod LY_i$ is the morphism corresponding to $\inli LY_i$, and the remaining notation is from Definition \ref{dcoulim}. According to Proposition 1.1.11 of \cite{bbd}, we can complete it to a commutative diagram 
\begin{equation}\label{ely}\begin{CD}
\coprod LY_i @>{La}>>\coprod LY_i@>{}>> LY@ >{}>>\coprod LY_i[1] \\
 @VV{\coprod a_{Y_i}}V@VV{\coprod a_{Y_i}}V @VV{}V@VV{\coprod a_{Y_i}[1]}V\\
D @>{a}>>D@>{}>> Y@>{}>> D[1]\\
@VV{}V @VV{}V @VV{}V@VV{}V \\
\coprod RY_i @>{}>>\coprod RY_i@>{}>> RY @>{}>> \coprod RY_i[1] \\
\end{CD}
\end{equation}
whose rows and columns are distinguished triangles. 
Then $LY$ is a homotopy colimit of $LY_i$ (with respect  to $l_i$) by definition. Since $LY_0$ and cones of $l_i$ 
belong to $\cu_{w\le 0}$, we also have $LY\in \cu_{w\le 0}$ according to assertion \ref{izs}. On the other hand, the 
bottom row  of (\ref{ely}) gives  $RY\in \cu_{w\ge 1}$ (since 
$\coprod RY_i\in  \cu_{w\ge 1}$). 
Thus the 
 third column of our diagram is a weight decomposition of $Y$ of the type desired.
\end{proof}

Let us also consider so-called dual towers (see \S3.1 of \cite{marg} and the footnote to our lemma) and their relation to wcc functors.

\begin{lem}\label{ldualt}
Assume that  $h_i:M_{i}\to M_{i+1}$  for $i\ge 0$ is a sequence of $\cu$-morphisms, and for  $M\in \obj \cu$  we have connecting morphisms $g_i\in \cu(M,M_i)$ compatible with $h_i$; take distinguished triangles $L_i\stackrel{b_i}{\to}M \stackrel{g_i}{\to}M_i\stackrel{f_i}{\to} L_i[1]$ and $P_i\stackrel{a_i}{\to} M_i\stackrel{b_i}{\to} M_{i+1}\to P_i[1]$.  
Then there exists a system of morphisms $s_i:L_i\to L_{i+1}$ compatible with 
$(b_i)$\footnote{So, if one complete $(h_i)$ and $(s_i)$ to negative values of $i$ by $\id_{M_0}$ and $\id_{L_0}$, respectively, then one obtains Postnikov towers for $M$ in the categories $\cu\opp$ and $\cu$, respectively.} such that for $L=\inli L_i$,  for any morphism $b:L\to M$ that is compatible with $(b_i)$ in the sense of Remark \ref{rcoulim}(5), and any wcc functor $H:\cu\to \au$ the following statements are fulfilled.

1. $\co(s_i)\cong P_i$. 

2. 
 If $H(g_1)=0$ 
  then $H(b)$ is an epimorphism.

3. The morphism $H(b)$ is monomorphic whenever all the restrictions of $H(b_{i+1})$ to the images of $H(s_{i})$ are monomorphic  
and one of the following conditions are fulfilled: 
(i) $H(g_1)=0=H(g_1[1])$ and the restrictions of $H(b_{i+1}[1])$ to the images of $H(s_{i}[1])$ are monomorphic for $i\ge 1$ as well; 

 (ii) $\au$ is an AB5 category.


4.  For any $i\ge 0$ the restriction of $H(b_{i+1})$ to the image of $H(s_{i})$ is monomorphic whenever $H(h_i[-1])=0$.

5. For any $i\ge 0$ if $H(a_i)$ is epimorphic then $H(h_i)=0$ and $H(g_{i+1})=0$. 

\end{lem}
\begin{proof} 
1. For all $i\ge 0$ we complete the commutative triangles $M\to M_{i}\to M_{i+1}$ to octahedral diagrams as follows:
\begin{equation}\label{eoct}
\xymatrix{M_{i+1}  \ar[dd]^{[1]} & & M \ar[ld]^{g_{i}} \ar[ll]^{g_{i+1}} \\ & M_i \ar[lu]^{h_i} \ar[rd]^{[1]}   & \\ P_i \ar[ru]^{a_i}
  \ar[rr]^{[1]} & & L_i\ar[uu]^{b_{i}}} 
\xymatrix{M_{i+1} \ar[rd]^{[1]}\ar[dd]^{[1]} & & M \ar[ll]^{g_{i+1}} \\ & L_{i+1} \ar[ru]^{b_{i+1}}\ar[ld]^{
} & \\ P_i\ar[rr]^{[1]} & & L_i\ar[lu]^{s_i}\ar[uu]^{b_{i}} }
\end{equation}
So we obtain the property 1 for this choice of $\{s_i\}$. Moreover, 
 we fix any morphism $b:L\to M$ that is compatible with $(b_i)$ (in the sense of Remark \ref{rcoulim}(5)).

2. 
Certainly, if $H(g_1)=0$ then $H(g_i)=0$ for all $i\ge 1$.

Next, the exact sequences 
\begin{equation}\label{eleso}
\dots\to H(L_i)\stackrel{H(b_i)}{\to}H(M) \stackrel{H(g_i)}{\to}H(M_i)\stackrel{H(f_i)}{\to} H(L_i[1])\to \dots \end{equation}
 yield that $H(b_i)$ are epimorphic for $i\ge 1$. Lastly, 
 the first statement in Lemma \ref{lcoulim}(3) allows us to pass to the limit and conclude the proof. 

3. In version (i) (resp. (ii)) of our assertion we should prove that $H(b)$ is isomorphic (resp. monomorphic). Now, in case (ii) we have  $H(L)\cong \inli H(L_i)$ (see  Lemma \ref{lcoulim}(3(ii))). Since  $\inli H(L_i)\cong \inli \imm(H(s_{i}))$, applying $H$ and passing to the limit in the formula (\ref{ecompat}) we conclude the proof.

Next, in case (i) we certainly have $H(g_i[m])=0$ for all $i\ge 1$ and $m$ being either $0$ or $1$; hence (\ref{eleso}) implies that 
 the corresponding $H(b_i[m])$ are epimorphic. Hence for any $i\ge 1$ and $m=0$ or $1$ we have $\imm H(s_i[m])\cong H(M[m])$, and the restriction of   $H(s_{i+1}[m])$ to $\imm H(s_i[m])$ is an isomorphism. Thus for any $i\ge 2$ the morphism $H(s_i[m])$ is isomorphic to $\id_{H(M)}\bigoplus 0:A_i^s\to A^s_{i+1}$ for certain $A^s_i\in \obj \au$. Hence applying Lemma \ref{lcoulim}(3(i)) (for $H'=H$ and also for $H'=H\circ [-1]$) 
 we obtain $H(L)\cong \inli H(L_i)\cong A$, and applying $H$ to the formula (\ref{ecompat}) we conclude the proof.

4. The restriction of $H(b_{i+1})$ to the image of $H(s_{i})$ is monomorphic if and only if $H(s_{i})$ kills 
$\ke H(b_i)$. Next, $\ke H(b_i)=\imm (H(M_i[-1])\to H(L_i)$; thus we should check that the composition morphism $H(M_i[-1])\to H(L_{i+1}$ vanishes. Lastly, the octahedral axiom (also) says that   this composition can be factored through $H(h_i[-1])$, and we obtain the 
result in question.

5. The corresponding long exact sequence implies that $H(h_i)=0$ if and only if $H(a_i)$ is epimorphic. Next, if $H(h_i)=0$ then certainly the morphism $H(g_i)$ is zero as well. 
\end{proof}

\begin{rema}\label{r27}   Proposition 2.7 of \cite{postov} says that we can also assume that there exists a distinguished triangle $L\to M\to \inli M_i\to L[1]$ whenever $\cu$ is a "stable derivator" triangulated category. We note that this additional assumption on $\cu$ is rather "harmless", and it can be used to simplify the proof of our lemma. However, it seems to be now way to avoid the conditions similar to that in Lemma \ref{lcoulim}(3(i)) completely (for our purposes). 

2. Certainly, the assumptions that $H(g_1)$ and $H(g_1)[1]$ vanish can be replaced by the vanishing of $H(g_i)$ and $H(g_i)[1]$ for an arbitrary $i>1$.
\end{rema}

\subsection{A classification of compactly generated torsion pairs}\label{scghop}

Now we generalize (and extend) 
Theorem 3.7 of \cite{postov} to arbitrary triangulated categories that have coproducts (yet cf. Remark \ref{r27}(1)).
 
\begin{theo}\label{tclass}
Assume that  
 $\cp\subset \obj \cu$ 
is a  set of compact objects (recall that $\cu$ has coproducts).

Then the following statements are valid.
\begin{enumerate}
\item\label{iclass1}
The strong extension-closure $\lo$   of $\cp$ and $\ro=\cp^\perp$ give a smashing torsion pair $s$ for $\cu$ (so, $s$ is the torsion pair generated by $\cp$). Moreover, $\lo$ equals the big hull of $\cp$, and for any $M\in \obj \cu$ there exists a choice of $L_sM$ (see   Remark \ref{rwsts}(3)) belonging to the naive big hull of $\cp$.

\item\label{iclass2} The class of compact objects in $\lo$ equals the $\cu$-envelope of $\cp$ (see \S\ref{snotata}).

\item\label{iclassts} The correspondence sending a compactly generated torsion pair $s=(\lo,\ro)$ for $\cu$ into $\lo\cap\cu^{\alz}$, where $\cu^{\alz}$ 
  is the class of compact objects of $\cu$, gives a one-to-one correspondence between the following classes: the class of  compactly generated torsion pairs for $\cu$ and the class of essentially small 
Karoubi-closed extension-closed subclasses  of $\cu^{\alz}$.
\footnote{Actually, $\cu^{\alz}$ is essentially small itself in "reasonable" cases; in this case the essential smallness of its classes of objects is automatic.} 

\item\label{iclasst}  If the torsion pair $s$ generated by $\cp$ is associated to a $t$-structure then the class $\lo$ ($=\cu^{t\le 0}$)  equals the naive big hull of $\cp$.

\item\label{iclass5}  Let 
$H$ be a cp (resp. a cc) functor from $\cu$ into an AB4* (resp. AB5) category $\au$ whose restriction to $\cp$ is zero. Then $H$ kills all elements of $\lo$ as well. 

\end{enumerate} 
\end{theo}

\begin{proof}
\begin{enumerate}
\item If $s$ is a torsion pair indeed then it is smashing according to Proposition \ref{phopft}(III).

Since $\cp\perp \ro$, for any $N\in \ro$ the cp functor 
$H_N$ (from $\cu$ into $\ab$) also kills $\lo$ according to Lemma \ref{lbes}(\ref{isescperp}). Hence $\lo\perp\ro$.\footnote{This statement was previously proved in \cite{postov} and our argument is just slightly different from the one of Pospisil and \v{S}\v{t}ov\'\i\v{c}ek; see Lemma 3.9 of  ibid.}
Since $\lo$ is Karoubi-closed by definition, Proposition \ref{phop}(\ref{itp9}) (along with Lemma \ref{lbes}(\ref{iseses})) reduces the assertion
 to the existence for any  $M\in \obj\cu$ of  an $s$-decomposition 
 such that the corresponding $L_sM$   belongs to the naive big hull of $\cp$. Now we apply Lemma \ref{ldualt}; 
our argument is also related to the proof of \cite[Theorem 4.5.2(I)]{bws} and to the construction of  crude cellular towers in \S I.3.2 of \cite{marg}. 

  So, we construct a certain sequence of $M_k\in \obj \cu$ for $k\ge 0$ by induction in $k$ starting
from $M_0=M$. 
Assume  that $M_k$ (for some $k\ge 0$) is  constructed; then we take $P_k=\coprod_{(P,f):\,P\in \cp,f\in \cu(P,M_k)}P$; $M_{k+1}$ is a cone of the morphism $\coprod_{(P,f):\,P\in \cp,f\in \cu(P,M_k)}f:P_k\to M_k$.
Then compositions of the morphisms $h_k:M_{k}\to M_{k+1}$ given by this construction yields morphisms $g_i:M\to M_i$ for all $i\ge 0$.

We apply Lemma \ref{ldualt}(1) and obtain the existence of connecting morphisms $0=L_0\stackrel{s_0}{\to}L_1\stackrel{s_1}{\to}L_2\stackrel{s_1}{\to}\dots$; we set $L=\inli L_i$. Moreover, we have a compatible system of morphisms $b_i:L_i\to M$ (cf. the formulation of that lemma) and we  choose  $b: L\to M$ to be compatible with  $(b_k)$ (see Remark \ref{rcoulim}(5)). 
 We complete $b$ to a distinguished triangle $L\stackrel{b}{\to} M\stackrel{a}{\to} R\stackrel{f}{\to} L[1]$;
it will be our candidate for an $s$-decomposition of $M$.

 Since $\co (s_i)\cong P_i$,  $L$ belongs to the naive big hull of $\cp$ by the definition of this hull. 

It remains to prove that $R\in \ro$, i.e., that $\cp\perp R$. 
For an element $P $ of $\cp$ we should check that $\cu(P,R)=\ns$, i.e.,  for the functor $H^P=\cu(P,-)$ we should prove that $H^p(R)=0$.

The long exact sequence  $$\dots  \to \cu(P,L)\to  \cu(P,M)\to \cu(P,R)\to \cu(P, L[1])\to \cu(P,M[1])\to\dots $$ translates this into the following assertion: $H^P(b)$ is surjective and $H^P(b[1])$ is injective. 

Now, $H^P$ is certainly a wcc (and actually a cc) functor, and its target is an AB5 category. Since $H^P(a_i)$ is epimorphic by construction for all $i\ge 0$, Lemma \ref{ldualt}(5,2) implies the surjectivity of $H^P(b)$. 

Next 
we apply   Lemma \ref{ldualt}(4,3(ii)) for $H=H^P\circ [1]$ and obtain that to verify the injectivity of  $H^P(b[1])$ it remains to check that $H^P(h_i)=0$ for $i\ge 0$. 
Applying part 5 of the lemma we reduce the statement in question to the aforementioned surjectivity of  $H^P(a_i)$. 
  
\item Given the previous assertion, the argument used in the proof of \cite[Theorem 3.7(ii)]{postov} goes through without any changes. We will describe another proof of our statement (that does not depend on ibid.) in Remark \ref{rnz}(1) below.

\item Recall that  for any set $D\subset \obj \cu$ the smallest strict (full) triangulated subcategory of $\cu$ containing $D$ is essentially small by  Lemma 3.2.4 of \cite{neebook}; hence $\lan D\ra_{\cu}$ (see \S\ref{snotata} for the notation)  is essentially small also (cf. Proposition 3.2.5 of ibid.). Thus for a set $\cp$ of compact objects of $\cu$ its $\cu$-envelope $\cp'$ is essentially small; its elements are compact according to Lemma 4.1.4 of \cite{neebook}. Since $\cp'\perpp=\cp\perpp$ (see Proposition \ref{phop}(\ref{itp1})),   the torsion pair $s$ given by 
assertion \ref{iclass1} is also generated by $\cp'$; hence it suffices to note that $\lo\cap \cu^{\alz}=\cp'$ according to assertion \ref{iclass2}. 

\item Recall from 
 assertion \ref{iclass1} 
that for any $M\in \obj \cu$ there exists a choice of $L_sM$ that belongs to the naive big hull of $\cp$. Now, if $s$ is associated to a $t$-structure and $M\in \lo=\cu_{t\le 0}$ then we certainly have $L_sM=M$ (see Remark \ref{rtst1}(\ref{it1},\ref{it3})); this concludes the proof.

\ref{iclass5}. Immediate from Lemma \ref{lbes}(\ref{isesperp}) (resp. \ref{isescperp}).

\end{enumerate}

\end{proof}

\begin{rema}\label{rnewt}
1.  As a particular case of part \ref{iclass1} of our theorem we obtain that for any set $\cp$ (of compact objects of $\cu$) such that $\cp[-1]\subset \cp$ there exists a weight structure on $\cu$ such that $\cu_{w\ge 1}=\cp^{\perp}$ (cf. Remark  \ref{rwhop}(1); note that this statement was originally proved in \cite{paucomp}). Thus if $\cu$ is compactly or 
  perfectly generated (see Definition \ref{dwg}(\ref{idpg}); in particular, this is certainly the case if $\cu=\cupr$, where the latter is the localizing subcategory of $\cu$ generated by $\cp$) then Theorem \ref{tadjt} implies that   $(\cp^{\perp}[-1], (\cp^{\perp})^{\perp})$ is a $t$-structure on $\cu$. Moreover, the couple $(\cp^{\perp}[-1]\cap \obj \cupr, \cp^{\perp}\cap \obj \cupr)^{\perp_{\cu'}})$ is a $t$-structure on $\cu'$ regardless of any extra restrictions on $\cu$ (here one should invoke Propositions \ref{phopft}(I) 
and   \ref{pcomp}(II)).

So, we obtain a statement on the existence of (cosmashing) $t$-structures that does not mention weight structures! This result appears to be new (unless $\cu$ is a  compactly generated  algebraic triangulated category; see Remark 
\ref{rwsym}(1) and Theorem 3.11 of \cite{postov}). 

Moreover, the results of the next subsection (see Remark \ref{revenmorews}) give an even vaster source of smashing weight structures  (and so, of their adjacent $t$-structures as well).

2. Recall that Theorem 3.7 and  Corollary 3.8 of \cite{postov} give parts \ref{iclass1}--\ref{iclassts} of our theorem in the case where $\cu$ is  a "stable derivator" triangulated category $\cu$. 

Moreover, as a consequence of part \ref{iclassts} we certainly obtain a bijection between compactly generated $t$-structures (resp. weight structures) and those essentially small 
Karoubi-closed extension-closed subclasses  
of $\cu^{\alz}$ that are closed with respect to $[1]$ (resp. $[-1]$); this  generalizes Theorem 4.5 of ibid. to arbitrary triangulated categories having coproducts.

3. Part \ref{iclasst} of our theorem generalizes Theorem A.9 of \cite{kellerw} where 
 "stable derivator"  categories were considered (similarly to the aforementioned results of \cite{postov}).

4. The question whether all  smashing weight structures  on a given compactly generated category $\cu$ are compactly generated is a certain weight structure version of  the (generalized) telescope conjecture (that is also sometimes called the smashing conjecture) for $\cu$; this question generalizes its "usual" stable version (see Proposition \ref{prtst}(\ref{it4sm})). As we have noted in Remark \ref{rnondeg}(3), the main result of \cite{kellerema} demonstrates that the answer to the shift-stable version of the  question is negative for a general $\cu$; hence this is only more so for our weight structure version. On the other hand, the answer to our question for $\cu=SH$ (the topological stable homotopy category)  is not clear.

5. The description of compact objects in $\lo$ provided by part \ref{iclass2} of our theorem is important for the continuity arguments in \cite{bcons}.
\end{rema}

\subsection{On perfectly generated weight structures and symmetrically generated $t$-structures}\label{sperfws}

Now we prove that an arbitrary cosuspended countably perfect set $\cp$ 
 gives a weight structure; this is an interesting modification of Theorem \ref{tclass}(\ref{iclass1}) (note that any class of compact objects is  perfect).

\begin{theo}\label{tpgws}
Let $\cp$ be a countably perfect (see Definition \ref{dwg}) cosuspended  (i.e., $\cp[-1]\subset \cp$) set 
of  objects of $\cu$. Then   the strong extension-closure $\lo$ of $\cp$ and $\ro=\cp^\perp$ give a weighty torsion pair for $\cu$ (i.e., $(\lo,\ro[-1])$ is a weight structure). 
Moreover, $\lo$ equals the big hull of $\cp$. 
\end{theo}
\begin{proof}
As we have  noted in Remark \ref{rwhop}(1), 
$s$ is weighty whenever $s$ is a torsion pair and $\lo\subset \lo[1]$; the latter certainly follows from $\cp[-1]\subset \cp$.

So, we should prove the remaining assertions. Repeating the beginning of the proof of Theorem \ref{tclass}(\ref{iclass1}), 
 we reduce them to the existence for any $M\in \obj \cu$  an $s$-decomposition $L_sM\to M\to R_sM\to L_sM[1]$ with $L_sM$ belonging to the naive big hull of $\cp$. 

For this purpose we construct a distinguished triangle $L\stackrel{b}{\to} M\stackrel{a}{\to} R\stackrel{f}{\to} L[1]$ using the method described in the proof of Theorem \ref{tclass}; 
so, $L$ belongs to the naive big hull of $\cp$ by construction.
To finish the proof we  should check that $R\in \ro$; for this purpose we "mix" the proof of Theorem \ref{tclass}(\ref{iclass1}) with that of \cite[Theorem A]{kraucoh}; cf. also Remark \ref{rigid}(1) below.
The idea is to replace the collection of functors 
$H^P$ for $P\in \cp$ with a single "more complicated" functor (that would be a wcc one in contrast with the functors $H^P$ in the case of  "general" $P$).

We will write $\cocp$ for the full subcategory of $\cu$ formed by the $\coprod$-closure of $\cp$;  
following \cite{kraucoh} (see also \cite[Definition 5.1.3]{neebook} and \cite{auscoh}) 
we consider the full subcategory $\hatc\subset  \psv^\z(\underline{\coprod}\cp)=\adfu((\cocp)^{op},\ab)$ (cf. Remark \ref{rdetect}; we will omit the index $\z$ in this notation below) of {\it coherent functors}. 
We recall (see \cite{kraucoh}) that a  functor $H\in \obj \psv(\cocp)$  
is said to be coherent whenever there exists a $\psv(\cocp)$-short exact sequence $\cocp(-,X)\to \cocp(-,Y)\to H\to 0$, where $X$ and $Y$ are some objects of $\cocp$ (note that this  is a projective resolution of $H$ in  $\psv(\cocp)$; see \cite[Lemma 5.1.2]{neebook}).

 According to \cite[Lemma 2]{kraucoh}, the category $\hatc$ is abelian; it has coproducts according to Lemma 1 of ibid. Since any morphisms of (coherent) functors is compatible with some morphism of their (arbitrary) projective resolutions, a $\hatc$-morphism is zero (resp. surjective) if and only if it is surjective in $\psv(\cocp)$.

Next, the Yoneda correspondence $\cu\to  \psv(\cocp)$  (sending $M\in \obj \cu$ to the restriction of $\cu(-,M)$ to $\cocp$) gives a  homological functor $H^{\cp}:\cu\to \hatc$  (see Lemma 3 of ibid.). $H^{\cp}$ is a wcc functor since $\cp$ is countably perfect (according to that lemma);  it also respects arbitrary ${\cocp}$-coproducts (very easy; see Lemma 1 of ibid.).
Lastly, our discussion of zero and surjective $\hatc$-morphisms certainly yields that $H^{\cp}(h)$ is zero (resp. surjective) for $h$ being a $\cu$-morphism if and only if $\cu(N,-)(h)=0$ (resp. surjective) for any $N\in \obj {\cocp}$; it is certainly suffices to take $N\in \cp$ in these "criteria" only. 

Now we prove that $R\in \ro$ using the notation introduced in the proof of Theorem \ref{tclass}(\ref{iclass1}) (along with Lemma \ref{ldualt}). 
As we have just proved, $R\in \ro$ whenever $H^{\cp}(R)=0$. Hence the long exact sequence  $$
\to H^{\cp}(L)\stackrel{H^{\cp}(b)} \to  H^{\cp}(M) 
\to   H^{\cp}(R)\to   H^{\cp}(L[1])\stackrel{H^{\cp}(b[1])} \to  H^{\cp}(M[1])\to
 $$
reduces the assertion to the epimorphness of $H^{\cp}(b)$ along with the monomorphness of $H^{\cp}(b[1])$.
Since $\cp$ is cosuspended, the morphism $H^{\cp}(a_i[j])$ is epimorphic (essentially by construction) for all $i,g\ge 0$.
Applying this observation in the case $i=j=0$ along with Lemma \ref{ldualt}(5,2) we obtain that $H^{\cp}(b)$ is epimorphic. 

It remains to apply part 3(i) of the lemma to verify that $H^{\cp}(b[1])$ is monomorphic; so we take $H=H^{\cp}\circ [1]$. Thus we should check $H^{\cp}(g_1[1])=0=H^{\cp}(g_1[2])$ and also  that  $H^{\cp}(h_i)=0=H^{\cp}(h_i[1])$ for all $i\ge 0$ (see part 4 of the lemma). 
Thus combining part 5 of the lemma with the aforementioned surjectivity of $H^{\cp}(a_i[j])$  we obtain the result.
\end{proof}

\begin{rema}\label{rigid}  
1. The 
 author was inspired to apply coherent functors in this context by  \cite{salorio}; 
  yet 
the proof of Theorem 2.2 of ibid. (where coherent functors are applied to the construction of $t$-structures) appears to contain a gap.\footnote{An argument  even more closely related to our one was used in the proof of \cite[Lemma 2.2]{modoi}; yet the assumptions of that lemma appear to require  a correction.}
The author believes that applying arguments of the sort used in the proof of our theorem in the case of a "general" (countably) perfect set $\cp\subset \obj \cu$ such that $\cp\subset \cp[1]$ one can (only) obtain a "semi-$t$-structure" for $\cu$, i.e., for any $M\in \obj \cu$ there exists a distinguished triangle 
 $L\to M\to R\to L[1]$  such that $L$ belongs to the big hull of $\cp$ and $R\in \cp\perpp[1]$.\footnote{Certainly, the  big hull of $\cp$ is contained in its big extension-closure, whereas the latter (for $\cp\subset \cp[1]$) equals the smallest coproductive extension-closed subclass of $\obj \cu$ containing $\cp$; cf. Corollary \ref{cgdb}(\ref{ict}) below.} 

The author wonders whether this result can be improved, and also whether semi-$t$-structures can be "useful".\footnote{Note however that {\it weak weight structures} (one replaces the orthogonality axiom in Definition \ref{dwstr}   by $\cu_{w\le 0}\perp\cu_{w\ge 2}$) were essentially considered in \cite{bsosn} (cf. Remark  2.1.2 of ibid.) and in 
Theorem 3.1.3(2,3) of \cite{bsnew}, whereas in Proposition 3.17 of \cite{brelmot} it was shown that they are relevant for the study of mixed \'etale $\ql$-adic sheaves over varieties over finite fields (actually, it was demonstrated that the weight filtration for the category $\dbm(X_0,\ql)$ satisfies the somewhat stronger Definition 3.11 of ibid., where $X_0$ is a variety over a finite field of characteristic $\neq l$).} 
 Note also that in Theorem \ref{tsymt} below we will prove the existence of a $t$-structure generated by a suspended $\cp$ whenever $\cp$ is symmetric to some set $\cp'\subset \obj \cu$; however, the author does not know whether $\cu^{t\le 0}$ equals the 
 big hull of $\cp$ whenever $\cp$ satisfies these conditions.

2. The arguments from the proof of our theorem can also be used (and  significantly simplified) if instead of requiring $\cp$ to be cosuspended (and countably perfect) we assume that ${\cocp}$ is {\it rigid}, i.e., ${\cocp}\perp {\cocp}[1]$. Indeed, then the distinguished triangle $P_0\to M\to M_1\to P_0[1]$ (see the proof of Theorem \ref{tclass}(\ref{iclass1})) is easily seen to be an $s$-decomposition of $M$ (and if we proceed as above then this triangle will actually be equal to $L\to M\to R\to L[1]$).

Note also that  ${\cocp}$ is rigid if and only if $\cp\perp {\cocp}[1]$. Moreover, if $\cp$ is perfect then these conditions are equivalent to  $\cp\perp \cp[1]$.

The author does not know whether any formulation of this sort is known.

3. The case $\cp=\cp[1]$ of our theorem (cf. Remark \ref{rtst2}) 
 is closely related to the proof of \cite[Theorem A]{kraucoh}. 

4. We will say that a weighty  torsion pair and the corresponding weight structure are perfectly generated whenever they can be obtained by means of our theorem.
Remark \ref{revenmorews}(1)  below will give a "more natural" equivalent of this definition. 
Note also that instead of assuming that $\cp$ is a set in the theorem it certainly suffices to assume that $\cp$ is essentially small.

 Moreover, Theorem \ref{twgws}(III.2) states that any smashing weight structure on a well generated triangulated category is perfectly generated.

5. In Theorem \ref{tsymt} and Corollary \ref{csymt} below we will study examples for Theorem \ref{tpgws} that are constructed using "symmetry"; this will yield some new results on $t$-structures. The idea to relate $t$-structures to symmetric sets and Brown-Comenetz duals comes from \cite{salorio} also; however the author doubts that one can get a "simple description" of a $t$-structure obtained using arguments of this sort (cf. Corollary 2.5 of ibid.). 
\end{rema}

Now we 
 prove a few simple definitions  and statements 
related to countably perfect classes (without claiming much originality in these results); recall that $\cp$-null morphisms (for $\cp\subset \obj \cu$)
are the ones annihilated by the functors $H^P$ for all $P\in \cp$ 
 (see Definition \ref{dhopo}(5)).

\begin{defi}\label{dapprox} Let  $\cp$ be a subclass of  $\obj \cu$, $h\in \cu(M,N)$ (for some $M,N\in \obj\cu$).

1.  We will say that $h$ is a {\it $\cp$-epic} whenever  for any $P\in \cp$ the homomorphism 
$H^P(h)$ is surjective, i.e., if any $g\in \cu(P,N)$ for $P\in \cp$ factors through $h$.
We will say that $h$ is  {\it $\cp$-monic} if all  $H^P(h)$ are injective.\footnote{These definitions (along with the definition of $\cp$-null morphisms) is taken from \cite{christ}.}

2. We will say that $h$ is a {\it $\cp$-approximation} (of $N$) if it is a $\cp$-epic and $M$ belongs to $\obj\cocp$.

3. We will say that $\cp$ is {\it contravariantly finite } (in $\cu$) if for any $N\in \obj \cu$ there exists its $\cp$-approximation.\footnote{Actually, the standard convention is to say that $\cocp$ is contravariantly finite if this condition is fulfilled; yet our version of this term is somewhat more convenient for the purposes of the current paper.}
\end{defi}
 
\begin{lem}\label{lperf}
Let $\cp$ 
be  a subclass of $\obj \cu$; denote by $\cpt$ the Karoubi-closure of $\cocp$ in $\cu$. 

\begin{enumerate}
\item\label{ipercl} If $\cp'$ is a subclass of $\cpt$ containing $\cp$ then a $\cu$-morphism $h$ is $\cp'$-null (resp. a $\cp'$-epic, resp. a $\cp$-approximation) if and only if it is  $\cp$-null (resp. a $\cp$-epic, resp. a $\cp$-approximation). 

\item\label{iperot} In a $\cu$-distinguished triangle $M\stackrel{h}{\to}N \stackrel{f}{\to}Q\stackrel{g}{\to} M[1]$ the morphism $h$  is a $\cp$-epic if and only if $f$ is $\cp$-null; this is also equivalent to $g$ being $\cp$-monic. 

\item\label{ipereq} The class $\cp$ is (countably) perfect if and only if any (countable) coproduct 
of  $\cp$-epic morphisms is  $\cp$-epic.

Moreover, if this is the case then any (countable) coproduct 
 of $\cp$-approximations is a $\cp$-approximation.

\item\label{ipertest} If $h:M\to N$ is a  $\cp$-approximation of $N$ then  a $\cu$-morphism $g:N\to N'$ is $\cp$-null if and only if $g\circ h=0$.


\item\label{ipercrit} Assume that for any (countable) 
collection of  $N_i\in \obj \cu$ the object $\coprod N_i$ possesses a $\cp$-approximation being the coproduct of some $\cp$-approximations of $M_i$. Then $\cp$ is   (countably) perfect. 

\item\label{iperfcovarf} If $\cp$ is a set then  it is  contravariantly finite. 

\item\label{iperloc}
 Let $F:\du\to \cu$ be an exact functor that possesses a right adjoint $G$ respecting (countable) coproducts.
 Then for any (countably) perfect class $\cp'$ of objects of $\du$ the class  $\cp=F(\cp')$  is (countably) perfect also.
\end{enumerate}
\end{lem}
\begin{proof} All of the assertions are rather easy.

\ref{ipercl}, \ref{iperot}: 
 obvious.

Assertion \ref{ipereq} follows from assertion \ref{iperot} immediately according to Proposition \ref{pcoprtriang}.

\ref{ipertest}. Since $M$ is a coproduct of elements of $\cp$, the composition of $h$ with any $\cp$-null morphism is zero. Conversely,  since any morphism from $\cp$ into $N$ factors through $h$, if $g\circ h=0$ then $g$ is $\cp$-null. 

\ref{ipercrit}. Let $f_i:N_i\to P_i$ be some $\cp$-null morphisms; choose $\cp$-approximations $h_i:M_i\to N_i$ such that $\coprod h_i$ is a $\cp$-approximation of $\coprod N_i$. Then $\coprod f_i \circ \coprod h_i=\coprod (f_i\circ h_i)=0$. Hence $\coprod f_i$ is $\cp$-null according to assertion \ref{ipertest}.

 \ref{iperfcovarf}. Easy and standard: a $\cp$-approximation of $M\in \obj \cu$ is given by $\coprod_{P,h_P}\stackrel{\bigoplus h_P}{\to}M$, where 
 $P$ runs through all elements of $\cp$ and $h_P$ runs through $\cu(P,M)$.

\ref{iperloc}. The adjunction immediately yields that a $\cu$-morphism $h$ is $\cp$-null if and only if $G(h)$ is $\cp'$-null. It remains to recall that $G$ respects (countable) coproducts. 
\end{proof}

\begin{rema}\label{requivdef} Our definition of perfect classes essentially coincides with the one used in \cite{modoi}. Moreover,  part \ref{ipereq} of the lemma gives the equivalence of our definition (\ref{dcomp}(\ref{idpc})) of countably perfect classes  to 
 the condition (G2)  in the definition of perfect generators in \cite{kraucoh}; hence a category is perfectly generated in the sense of ibid. if and only if it is so in the sense of Definition \ref{dwg}. Similarly, our condition \ref{isym} in Definition \ref{dsym} 
 is equivalent to  condition (G3) in \cite[Definition 2]{kraucoh}; hence $\cu$ is symmetrically generated in the terms of loc.cit. whenever it has products and contains a set $\cp$ that  Hom-generates $\cu$ and is symmetric to some set $\cp'\subset \obj \cu$. 

Furthermore, any  class that is $\aleph_1$-perfect in the sense of  \cite[Definition 3.3.1]{neebook} is countably perfect.
We also obtain  that our definition of $\al$-well generated categories is equivalent to the one given in \cite{krauwg}.
Moreover, recall that the 
	latter definition is equivalent to the definition 
	given in \cite{neebook} 
 according to Theorem A of ibid.

\end{rema}

We deduce some immediate consequences from the lemma.

\begin{coro}\label{cwftw}
1. Assume that $\cp$ is a (countably)  perfect  cosuspended set of objects of $\cu$. 
Then the weight structure $w$ constructed in Theorem \ref{tpgws} is (countably) smashing. 

2. Assume that $\{\cp_i\}$ is a set of (countably) perfect  
sets of objects of $\cu$. Then the couple $w=(\cu_{w\le 0}, \cu_{w\ge 0})$ is a (countably) smashing weight structure on $\cu$, where $\cu_{w\le 0}$ is the big hull of $\cup_{j\ge 0,i} \cp_i[-j]$ and $\cu_{w\ge 0}=\cap_{j\ge 1,i} (\cp_i\perpp[-j])$. 

3. Assume that  $\{w_i\}$ is a set of perfectly generated weight structures on $\cu$, i.e., assume that there exist countably perfect sets $\cp_i\subset \obj \cu$ that generate $w_i$ (see Remark \ref{rwhop}(1)).
Then the couple $w=(\cu_{w\le 0}, \cu_{w\ge 0})$ is a weight structure, where   $\cu_{w\le 0}$ is the big hull of $\cup_i \cu_{w_i\le 0}$ and  $\cu_{w\ge 0}=\cap_i \cu_{w_i\ge 0}$. Moreover, $w$ is perfectly generated in the sense of Remark \ref{rigid}(4);  it is  smashing whenever all $w_i$ are.

\end{coro}
\begin{proof}
1. Recall that we should check whether $\cu_{w\ge 0}=\cp\perpp[-1]$ is closed with respect to small (resp. countable) $\cu$-coproducts. Hence the statement follows immediately from  Proposition \ref{psym}(\ref{isymeu}). 

2. $\cup_{j\ge 0;i} \cp_i[-j]$ is a  (countably) perfect set according to  Proposition \ref{psym}(\ref{isymuni}); 
 it is certainly cosuspended. Hence $w$ is a weight structure according to Theorem \ref{tpgws}. Lastly, the smashing property  statements follow immediately from the previous assertion.

3. 
According to the previous assertion, the couple $(\cu_{w'\le 0}, \cu_{w'\ge 0})$ is a weight structure on $\cu$, where $\cu_{w'\le 0}$ is the big hull of $\cup_{j\ge 0,i} \cp_i[-j]$ and $\cu_{w'\ge 0}=\cap_{j\ge 1;i} (\cp_i\perpp[-j])$. 

Now we compare $w$ with $w'$. Since $\cp_i$ generate $w_i$, we certainly have $\cu_{w\ge 0}=\cu_{w'\ge 0}$. Next,  $\cu_{w\le 0}\perp \cu_{w\ge 0}[1]$ according to Lemma \ref{lbes}(\ref{isesperp}). Since $\cu_{w\le 0}$ contains $\cu_{w'\le 0}$, these classes are equal. 

Thus $w$ is a perfectly generated weight structure. It is smashing if all $w_i$ are; indeed,   $\cu_{w\ge 0}$ is coproductive as being the intersection of coproductive classes.
\end{proof}

\begin{rema}\label{revenmorews} 
1.  In particular, we obtain that $w$ is perfectly generated in the sense of Remark \ref{rigid}(4) if and only if it is generated by a 
  countably perfect set (i.e., if $\ro=\cu_{w\le -1}$ equals $\cp^\perp$ for some countably perfect set $\cp$). 
Note here  that we do not have to assume that $\cp$ is cosuspended, since $\cup_{i\le 0}\cp[i]$ generates $w$ whenever $\cp$ does.

Moreover, part 3 of the corollary  gives a certain "join" operation on perfectly generated $t$-structures (and so, we obtain a monoid). 
 Note also that the join of any set  of  smashing  weight structures is smashing also.

2. Thus our corollary gives a vast source of smashing weight structures. 
Now, the results of \S\ref{sadjt} allow to construct "new" $t$-structures that are right adjacent to these weight structures (cf. Remark \ref{rnewt}(1)) and also describe their hearts. 

Note also that for $t_i$ being right adjacent to (perfectly generated) $w_i$ (for $i\in I$) and the corresponding $w$ and $t$ we obviously have $\cu^{t\ge 0}=\cap_{i\in I} \cu^{t_i\ge 0}$, whereas $\cu^{t\le 0}$ is the big hull of $\cup_{i\in I} \cu^{t_i\le 0}$.

\end{rema}

Now we prove that {\bf suspended} symmetric sets generate $t$-structures. 

\begin{theo}\label{tsymt}
Assume that $\cu$ (also) has products; for a class $\cp\subset \obj \cu$ assume that $\cp$ is {\it suspended} (i.e.,  $\cp[1]\subset \cp$) and $\cp$ is symmetric (see Definition \ref{dsym}(\ref{isym})) to a {\bf set} $\cp'$ (of objects of $\cu$).

Then the following statements are valid.

1. There exists a $t$-structure $t$ on $\cu$ such that $\cu^{t\ge 1}=\cp\perpp$ (i.e., $t$ is generated by $\cp$) 
 and also a cosmashing  weight structure $w$  that is right adjacent to $t$. 

2. $\hrt$ has 
an injective cogenerator and satisfies the AB3* axiom. Moreover, $\hrt$ is naturally anti-equivalent to the subcategory of  $\adfu(\hw,\ab)$ consisting of those functors that respect products; $\hw$ is  naturally equivalent to the subcategory of  injective objects of $\hrt$.

3. If  
 $\cp'$ is suspended as well then $\cu_{w\ge 0}$ equals the big hull of $\cp'$ in $\cu\opp$. 
\end{theo}
\begin{proof}
Firstly we note that $\cp$ is also symmetric to $\cup_{i\ge 0}\cp'[i]$ (see Proposition \ref{psymb}(I.\ref{iws1})); hence it suffices to consider the case where 
$\cp'$ is suspended.

Now we adopt the notation of Proposition \ref{psymb}(I.\ref{isym4}). 

Assume first that  $\du=\ns$ (in the notation of the proposition), and so, $\cu'=\cu$. According to (parts I.\ref{iwsym1}  and I.\ref{iws2} of)  the aforementioned proposition, the set $\cp'$ is perfect in the category $\cu\opp$. 
 Since $\cp'$ is certainly cosuspended in $\cu\opp$,  Theorem \ref{tpgws} yields a smashing weight structure $w_{\cp'}^{op}$  that is $\cu\opp$-generated by $\cp'$ with $\cu\opp_{w^{op}\le 0}$ being the  $\cu\opp$-big hull of $\cp'$. 
 The corresponding weight structure $w_{\cp'}$ on $\cu$ (see Proposition \ref{pbw}(\ref{idual}))  will be our candidate for $w$.  $w_{\cp'}$ is certainly  cosmashing; hence $\hw_{\cp'}$ is 
 closed with respect to $\cu$-products. Moreover,  $\hw_{\cp'}$ has a {\it cogenerator}, i.e., any its object is a retract of a product of (copies of) some $M\in \cu_{\hw_{\cp'}=0}$; here we apply (the dual to) Corollary \ref{cvttbrown}(I.2) 
(see Theorem 2.4.1(3) and Remark 2.4.2(3) of \cite{bpws} for a more detailed argument). 

Next we apply Corollary \ref{cdualt} to obtain that there exists a $t$-structure $t_{\cp'}$ on $\cu$ that is left adjacent to $w_{\cp'}$. Moreover, $\hrt_{\cp'}$  is anti-equivalent to the subcategory of  $\adfu(\hw_{\cp'},\ab)$ consisting of those functors that respect products (according to the corollary); $\hw_{\cp'}$ is equivalent to the subcategory of  injective objects of $\hrt_{\cp'}$.  Hence $\hrt_{\cp'}$ 
has an injective cogenerator; it satisfies the AB3* axiom since $\hw_{\cp'}$ has products.

So to finish the proof in this case it remains to note that $t_{\cp'}$ is precisely the $t$-structure  generated by $\cp$ according to  Proposition \ref{pwsym}(\ref{iws4}). 

Now we proceed to the general case of our setting  (using Proposition \ref{psymb}(I.\ref{isym4}) also). We recall that the corresponding category $\du'$ is closed with respect to $\cu$-products and $\du'^{op}$ is perfectly generated by $\cp'$. Thus (by Theorem \ref{tpgws}) there exists a weight structure $w_{\du'}$ on $\du'$ with $\du'_{w_{\du'}\le 0}$ 
 equal to the $\du'^{op}$-big hull of $\cp'$ and $\du'_{w_{\du'}\ge 0}=({}^{\perp_{\du'}} \cp')[1]$. Once again, we apply Corollary \ref{cdualt} to obtain a $t$-structure $t_{\du'}$ on $\du$ that is left adjacent to $w_{\du'}$. We also obtain 
 that $\hrt_{\du'}$ and $\hw_{\du'}$ are related similarly to assertion 2.

Now we "extend" $t_{\du'}$ and $\hw_{\du'}$ to $\cu$. Recall that the embedding $\du'\to \cu$ respects products and possesses a left adjoint; hence (the dual to) 
 Remark \ref{rwhop}(2) gives the existence of a weight structure $w_{\cp'}$ as above. 
 
Next, we consider the equivalence $j:\du'\to \cu'$ induced by the functor $G$ (that is right adjoint to $i:\cupr\to\cu$) and denote by $t_{\cupr}$ the $t$-structure on $\cupr$ obtained from $t_{\du'}$ via 
$j$. Since the embedding $i:\cupr\to \cu$ possesses an (exact) right adjoint, we obtain (using Proposition \ref{phopft}(II.1))  that $(\cupr^{t_{\cupr}\le 0},(\cupr^{t_{\cupr}\le 0})\perpp[1])$ is a $t$-structure $t$ on $\cu$. Note also that $\hrt$ equals $j(\hrt_{\du'})$,  $\hw_{\cp'}=\hw_{\du'}$, and $\cu_{w_{\cp'}\ge 0}=\du'_{w_{\du'}\le 0}$; hence assertions 2 and 3 follow from assertion 1, and 
 to prove the latter it suffices to verify that $(\cupr^{t_{\cupr}\le 0})\perpp=\cp\perpp$. Since $i$ possesses a right adjoint, for the latter purpose one should compare  $\cp\perpp\cap \obj \cupr$ with $(\cupr^{t_{\cupr}\le 0})\perpp\cap \obj \cupr=\cupr^{t_{\cupr}\ge 1}=j(\cp')^{\perp_{\cupr}}$ (here we apply  Remark \ref{rwhop}(2)). It remains to recall that $\cp$ is symmetric to $j(\cp')$ in $\cupr$ and apply  
 Proposition \ref{pwsym}(\ref{iws4}) once again.
\end{proof}

To demonstrate the relevance of our theorem, we apply it to the study of compactly generated $t$-structures.

\begin{coro}\label{csymt}
Let $\cq$ be a set of compact objects of $\cu$. 

1. Then $t=(\cu^{t\le 0},\cu^{t\ge 0})$ is a smashing $t$-structure on $\cu$, where $\cu^{t\le 0}$ is the 
smallest coproductive extension-closed subclass of $\obj \cu$ containing $\cup_{i\ge 0}\cq[i]$ and $\cu^{t\ge 0}=\cap_{i\ge 1} \cq\perpp[i]$.\footnote{Certainly, $t$ is generated by the set $\cup_{i\ge 0}\cq[i]$. Hence Theorem \ref{tclass}(\ref{iclass1},\ref{iclasst}) 
 gives a "more precise" description of the class $\cu^{t\le 0}$. }
Moreover, $\hrt$ has an injective cogenerator and satisfies the AB3* axiom; it is anti-equivalent to the 
subcategory of $\adfu((\injh)\opp,\ab)$ consisting of those functors that send $\injh$-products into products of abelian groups.\footnote{This relation of $\hrt$ with its subcategory $\injh$ of  injective objects does not mention weight structures; yet it appears to follow from the existence of an injective cogenerator along with the AB4 property.}

2.  Assume in addition that $\cu$ satisfies the Brown representability condition. Then there exists a weight structure $w$ that is right adjacent to $t$ with $\hw\cong \injh$, and $\cu_{w\ge 0}$ equals the $\cu\opp$-big hull of $\{\hat{Q}[i]_{\cu}:\ Q\in \cq, i\ge 0\}$, where $\hat{Q}$ is the Brown-Comenetz dual of $Q$ (that represents the functor $M\mapsto \ab(\cu(Q,M),\q/\z)$).\footnote{Recall here that $\cu\opp$ has coproducts according Proposition \ref{pcomp}(II.2).}

3. Assume that the $t$-structure $t$ mentioned in assertion 1 is non-degenerate. Then $\hrt$ is Grothendieck abelian.

\end{coro}
\begin{proof}
1. $t$ is a $t$-structure on $\cu$ according to Theorem A.1. of \cite{talosa}; 
 it is certainly smashing. 
Next we take $\cu'$ to be the localizing subcategory of $\cu$ generated by $\cq$.  
 Then $\cq$ is symmetric to the set $\cq'_{\cupr}$ of $\cu'$-Brown-Comenetz duals of elements of $\cq$ according to Proposition \ref{psymb}(II.\ref{iws23}). 
Hence we can apply Theorem \ref{tsymt} to the category $\cu'$ with the corresponding $\cp$ and $\cp'$ being equal to $\cup_{i\ge 0}\cq[i]$ and to $\cup_{i\ge 0}\cq'_{\cupr}[i]$, respectively. This yields the result since for the corresponding $t_{\cu'}$ we have $\cu^{t\le 0}=\cu'^{t_{\cu'}\le 0}$ and $\cu^{t\ge 0}\cap \obj \cu'=\cu'^{t_{\cu'}\ge 0}$ (here we note that the embedding $\cu'\to \cu$ has a right adjoint according to Proposition \ref{pcomp}(II.2), and apply 
 Remark \ref{rwhop}(2)).

2. Once again, it suffices to combine Proposition \ref{psymb}(II.\ref{iws23}) with Theorem \ref{tsymt}.

3. Since $t$ is non-degenerate, the set $\cq$ Hom-generates $\cu$. Hence the previous assertion gives the existence of $w$ that is right adjacent to $t$.

Now to prove the result it suffices to repeat the argument used in the proof of \cite[Corollary 4.9]{humavit}. We do so briefly here (without recalling the corresponding definitions).

Firstly, the corresponding shifts of the classes $\cu^{t\le 0}$, $\cu^{t\ge 0}=\cu_{w\le 0}$, and  $\cu_{w\ge 0}$ give  a {\it cosuspended TTF triple}; see Definition 2.3 of ibid.\footnote{Note that in Definition 2.1 of ibid. $t$-structures and co-$t$-structures (i.e., weight structures) were defined as the corresponding types of torsion pairs; so our definition differs from loc. cit. by shifts of the corresponding $\ro$ (cf. Remarks \ref{rtst1}(\ref{it1}) and \ref{rwhop}(1)).} Next, $\cu_{w\le 0}$ is definable in the sense of Definition 4.1 (since it is the zero class of the set $\{\cu(Q[i],-),\ Q\in \cq,\ i>0\}$ of coherent functors; see Lemma \ref{lbes}(\ref{izs}) for the definition of zero classes). Applying Theorem 4.8 of ibid. we obtain that  $\cu^{t\le 0}=\perpp\{I[j],\ j< 0\}$ for some {\it pure-injective cosilting object} $I$ of $\cu$ (this is where we use the non-degeneracy assumption on $t$). Thus it remains to apply Theorem 3.6 of ibid. 
\end{proof}

\begin{rema}\label{rsymt}
\begin{enumerate}
\item\label{irsymt2}
The author does not know how to deduce 
 the existence of $w$ (in part 2 of our corollary) from the results of \S\ref{sadjw}. Another interesting fact related to this statement is Theorem 3.11 of \cite{postov} (where algebraic triangulated categories were considered); cf. Remark \ref{rwsym}(1) above.\footnote{Note also  that the reasoning of Pospisil and  \v{S}\v{t}ov\'\i\v{c}ek in the proof of loc. cit. works for arbitrary torsion pairs; this is certainly not the case for our arguments.} So, we generalize "the $t$-structure case" of loc. cit.; this immediately yields the corresponding generalization of \cite[Corollary 4.9]{humavit} (in part 3 of our corollary). 

Recall also that $\cu^{t\le 0}=\perpp\{I[j],\ j<0\}$ for some pure-injective cosilting object $I$ under the assumptions of our corollary (see its proof). It is easily seen that 
any cogenerator of $\hw$ (see the proof of Theorem \ref{tsymt}) 
can be taken for $I$ in this statement.

\item\label{irsymtab5}
In Theorem \ref{tab5} below we prove under certain restrictions on $\cu$ (and without assuming that $t$ is non-degenerate) that $\hrt$ is an AB5 category; this argument is completely independent from ibid. (and gives some interesting additional information on $\hrt$).  On the other hand, the generators of $\hrt$ given by part 1 of that theorem are 
 the same as the ones given by the proof of  \cite[Theorem 3.6]{humavit} (yet they were not specified explicitly in loc. cit.). 

\item\label{irsymt1}
It appears to be quite difficult to produce perfect and symmetric sets "out of nothing"; note that the existing literature on this subject mostly concentrates on the search of shift-stable sets of perfect generators of triangulated categories. So, the  weight structures opposite to those given by Corollary \ref{csymt}(2) appear to be (essentially) the only known "type" 
 of perfectly generated weight structures that are not compactly generated (yet cf. Remark \ref{rnondeg}(3)). 

Let us prove that the corresponding weight structure $w^{op}$ is not compactly generated if the set $\cq$ Hom-generates $\cu$ (certainly, this condition is fulfilled if and only if $\cq$ generates $\cu$ as its own localizing subcategory). Indeed, in this case the symmetric  set  $\hat{Q}$ Hom-generates $\cu^{op}$. Then Remark \ref{rwhop}(8) says that   $w^{op}$ is left non-degenerate; 
 thus any class generating $w^{op}$ also Hom-generates $\cu^{op}$. On the other hand, $\cu^{op}$ is  not compactly generated 
according to Corollary E.1.3 (combined with Remark 6.4.5) of \cite{neebook}. 

Certainly, one can also consider the direct sum of an example of this sort with a "compactly generated" one.  

\item\label{irsymt4}
 Probably the most interesting case of  Theorem \ref{tsymt} and Corollary \ref{csymt} is the one where $\cp$ (resp. $\cq$) Hom-generates $\cu$; note that in this case $\cu$ has products and satisfies the Brown representability condition automatically. Note also that $\cu$ is Hom-generated by the corresponding set if and only if $t$ is right non-degenerate (see Definition \ref{dtstr}(\ref{ito3}); this is certainly equivalent to the right non-degeneracy of $w$). 

\item\label{irsymt5} It is easily seen that 
the Brown-Comenetz duals of any family $\{F_i\}$ of cc functors $\cu\to \ab$ that are also pp ones form a perfect class in $\cu\opp$. Yet this observation can scarcely give any "new" weight structures 
 since all "known" functors satisfying these conditions appear to be corepresented by compact objects of $\cu$ (cf. \cite[Proposition 2.9]{krause}). Moreover, when we pass from the weight structure $w$ to its left adjacent $t$ we apply the dual Brown representability condition, whereas the latter says that all pp functors are corepresentable.
\end{enumerate}
\end{rema}

\subsection{On well generated weight structures and torsion pairs}\label{swgws}

Now we study the 
relation of (countably) perfect classes to torsion pairs and (especially) to weight structures. In particular, we obtain a complete "description" of smashing weight structures on well generated triangulated categories (see Theorem \ref{twgws}(III)). 

We will need some new definitions to deal with well generated categories. Most of them are simple variations of the notions described above; we  also recall the notion of $\be$-compact objects.

\begin{defi}\label{dbecomp}
 Let 
 $\be$ be a regular infinite cardinal.
\begin{enumerate}
\item\label{idclass} We will say that a class $\cpt$ of objects of $\cu$ is {\it $\be$-coproductive} if it is closed with respect to $\cu$-coproducts of less than $\be$ objects.

\item\label{idchop} We will say that a torsion pair $s=(\lo',\ro')$ for a full triangulated subcategory $\cu'$ of $\cu$ is {\it $\be$-coproductive} if both $\obj \cu'$ and $\ro'$ are $\be$-coproductive.

\item\label{idcomp} We will say that an object $M$ of $\cu$ is {\it $\be$-compact} if it belongs to the maximal   perfect class of $\be$-small objects of $\cu$ (whose existence is given by Proposition \ref{psym}(\ref{isymuni})). We will write $\cu^{\be}$ for the full subcategory of $\cu$ formed by $\be$-compact objects.
\end{enumerate}
\end{defi}

\begin{rema}\label{rbecomp}
1. Our definition of $\be$-compact objects is equivalent to the one used in \cite{krauwg}.
Indeed, coproducts of less than $\be$ of  $\be$-small objects are obviously $\be$-small; thus $\obj \cu^\be$ is $\be$-coproductive. Hence the equivalence of definitions follows from Lemma 4 of ibid. 
Furthermore, Lemma 6 of ibid. states  that 
(both of) these definitions are equivalent to Definition 4.2.7 of \cite{neebook} if we assume in addition that $\cu^{\be}$ is essentially small.

2. Now we recall some more basic properties of $\be$-compact objects in an $\al$-well generated category $\cu$ assuming that $\be\ge \al$.

Theorem A of \cite{krauwg} yields immediately that $\cu^\be$ is an essentially  small triangulated subcategory of $\cu$. 

Moreover, the union of  $\cu^{\gamma}$ for $\gamma$ running through all regular cardinals ($\ge \al$) equals $\cu$ (see the Corollary in loc. cit. or Proposition 8.4.2 of \cite{neebook}). 

3. Lastly, we recall a part of \cite[Lemma 4]{krauwg}. 
For any $\be$-coproductive essentially small  perfect class $\cp$ of $\be$-small objects of a triangulated category $\cu$ (that has coproducts) it says the following:
for any $P\in \cp$ and any set of  $N_i\in \obj \cu$ any morphism $P\to \coprod N_i$ factors through the coproduct of some $\cu$-morphisms $M_i\to N_i$ with $M_i\in \cp$. 

\end{rema}

\begin{theo}\label{twgws}
Let $s=(\lo,\ro)$ be a torsion pair for $\cu$ that is countably smashing, $\cp\subset \obj \cu$. 

I. Consider the class $J$ of $\cu$-morphisms characterized by the following condition: $h\in \cu(M,N)$ (for $M,N\in \obj \cu$) belongs to $J$ whenever for any chain  of morphisms $ L_sP\stackrel{a_P}{\to} P\stackrel{g}{\to}M \stackrel{h}{\to}N$ its composition is zero if $P\in \cp$ and $a_P$ is an $s$-decomposition morphism.

Then the following statements are valid.

\begin{enumerate}
\item\label{indep}  The class $J$ will not change if we will fix $a_P$ for any $P\in \cp$ in this definition.

\item\label{icontraf} Assume that $\cp$ is  contravariantly finite and $s$ is smashing. Then $h$ 
belongs to $J$ if and only if there exists a
$\cp$-approximation morphism $AM\stackrel{g}{\to} M$ and an $s$-decomposition morphism $a_{AM}: L_sAM\to AM$ such that 
$h\circ g \circ a_{AM}=0$. Moreover, the latter is equivalent to the vanishing of all compositions of this sort.

\item\label{icoprcl}
 Assume that $\cp$ is  contravariantly finite and  (countably) perfect, 
and $s$ is smashing. Then the class $J$ is closed with respect to (countable) coproducts.

\item\label{ilscp}
 Assume that for any $P\in \cp$ there exists a choice of $ L_sP\in \cp$; denote the class of these choices by $\lscp$. Then $J$ coincides with the class of $\lscp$-null morphisms.

\item\label{ilscper} Assume in addition (to the previous assumption) that  $\cp$ is a (countably) perfect   contravariantly finite  class 
and $s$ is smashing. Then $\lscp$ is a (countably) perfect   contravariantly finite class also.

\item\label{ilscperw}
 Assume in addition that $s$ is weighty; suppose that the class $\cp$ is essentially small, contains $\cp[-1]$,   and Hom-generates $\cu$. Then  the class $L_s\cp$  generates $s$ and $\lo$ is the big hull of $L_s\cp$; thus $s$ is perfectly generated in the sense of Remark \ref{rigid}(4).

\end{enumerate}

II. For a regular cardinal $\be$ let  $s'=(\lo',\ro')$ be a $\be$-coproductive torsion pair for a full triangulated category $\cupr$ of $\cu$ such that $\obj \cupr$ is a  perfect essentially small 
class of  $\be$-small objects.  Then $\lo'$ is  perfect also. 

Moreover, if $s'$ is weighty in $\cupr$ then $\lo'$ generates a weighty smashing torsion pair  for $\cu$.

III. Assume in addition that $\cu$ is $\al$-well generated for some regular cardinal $\al$, and that  $s$ is smashing. 

1. Assume that $s$ restricts (see Definition \ref{dhopo}(4)) to $\cu^{\be}$  for a regular cardinal $\be \ge \al$. Then $\lo\cap \obj \cu^\be$ is an essentially small  perfect class.

2. If $s$ is weighty then it restricts to $\cu^{\be}$ for all large enough regular $\be\ge \al$. Moreover, the class $\lo\cap \obj \cu^\be$ perfectly generates $s$ for these $\be$. 
\end{theo}
\begin{proof}

I.\ref{indep}. It suffices to note that any $s$-decomposition morphism for $M$ factors through any other one according to Proposition \ref{phop}(\ref{itp7}).

\ref{icontraf}. We fix $h$ (along with $M$ and $N$).  

The definition of approximations along with Proposition \ref{phop}(\ref{itp7}) implies that any composition  $ L_sP\stackrel{a_P}{\to} P\stackrel{g}{\to}M$ as in the definition of $J$ factors through the composition morphism  $L_sAM\to M$. Hence if the composition $L_sAM\to N$ is zero then $h\in J$.

Conversely, assume that $h\in J$.
Since  $\cp$ is  contravariantly finite, we can choose  a $\cp$-approximation morphism $g\in \cu(AM, M)$. 
Present $AM$ as a coproduct of some $P_i\in \cp$; choose some $s$-decomposition morphisms $L_sP_i\stackrel{a_{P_i}}\to P_i$. Since $s$ is smashing,  the morphism $a_{AM}^0=\coprod a_{P_i}$ is an $s$-decomposition one also according to Proposition \ref{phop}(\ref{itp4}). Since $h \circ g\circ  a_{P_i}=0$ for all $i$, we also have $h \circ g\circ  a_{AM}^0=0$. Lastly,  any other choice of $a_{AM}$ factors through $a_{AM}^0$ (by Proposition \ref{phop}(\ref{itp7}); cf. the proof of assertion I.\ref{indep}); this gives the "moreover" part of our assertion.

\ref{icoprcl}. This is an easy consequence of the previous assertion. Indeed, to prove that $\coprod h_i\in J$ for a small  (resp. countable) collection of $h_i\in J\cap \cu(M,N)$  note that for any choices of $\cp$-approximations $AM_i\to M_i$ their coproduct is a  $\cp$-approximation of $\coprod M_i$ (by Lemma \ref{lperf}(\ref{ipereq})). The assertion follows easily since the coproduct of any choices of $L_sAM_i\to AM_i$ of $s$-decomposition morphisms is an  $s$-decomposition morphism also (according to Proposition \ref{phop}(\ref{itp4})); thus it remains to apply assertion I.\ref{icontraf}.

\ref{ilscp}. 
 Assertion I.\ref{indep} certainly implies that any $\lscp$-null morphism belongs to $J$. The converse implication is immediate from $\lscp\subset \cp$.

\ref{ilscper}. This is an obvious combination of the previous two assertions.

\ref{ilscperw}. Since $\lo$ contains $\lscp$, it also contains its big hull (see Lemma \ref{lbes}(\ref{iseses}, \ref{isesperp})). Thus it suffices to verify the converse inclusion. 

Now, since $\cp$ is essentially small, countably perfect, and $\cp = \cp[1]$, the big hull of $\cp$ along with $\cp^\perp$ is a (weighty) torsion pair according to Theorem \ref{tpgws}. 
   Since $\cp$ Hom-generates $\cu$,  we obtain 
$\cp^\perp=\ns$; 
 thus any object of $\cu$ belongs to this big hull.  

Now let $P$ belong $\lo$. As we have just proved, it is a retract of some $P'$ belonging to the naive big hull of $\cp$. We present $P'$ as $\inli Y_i$ so that 
$Y_0$ and cones of the connecting morphisms $\phi_i$ belong to 
$\cocp$. 
 Thus for $Z_i$ being as in Lemma \ref{lbes}(\ref{ilwd}) we can choose $L_sZ_i\in {\underline{\coprod} \lscp}$. Applying the lemma we obtain the existence of an $s$-decomposition triangle $L'\to P'\stackrel{n_{P'}}{\to} R'\to L[1]$ 
with $L'$ belonging to the naive big hull of $ {\underline{\coprod} \lscp}$. Thus applying Proposition \ref{phop}(\ref{itp7p}) we obtain that $P$ belongs to the  big hull of $ {\underline{\coprod} \lscp}$.


II. Let $f_i\in \cu(N_i,Q_i)$ for $i\in J$ be a set of $\lo'$-null morphisms; for $N=\coprod N_i$, $f=\coprod f_i$, and $P\in \lo'$ we should check that the composition of any $e\in \cu(P,N)$ with $f$ vanishes. The $\be$-smallness of $P$ allows us to assume that $J$ contains less than $\be$ elements.

Next, Remark \ref{rbecomp}(3) gives a factorization of $e$ through the coproduct of some $h_i\in \cu(M_i,N_i)$ with $M_i\in \obj \cupr$. We choose some $s'$-decompositions $L_i\to M_i\to R_i\to L_i[1]$ of $M_i$. Our assumptions easily imply that $\coprod L_i\to \coprod M_i\to \coprod R_i$ is an $s'$-decomposition of $\coprod M_i$ (cf. Proposition \ref{phop}(\ref{itp4})). Hence part  \ref{itp7} of the proposition implies that $e$ factors through the coproduct $g$ of the corresponding morphisms $L_i\to N_i$. Now, since $f_i$ are $\lo'$-null and $L_i\in \lo'$ then $f\circ g=0$; hence $f\circ e=0$ as well.

Lastly, if $s'$ is weighty then $\lo'$ is cosuspended. 
  Since it is also essentially small it remains to apply Theorem \ref{tpgws} along with Corollary \ref{cwftw}(1).

III. For $\be\ge \al$ being a regular cardinal we take $\cp=  \obj \cu^\be$. This is certainly a perfect essentially small class that Hom-generates $\cu$; 
 we also have $\cp=\cp[1]$. 

 To prove assertion III.1 it suffices to note that $\lo\cap \obj \cu^\be$ is a possible choice of $L_s\cp$ (in the notation of assertion I) and apply assertion I.\ref{ilscper}. 

Next, assertion I.\ref{ilscperw} (combined with Remark \ref{revenmorews}(1)) implies that to prove assertion III.2 it suffices to verify that $s$ restricts to $\cu^{\be}$ for all large enough regular $\be\ge \al$.

Now we choose some $L_sM$ for all $M\in \obj \cu^\al$, and and take   a regular cardinal $\al'$  such that all elements of $L_s\cp$ belong to $\cu^{\al'}$ (see Remark \ref{rbecomp}(2)). Then for any regular $\be\ge {\al'}$ the pair $s$ restricts to $\cu^{\be}$, since the corresponding weight decompositions exist according to the "furthermore" part of Proposition \ref{ppcoprws}(\ref{icopr7p}). 
\end{proof}

\begin{rema}\label{rtkrau}
1. Our theorem suggests that it makes sense to define (at least) two distinct notions of $\be$-well generatedness for smashing torsion pairs and weight structures  in an $\al$-well generated category $\cu$. One may say that $s$ is {\it weakly $\be$-well generated} for some regular $\be\ge \al$ if it is generated by a perfect set 
 of $\be$-compact objects.
$s$ is {\it strongly  $\be$-well generated} if in addition to this condition, $s$ restricts to $\cu^\be$.

Certainly, compactly generated torsion pairs (see Definition \ref{dhopo}(3)) 
 are precisely the weakly $\alz$-well generated ones (since any set of compact objects is perfect; cf. Proposition \ref{psym}(\ref{isymcomp})). Hence  our two notions of $\be$-well generatedness are not equivalent (already) in the case $\al=\be=\alz$; this claim follows from \cite[Theorems  4.15, 5.5]{postov} (cf. also Corollary 5.6 of ibid.) where (both weakly and strongly $\alz$-well generated) weight structures on $\cu=D(\modd-R)$ were considered in detail.

Moreover, for $k$ being a field of cardinality $\gamma$ the main subject of \cite{bgn} gives the following example: the opposite (see Proposition \ref{pbw}(\ref{idual})) to (any version of) the Gersten weight structure over $k$ (on the category $\cu$ that is opposite  to the corresponding category of {\it motivic pro-spectra}; note that $\cu$ is compactly generated) is 
weakly  $\alz$-well generated (by definition) and it does not restrict to the subcategory of $\be$-compact objects for any $\be\le \gamma$. On the other hand, this example is "as bad is possible" 
  for weakly  $\alz$-well generated weight structures in the following sense: combining the arguments used the proof of part III.2 of our theorem with that for Theorem \ref{tclass} one can easily verify  that any $\alz$-well generated weight structure is $\al$-well generated whenever the set of  (all) isomorphism classes of morphisms 
in the subcategory $\cu^\alz$ of compact objects of $\cu$ is of cardinality less than $\al$. 

2. Obviously the join (see Remark \ref{revenmorews}(1) and Corollary \ref{cwftw}(3)) of any set of  weakly $\be$-well generated weight structures  is    weakly $\be$-well generated; so, we obtain a filtration (respected by joins) on the "join monoid" of weight structures. 
 The natural analogue of this fact  for strongly $\be$-well generated weight structures is probably wrong. Indeed, it is rather difficult to believe
that for a general compactly generated category $\cu$ the class of weight structures on the subcategory $\cu^{\alz}$ of compact objects (cf. Proposition \ref{psym}(\ref{isymcomp}))   would be closed with respect to joins; note that  joining compactly generated weight structures $w_i$ on $\cu$ corresponds to intersecting the classes $\cu_{w_i\ge 0}\cap \obj \cu^{\alz}$.

On the other hand, Corollary 4.7 of \cite{krause}  suggests that the filtration of   the class of smashing weight structures  by the sets of weakly $\be$-well generated ones (for $\be$ running through regular cardinals) may be "quite short".  

3. According to part III.2 of our theorem, any weight structure on a well generated $\cu$ is strongly $\be$-well generated for $\be$ being large enough.  Combining this part of the theorem with its part II 
 we also obtain a bijection between strongly  $\be$-well generated weight structures on $\cu$ and $\be$-coproductive weight structures on $\cu^\be$. 
Note that (even) the restrictions of these results to compactly generated categories appear to be quite interesting. 

4. Now assume that a weight structure $w$ on $\cu$ is strongly $\alz$-well generated; this certainly means 
 that  $\cu$ is compactly generated (see Proposition \ref{psym}(\ref{isymcomp})) and $w$ restricts to its subcategory $\cu^\alz$ of compact objects. Then Proposition \ref{ppcoprws}(\ref{icopr7p}) implies that $\cu_{w\ge 0}$ is the smallest coproductive extension-closed subclass of $\obj \cu$ that contains $\cu_{w\ge 0}\cap \obj \cu^\alz$ 
 (cf. the proof of Theorem \ref{twgws}(III.2)). Thus the $t$-structure right adjacent to $w$ is generated by the essentially small class $\cu_{w\ge 0}\cap \obj \cu^\alz$;
so it is compactly generated (and hence smashing).

5. For $\cu$ as above and 
 a weakly $\be$-well generated weight structure $w$ on it one can easily establish a natural weight structure analogue of \cite[Theorem B]{krauwg} that will 
"estimate the size" of an element $M$ of $\cu_{w\le 0}$ in terms of the cardinalities of $\cu(P,M)$ for $P$ running through $\be$-compact elements of  $\cu_{w\le 0}$ (modifying the proof of loc. cit. that is closely related to our proof of  Theorem \ref{tpgws}). Moreover, this result should generalize loc. cit. Note also that there is a "uniform" estimate of this sort that only depends on $\cu$ (and does not depend on $w$).
 This argument should also yield that a weakly $\be$-well generated weight structure is always strongly $\be'$-well generated for a regular cardinal $\be'$ that can be described explicitly.

Moreover, similar arguments can possibly yield that any smashing weight structure on a perfectly generated triangulated category $\cu$ is perfectly generated (cf. Theorem \ref{twgws}(III.2)).

6. Our understanding of "general" well generated torsion pairs is much worse than the one of (well generated) weight structures. In particular, the author does not know which properties of weight structures proved in this section can be carried over to $t$-structures.
\end{rema}

\section{On torsion pairs orthogonal with respect to dualities}\label{skan}

In this section we study dualities between triangulated categories (generalizing the bifunctor $\cu(-,-)$ along with its restrictions to pairs of triangulated subcategories of $\cu$). Our main construction tool are Kan extensions of homological functors from triangulated subcategories of $\cu$ (to $\cu$; we call the resulting functors {\it coextensions}); these are interesting for themselves.

In \S\ref{scoext} we study 
coextensions of homological functors (into an AB5  category $\au$) from a triangulated subcategory  $\cuz$   to $\cu$ following \cite{krause}. If $\cu$ is compactly generated and $\cuz$ is its  subcategory of compact objects then the {\it coextended} functors are precisely the cc ones.  As an application we demonstrate that for any compactly generated torsion pair  $s=(\lo,\ro)$ there exists an object $N$ of $\cu$ such that the functor $H^N$ kills precisely those compact objects of $\cu$ that (also) belong to $\lo$.

In \S\ref{sdual1} we recall (from \cite{bger}) the definition of a duality $\Phi:\cu\opp\times \cupr \to \au$ (we are mostly interested in the case $\cu=\ab$). The corresponding notion of {$\Phi$-orthogonal}  torsion pairs generalizes the notion of adjacent ones. In the  case where $\cu\subset \cupr$ and $\Phi$ is just the restriction of $\cupr(-,-)$  to $\cu\opp\times \cupr$ we are able to prove two natural generalizations of Proposition \ref{psatur}; so we prove (assuming some additional  conditions) 
 that for any weight structure $w$  on $\cu$ (resp. on $\cu'$) there exists a $t$-structure on $\cupr$ (resp. on $\cu$) that is right (resp. left) $\Phi$-orthogonal to $w$. The results of \cite{neesat} and \cite{roq} demonstrate that these results can be applied for $\cu$ being the derived category of perfect complexes 
 and $\cu'$ being the bounded derived  category of coherent sheaves on a scheme proper over the spectrum of a noetherian ring. 

In \S\ref{sdual2} we study in detail the case where $\cu$ and $\cupr$ contain a common triangulated subcategory $\cuz$ whose objects are cocompact in $\cu$ and compact in $\cupr$, and $\Phi$ is a certain "bicontinuous biextension" of the bi-functor $\cuz(-,-)$. We obtain that any set $\cp$ of objects of $\cuz$ gives a couple of $\Phi$-orthogonal torsion pairs (in $\cu$ and $\cu'$, respectively).

For 
 a suspended $\cp\subset \obj \cuz$  this gives (in \S\ref{sgengroth}) a $t$-structure $t$ on $\cupr$ that is generated by $\cp$ and a weight structure $w$ on  $\cu$ that is cogenerated by $\cp$ and is (left) $\Phi$-orthogonal to $t$.  A particular case of this setting is considered in  \cite{bgn} (where our results are applied to the study of  
 various motivic homotopy categories, homotopy $t$-structures, and coniveau spectral sequences). 
We also  apply this result  to the study of compactly generated $t$-structures 
The elements of the heart of $w$ give a faithful family of exact 
  functors $\hrt\to \ab$ that respect coproducts; hence $\hrt$ is an AB5 category. Next an easy argument yields the  existence of  a generator for $\hrt$ (so,  it is a Grothendieck abelian category). 

In \S\ref{sprospectra} we prove (using the results of \cite{tmodel}) that taking $\cupr$ to be the homotopy category of a proper simplicial stable model category $\gm$ and $\cu$ to be the homotopy category of the category $\gmp$ of (filtered) pro-objects of $\gm$ 
 one obtains an example of the aforementioned setting. Thus we obtain that $\hrt$ is an Grothendieck abelian category for any compactly generated $t$-structure on a "topological" triangulated category.
Moreover, this pro-object construction is used in \cite{bgn} for the description of various triangulated categories of {\it motivic pro-spectra} (and {\it comotives}) and for the study of their relation to the corresponding motivic stable homotopy categories. Another example of the couple $(\cu,\cupr)$ 
 that may be obtained this way  is $(\operatorname{SH}^{op},\operatorname{SH})$; this observation is closely related to the main subject of \cite{prospect}.

In \S\ref{slocoeff} we recall a few results on "localizing coefficients" in (compactly generated) triangulated categories and the relate this matter to torsion pairs and their orthogonality. We also study (from a similar perspective) decompositions of triangulated categories as direct sums of their subcategories. The section is included here for the purpose of using it in \cite{bgn}; still some of its results may be interesting for themselves.

\subsection{On Kan extensions of homological functors }\label{scoext}

Now we will study a method of extending of homological functors from a triangulated subcategory of $\cu$; this is a version of left Kan extensions that was studied in detail in \cite[\S2]{krause}. In particular, we obtain a description of all cc functors from $\cu$ (into an AB5 abelian category) if it 
 is compactly generated. 
Since the construction  is dual to the one applied in \cite{bger} and \cite{bgn} (and it "usually respects coproducts") we will call the resulting functors {\it coextended} ones; we will explain that they are actually Kan extensions below. 

To formulate some of the properties of the construction we will use (a few times) the following definition. 

\begin{defi}\label{drelim} 

Let 
$\cuz$ be a 
subcategory of $\cu$ (or just a class of objects), $M\in \obj \cu$; let $L$ be a small (index) category and fix a functor $L\to \cu: l\mapsto N_l$. 

1. Let $M$ be a co-cone of  
this functor (i.e., $M$ is equipped with compatible morphisms from $N_l$ for $l\in \obj L$).
Then we will say that $M$ is a {\it $\cuz$-colimit} of $(N_l)$  
if the restriction $H_M$ of the functor $\cu(-,M)$ to $\cuz$ equals the colimit of 
  $ H_{N_l}$ (in $\psv(\cuz)$, i.e., if for any $Y\in \obj \cuz$ the connecting morphisms induce an isomorphism $\cu(Y,N)\cong \inli \cu(Y,N_l)$.

2.  If $M$ is a cone of $N$ then we will say that  $M$ is a {\it $\cuz$-limit} of $(N_l)$ if   $M$ is a $\cuz^{op}$-colimit of $(N_l)$ in $\cu^{op}$.
\end{defi}

\begin{rema}\label{rlim}
We will not need much of this definition in the current paper. Moreover, it appears that the most "useful" case of part 1 (resp. 2) of the definition is the filtered one, i.e., the one where $(N_l)$ is an inductive (resp. projective) system.
\end{rema}

Now let $\cuz$ be an essentially small triangulated subcategory of $\cu$; we will also assume that $\cu$  has coproducts.  We  consider the category $\psv(\cuz)=\adfu(\cu_0^{op},\ab)$  (cf. the proof of Theorem \ref{tpgws}); 
 recall that this is a locally small  abelian category. 
It it easily seen that any  $H\in \obj  \psvcuz$ possesses a  (projective) resolution 
$ \coprod\cu_0(-, C_i)\to \coprod \cu_0(-,C_j)\to H\to 0 $
where $\{C^i\}$ and $\{C^j\}$ are some 
 families of objects of $\cu_0$; cf. Lemma 2.2 of \cite{krause}. 

\begin{pr}\label{pkrause}
Let $H_0:\cu_0\to \au$ be a homological functor, where $\au$ is an AB5 abelian category; fix some $N\in \obj \cu$. For any $M\in \obj \cu$ we fix a resolution (as above)
\begin{equation}\label{ekrause}
 \coprod_{i\in I}
H_{C_M^i}\to \coprod_{j\in J} 
H_{C_M^j}\to H_M\to 0,\end{equation} where  we use the notation $H_M$ for the restriction of the functor $\cu(-,M)$ to $\cuz$.

Then for the association $H:M\mapsto \cok(\coprod H_0(C_M^i)\to  \coprod H_0(C_M^j))$  the following statements are valid.

 \begin{enumerate}
\item\label{ikr1} $H$ is a homological functor $\cu\to \au$. 

\item\label{ikrchar} For any $\psv(\cuz)$-resolution $ \coprod H_{C'{}_M^{i}}\to \coprod H_{C'{}_M^{j}} \to H_M\to 0$  of $H_M$ with $C'{}_M^{i}$ and $C'{}_M^{j}$ being some objects of $\cuz$, the object $\cok(\coprod H_0(C'{}_M^{i})\to  \coprod H_0(C'{}_M^{j}))$ is canonically isomorphic to $H(M)$. 

In particular, the restriction of $H$ to $\cuz$ is canonically isomorphic to $H_0$. 
We will call $H$ the {\it coextension} of $H_0$ to $\cu$, and say that it is  a {\it coextended} (from $\cuz$) functor.

\item\label{ikres} 
More generally, we have $H(M)\cong \cok(\coprod H(C'{}_{M^i})\to  \coprod H(C'{}_{M^j}))$ (also) if  the $\psvcuz$-sequence  $ \coprod H_{C'{}_M^{i}}\to \coprod H_{C'{}_M^{j}} \to H_M\to 0$  is exact for some objects $C'{}_M^{i}$ and $C'{}_M^{j}$ of $\cu$. 

\item\label{ikrtriv} If $H_0:\cuz\to \ab$ is corepresented by some $M_0\in \obj \cuz$ then $H$ is $\cu$-corepresented by $M_0$ also.

\item\label{ikr7} For any exact sequence $\coprod H^i_0\to \coprod H^j_0\to H_0\to 0$ the corresponding $\psv(\cu)$-sequence of coextensions  $\coprod H^i\to \coprod H^j\to H\to 0$ is exact also. In particular, the coextension of $\coprod H^j_0$ equals  $\coprod H^j$.

\item\label{ikr2} $H(N)$ for $N\in \obj \cu$ only depends on the restriction of $\cu(-,N)$ to $\cuz$.

\item\label{ikradj} Let $\eu$ be a full triangulated subcategory of $\cu$ that contains $\cuz$ 
and  assume that there exists a right adjoint $G$ to the embedding $\eu\to \cu$ (so, $\eu$ has coproducts). 
Then we have $H\cong H^{\eu}\circ G$,
where the functor $H^{\eu}:\eu\to \au$ is defined on $\eu$ using the same construction as the one used for the definition of $\cu$.

\item\label{ikr3} 
If $N$ is a $\cuz$-colimit of some 
$(N_l)$ (see Definition \ref{drelim}(1)), then $H(N)\cong \inli H(N_l)$.

Moreover, such a set of $(X_l,f_l)$ exists for any $X\in \obj \cu$.

\item\label{ikrtransf}  For $H'_0$ being another homological functor $\cuz\to \au$ and the corresponding coextension $H'$ we have the following:
the restriction of  natural transformations $H\implies H'$ to the subcategory $\cuz$ gives a one-to-one correspondence between them and the transformations $H_0\implies H'_0$.

\item\label{ikr6} Let $H_0\stackrel{f_0}{\to} H'_0 \stackrel{g_0}{\to} H''_0$
be a (three-term) complex of  homological functors $\cuz\to \ab$ 
 that is exact in the term $H'_0$. 
 Then 
 the complex $H\stackrel{f}{\to} H' \stackrel{g}{\to} H''$ (here $H,H',H'',f,g$ are the corresponding coextensions) is exact  in the middle also.

\item\label{ikr8}  Assume that all objects of $\cuz$ are compact.
Then  $H$ 
is determined (up to a canonical isomorphism) by the following conditions: it respects coproducts, and its restriction to $\cuz$ equals $H_0$, and it kills $\cuz^\perp$.
\end{enumerate}

\end{pr}
\begin{proof}
\ref{ikr1}, \ref{ikrchar}. Immediate from \cite[Lemma 2.2]{krause} (see also Proposition 2.3 of ibid.).

\ref{ikres}--\ref{ikr3}. 
The proofs are straightforward (and very easy); cf. also Remark \ref{rkrause}(I.1) below.

Assertion \ref{ikrtransf} is easy also. The injectivity of this restriction correspondence follows easily from the previous assertion and the surjectivity is immediate from the naturality of the coextension construction in $H_0$. 

\ref{ikr6}. We should check that the sequence  
$H(M)\stackrel{f(M)}{\to} H'(M) \stackrel{g(M)}{\to} H''(M)$ of abelian groups  is exact (in the middle term) for any $M\in \obj \cu$. Certainly, the functoriality of coextensions gives $g(M)\circ f(M)=0$.

Our exactness assertion is obviously valid if $M\in \obj \cuz$. We reduce the general case of this statement to this one.

We start from analysing the sequence (\ref{ekrause}).  The Yoneda lemma immediately implies that the natural transformations in it are given by certain $\cu$-morphisms
 $f_j: C_M^j  \to M$ for all
$j\in L$ and $g_{ij}:C_M^i\to C_M^j $ for all $i\in I$, $j\in J$; moreover,  for any $i\in I$ almost all $g_{ij}$ are $0$ 
 and  we have $\sum_{j\in J} f_j\circ g_{ij}=0$.

We should check that  if for $a\in H'(M)$ we have $g(M)(a)=0$,  then $a=f(M)(x)$ for some $x\in H(M)$.

Using the additivity of $\cuz$  we can gather finite sets of $C_M^i$ and $C_M^i$ in (\ref{ekrause}) into single objects. Hence we can assume
that the following assumptions are fulfilled: $a=H'_0(f_{j_0})(a_0)$ for some $j_0\in J$ and $a_0\in H'_0(C_M^{j_0} )$, $g_0(C_M^{j_0})(a_0)=H''_0(g_{i_0j_0}) (b_0)$ for  some $i_0\in I$ and $b_0\in H''_0(C_M^{i_0} )$ (recall that $H''(M)$ is defined as the corresponding cokernel!). Moreover, we can assume that $g_{i_0j}=0$ for any $j\neq j_0$; thus $f_{j_0}\circ g_{i_0j_0}=0$. 
Complete $g_{i_0j_0}$ to a distinguished triangle $C_M^{i_0}
\to C_M^{j_0}\stackrel{\al}{\to} Y$; then we can assume 
$Y\in \obj \cuz$, and the equality $f_{j_0}\circ g_{i_0j_0}=0$ implies that $f_{j0}$ can be decomposed as $\be \circ \al$ for some $\be \in \cu(Y,M)$.

Since $H''_0$ is homological, $H''(\al)(g_0(C_M^{j_0})(a_0))=0$. Applying the exact sequence $H_0\to H_0'\to H_0''$ to $Y$ we obtain that $H'_0(\al)(a_0)\in H'(Y)$ can be presented as $f_0(Y)(x_Y)$ for some $x_Y\in H_0(Y)=H(Y)$. Hence $a=f(M)(x)$ for $x=H(\be)\in H(M)$; see the commutative diagram
$$\begin{CD}
 H_0(C_M^{j_0}) @>{H_0(\al)}>> H_0(Y)=H(Y)@>{H(\be)}>>H(M)\\
@VV{f_0(C_M^{j_0})}V@VV{f_0(Y)=f(Y)}V@VV{f(M)}V \\
 H'_0(C_M^{j_0})@>{H'_0(\al)}>>H'_0(Y)=H'(Y) @>{H'(\be)}>>H'(M)\\
@VV{g_0(C_M^{j_0})}V@VV{g_0(Y)=g(Y)}V@VV{g(M)}V \\
 H''_0(C_M^{j_0})@>{H''_0(\al)}>>H''_0(Y)=H''(Y)@>{H''(\be)}>>H''(M)
\end{CD}$$

\ref{ikr8}. 
In the case where $\cuz$ generates $\cu$ as its own localizing category the assertion is given by Proposition 2.3 of \cite{krause}. 
Now recall that (in the general case) the embedding of the  localizing category generated by $\cuz$ into $\cu$ possesses an (exact) right adjoint $G$ that respects coproducts (according to Proposition \ref{pcomp}(II)). Hence the general case of the assertion reduces to loc. cit. if we apply assertion \ref{ikradj}.

\end{proof}

\begin{rema}\label{rkrause}
I.1. Now we explain that 
$H$ is actually the left Kan extension of $H_0$ along the inclusion $\cuz\to \cu$.

Indeed, we can certainly assume that $\cuz$ is small. Then the standard pointwise construction of the left Kan extension is easily seen to correspond to "the most obvious" resolutions of the functors $H_M$ in $\psvcuz$.

This observation certainly gives an alternative proof of part \ref{ikrtransf} of the proposition.

Below we will also mention extensions of cohomological functors $\cu\to \au$  obtained via applying the coextension construction to the corresponding (homological) functors $\cu\opp\to \au$. So they can also be described via right Kan extensions of the opposite homological functors $\cu\to \au\opp$. 


2. It appears that  it is not necessary to assume that $\cuz$ is triangulated to construct $H$ from $H_0$. 

So, assume that $\cuz$ is an essentially small additive subcategory of $\cu$ that satisfies the following condition: for any $\cu$-distinguished triangle 
\begin{equation}\label{edi}
Z[-1]\to X\to Y\to Z 
\end{equation}
the object $X$ belongs to $\obj\cuz$ whenever $Y$ and $Z$ do. Note that any $\cuz$ satisfying this condition is cosuspended, and the extension-closure of any cosuspended subcategory of $\cu$ satisfies this condition. 

Now, our restriction on $\cuz$ certainly implies that it has {\it weak kernels} (for all morphisms); hence the subcategory of coherent functors inside $\psv(\cuz)$ is abelian (see \S1.2 of \cite{krause}). Take an additive functor $H_0:\cuz\to \au$ for an AB5 category $\au$ such that for any triangle (\ref{edi}) with $Y,Z\in \obj \cuz$ the  sequence $H(X)\to H(Y)\to H(Z)$ is exact (in the middle term). Using \cite[Lemmas 2.1, 2.2]{krause} (cf. also the proof of Proposition 2.3 of ibid.) one can easily prove that  the ("pointwise Kan") recipe described in Proposition \ref{pkrause} gives a homological functor $H:\cu\to \au$ whose restriction to $\cuz$ is isomorphic to $H_0$.

Assume in addition that $\cuz$ consists of compact objects of $\cu$. Then the functor $H$ (coming from any $H_0$ as above)   is certainly a cc functor. Moreover, $\obj \cuz$ generates a weight structure $w$ (see  Remark \ref{rnewt}(1))  
 and  $H$ obviously annihilates $\ro=\cu_{w\ge 1}$. 

 Now we prove that the functors $H$ that can be obtained using coextensions of this sort are precisely the cc functors satisfying this vanishing condition (similarly to  Definition \ref{drange} one may say that these functors are of weight range $\le 0$). It follows that for $H_0$ being the restriction to $\cuz$ of a cc functor $H':\cu\to \au$ we have $H\cong \tau^{\ge 0}H'$ (see Remark \ref{rwrange}(4)).

We start from noting that for any $H'$ as above the definition of $H$ gives a canonical transformation $\Psi: H\to H'$. 
Now assume that $H'$  annihilates $\ro$. First we  prove that $\Psi(M_0)$ is injective for any $M_0\in \obj \cu$. Considering a $w$-decomposition $(w_{\ge 1}M_0)[1]\to w_{\le 0}M_0\to M\to w_{\ge 1}M_0$ we obtain that $H'(M_0)\cong  H'(w_{\le 0}M_0)$ and $H(M_0)\cong  H(w_{\le 0}M_0)$. Hence it suffices to verify the injectivity of $\Psi(M)$ for any $M\in \cu_{w\le 0}$. We will apply a certain inductive argument with  the base being the obvious fact that $\Psi(M)$ is  an isomorphism if $M$ is a coproduct of objects of $\cuz$. 

We use the construction (and adopt the notation) used in the proof of Theorem \ref{tclass}(\ref{iclass1}). Since $M$ belongs to $M\in  \cu_{w\le 0}$, it is a retract of the corresponding $L$ (according to Proposition \ref{phop}(\ref{itp7})), whereas the latter is a homotopy colimit of $L_i$  with cones of the connecting morphisms belonging to $\cocp$.

So, for all $k\ge 0$ we have commutative diagrams  
$$\begin{CD}
 H(P_k[-1])@>{}>>H(L_k)@>{}>>H(L_{k+1})@>{}>>H(P_k) \\
@VV{\cong }V@VV{}V@VV{}V @VV{\cong}V\\
H'(P_k[-1])@>{}>>H'(L_k)@>{}>>H'(L_{k+1})@>{}>>H'(P_k)
\end{CD}$$
whose rows are exact.  It easily follows by induction (starting from $\Psi(L_0)=\Psi(0)=0$) that $\Psi(L_k)$ is injective for any $k\ge 0$. Passing to the colimit (see Lemma \ref{lcoulim}(3(ii))) we obtain that $\Psi(L)$ is monomorphic; hence $\Psi(M)$ is monomorphic also. 

Next we prove that $\Psi(M)$ is also epimorphic (for any $M\in \obj \cu$). We set $H''(M)= \cok(\Psi(M))$. Since $\Psi(M)$ is monomorphic for all $M\in \obj\cu$, we obtain that $H''(M)$ is a homological functor ($\cu\to\au$); 
certainly it is a cc functor. So we should prove that $H''$ is zero.  For any $M$, $L$,  $L_k$, and $P_k$ as above  we certainly have $H''(P_k)=H''(P_k[-1])=0$ for any $k\ge 0$; hence obvious induction (similar to that used in the proof of the injectivity of $\Psi(M)$) gives the vanishing of  $H''(L_k)$ for $k\ge 0$. Passing to the colimit once again we obtain $H''(L)=H''(M)=0$ (for any $M\in \cu_{w\le 0}$). Since $H''$ also annihilates $\ro$, we obtain that it is zero.

Possibly, the author will write these arguments in more detail in a new version of this paper to obtain a  set of generators of a (compactly generated) $t$-structure $t$ under the assumptions of Theorem \ref{tab5}(2) below 
(this  set of generators would be much smaller than the one given by part 1 of this theorem; see Remark \ref{rab5}(\ref{ismgen})). 
 
II. Now assume (once again) that all objects of $\cuz$ are compact in $\cu$. 

 1. Then part \ref{ikr8} of our proposition easily implies that 
$\adfu(\cuz,\au)$ is equivalent to the category of those cc functors $\cu\to \au$ that kill  $\cuz^\perp$. 

2. Let $w$ be a smashing weight structure  for $\cu$ that restricts to $\cuz^\perp$. Since the corresponding virtual $t$-truncations of cc functors are cc ones according to Proposition \ref{ppcoprws}(\ref{icopr6}), we also obtain that virtual $t$-truncations of coextended functors are coextended. 

It certainly follows that (in this case) virtual $t$-truncations of coextended functors satisfy the "continuity" property described in part  \ref{ikr3} of our proposition. Recall also that  we have a similar continuity for natural transformations between coextended functors 
 according to part \ref{ikrtransf} of the proposition. 

3. Now we describe an interesting application of these observations. We recall that for any (co)homological functor $J:\cu\to \au$ and any $N\in \obj \cu$ a certain {\it weight spectral sequence} $T_w(J,N)$ for $J_*(N)$ is defined. We will not need its definition here. We will only recall (see Theorem 2.3.2 of \cite{bws}) that this spectral sequence is defined starting from the $E_1$-page;  $T_w(J,N)$ becomes independent from any choices (and functorial in $N$)  starting from $E_2$ (and we will write $T^{\ge 2}_w(J,N)$ for  this "part" of $T_w(J,N)$). Moreover, Theorem 2.4.2(II) of \cite{bger} immediately implies the following: 
 $T^{\ge 2}_w(J,N)$  can be $N$-functorially described in terms of certain 
virtual $t$-truncations of $J$ along with canonical transformations between them (and the transformations of the second level essentially come from (\ref{evtt}); note that this statement follows from the description of the derived exact couple $(E_2^{**},D_2^{**})$ for the exact couple $(E_1^{**},D_1^{**})$ that gives  $T_w(J,N)$). Hence we obtain the following: if $J$ is a coextended functor (and so, $\au$ is an AB5 category) and
$N$ is a 
$\cuz$-colimit of an {\bf 
inductive  system} $(N_l)$ then  
$T^{\ge 2}_w(J,N)$ is the direct limit of $T^{\ge 2}_w(J,N_l)$. 

The author has applied (the dual to) this statement in \cite{bgn} (where weight spectral sequences were described and studied in much more detail). It was applied to extended (see part I of this remark) cohomological functors from a cocompactly cogenerated category $\eu$, where $\eu$ is a certain category of   {\it motivic pro-spectra} (or of  {\it comotives}). This allows to compute 
 ("generalized") coniveau spectral sequences for the corresponding cohomology of a projective 
 limits of smooth varieties $X_l$ (over the base field $k$ that is perfect). 

4.  Now we give one more proof of the fact that 
  $\inli T^{\ge 2}_w(J,N_l)\cong  T^{\ge 2}_w(J,N) $ (if $N$ is a $\cuz$-colimit of 
	an inductive system $(N_l)$ and $J$ is extended); this argument avoids the consideration of $D_*^{**}$.

The functoriality of 
 $T^{\ge 2}_w(J,-)$ yields canonical compatible morphisms between $T^{\ge 2}_w(J,N_l)$ and from them into $T^{\ge 2}_w(J,N)$. Now, since $(N_l)$ form 
 an inductive system, the inductive limit $T'^{\ge 2}_w(J,N)$ of $T^{\ge 2}_w(J,N_l)$ is a spectral sequence (that starts from $E_2$); we also have a canonical morphism $T'^{\ge 2}_w(J,N)\to T^{\ge 2}_w(J,N)$. Thus it remains to verify that this morphism is an isomorphism at $E_2$, i.e., that $\inli E_2^{**}T_w(J,N_l)\cong E_2^{**}T(J,N_l)$. 
 
Now, Theorem 2.3.2 of \cite{bws} implies immediately that the $E_2$-terms of $T_w(J,-)$ are given by the pure functors $H^{A_{J\circ [i]}}\circ [j]$ for $i,j\in \z$
(see  Proposition \ref{ppure}). It is easily seen that these functors kill $\cuz^\perp$, and they are cc according to Proposition \ref{ppcoprws}(\ref{icopr5}). Hence they are extended according to Proposition \ref{pkrause}(\ref{ikr8})\footnote{This fact can also be proved by noting that $H^{A_J}$ can be obtained from $J$ "by means" of the corresponding virtual $t$-truncations; see Theorem 2.4.2(II) of \cite{bger}.}, and it remains to apply  Proposition \ref{pkrause}(\ref{ikr3}).
\end{rema}

Now we combine the results of this section with the ones of \cite{bsnull}; this gives an alternative proof of the "classification" of compactly generated torsion pairs given by Theorem \ref{tclass}. 

\begin{rema}\label{rnz} 
1. So, we want to give 
one more proof of the existence of a one-to-one correspondence between compactly generated torsion pairs for $\cu$ and essentially small Karoubi-closed extension-closed subclasses of the class $\cu^{\alz}$ of compact objects of $\cu$; the argument should not depend on Theorem \ref{tclass}(\ref{iclass2}).

Recall that any  essentially small subclass  $\cp$ of $\cu^{\alz}$ generates a compactly generated torsion pair $s=(\lo,\ro)$ for $\cu$ according to  Theorem \ref{tclass}(\ref{iclass1}) (this statement is also given by 
 the easier Theorem 4.3 of \cite{aiya}). Since $\cu^{\alz}$ gives a triangulated subcategory of $\cu$ (see Lemma 4.1.4 of \cite{neebook}), the class $\cu^{\alz}\cap \lo$ contains the envelope of $\cp$. Next (cf. Theorem \ref{tclass}(\ref{iclasst}) or its proof)
 this envelope is essentially small. Thus we should prove that 
 $\cu^{\alz}\cap \lo$ equals $\cp$ whenever the generating class $\cp$ is essentially small, Karoubi-closed, and extension-closed in $\cu$.  

Now, in \cite{bsnull} zero classes (see Lemma \ref{lbes}(\ref{izs})) of (co)homological functors were studied. Corollary 3.11 of ibid. (applied to the category $\cuz\opp$) gives the following remarkable statement: if $\cuz$ is a small triangulated category then a set $\cp_0$ of its objects is the zero class of some "detecting" homological functor $H_0:\cuz\to \ab$ if and only if $\cp_0$ is extension-closed and Karoubi-closed in $\cuz$. The author believes that this statement will become an important tool of studying (compactly generated) torsion pairs.

We take $\cuz$ to be  a small skeleton of the category $\lan \cp\ra$, $\cp_0=\obj \cuz\cap \cp$, and take  $H_0$ to be the corresponding "detector functor".
Since all objects of $\cuz$ are compact in $\cu$,  the coextension $H$ of the aforementioned functor $H_0$ to $\cu$ 
is a cc functor according to Proposition \ref{pkrause}(\ref{ikr8}). Next we consider the Brown-Comenetz dual functor $\hdu$ 
(see Definition \ref{dsym}(\ref{ibcomf}); recall that $\hdu: M\mapsto \ab(H(M),\q/\z)$  is  a cp functor from $\cu$ into $\ab$ whose zero class coincides with that of $H$.

We take $\cu'$ to be localizing subcategory of $\cu$ generated by $\cp$. Certainly, this subcategory contains $\lo$ (see the previous assertion). Moreover, Lemma 4.4.5 of \cite{neebook} implies that that the class $\cu'{}^{\alz}$ essentially equals $\obj\cuz$; hence the class  $\obj \cu'\cap  \cu^{\alz}$ 
 essentially equals $\obj \cuz$ as well.

Since $\cu'$ is generated by a set of compact objects as its own localizing subcategory, $\cu'$ satisfies the Brown representability condition according to Proposition \ref{pcomp}(II.1). Thus the restriction $\hdu'$ of the functor $\hdu$  to $\cu'$ is $\cu'$-representable by some $I\in \obj\cu'$.  Since the zero class of $\hdu'$  contains $\cp$, we have $I\in \ro$. Hence for $M\in  \obj \cu'\cap  \cu^{\alz}$ we have $M\in \lo$ if and only if $M\in \cp$.


Note also that the existence of a "detector object" $I$  is a new result. Moreover, if 
the class $\cu^{\alz}$ is essentially small itself then we can take $\cu'=\cu$ in this argument.

2. The aforementioned result of \cite{bsnull} (along with some of its variations also proved in ibid.) is a sort of Nullenstellensatz  for (co)homological functors from (small) triangulated categories (whence the name of the paper). It would certainly be interesting to obtain some analogue of this statement for cc and cp functors; 
 Theorem \ref{tclass}(\ref{iclass5}) is certainly related to this matter. Note however that a priori the intersection of zero classes of all  those smashing 
virtual $t$-truncations of representable functors that contain a given $\cp\subset \cu$ (for $\cu$ having coproducts) may be bigger than the strong extension-closure of  $\cp$. 

\end{rema}

\subsection{Dualities between triangulated categories and orthogonal torsion pairs; applications to categories of coherent sheaves}\label{sdual1}

Now we study certain pairings between triangulated categories and define the notion of orthogonality of torsion pairs (as well as of weight and $t$-structures) generalizing the one of adjacent structures.
We also define the notion of a nice duality; yet we will not use it in the current paper (anywhere except in Proposition \ref{pnice}). 

\begin{defi}\label{dort}
Let $\cu'$ be a triangulated category. 
\begin{enumerate}
\item\label{idu}
 We will call a (covariant) bi-functor
  $\Phi:\cu^{op}\times\cu'\to \au$ a {\it duality} if  it is 
	 homological with respect
to both arguments, and is equipped with a (bi)natural 
 transformation $\Phi(A,Y)\cong \Phi (A[1],Y[1])$.

\item\label{iorthop} Suppose that  $\cu$ is endowed with a torsion pair $s=(\lo,\ro)$ and $\cu'$ is endowed with a torsion pair $s'=(\lo',\ro')$. Then we will say that $s$
 is (left) {\it orthogonal} to $s'$ with respect to $\Phi$ (or just {\it left $\Phi$-orthogonal} to it) if the following {\it orthogonality condition} is fulfilled:
 $\Phi (X,Y)=0$ whenever $X\in \lo$ and $Y\in \lo'$ or if $X\in \ro$ and $Y\in \ro'$.

\item\label{iorthtw} Suppose that  $\cu$ is endowed with a weight structure $w$,
 $\cu'$ is endowed with a $t$-structure $t$. Then we will say that $w$
 is (left) {\it orthogonal} to $t$ with respect to $\Phi$  if the following condition is fulfilled:
 $\Phi (X,Y)=0$ if $X\in \cu_{w\ge 0}$
and $Y\in \cu'^{t \ge 1}$ or if $X\in \cu_{w\le 0}$ and $Y\in \cu'^{t \le -1}$.

\item \label{ini} We will say that $\Phi$ is {\it nice} if for any distinguished
triangles $T=A\stackrel{l}{\to} B \stackrel{m}{\to} C\stackrel{n}{\to} A[1]$ in $\cu$ and $X\stackrel{f}{\to} Y\stackrel{g}{\to} Z\stackrel{h}{\to}X[1]$ in $\cu'$ 
we have the following:
 the natural morphism $p$:
$$ \begin{gathered} 
 \ke (\Phi(A,X)\bigoplus \Phi(B,Y) \bigoplus \Phi(C,Z))\\
\xrightarrow{\begin{
pmatrix}
\Phi(A,-)(f) & -\Phi(-,Y)(l) &0  \\
0& g(B) &-\Phi(-,Z)(m)  \\
- \Phi(-,X)([-1](n)) & 0 &\Phi(C,-)(h)
\end{
pmatrix}}
\\ (\Phi(A,Y) \bigoplus \Phi(B,Z) \bigoplus \Phi(C[-1],X))
 \stackrel{p}{\to} \ke ((\Phi(A,X)\bigoplus \Phi(B,Y))\\ \xrightarrow{\Phi(A,-)(f)\oplus - \Phi(-,Y)(l)}
 \Phi(A,Y)) 
 \end{gathered}$$
is epimorphic.
\end{enumerate}

\end{defi}

\begin{rema}\label{rort}
1. For $\cu'=\cu$ and $\Phi=\cu(-,-)$ the definition of orthogonal torsion pairs restricts to the one of adjacent ones. Indeed, in this case the $\Phi$-orthogonality of $s$ and $s'$ means that $\lo\perp_{\cu} \lo'$ and $\ro\perp_{\cu} \ro'$; these inclusions are certainly equivalent to $\lo'\subset \ro$  and $\ro\subset \lo'$, respectively. 

2. Similarly to the notions of adjacent (weight and $t$-) structures and torsion pairs, part \ref{iorthtw} of Definition \ref{dort} is essentially a particular case of part \ref{iorthop} if one takes $w$ and $t$  associated with $s$ and $s'$, respectively, and shifts one of them; see Remarks \ref{rstws}(4),  \ref{rwhop}(1), and \ref{rtst1}(\ref{it1}). 
\end{rema}

Now we give some definitions

\begin{defi}\label{dortt}
1. We will say that a weight structure $w$ on $\cupr$ is {\it right $\Phi$-orthogonal} to a $t$-structure $t$ on $\cu$ whenever $w$ is left orthogonal to $t$ with respect to the (obviously defined) duality "opposite to $\Phi$". 

2. For an $R$-linear triangulated category $\cu$ (where $R$ is an associative commutative unital ring) we will say that an $R$-linear cohomological functor $H$ from $\cu$ into $R-\modd$ is  {\it almost of $R$-finite type} whenever for any $M\in \obj\cu$ the $R$-module $H(M)$ is finitely generated and $H(M[i])=\ns$ for  $i\ll 0$ (cf. Definition \ref{dsatur} and Definition 0.1 of \cite{neesat}).
\end{defi}

Now we prove a generalization of Proposition \ref{psatur}(II) that enables constructing $t$-structures orthogonal to certain weight ones; its formulation is motivated by related results of  \cite{neesat} and \cite{roq}.


\begin{pr}\label{psaturdu}
Assume  that  $\cu\subset \cupr$ are $R$-linear categories (for $R$ being a unital commutative ring; see Definition \ref{dsatur} for some of our notation). 
Denote by $\Phi$ the restriction of the bifunctor $\cupr(-,-)$ to $\cu\opp\times \cupr$.

1. Assume that functors $\cupr\to R-\modd$ that are corepresented by objects of $\cu$ are precisely those ones that are of $R$-finite type as functors from $\cu'\opp$. Then for any bounded weight structure $w'$ on $\cupr$ the couple $t=(t_1,t_2) =(\perpp(\cupr_{w'\ge 1})\cap \obj \cu, \perpp (\cupr_{w'\le -1})\cap \obj \cu)$ is a $t$-structure on $\cu$. 

 Moreover, $w'$ is right $\Phi$-orthogonal (see 
 Definition \ref{dortt}) to this $t$-structure.  

2. Assume that $\cu$ is essentially small 
 and there exists a triangulated category $\du$ that satisfies the following conditions: it has coproducts, $\cupr\subset \du$, all objects of $\cu$ are compact in $\du$, and for $N\in \obj\du $ the  restriction of the functor $\du(-,N)$ to $\cu$ is of $R$-finite type (resp. almost of $R$-finite type) if and only if $N\in \obj \cupr$. Then for any  bounded weight structure $w$ on $\cu$ the couple $t'=(t_3,t_4) =((\cu_{w\ge 1})^{\perp_{\cupr}},  (\cu_{w\le -1})^{\perp_{\cupr}})$ is a $t$-structure on $\cupr$. 

Moreover, $w$ is left  $\Phi$-orthogonal  to this $t'$. 

Furthermore, if $R$ is noetherian and  the correspondence $N\mapsto H_N$ is a full  functor from $\cupr$ into $\adfur(\cu\opp,R-\mmodd)$ (resp. gives an equivalence of $\cupr$ with the subcategory of $\adfur(\cu\opp,R-\mmodd)$ consisting of functors of $R$-finite type) then the obvious functor from $\hrt'$ into $\adfur(\hw\opp,R-\mmodd)$ is full (resp.  gives an equivalence of $\hrt'$ with  the category $\adfur(\hw, R-\mmodd)$).
\end{pr}
\begin{proof}
In both assertions the corresponding $\Phi$-orthogonalities are automatic; so we should only check that the corresponding couples are $t$-structures indeed, and study $\hrt'$ in assertion 2.

1. This statement is rather similar to (the dual to) Proposition \ref{psatur}. 

Axioms (i) and (ii) of Definition \ref{dtstr}  are obvious for $t$.

Next,  ${}^{\perp_{\cu'}}(\cupr_{w'\ge 1})=\cu'_{w'\le 0}$; thus $t_1[1]\perp t_2$. 

It remains to verify the existence of $t$-decompositions. For $M\in \obj \cu$ we note that all the virtual $t$-truncations of the functors $\cupr(M,-)$ (see Remark \ref{rwrange}(4)) are of $R$-finite type as considered as cohomological functors from $\cu'\opp$ 
according to Proposition \ref{psatur}(I). Hence they are representable by objects of $\cu$ according to our assumptions. Thus  arguments similar to that used for the proof of Proposition \ref{padjt}(\ref{ile4}) easily allow us to conclude the proof (note here that similarly to Proposition \ref{padjt}(\ref{ile2}) we can apply the Yoneda lemma here since $\cu\subset \cu'$).

2. 
 Since $w$ is bounded, the category $\cu$ is densely generated by $\cu_{w=0}$ (by Proposition \ref{pbw}(\ref{igenlm}); cf. Remark \ref{rsatur}(2)) and $\hw$ is negative in it. Moreover, we can assume that $\du$ is generated by $\obj \cu$ as its own localizing subcategory (see 
  Remark \ref{rwhop}(2)); thus $\obj \cu^{\perp_{\du}}=(\cup_{i\in \z}\cu_{w=i})^{\perp_{\du}}=\ns$. Hence we can apply \cite[Theorem 4.5.2(I)]{bws} 
	to obtain that $t_{\du}=((\cu_{w\ge 1})^{\perp_{\du}},  (\cu_{w\le -1})^{\perp_{\du}})$ is a $t$-structure on $\du$.\footnote{One can also prove this statement using  Remark \ref{rtkrau}(4);   note  that is actually not necessary to assume  that $\obj \cu^{\perp_{\du}}=\ns$ to get $t_{\du}$. Note also that Theorem 1.3 of \cite{hoshimi} is essentially an important particular case of \cite[Theorem 4.5.2(I.1)]{bws}.}

Next, 	$w$ is certainly left orthogonal to $t_{\du}$ with respect to the restriction of $\du(-,-)$ to $\cu\opp\times \du$. 
 Thus Proposition 2.5.4(1) of \cite {bger}  (cf. Proposition \ref{pwfil}(4)) implies that for $N\in \obj \du$ the restriction of $\du(-,N^{t_{\du}\le 0})$ to $\cu$ is isomorphic to  $ \tau^{\le 0 }(H_N)$ (where $H_N$ is the restriction of $\du(-,N)$ to $\cu$). Now, if $H_N$ is (almost) of   $R$-finite type then this virtual $t$-truncation is so as well according to (the corresponding obvious modification of) Proposition \ref{psatur}(I). Hence $N^{t_{\du}\le 0}$ is an object of  $ \cupr$ whenever $N$ is. It obviously follows that $t_{\du}$ restricts to $\cu'$, i.e., $t'=(t_3,t_4)$ is a $t$-structure on $\cu'$.

The proof of the "furthermore" part of the assertion is an easy application of Proposition \ref{phadj}; see the proof of Proposition \ref{psatur}(3) where a particular case of this fact is established.
\end{proof}

\begin{rema}\label{roq}
1.  
Let $X$ be a scheme, and take $\cu$ to be the triangulated category of perfect complexes on $X$, $\cu_1'=D^b(X)$ (the bounded derived category  of coherent sheaves on $X$),  $\cu_2'=D^-(X)$ (the bounded above category), $\du=D(QCoh)(X)$ (the unbounded derived category of quasi-coherent sheaves on $X$). Certainly, $\cu\subset \cu_1'\subset \cu_2'\subset \du$, and objects of $\cu$ compactly generate $\du$.

Next, assume that $X$ is proper over $\spe R$, where $R$ is a commutative unital  noetherian ring. 
Then Corollary 0.5 of \cite{neesat} says that  for $N\in \obj \du$   the  restriction $H_N$ of the functor $\du(-,N)$ to $\cu$ is of $R$-finite type (resp. almost of $R$-finite type) if and only if $N\in \obj \cu'_1$ (resp. $N\in \obj \cu'_2$). Thus we obtain the existence of the corresponding $t$-structures $t'_i$ both on $\cu_1'$ and $\cu_2'$; certainly, these $t$-structures are the restrictions of the 
 $t$-structure $t_{\du}$ on $\du$ (to $\cu_1'$ and $\cu_2'$, respectively; see the proof of part 2 of  our proposition).

Moreover, the properties of the correspondence  $N\mapsto H_N$ allow us to apply this part 2 to obtain a full functor from $\hrt'_2$ into $\adfur(\hw, R-\mmodd)$ that restricts to an equivalence of $\hrt'_1$ with the latter category. It easily follows that we have $\hrt'_1=\hrt'_2$ in this case.


2. Assume in addition that $R$ is a field and $X$ is projective over $\spe R$. Then 
Lemma 7.49 and Corollary 7.51(ii) of \cite{roq} (cf. also \cite[Theorem A.1]{bvdb}) also enable us to apply Proposition \ref{psaturdu}(1) for $\cu'=\cu'_1$.



3. 
So, to prove the existence of orthogonal $t$-structures when there is a duality $\Phi:\cu\opp\times \cu'\to \ab$  using our methods one needs some version of the Brown representability condition or its dual; 
 i.e., one should have a  description of the functors $\Phi(M,-)$ for $M\in \obj \cu$ (resp.  of $\Phi(-,N)$ for $N\in \obj \cu'$) that should be respected by the corresponding virtual $t$-truncations (cf. Remark \ref{rkrause}(II)). 
\end{rema}

Now we give two simple recipes for constructing nice dualities.

\begin{pr}\label{pnice}
1. 
If  $F:\cu\to \du$ and $F':\cu'\to \du$ are exact functors, 
 then  $\Phi(X,Y)=\du(F(X),F'(Y)):\cu\opp\times\cu\to \ab$ is a nice duality.

2. For triangulated categories $\cu$, $\cuz\subset \cu'$, $\cuz$ is skeletally small,  
let $\Phi_0:\cu^{op}\times \cuz\to \au$ be a duality.
For any $P\in \obj \cu'$ 
 denote by  $H^P$  the coextension (see Proposition \ref{pkrause}) to $\cupr$ of the functor $\Phi_0(-,Y)$;
denote by $\Phi$ the  pairing $\cu^{op}\times \cupr\to \ab:\ \Phi(P,Y)=H^P(Y)$. Then $\Phi$ is a duality ($\cu^{op}\times \cu'\to \au$); it is nice whenever $\Phi_0$ is.
\end{pr}
\begin{proof}
1. Easy; it suffices to note that  the niceness restriction is a generalization  of the axiom (TR3) of triangulated categories
(any commutative square can be completed to a morphism of distinguished triangles) to the setting of dualities of triangulated categories. 

2. This is the categorical dual to Proposition 2.5.6(3) of \cite{bger}.
\end{proof}

 \subsection{On "bicontinuous" dualities} 
\label{sdual2}

Now we describe a interesting type of 
(nice) dualities of triangulated categories and orthogonal torsion pairs in them.
We will apply it (below and in \cite{bgn}) in the case where $\cu$ is a certain category of pro-objects; so it is no wonder that we will consider co-compact objects in it. 

So we will say that a cohomological functor $\cu\to \au$ ($\au$ is an abelian category) is pc one if it converts $\cu$-products into $\au$-coproducts; recall that an   object $M$ of  $\cu$  is said to be cocompact if the functor $\cu(-,M)$ is a pc one (cf. Remark \ref{rwhop}(5)).

\begin{pr}\label{porthop}
Let $\cuz$ be an essentially small common subcategory of (triangulated categories) $\cu$ and $\cupr$
whose objects are compact in $\cupr$ and cocompact in $\cu$ (and so, $\cu$ has products and $\cupr$ has coproducts);
 let $\cp$ be a 
 set of objects of $\cuz$, $M\in \obj \cu$, $N\in \obj \cupr$. 

For each $P\in \obj \cu$ denote by  $H^P$  the coextension (see Proposition \ref{pkrause}) to $\cupr$ of the restriction to $\cuz$ of the functor $\cu(P,-)$; 
denote by $\Phi$ the  pairing $\cu^{op}\times \cupr\to \ab:\ \Phi(P,Y)=H^P(Y)$.

Then the following statements are valid.

\begin{enumerate}
\item\label{icudupa} 
 $\Phi$ is a nice duality of triangulated categories. 

\item\label{icuduadj}
Denote by  $\eu$ the colocalizing triangulated subcategory of $\cu$ cogenerated by $\cuz$. Then there exists a  left adjoint $L$ 
to the embedding $\eu\to \cu$ and $\Phi(-,-)\cong \Phi^{\eu}(L(-),-)$, where $\Phi^{\eu}$ is the  restriction of $\Phi$ to $\eu\opp\times \cupr$. 

\item\label{icontcu} 
Assume that $M\in \obj \cu$ is a 
$\cuz$-limit of 
 some $(M_l)$ 
 (see Definition \ref{drelim}(2)). 
 Then $\Phi(M,N)\cong \inli \Phi(M_l,N)$. In particular,  the functor  $\Phi(-,N)$ 
is a pc one.

\item\label{icontdu}
Assume that $N\in \obj \cu$ is  a 
$\cuz$-colimit of 
 some $(N_l)$ (see Definition \ref{drelim}(1)).  Then  $\Phi(M,N)\cong \inli \Phi(M,N_l)$. 

In particular, 
the functor $\Phi(M,-):\cupr\to \ab$ 
is a cc one.

\item\label{ibiext}
 The functor $\Phi$ is  determined (up to a canonical isomorphism) by the following conditions: it converts $\cu$-products and $\cupr$-coproducts into direct sums of abelian groups, it restriction to $\cuz\opp\times \cuz $ equals $\cuz(-,-)$, and it annihilates both ${}^{\perp_{\cu}}\cuz \times \obj \cupr$ and $\obj\cu \times \cuz^{\perp_{\cupr}}$. So, in this case we will say that $\Phi$ is the {\it biextension} of $\cuz(-,-)$ to $\cu^{op}\times \cupr$.

\item\label{icupr} 
 There exists a smashing torsion pair $s'=(\lo',\ro')$ for $\cu'$ such that  $\lo'$ is the $\cupr$-strong extension-closure of $\cp$ and $\ro'=\cp^{\perp} $.

\item\label{icup} 
 There exists a cosmashing torsion pair $s=(\lo,\ro)$ for $\cu$ such that  $\lo=\perpp \cp$  and $\ro$ is the strong extension-closure of $\cp$ in $\cu\opp$. 

Moreover, $s$ restricts to a torsion pair $s_{\eu}$ for $\eu$ (see assertion \ref{icuduadj} for the latter notation). 

\item\label{ihoport} 
 The torsion pairs  $s$ and $s_{\eu}$ mentioned in the previous  assertion are (respectively)  $\Phi$-orthogonal and 
$\Phi^{\eu}$-orthogonal  to the  torsion pair $s'$ from assertion \ref{icupr}. 
Moreover, $\lo={}^{\perp_{\Phi}}(\lo')$ and $\lo_{\eu}={}^{\perp_{\Phi^{\eu}}}(\lo')$.

\item\label{inters} Both $\lo'\cap \obj \cuz$ and $\ro\cap \obj \cuz$ are equal to the $\cuz$-envelope of $\cp$.
\end{enumerate}
\end{pr}
\begin{proof}
\begin{enumerate}
\item Immediate from Proposition \ref{pnice}.

\item The existence of $L$ is immediate from Proposition \ref{pcomp}(II). The rest of the assertion  follows from the adjunction property of $L$ immediately.  

\item It  suffices to note  that the coextension construction respects 
 colimits;  
 this is obvious (cf. also Proposition \ref{pkrause}(\ref{ikr7})).

\item Immediate from Proposition \ref{pkrause}(\ref{ikres}).

\item  The previous assertions easily imply that $\Phi$ satisfies all the properties desired.

Next, Proposition \ref{pkrause}(\ref{ikr8}) implies that $\Phi$ is determined by its restriction to $\cuz\opp\times \cupr$ along with the conditions that it respects $\cupr$-coproducts and annihilates $\obj\cu \times \cuz^{\perp_{\cupr}}$. Thus one can easily conclude the proof by applying the categorical dual of  the aforementioned statement. 

\item This is just Theorem \ref{tclass}(\ref{iclass1}).

\item The first part of the assertion is the dual of   assertion \ref{icupr}. 

To prove the "moreover" statement we note that the embedding $\du\to \cu$ admits an (exact) left adjoint whose kernel is $\obj \du^{\perp_{\cu}}=\cp^{\perp_{\cu}}$ (by the dual to Proposition \ref{pcomp}(II.2)). 
 Thus it remains to apply (the dual to) Proposition \ref{phopft}(II.1).

\item To verify the first part of the assertion it certainly suffices to prove that $s\perp_{\Phi}s'$.

Now recall that all the functors of the type $\Phi(M,-):\cupr\to \ab$ are cc ones and functors of the type $\Phi(-,N):\cu\opp\to \ab$ are pc-ones
(see assertions \ref{icontcu} and \ref{icontdu}).
Thus  
Theorem \ref{tclass}(\ref{iclass5}) (along with its dual) 
 reduces the $\Phi$-orthogonality checks to the following ones: $\Phi(X,Y)=0$ if either $X\in \cp\subset \obj\cu$ and $Y\in \cp^{\perp_{\cu}}$ or if  $X\in {}^{\perp_{\du}}\cp$ and $Y\in \cp\subset \obj\du$. Thus it suffices to note that $\Phi(A,B)$ is isomorphic to $\cupr(A,B)$ if $A\in \obj \cuz (\subset \obj\cu)$ and is isomorphic to  $\cu(A,B)$ if $B\in \obj \cuz (\subset \obj\cupr)$. The latter statements are immediate from Proposition \ref{pkrause} (parts I.\ref{ikrchar} and I.\ref{ikrtriv}, respectively). 

Now, to prove the "moreover" statement it remains to note that ${}^{\perp_{\Phi^{\eu}}}(\lo')\subset \lo$ 
 since $\lo'$ contains $\cp$.

\item Immediate from Theorem \ref{tclass}(\ref{iclass2}).
\end{enumerate}
\end{proof}

\begin{rema}\label{rcudu} 
\begin{enumerate}
\item\label{ircocomp} 
One may say that all objects of $\cu$ are "compact with respect to $\Phi$" and objects of $\cupr$ are "cocompact with respect to $\Phi$"; see  Proposition \ref{porthop}(\ref{icontcu}, \ref{icontdu}).
Note that both of these properties fail for the duality $\cu(-,):\cu\opp\times \cu\to \ab$ (and $\cupr=\cu$); so one may say that $\Phi$ is a "regularized Hom bifunctor" that is "bicontinuous".

\item\label{ibcf} Now assume that $\Phi:\cu\opp\times \cupr\to \ab$ is any duality satisfying these bicontinuity conditions and $\cupr$ satisfies the Brown representability condition (this is certainly the case whenever $\cupr$ is compactly generated). Then one can define a curious functor $\cu\to \cupr$ as follows.

We note that for any $P\in \obj \cu$ there exists its "$\Phi$-Brown-Comenetz dual" $\bc(P)\in \obj\cupr$ that $\cupr$-represents the functor $\ab(\Phi(P,-), \q/\z)$ (since this functor is obviously a cp one). Moreover, the correspondence  $\bc$ is easily seen to be exact; it also respects products.  

Next, for the torsion pairs $s$ and $s'$ as above we obviously have $\bc(\lo)\subset \ro'\perpp$ and $\bc(\ro)\subset \lo'\perpp=\ro'$ (see Proposition \ref{psymb}(I.\ref{isbcd}). Lastly, if  $\cp$ is suspended then for the corresponding $w$ and $t$ (see Corollary \ref{cwt} below) and for the weight structure $w'$ right adjacent to  $t$ (whose existence is given by Corollary \ref{csymt}(2)) we obtain that $\bc$ is weight-exact (with respect to $w$ and $w'$; see Definition \ref{dwso}(\ref{idwe})). Note however that the existence of $\Phi$ is a stronger assumption (in this case) then the existence of a weight-exact functor $\bc$ that respects products since the former condition implies that the image of $\bc$ consists of "Brown-Comenetz duals" only. 

We will demonstrate the utility of dualities in Theorem \ref{tab5} below.

\item\label{ir1} For any $\cu$ and $\cupr$ as in our proposition and their small subcategories $\cu_0$ and $\cu_0'$ respectively, one can extend any (nice) duality $\cu_0\opp\times \cu_0'\to \ab$ 
first to a duality $\cu\opp\times \cu_0'\to \ab$ using the dual to Proposition \ref{pkrause} and next proceed as above to obtain $\Phi: \cu\opp\times \cu'\to \ab$. It is easily seen that 
 some of the statements proved above 
 have natural (and easy to prove) analogues for this setting.\footnote{Moreover, it would be interesting to prove some version of part \ref{ihoport} of the proposition without assuming that objects of $\cuz$ are compact in $\cupr$.} 

Note also that a nice duality $\cu_0\opp\times \cu_0'\to \ab$ can be obtained from any exact functors $\cu_0\to \du$ and $\cu'_0\to \du$ (see Proposition \ref{pnice}(1)).

\item\label{ir2} Dually, 
  one may start from "coextending" a duality  $\cu_0\opp\times \cu_0'\to \ab$  to a duality  $\cuz\opp\times \cu'\to \ab$ and next extend the result to a duality $\cu\opp\times \cu'\to \ab$ (using the dual to  Proposition \ref{pkrause}). 
	 Combining part \ref{ikrchar} and \ref{ikr7} of  this proposition we obtain  the duality obtained this way is isomorphic to the one described above (and obtained in the "reverse order"). 
Note also that  this statement easily follows from the corresponding generalization of  Proposition \ref{porthop}(\ref{ibiext})  
whenever all objects of $\cuz'$ are compact in $\cupr$ and objects of $\cuz$ are cocompact in $\cu$.

\item\label{ir3} If one applies the "reverse biextension method" of part \ref{ir2} of this remark to the duality $\cuz(-,-)$ (that was the "starting one" in our proposition) then the "intermediate" duality  $\cuz\opp\times \cu'\to \ab$ would be the restriction of the bi-functor $\cupr(-,-)$ to $\cuz\opp\times \cu'$; see Proposition \ref{pkrause}(\ref{ikrtriv}).

\item\label{ikrausl} An interesting family of examples for our proposition can easily be constructed using Theorem 4.9 of \cite{krauslender}.
\end{enumerate}
\end{rema}

\subsection{An application: hearts of compactly generated $t$-structures are "usually" Grothendieck abelian}\label{sgengroth}

First  we list the consequences of  Proposition \ref{porthop} (essentially) in the case where $s$ is a  weight structure and $s'$ is a $t$-structure.

\begin{coro}\label{cwt}
In the setting of the previous proposition (and adopting its notation including the one of  its part \ref{icuduadj}) assume that 
  $\cp$ is a  suspended subset of $\obj \cuz$ (i.e., $\cp[1]\subset \cp$). 

Then the following statements are valid.

\begin{enumerate}

\item\label{ic1} For any $N'\in \obj \cupr$ denote by $H^{N'}$ the extension to $\cu$  of the functor 
$H_0^{N'}:\cuz\opp\to \ab$ (that is the restriction of $\cupr(-,N')$ to $\cuz$)  obtained via the  dual to  Proposition \ref{pkrause}. Then the bi-functor $\Phi:\cu\opp\times \cupr\to \ab: (M,N')\mapsto H^{N'}(M)$ is a nice duality that is naturally isomorphic to the one given by Proposition \ref{porthop}(\ref{icudupa}).

\item\label{iccc} For any $M\in \obj \cu$ the functor $\Phi(M,-):\cupr\to \ab$ is a cc one.

\item\label{ic2} The 
 functor $L:\cu\to\eu$ (left adjoint to the embedding $\eu\to \cu$) respects products, $\cuz$-limits, and is identical on $\obj \eu$. 

\item\label{ic15}   The restriction $\Phi^{\eu}$ of $\Phi$ to $\eu\opp\times \cupr$ is a nice duality also, and we have $\Phi(-,-)\cong \Phi^{\eu}(L(-),-)$. 

\item\label{ic25} 
For any set of $M_i\in \obj \cu$ and $N\in \obj \cupr$ we have $\Phi(\prod M_i,N)\cong \bigoplus \Phi (M_i,N)$.

\item\label{ict} There exists a smashing $t$-structure $t$  on $\cupr$ such that $\cupr^{t\le 0}$ is the 
smallest coproductive extension-closed subclass of $\obj \cupr$ containing $\cp$ and $\cupr^{t\ge 0}=\cp\perpp[1]$. Moreover,  $\cupr^{t\le 0}$ also equals the big hull of $\cp$ in $\cupr$.

\item\label{icw} There exists a cosmashing weight structure $w$  on $\cu$ such that $\cu_{w\ge 0}$ is the big hull of $\cp$ in $\cu\opp$ and $\cu_{w\le 0}=({}^{\perp_{\cu}}\cp)[1]$. Moreover,  $\cu_{w\le 0}$ is the extension-closure of $({}^{\perp_{\eu}}\cp)[1]\cup \perpp(\cup_{i\in \z}\cp)$.

\item\label{icwe} The couple $w_{\eu}= (({}^{\perp_{\eu}}\cp)[1], \cu_{w\ge 0})$ is a cosmashing weight structure  on $\eu$, $L$ is weight-exact (with respect to $w$ and $w_{\eu}$, respectively), and  $\cu_{w=0}=\eu_{w_{\eu}=0}$. Moreover, if $\cp$ densely generates $\cuz$ then $w_{\eu}$ is right non-degenerate.

\item\label{icort} $w$ and $w_{\eu}$ are orthogonal to $t$ with respect to $\Phi$ and $\Phi^{\eu}$, respectively. Moreover, $\cu_{w\le 0}={}^{\perp_{\Phi}}\cupr^{t\le -1}$ and $\eu_{w_{\eu}\le 0}={}^{\perp_{\Phi^{\eu}}}\cupr^{t\le -1}$.

\item\label{iwheart} Choose some $w_{\eu}$-weight complexes for elements of $\cp$; denote their terms by $P^k_l$.  Then the object $I=\prod P^k_l$ cogenerates 
$\hw$  (cf. Corollary \ref{cvttbrown}(I.2)), i.e., any object of $\hw$ is a retract of   a product of copies of $I$. 

\item\label{icint} The $\cuz$-envelope of $\cp$ equals both $\cupr^{t\le 0}\cap \obj \cuz$ and $\cu_{w\ge 0}\cap \obj \cuz$.

\item\label{icvirt}
If $\au$ is an AB5 category then for the extension (obtained via the  dual to  Proposition \ref{pkrause})  $H:\cu\opp\to \au$ of a cohomological functor $H_0$ from $\cuz $ into $\au$  all its  virtual $t$-truncations are extended functors (in the same sense as $H$ is).
\end{enumerate}
\end{coro}
\begin{proof}
\ref{ic1}. Easy from Proposition \ref{porthop}(\ref{ibiext});  see Remark \ref{rcudu} (\ref{ir2},\ref{ir3}).

\ref{iccc}. Immediate from the previous assertion combined with  Proposition \ref{porthop}(\ref{icontdu}).

\ref{ic2}. 
Obviously,   $L$ respects $\cuz$-limits  and is identical on $\obj \eu$. Dualizing Proposition \ref{prtst}(\ref{it4sm})
we 
 obtain that $L$ respects products. 

\ref{ic15}. Certainly, (arbitrary) restrictions of nice dualities are nice dualities also. It remains to apply Proposition \ref{porthop}(\ref{icuduadj}).

\ref{ic25}. Immediate from Proposition \ref{porthop}(\ref{icontcu}).

\ref{ict}. Certainly, $\cp$ generates a smashing torsion pair $s'=(\lo',\ro')$ for $\cupr$ with $\ro'=\cupr^{t\ge 0}$ and $\lo'$ being the big hull of $\cp[1]$; see  Proposition \ref{porthop}(\ref{icupr}). Hence the assertion follows from Corollary \ref{csymt}(1) easily.

\ref{icw}.  The first part of the assertion is immediate from Proposition \ref{porthop}(\ref{icup}); see Remark \ref{rwhop}(1). The second part is immediate from Proposition \ref{phopft}(II.1).

\ref{icwe}. 
Proposition \ref{porthop}(\ref{icup})  gives the 
existence of $w_{\eu}$, which certainly implies the weight-exactness of $L$ and $\cu_{w=0}=\eu_{w_{\eu}=0}$.
 To prove the "moreover" part  one should note that ${}^{\perp_{\eu}}\obj \lan \cp\ra ={}^{\perp_{\eu}}\obj \cuz=\ns$ according to (the dual to) Proposition \ref{pcomp}(I.1); it remains to apply (the dual to) Remark \ref{rwhop}(8).

\ref{icort}. Immediate from Proposition \ref{porthop}(\ref{ihoport}) (see Remark \ref{rort}(2)). 

\ref{iwheart}.  Immediate from (the dual to) Proposition \ref{ppcoprws}(\ref{icopr7}) (cf. Corollary \ref{cvttbrown}(I.2)).

\ref{icint}. This is just a particular case of  Proposition \ref{porthop}(\ref{inters}).

\ref{icvirt}. Immediate from Remark \ref{rkrause}(II.2).
\end{proof}

Now we apply this corollary to the study of compactly generated $t$-structures. First we study the AB5 condition.

\begin{theo}\label{tab5}
Let $\du$ be a triangulated category having coproducts and $t$  be a smashing $t$-structure on it. 

Then the following statements are valid.

1. Assume moreover that  $\hrt$ is an AB5 category  and $\du$ is compactly generated. Then the (essentially small) class $H_0^t(\du^{\alz})$ generates $\hrt$ (and so, the coproduct of any small skeleton of $H_0^t(\du^{\alz})$ generates it also).\footnote{Recall that $\du^{\alz}$ is the (essentially small) class of compact objects of $\du$. Following Remark 4.2.11 of \cite{bondegl}, one may call elements of  $H_0^t(\du^{\alz})$ {\it strongly constructible} objects of $\hrt$.}

2. Assume that $t$ is generated by a suspended set $\cp\subset \obj \du$  of compact objects.  Denote by $\duz$ the subcategory of $\du$ that is densely generated (see \S\ref{snotata}) by $\cp$ and assume that there exists a triangulated category $\du'$ that contains $\du_0^{op}$ as a full subcategory of compact objects. 
 Then $\hrt$ is  a Grothendieck abelian 
 category, and there exists a 
 faithful exact functor $\ec:\hrt\to \ab$  
 that respects coproducts.
\end{theo}
\begin{proof}
1. Proposition \ref{pkrause}(\ref{ikr8}) implies that the functor $H_0^t:\du\to \hrt$ is coextended (from the subcategory $\du'$ of $\du$ whose object class equals $\du^{\alz}$). Hence for any $M\in \obj \du$ and a  family $C_M^j\in \du^{\alz}$ as in (\ref{ekrause}) we obtain that $H_0^t(M)$ is a $\hrt$-quotient of  $\coprod H_0^t(C_M^j)$. Thus the class $H_0^t(\du^{\alz})$ generates $\hrt$ indeed. 

2. The embedding into $\du$ of its localizing subcategory generated by $\duz$ possesses an exact right adjoint (by Proposition \ref{pcomp}(II.2)); hence Remark \ref{rwhop}(2) allows us to assume  that $\du$ equals this subcategory. Thus $\du$ is 
  compactly generated.  Hence assertion 1 says that it suffices to verify whether $\hrt$ is an AB5 category.

Next, $\hrt$ is an AB4 category according to Proposition \ref{prtst}(\ref{it2}). Hence it is cocomplete; thus to check that it is AB5  we should prove that filtered colimits of $\hrt$-monomorphisms are $\hrt$-monomorphic. Thus it suffices to construct $\ec$ (since small colimits can be expressed in terms of coproducts). 
Note also that an exact functor (of abelian categories) is faithful if and only if it is conservative (and this is equivalent to the assumption that $\ec$ does not kill non-zero objects). 

Next we note that exact functors $\hrt\to \ab$ are precisely the restrictions to $\hrt$ of those homological functors $\cu\to \ab$ that annihilate $\cu^{t\le -1}\cup \cu^{t\ge 1}$. Hence it suffices to find a cc functor $\ec^{\du}:\du\to \ab$ that kills $\cu^{t\le -1}\cup \cu^{t\ge 1}$
  and whose restriction to $\hrt$ is conservative.

We start with constructing a "big" conservative family of  cc functors  $\du\to \ab$ satisfying the vanishing condition above; we call them    {\it stalk} ones for the reason that will be explained in Remark \ref{rsheaves}(1) below.\footnote{Note also that our argument has inspired the proof of  \cite[Corollary 4.2.4]{bondegl}.}
 For this purpose we apply the previous corollary for $\cu'$ equal to $\du$ and $\cu=\du'^{op}$. We obtain the existence of a duality $\Phi:\cu\opp\times \du\to \ab$ and a weight structure $w$ on $\cu$ that is (left) $\Phi$-orthogonal to $t$.

Our stalk functors are   the functors  $\Phi(P,-):\du\to \ab$ for $P$ running through $\cu_{w=0}$. 
The stalk functors are certainly exact (by the definition of a duality);  they annihilate $\cu^{t\le -1}\cup \cu^{t\ge 1}$ since $w\perp_{\Phi}t$. 
 The stalk functors are cc according to Corollary \ref{cwt}(\ref{iccc}). 

Let us  verify the conservativity of our family. Let $N$ be a non-zero element of $\du^{t=0}$; we should verify that the restriction $A$ of the functor  $\Phi_N=\Phi(-, N):\du\opp\to \ab$ to $\hw\opp$ is not zero. 
Now, the $\Phi$-orthogonality of   $w$ to $t$ along with Proposition \ref{pwrange}(\ref{iwrpure}) implies that $\Phi_N$ is  a pure functor; hence it equals $H_\ca$ in the notation of  Definition \ref{drange}(2). Thus to check that $A\neq 0$ it suffices to prove that  $\Phi_N\neq 0$. 
Now, the class $\obj\duz$ Hom-generates $\du$; hence there exists $M_0\in \obj \duz$ such that $\du(M_0,N)\cong \Phi(M_0, N)\neq\ns$. 

Lastly, we recall that $\hw$  has a cogenerator (see Corollary \ref{cwt}(\ref{iwheart})). Thus according to Corollary \ref{cwt}(\ref{ic25})  one can take $\ec^{\du}$ to be the stalk functor corresponding to this cogenerator.
\end{proof}

\begin{rema}\label{rab5}
\begin{enumerate}
\item\label{irrel}
Note that in the case where $\duz\opp$ embeds into the subcategory of compact objects of $\du$ (in particular, this is the case if the latter category is anti-isomorphic to itself) one can take $\du'=\du$. A toy example of this situation is $\du$ being the derived category 
$D(R-\modd)$, where $R$ is a commutative ring; 
$\duz$ is the category of perfect complexes. 
 
More generally, for $\du=R-\modd$,  
 where $R$ is a not (necessarily) commutative ring, one may take $\du'$ to be the derived category of right $R$-modules; this example has natural "differential graded" (see \S5.2 of \cite{bger}) and probably "spectral" (see \cite{schwmod}) generalizations.

On the other hand, Corollary \ref{cgdb}(\ref{iab5}) below demonstrates that our theorem can be applied for $\du$ being the homotopy category of an arbitrary proper simplicial stable (Quillen) model category.\footnote{Possibly, a somewhat more general statement of this sort may be obtained by using stable $\infty$-categories; see Denis Nardin's answer at \url{http://mathoverflow.net/q/255440}; 
 yet the author does not know much about these matters. Respectively, the author is not sure that the construction of $\du'$ in \S\ref{sprospectra} is "optimal" for the purposes of Theorem \ref{tab5}; yet this argument enables certain "computations" that are important for \cite{bgn}.}\ Hence our theorem gives a positive answer to 
  Question 3.8 of  \cite{parrasao} for a really wide range of triangulated categories. Recall also that Theorem 3.7 of loc. cit.
says that {\bf  countable} colimits in $\hrt$ are exact  for any compactly generated $t$.\footnote{The author wonders whether some version of our argument can work for arbitrary $\du$. For this purpose it seems necessary to "get rid of $\du'$" in the proof. One may try to construct certain version of the stalk functors $\Phi(P,-)$ 
  "directly". It appears that a minor modification of  the reasoning used in the proof of \cite[Proposition 2.1]{bsnull} gives the result whenever $\duz$ is {\it countable} (i.e., $\mo \duz$ is a countable set). It is not clear whether the result can be extended to the general case (possibly, using the arguments of \S3.1 of ibid.).}

Thus our theorem demonstrates once again that "weight structures can shed some light" on $t$-structures; cf. Corollary \ref{csymt}.  
 The author wonders whether one can mimic the argument above (and so, obtain certain stalk functors) using the "naive" category $\proo-\duz$ instead of $\du$ (note that $\proo-\duz$ is a {\it pro-triangulated} category that is "very rarely" triangulated).

\item\label{ismgen} The author also wonders whether one can describe a set of generators for $\hrt$ that is smaller than that described in part 1 of our theorem. If $t$ is 
 generated by a set $\cp$ of compact objects then a natural candidate here  is the (essentially small) class of $H_0^t(P)$ for $P$ running through elements of $\du^{t\le 0}\cap \du^{\alz}$. 

Note that to prove that this class generates $\hrt$ is 
 suffices to describe a generating family of stalk functors that could be presented as  colimits of functors corepresented by elements of $\du^{t\le 0}\cap \du^{\alz}$ (see Proposition \ref{pgen}(\ref{ipgen1})).  In particular, the proof of \cite[Proposition 4.2.10]{bondegl} is essentially an argument of this sort. 

Moreover,  
 the argument described in Remark \ref{rkrause}(I.2) above 
 yields that $\hrt$ is generated by 
 the (even smaller) class $H_0^t(\cq)$, 
where $Q$ is the smallest subclass of $\obj \cu$ containing $\cp$ such that for any triangle $Z[-1]\to X\to Y\to Z$ with $Y,Z\in \cq$ the object $X$ belongs to $\cq$ also. 

\item\label{irpostov} Proposition \ref{porthop}(\ref{inters}) (cf. also  Remark \ref{rnewt}(2)) certainly gives a one-to-one correspondence between weight structures (resp. $t$-structures) generated by 
	subsets of $\obj\duz$ in $\du$ and $t$-structures (resp. weight structures) generated by these sets in $\du'$. 
	This  observation was very successively applied in \cite[\S4]{postov}. However, (the proof of) our theorem (cf. also the next part of this remark) demonstrates that introducing a duality between $\du'^{op}$ and $\du$ can give information that can hardly be obtained if $\du$ and $\du'$ are considered "separately" only. 

\item\label{iremb} 
Proposition 6.2.1 of \cite{bger} 
 suggests the following conjecture:  $\hrt$ is equivalent to the 
category of those functors $\hw\opp\to \ab$ that respect coproducts. The 
conjecture certainly contains more information on $\hrt$ than Theorem \ref{tab5} (along with its proof); however, this description of $\hrt$ may be not that "useful" (since $\hw$ can be rather "complicated").

\item\label{ir4} An interesting question (that is closely related to the aforementioned conjecture) is
which  homological functors from $\du$ (resp. from $\du$) are {\it $\Phi$-corepresentable}, i.e., have the form $\Phi(-,N)$ (resp. $\Phi(M,-)$) for some $N\in \obj \du$ (resp. $M\in \obj \du'$). 
Now, the functors of    the form $\Phi(-,N):\du'\to \ab$
 are precisely the "extensions" (obtained using the dual to  Proposition \ref{pkrause}) of the functors that are $\du$-represented on $\duz$. 
 In the case where $\duz$ is {\it countable} (i.e., its object class is essentially countable and all morphism sets are at most countable) all cohomological functors from $\duz$ into $\ab$ are represented by objects of $\du$ according to Theorem  5.1 of \cite{neebrh}; so, we obtain a complete description of $\Phi$-representable functors in this case. Unfortunately, this argument cannot be extended to the case of a general $\duz$. 

\item\label{istalks} 
The stalk functors $\Phi(P,-)$ for $P\in \du'^{op}_{w=0}$ 
in the proof of Theorem \ref{tab5} essentially play the role of functors corepresented by $t$-projective objects of $\cu$ (see Definition \ref{dpt}). Note however we cannot have "enough" of the latter unless $\hrt$ has enough projectives (see Proposition \ref{pgen}(\ref{ipgen25})). If we assume in addition  
the existence of (a set of) "compact generators" for $P_t$ (that are necessary to obtain enough cc functors) 
then it would imply that $\hrt$ is isomorphic to $\adfu (\bu,\ab)$ for $\bu$ being the corresponding  small additive category. Certainly, $\hrt$  "rarely" can be presented in this form;  (cf. Remark \ref{rsheaves}(1) below). 
Hence 
 constructing a duality of $\du$ with an  "auxiliary" category $\du'$ is necessary for  the proof of Theorem \ref{tab5}.

\item\label
{irotimes}  In  Proposition \ref{porthop} the "starting duality" was $\cuz(-,-)$. It is an interesting question whether Theorem \ref{tab5} can be generalized by treating (biextensions of) other dualities (cf. Remark \ref{rcudu}(\ref{ir1},\ref{ir2})). 

Anyway, it appears that one can construct "interesting" dualities using various tensor products. 
 Assume that we are given triangulated categories $\cu$ and $\cupr$ as above, a  triangulated category  $\eu$ having coproducts,
a  bi-exact functor $\otimes: \cu\opp\times \cupr\to \eu$ that respects coproducts when any of the arguments is fixed, and a 
 cc functor $H:\eu\to \au$  (for some abelian $\au$). Then $\Phi(-,-)=H(-\otimes -)$ is a duality $\cu\opp\times \cupr\to \au$ that converts $\cu$-products and $\cupr$-coproducts into $\au$-coproducts. Certainly, one may take $H$ to be the functor corepresented by a compact object of $\eu$ (and $\au=\ab$).
Note also that $\Phi$ is canonically characterized by its "coproductivity properties" along with its restriction $\Phi_0$ to $\cu'_0{}\opp\times \cuz$ whenever $\cuz$  and $\cu'_0{}^{op}$  are categories of compact objects in $\cu\opp$ and $\cu'$ (respectively) that Hom-generate these categories (and so, generate them as their own localizing subcategories). 

The author believes that dualities of this sort may be useful for "computing" dualities of the type treated in  Proposition \ref{porthop} in the case where $\cu\opp=\cupr$, $\cuz$ is self-dual with respect to the tensor product on $\cupr$ (cf. part \ref{irrel} of this remark) and Hom-generates $\cupr$, and the unit object $\oo_{\cupr}$ is compact in $\cupr$. 
Note that all these condition are fulfilled for $\cupr$ being the stable homotopy category of ("topological") spectra; we will say more on this setting (that was essentially treated in \cite{prospect})  in Remark \ref{rsheaves}(4) below.

\item\label{krauseideals} Recall that Corollary 4.7 of \cite{krause} gives a description of all smashing (see Remark \ref{rtst2}(\ref{ismashs})) subcategories of a compactly generated triangulated category $\du$ in terms of certain ideals of morphisms 
 in its subcategory $\duz$ of compact objects. Along with Theorem 4.9 of ibid. and the example from \cite{kellerema} this appears to yield an example of a not compactly generated smashing triangulated subcategory $L$ of $\du$ 
such that the corresponding "shift-stable $t$-structure" 
 (see   Remark \ref{rtst2}(\ref{ismashs}) again) possesses a (left) $\Phi$-orthogonal shift-stable weight structure in the corresponding $\du'$. The authors wonders whether one can also obtain similar non-shift-stable   examples. 
\end{enumerate}
\end{rema}

\subsection{Relation to triangulated 
categories of pro-objects}\label{sprospectra} 


Now 
we describe a method for constructing a vast family  of examples for Proposition \ref{porthop} (and so, also of Corollary \ref{cwt}) along with  Theorem \ref{tab5}. 
Its main ingredient is a construction of a triangulated category of "homotopy pro-objects" for a  stable model category that is a straightforward application of the results of \cite{tmodel}.  

So let $\gm$ be a proper simplicial stable (Quillen) model category;  denote its homotopy category by $\cupr$. 
We construct another model category $\gmp$ whose underlying category is the category of (filtered) pro-objects of $\gm$ (cf. \S5 of \cite{tmodel}).

\begin{rema}\label{rproobj}
In \cite{tmodel} pro-objects were defined via filtered diagrams, i.e., via contravariant functors from filtered small categories. However the author prefers (for the reasons of minor notational convenience)  
  to consider inverse limits indexed by filtered sets instead. The latter notion is certainly somewhat more restrictive formally; however, Theorem 1.5 of \cite{adr}  says that any "categorical" filtered colimit can be presented as a certain "set-theoretic" one. It easily follows that it is no difference between the usage of these two notions (for the purposes of the current papers as well as for that of \cite{bgn}); cf.  the discussion in \S2 of \cite{isalim}. 
\end{rema}

We endow $\gmp$ with the {\it strict} model structure; see \S5.1 of ibid. (so, weak equivalences and cofibrations are {\it essential levelwise} weak equivalences and cofibrations of pro-objects).
  An important observation here is that this model structure is a particular case of a {\it $t$-model structure} in the sense of \S6 of ibid. if one takes the following "totally degenerate"  $t$-structure $t_{deg}$ on $\cupr$: $\cupr^{t_{deg}\ge 0}=\obj \cupr$, $\cupr^{t_{deg}\le 0}=\ns$; see Remark 6.4 of ibid.
	Indeed, 
  one can  take the following  functorial factorization of morphisms: for  $f\in \gm(X,Y)$ (for $X,Y\in \obj \gm$)  we can present it as $f\circ \id_X$; note that $\id_X$ is an {\it $n$-equivalence} and $f$ is a  {\it co-$n$-equivalence} in the sense of Definition 3.2 of ibid. for any $n\in \z$ (pay attention to Remark \ref{rstws}(3)!). 
	
	Now let us describe some basic properties of $\gmp$ and its homotopy category $\cu$ (we will apply some of them below, whereas other ones are important for \cite{bgn}). The pro-object corresponding to a  projective system 
	$M_i$ for $i\in I$ where $I$ is 
	 an inductive set and $M_i\in \obj \gm$, will be denoted by $(M_i)$. Note that $(M_i)$ is precisely the (inverse) limit of the system $M_i$ in $\gmp$ (by the definition of morphisms in this category).

\begin{pr}\label{pgdb}

Let  $X_i,Y_i,Z_i\ i\in I$, be projective systems in $\gm$. Then the following statements are valid.

1. $\gmp$ is a proper stable simplicial model category; hence $\cu$ is triangulated. 

2.  
If  some morphisms $X_i\to Y_i$ for all $i\in I$  yield a compatible system of cofibrations (resp. of weak equivalences; resp. some couples of morphisms 
$X_i\to Y_i\to Z_i$ yield a compatible system of 
   homotopy cofibre sequences) then the corresponding  morphism $(X_i)\to (Y_i)$ is a cofibration also (resp. a weak equivalence; resp. the couple of morphisms $(X_i)\to (Y_i)\to (Z_i)$ is a 
 homotopy cofibre sequence). 
 
3. The natural embedding $c:\gm\to \gmp$ is a left  Quillen functor; it also respects weak equivalences and fibrations.

4. For any $N\in \obj \gm$ 
we have $\cu((X_i),c(N))\cong \inli \cupr(X_i,N)$.

In particular, the 
 homotopy functor $\ho(c):\cupr\to \cu$ 
 is a full embedding, and $(X_i)$ is a $\cupr$-limit of $X_i$ in $\cu$ (with respect to this embedding; see Definition \ref{drelim}(2)).

5. More generally, for any projective system 
$\{M_i\}$ in $ \gmp$ and any $N\in \obj \cupr$ the inverse limit of $M_i$ exists in $\gmp$ and 
 we have $\cu(\prli M_i,c(N))\cong \inli_{i\in I} \cu (M_i,c(N))$. 

6. $\cu$ has products and all objects of $\ho(c)(\cupr)$  
are cocompact in $\cu$. 

7. The class $\ho(c)(\obj \cupr)$  cogenerates $\cu$.

\end{pr}
\begin{proof}
1. Theorems 6.3 and 6.13 of 
 \cite{tmodel} yield everything except the existence of functorial factorizations for morphisms in $\gmp$. 
  The existence of functorial factorizations is given by Theorem 1.3 of \cite{fufa} (see the text following Remark 1.5 of ibid.).
	One can also deduce this statement from \cite[Remark 4.5]{prospect}.

2. The first two parts of the assertion  are contained in the definition of the strict model structure. The last part follows from the previous ones immediately
(recall that pushouts can be computed levelwisely);  this fact is also mentioned in the proof of \cite[Proposition 9.4]{tmodel}.

3. The first part of the assertion is given by Lemma 8.1 (an also by \S5.1) of \cite{tmodel}. 
The second part is immediate from the description of weak equivalences in $\gmp$ given in loc. cit. 

4. The first of the statements is immediate from Corollary 8.7 of ibid.; the other ones are its obvious consequences.

5. The first part of the assertion is provided by  Theorem 4.1 of \cite{isalim}. Since loc. cit. also  (roughly) says that filtered limits in $\gmp$ can be naturally expressed in terms of  limits in $\gm$, combining this statement with assertion 4 we obtain the second part of assertion 5 immediately. 

6. 
$\cu$ has products since it is the homotopy category of a model category (see Example 1.3.11 of \cite{hovey}). Hence we should verify that $\cu(\prod_{i\in I} Y^i,X)=\bigoplus_{i\in I} \cu(Y^i,X)$ for $Y^i$ being fibrant objects of $\gmp$, $X\in \ho(c)(\obj \cupr)$. Now (see loc. cit.) the product of $Y^i$ in $\cu$ comes from their product in $\gmp$. Certainly, if $Y^i=(Y_{j}^i)$ then the product of $Y^i$ in $\gmp$ 
 can be presented by the projective system of all $\prod_{i\in J} Y^i_{j_i}$ for finite $J\subset I$. 
 Hence the statement follows 
from the previous assertion (recall  that products are particular cases of inverse limits).\footnote{
Alternatively, one can combine the argument dual to the one in the proof of Theorem 7.4.3 of \cite{hovey} 
with the 
fact that $c(f)$ 
for $f$ running through all fibrations in $\gm$ yield a set of 
generating fibrations for $\gmp$ (see Theorem 6.1 of \cite{chor}).}

7. 
Theorem 6.1 of \cite{chor} implies that $\gmp$ admits a non-functorial version of the generalized cosmall object argument with respect to $c(f)$. Hence we can apply the dual of the argument used in the proof of Theorem 7.3.1 of \cite{hovey}.
\end{proof}

We deduce some  consequences from this statement (mostly) using Proposition \ref{porthop}. We will consider $\cupr$ as a full subcategory of $\cu$ (via the embedding $c$ that we will not mention) in this corollary.

\begin{coro}\label{cgdb}

In the setting of the previous proposition assume that
$\cuz$ is an essentially small subcategory of $\cupr$ consisting of compact objects; let
  $\cp$ be  a suspended subset of  $\obj \cuz$. Denote by $\eu$ the colocalizing subcategory of $\cu$ cogenerated by $\obj \cuz$; 
	let $X_i$ be a projective system  in $\gmp$ and $N\in \obj \cupr$.

		Then the following statements are valid.
	
	\begin{enumerate}
\item\label{iprev} One can apply Corollary \ref{cwt} to this setting. 

\item\label{iab5}  $\hrt$ is a Grothendieck abelian category and there exists a faithful exact  functor $\ec:\hrt\to \ab$ that respects coproducts.

\item\label{ic3} 
$\Phi^{\eu}(L(\prli X_i),N)\cong \inli \Phi^{\eu} (L(X_i),N)$.

\item\label{icf} More generally, if $\au$ is an AB5 category and $H$ from $\cu$ into $\au$ is an extended functor (i.e., it is obtained via the  dual to  Proposition \ref{pkrause} from a  cohomological functor $H_0$ from $\cuz $ into $\au$)  then we have $H(\prli X_i)\cong H(L(\prli X_i))\cong \inli H(L(X_i))\cong \inli H(X_i)$. 

\item\label{icfig} 
 $\Phi^{\eu}$ is isomorphic to  the restriction to $\eu\opp\times \cu'$ of $\cu(-,-)$. 

\item\label{icharw} 
We have  $\eu_{w_{\eu}\le 0}=({}^{\perp_{\cu}}\cupr^{t\le -1})\cap \obj \eu $.

\end{enumerate}
\end{coro}
\begin{proof}
\ref{iprev}. Immediate from the previous proposition.

\ref{iab5}. According to the previous assertion, we can apply Theorem \ref{tab5} to our setting.

Assertion \ref{ic3} is a particular case of  assertion \ref{icf} indeed (by the definition of $\Phi$).  Next,  since $L$ respects $\cuz$-limits, 
 assertion \ref{icf} follows from Proposition \ref{pgdb}(6) according to (the dual to) Proposition \ref{pkrause}(\ref{ikres}).

\ref{icfig}. According to Proposition \ref{pkrause}(\ref{ikr8}) (cf. also Proposition \ref{porthop}(\ref{ibiext}))  
it suffices to verify that for any 
$M\in \obj \eu$ the functor $\cu(M,-):\cu\to \ab$ is a cc one. 
Now, this condition is certainly fulfilled if $M\in \obj \cuz$. Next,   the class of objects of $\eu$ satisfying this condition is shift-stable; hence it also closed with respect to extensions. Thus it remains to verify for a set of $M_i\in \obj \eu$ that the functor   $\cu(\prod M_i,-):\cupr\to \ab$ is a cc one if all $\cu(M_i,-)$ are.
The latter implication follows immediately from Proposition \ref{pgdb}(6).

\ref{icharw}. 
The previous assertion certainly implies that $({}^{\perp_{\cu}}\cupr^{t\le -1})\cap \obj \eu ={}^{\perp_{\Phi^{\eu}}}\cupr^{t\le -1}$. Hence it remains to apply Corollary \ref{cwt} (\ref{icort}).

\end{proof}

This proposition is 
applied in \cite{bgn} to various motivic homotopy categories. 
The corresponding $t$ is a (version of) the Voevodsky-Morel homotopy $t$-structure, whereas $w$ is called (a version of) the {\it Gersten} weight structure; the corresponding weight filtrations and weight spectral sequences generalize coniveau ones.

\begin{rema}\label{rsheaves}

1. Certainly, these methods can  be applied for $\cupr$ being some (other) triangulated category "constructed from sheaves"; one can use  Proposition 8.16 (and other results) of \cite{jardloc} to present $\cupr$ as the homotopy category of a proper stable model category. Note also that a category $\cupr$ of this sort is "usually" compactly generated (still cf. \cite{neeshman} and Remark \ref{rwg}(2)  above); 
in this case there exist plenty of possible $\cp$.

On the other hand, the heart of $t$ is "quite rarely" of the form $\adfu(\bu,\ab)$ for these examples; this justifies the claim made in Remark \ref{rab5}(\ref{istalks}). Note also that any stalk functor in this case should  
 send a complex of sheaves into a certain "stalk" of the zeroth (co)homology sheaf of this complex;  whence the name. 
 Moreover, for the motivic examples considered in \cite{bgn} the stalk functors are (retracts of coproducts of) certain "twists" of "actual stalks" (cf.  also \cite[Theorem 3.3.1]{bondegl} for a certain "relative" version of this observation). On the other hand, for other triples $(\du,t,\du')$ as in Theorem \ref{tab5} the stalk functors may give a new (and non-trivial) object of study. In particular, we obtain stalk functors for arbitrary compactly generated $t$-structures on arbitrary motivic homotopy categories (and  on their localizing subcategories); these may have quite non-trivial "geometrical meaning" (and do not require any resolution of singularities assumptions in contrast to loc. cit.). 
	
	2. Part \ref{iab5} of our corollary generalizes Corollary 4.9 of \cite{humavit} (where $\cu$ was assumed to be {\it algebraic} and $t$ was assumed to be non-degenerate; cf. Corollary \ref{csymt}(3)  and Remark \ref{rsymt}(\ref{irsymt2}) above). Note that both of these conditions are rather restrictive if one studies motivic homotopy categories.

3. Since all objects of $\cupr$ are cocompact in $\cu$, the class of cocompact objects of $\cu$ is not essentially small (in contrast to that for $\eu$). 

4. Recall 
  that 
 $\eu$ can be presented as the (Verdier) localization of $\cu$ by the subcategory $\perpp\eu$. In \cite{prospect} the case $\cupr=SH$ (the topological stable homotopy category) was considered, and the corresponding $\eu$ was constructed as the homotopy category of  a certain right Bousfield localization of $\cu$. Moreover, 
an exact equivalence $F:\eu\to \cupr\opp$ was constructed (in this case). 

Furthermore, Remark 6.9 of ibid. appears to imply that the duality  $\Phi^{\eu}$ in this case is isomorphic to the bifunctor  $SH(
S^0,F(-)\otimes -)$ ($S^0=\oo_{SH}$ is the sphere spectrum); note that it suffices to construct the restriction of this isomorphism to finite spectra 
(see Remark \ref{rab5}(\ref{irotimes})). 
Next, if we take $\cp=\{S^0[i]: i\ge 0\}$ then the corresponding $t$ is certainly the Postnikov $t$-structure for $SH$, whereas $w_{\eu}$ is easily seen to be the opposite (see Proposition \ref{pbw}(\ref{idual})) to the {\it spherical} weight structure on $\eu\opp\cong SH$ (see 
\S4.2 of \cite{bwcp} and \S4.2 of \cite{bkwn}).

Moreover, the author hopes that applying Remark \ref{rab5}(\ref{irotimes}, \ref{krauseideals}) in this context may shed some light on the seminal telescope conjecture.

5. We conjecture that 
 an  isomorphism $F:\eu\to \cupr\opp$  
 (in the setting of this section)  exists for a wide range of stable monoidal model categories such that $\cuz$ is self-dual with respect to $\otimes$ (and it generates $\cu'$ as its own localizing subcategory). 
Note that in the aforementioned particular case $\cupr=SH$ the existence of $F$ may be deduced from Theorem 5.3 of \cite{schwemarg};\footnote{For this purpose one should  use the fact that $\eu$ has a model; however the Margolis' uniqueness conjecture (see \cite[\S3]{schwemarg}) predicts  (in particular) that this conditions is fulfilled automatically.} cf. also  \cite[\S6.4]{bger} for a certain motivic observation related to our conjecture. 
 Note also that in this case the conjecture stated in Remark \ref{rab5}(\ref{iremb}) is easily seen to be fulfilled also.

6. Part \ref{icharw} of our corollary can easily be "axiomatized" somewhat similarly to Corollary \ref{cwt}. However, the existence of 
a category $\cu\supset \cu'$ such that all objects of $\cu'$ are cocompact in it seems to be rather "exotic".
\end{rema}

\subsection{On localizations of coefficients and "splittings" for triangulated categories}\label{slocoeff}

 In this subsection we gather a few results related to localizations of coefficients and decomposing triangulated categories into direct summands for the purpose of applying these statements in \cite{bgn}.

First we recall the "naive" method of localizing coefficients (for a triangulated category).

\begin{rema}\label{rcgws}
 1. Let 
$S\subset \p$ be a set of prime numbers; denote the ring  $\z[S\ob]$ by $\lam$.

Then for any triangulated $\cu$ one can consider the category $\cu\otimes \lam$ with the same object class and $\cu\otimes \lam(M,N)=\cu(M,N)\otimes_\z \lam$ for all $M,N\in \obj \cu$. 
 The category $\cu\otimes \lam$ has a natural structure of a triangulated category; see Proposition A.2.3 of \cite{kellyth}. Next, if $s$ is a torsion pair  for $\cu$ then   
 the Karoubi-closures  of the classes $(\lo,\ro)$ in $\cu\otimes \lam$ give  a torsion pair for $\cupr$ according to Proposition \ref{phop}(\ref{itp9}).

2. Certainly, for $\cu$ being an $R$-linear category (see Remark \ref{rsatur}(4)) we can also localize (using any of the methods described in this section) by any multiplicative subset of $R$.
\end{rema}

However, this method of "localizing coefficients" of a triangulated category does not seem to be "appropriate" if $\cu$ is has coproducts.
So we  describe an alternative construction that "works fine" for compactly generated categories. 

\begin{pr}\label{plocoeff} 
Assume that  $\cu$ is compactly generated; denote its subcategory of compact objects by $\cu^c$. 

For a set $S$  
as above denote by 
$\eu$ the localizing subcategory of $\cu$ generated by cones of 
 $p\id_M$ for $M\in \obj \cu,\ p\in S$; denote by $\du$ the full subcategory of $\cu$ whose object class is $\obj \eu\perpp$. 

I. Then the following statements are valid.

\begin{enumerate}
\item\label{icge} $\eu$ is compactly generated by cones of 
$p\id_N$ for $N$ running through 
 $\obj \cu^c$ (and $p\in S$).

\item\label{icgd} $\obj\du$  equals the class of $\lam$-linear objects of $\cu$, i.e.,  of those $M\in \obj \cu$  such that $p\id_M$ is an automorphism for any $p\in S$; so, it is closed with respect to $\cu$-coproducts. 

\item\label{icgadj} The embedding $i:\du\to \cu$ 
 possesses an exact  left adjoint $l_S$ that gives an equivalence $\cu/\eu\to \du$.

\item\label{icg3} For any $M\in \obj \cu$ and 
 $N\in \obj \cu^c$ we have $\du(l(N),l(M))\cong \cu(N,M)\otimes_\z \lam$; thus $\du$ contains $\cu^c\otimes\lam$ as a full subcategory.

\item\label{icgls}
 $l_S$ respects coproducts and converts compact objects into $\du$-compact ones. Moreover,   $\du$ is generated by $l_S(\obj \cu^c)$ as its own localizing subcategory, and the class of compact objects of $\du$ equals $\obj \lan l_S (\obj \cu^c\ra)\ra=\obj \kar_{\du} (\cu^c\otimes \lam)$. 
 
\item\label{icgtr} For any $M\in \obj \cu$  there exists a distinguished triangle 
\begin{equation}\label{ecdec}
N\to M\to i\circ l_S(M)\to N[1]
\end{equation}
 for some $N\in \obj \eu$; this triangle  is unique up to a canonical isomorphism.

\item\label{icgfun} Let $H:\cu\to \ab$ be a cc-functor. Then for any $M\in \obj \cu$ we have $H( i\circ l_S(M))\cong H(M)\otimes_\z\lam$.

\item\label{icgdi} Assume that $\cu'$ is also a compactly generated  (triangulated) category; define $\du'$, $i'$ and $l'_S$ as the $\cu'$-versions of $\du$, $i$, and $l_S$, respectively. Then  any functor $F:\cu\to \cu'$ that respects coproducts can be 
  canonically completed to a diagram
\begin{equation}\label{ediaf} 
\begin{CD}
 \cu@>{l_S}>>\du @>{i}>>\cu\\
@VV{F}V@VV{G}V@VV{F}V \\
\cu'@>{l'_S}>>\du' @>{i'}>>\cu'
\end{CD}
\end{equation}
where $G$ is a certain exact functor respecting coproducts. 
\end{enumerate}
II. Assume in addition that $\cp\subset \obj \cu$ 
is a  class of compact objects. 
Denote by $s=(\lo,\ro)$ the torsion pair generated by $\cp$ (whose existence is given by Theorem \ref{tclass}(\ref{iclass1}); note that $\cp$ is essentially small). 

1. Then 
the couple $s[S\ob]=(\kar_{\du}(l_S(\lo)),\kar_{\du}(l_S(\ro)))$ gives  the torsion pair generated by $l_S(\cp)$ in $\du$. 

2. If $s$ is weighty then $s[S\ob]$ also is and the functor $l_S$ is weight-exact with respect to the corresponding weight structures $w$ and $w[S\ob]$, respectively (and so $l_S(\cu_{w=0})\subset \du_{w[S\ob]=0}$).

3. If $s$ is associated to a $t$-structure $t$ then $s[S\ob]$ also is. Moreover,  both $l_S$ and $i$ are $t$-exact (with respect to the corresponding $t$-structures), and so $\hrt[S\ob]\subset \hrt$. Furthermore, $\obj \du\cap \cu^{t\le 0}=\du^{t[S\ob]\le 0}$ and $\obj \du\cap \cu^{t\ge 0}=\du^{t[S\ob]\ge 0}$.

4. If $N$ is a $\cp$-colimit of some $(N_j)$ (see Definition \ref{drelim}) then $l_S(N)$ is a $l_S(\cp)$-colimit of the corresponding   $(l_S(N_j))$.

III. 
 Adopt the notation and conventions of Proposition \ref{porthop}. Then for $\du'\subset \cu'$ (resp. for $\du\subset \cu$) 
being the corresponding  subcategories of $\lam$-linear object the restriction of $\Phi(-,-)$ to $\du\opp\times \du'$ is a nice duality $\Phi[S\ob]:\du\opp\times \du'\to \ab$. Moreover,  one can obtain it by "bi-extending" (see Remark \ref{rcudu}(\ref{ir1})) the duality $\du_0(-,-)$, where $\du_0=l_S(\cuz)\subset \du'$ is isomorphic to  the "naive" $S$-localization of $\cuz$ as described in Remark \ref{rcgws}(1). Furthermore,  the corresponding $w[S\ob]$ and $t[S\ob]$ are $\Phi[S\ob]$-orthogonal.
\end{pr}
\begin{proof}

I. This construction was described in detail in Appendix A.2 of \cite{kellyth}.  
So 
ibid. would yield our assertions \ref{icgd}--\ref{icgls} if we replace $\eu$ in it by the 
subcategory $\eu'$ defined by means of assertion \ref{icge}. Now, we certainly have $\eu'\subset \eu$, and the converse inclusion follows from ibid. since $l_S$ is easily seen to kill all 
 $\co(p\id_M)$ (for $M\in \obj \cu,\ p\in S$).  Hence $\eu=\eu'$ and we obtain assertions \ref{icge}--\ref{icgadj}.

Assertion \ref{icgtr} easily follows from Proposition \ref{pbouloc}(III.\ref{ibou1}).

\ref{icgfun}. First we note that $H(\obj \du)$ consists $\lam$-linear objects of $\ab$ (i.e., of $\lam$-modules) immediately from assertion \ref{icgd}. 

Next, the group 
$H(\co(p\id_M))$ is certainly killed by the multiplication by $p^2$ (for any $M\in \obj \cu,\ p\in S$). Since the subcategory $\au$ of $\ab$ consisting of groups all of whose elements are $S$-torsion (i.e., annihilated by multiplying by some products of elements of $S$) is a Serre subcategory closed with respect to coproducts, and $H$ is a cc functor, we obtain $H(\obj \du)\subset \obj \au$. 

Now, for any $M\in \obj \cu)$ we apply $H$ to the triangle (\ref{ecdec}) 
to obtain 
 an exact sequence  $H(N)\to H(M)\to H(i\circ l_S(M))\to H(N[1])$ (for some $N\in \obj \eu$).
Since $\lam$ is a flat $\z$-module, we can tensor this sequence by $\lam$ to obtain the exact sequence $H(N)\otimes_{\z}\lam \to H(M)\otimes_{\z}\lam \to H(i\circ l_S(M)) \otimes_{\z}\lam \to H(N[1]) \otimes_{\z}\lam$. Since $H(N)$ and $H(N[1])$ belong to  $\obj \au$, they are annihilated by $-\otimes_{\z}\lam$. Lastly, since $H(i\circ l_S(M))$ is $\lam$-linear, we obtain  $H(M)\otimes_{\z}\lam \cong  H(i\circ l_S(M)) \otimes_{\z}\lam\cong  H(i\circ l_S(M))$.

\ref{icgdi}. Certainly, $F$ maps $\lam$-linear objects of $\cu$ into $\lam$-linear ones; hence we can define $G$ as the corresponding restriction of $F$.
It remains to verify that the left hand square in (\ref{ediaf}) is essentially commutative. Applying assertion I.\ref{icgtr} we obtain the following: it suffices to check that for any $M\in \obj \cu$  the functor $F$ maps the distinguished triangle (\ref{ecdec}) into the corresponding triangle for $F(M)$. Hence it remains to note that $F(N)$ belongs to the corresponding localizing subcategory $\eu'$ of $\cu$ since $F$ respects coproducts. 

II.1. For any $M\in \obj \cu$ the functor $H_{M,S}: N\mapsto \du(l_S(N),l_S(M))$ is a cp functor from $\cu$ into $\ab$. Assertion \ref{icg3} implies that for any $M\in \ro$ the functor  $H_{M,S}$ kills $\cp$; thus it kills $\lo$ also (by Theorem \ref{tclass}(\ref{iclass5}). 
  Hence  $\kar_{\du}(l_S(\lo))\perp_{\du} \kar_{\du}(l_S(\ro))$. Next, objects of $\du$ certainly possess $s[S\ob]$-decompositions (since all objects come from $\cu$ and one can apply $l_S$ to $s$-decompositions). 

According to  Proposition \ref{phop}(\ref{itp9}), it remains to prove that $l_S(\ro)$ is Karoubi-closed in $\du$.\footnote{Note that $l_S(\lo)$ does not necessarily have this property.} 
This certainly reduces to $l_S(\ro)=l_S(\cp)^{\perp_{\du}}$ since $l_S$ respects the compactness of objects (see assertion I.\ref{icgls}). Now, $l_S(\cp)^{\perp_{\du}}\subset \ro$ by assertion I.\ref{icg3}, and it remains to note that $l_S$ maps objects of $\du$ into isomorphic ones.  

2. We certainly have $\lo[S\ob]\subset \lo[S\ob][1]$; thus $s[S\ob]$ is weighty indeed (see Remark \ref{rwhop}(1)). The remaining parts of the assertion follow immediately.

3. Similarly to the previous assertion, Remark \ref{rtst1}(\ref{it1}) implies that $s[S\ob]$ is associated to a $t$-structure. Thus $l$ is $t$-exact. 

According to assertion I.\ref{icgd}, to check the $t$-exactness of $i$ we should prove that for any $\lam$-linear object $M$ of $\cu$ its $t$-truncations are $\lam$-linear also. Now, for any $p\in S$ the functoriality of the $t$-decomposition triangle $L\to M\to R\to L[1]$  implies that $\cu(L,L)$ contains a morphism $1/p\id_L$ inverse to $p\id_M$ and  $\cu(R,R)$ contains a morphism $1/p\id_R$ inverse to $p\id_R$. Thus $L$ and $R$ are $\lam$-linear indeed.

It certainly follows that $\hrt[S\ob]\subset \hrt$. Lastly, $\du^{t[S\ob]\le 0}\subset \obj \du\cap \cu^{t\le 0}$ and $\du^{t[S\ob]\ge 0}\subset \obj \du\cap \cu^{t\ge 0}$. The converse implication is immediate from the $t$-exactness of $l_S$ (along with the fact that $l_S$ maps objects of $\eu$ into isomorphic ones). 

4. Immediate from assertion I.\ref{icg3}. 

III. $\Phi[S\ob]$ is a nice duality since it is a restriction of  a nice duality. Moreover, assertion I.\ref{icg3} gives the "description" of $\du_0$ in question.

Now, $\du$ has products, $\du'$ has coproducts, and the objects of  $\du_0$ are compact in $\du'$ and cocompact in $\du$ according to  assertions I.\ref{icgd}--\ref{icgadj} along with their duals.
So we will prove that $\Phi[S\ob]$ is isomorphic to the corresponding biextension of $\du_0(-,-)$ using  Proposition \ref{porthop}(\ref{ibiext}).
Since the embeddings $\du'\to \cu'$ and $\du\opp\to \cu\opp$ respect coproducts (see assertion I.\ref{icgd}), we obtain that $\Phi[S\ob]$ respects  $\du'$-coproducts as well as $\du\opp$-ones. Next,  assertion I.\ref{icg3} (along with assertion I.\ref{icge} and the duals of these two) implies that $\Phi[S\ob]$ annihilates both ${}^{\perp_{\du}}\duz \times \obj \du'$ and $\obj\du \times \duz^{\perp_{\du'}}$.

 Hence 
Proposition \ref{porthop}(\ref{ibiext}) implies the "moreover" part of the assertion. It remains to apply part \ref{ihoport}  of that proposition to obtain the orthogonality result in question.
\end{proof}

\begin{rema}\label{rloc} 
1. Combining part II.4 of our proposition (and adopting its notation) with Proposition \ref{pkrause}(\ref{ikr3}) we obtain the following: if $\cuz$ is a full triangulated subcategory of $\cu$, $\obj \cuz\subset \cp$, and $H:\du\to \au$ is an exact functor coextended from $l_S(\cuz)$ then $H(l_S(N))\cong \inli H(l_S(N_j))$.

Since $i$ is a full embedding, we also obtain that $H'(i(l_S(N)))\cong \inli H'(i(l_S(N_j)))$ for any functor exact functor $H':\cu\to \au$ coextended from $\cuz$.

2. One can easily generalize part I.\ref{icgfun} of our proposition to cc functors into arbitrary AB5 categories. 

3. Certainly, 
those results of of this subsection that concern compactly generated categories
 can easily be dualized. 
\end{rema}

Now we study certain "splittings" of triangulated categories.

\begin{pr}\label{psplit}
I. Assume that  $\cu$ is compactly generated, $\cuz$ is 
its subcategory of compact objects, and that decomposes into the direct sum of two triangulated subcategories $\cuz^1$ and $\cuz^2$. Denote by $\cu^1$ the localizing subcategory of  $\cu$  generated by $\cuz^1$.
 
Then the following statements are valid.

\begin{enumerate}
\item\label{idec}
$\cu^1$ is a direct summand of the category $\cu$; so, there exists an exact functor $l$ projecting $\cu$ onto $\cu^1$.
Moreover, $l$ is  essentially the only exact projector functor that respects coproducts and restricts to the projection $\cuz\to \cuz^1$. Furthermore, if $N\in \obj \cu$ is a $C$-colimit of some $(N_j)$ (for some $C\in \obj \cu$)  then $l(N)$ is a $l(C)$-colimit of the corresponding   $(l(N_j))$.

\item\label{isplcg} $l$ respects the compactness of objects, $\cu^1$ is compactly generated, and $\cuz^1$ is its subcategory of compact objects.

\item\label{isplitp} Assume that $s=(\lo,\ro)$ is a torsion pair generated by a class $\cp\subset \obj \cu$. Then $l(\lo)=\lo\cap \obj \cu^1$, $l(\ro)=\ro\cap \obj \cu^1$, and the couple $l(s)=(l(\lo),l(\ro))$ is a torsion pair for $\cu^1$. 
In particular, if $s$ is weighty then the functor $l$ is weight-exact with respect to the corresponding weight structures.

Moreover, the pair $l(s)$  is smashing whenever $s$ is.

\end{enumerate}
II. Adopt the notation and conventions of Proposition \ref{porthop};  assume  in addition that $\cuz$ generates $\cu'$ and cogenerates $\cu$, and that $\cuz\cong \cuz^1\bigoplus \cuz^2$. Denote by $\cu'^1$ (resp. $\cu^1$)  the subcategory of  $\cu'$ (resp. of $\cu$) that is (co)generated by $\cuz^1$. Then the restriction $\Phi^1$  of $\Phi(-,-)$ to  $\cu^1{}\opp\times\cu'^1$ is a nice duality, and one can obtain it by "bi-extending" (see Remark \ref{rcudu}(\ref{ir1})) the duality $\cu^1_0(-,-)$.
Moreover, the corresponding $w^1$ is $\Phi^1$-orthogonal to $t^1$.
\end{pr}
\begin{proof}
I.\ref{idec}. Since $\obj \cp^1\perp \obj \cp^2$ and vice versa,  our compactness assumptions easily imply that the natural functor $\cu^1\bigoplus \cu^2\to \cu$ is an equivalence, where  $\cu^2$ is the localizing subcategory of $\cu$ generated by  $\cp^2$. So, we obtain the existence of $l$. Next, $l$ is essentially unique since it should be identical on $\cu^1$ and should kill $\obj\cu^1{}^\perp=\obj \cu^2$. The "furthermore" part of the assertion follows immediately. 

\ref{isplcg}. Immediate from the previous assertion. 

\ref{isplitp}. Obvious from the fact that $l$ is a projection functor. 

II. The proof is rather similar to that of Proposition \ref{plocoeff}(III).

So, $\Phi^1$ is a nice duality since it is a restriction of  a nice duality. 

Assertion I.\ref{idec} easily implies that  $\Phi^1$ is isomorphic to the corresponding biextension of $\cu_0^1(-,-)$.
$\du_0(-,-)$ using  Proposition \ref{porthop}(\ref{ibiext}).
It remains to apply part \ref{ihoport}  of that proposition to obtain that $w_1\perp_{\Phi^1}t_1$. 
\end{proof}


\begin{thebibliography}{1}
\bibitem[AdR94]{adr} Ad\'amek J.,  Rosicky J., Locally presentable and accessible categories, London  Math. Soc. Lect. Notes, vol. 189. Cambridge University Press, 1994.

\bibitem[AiI12]{aiya} Aihara T.,  Iyama O., Silting mutation in triangulated categories// J. Lond. Math. Soc. (2) 85(3), 2012,  633--668.

\bibitem[AJS03]{talosa} Alonso L., Jerem\'ias A., Souto M.J., Construction of $t$-structures and equivalences of derived categories// Trans. of the AMS, 355(6), 2003, 2523--2543.

\bibitem[AKS02]{andkahn} Andr\'e Y., Kahn B., O'Sullivan P., Nilpotence, radicaux et structures mono\"idales //Rendiconti del Seminario Matematico della Universit\'a di Padova, vol. 108 (2002), 107--291.

\bibitem[Aus66]{auscoh} Auslander M., Coherent functors, in: Proc. conf. Categorical Algebra (La Jolla, 1965), Springer, Berlin, 1966, 189--231.

\bibitem[Bac17]{bach} Bachmann T.,  On the invertibility of motives of affine quadrics// Doc. Math. 22 (2017), 363--395. 

\bibitem[BaS15]{fufa} Barnea I., Schlank T., A new model for pro-categories// J. Pure Appl. Algebra 219(4), 2015, 1175--1210. 

\bibitem[Bar05]{barrabs} Barr M., Absolute homology// Theory and Applications of Categories vol. 14 (2005),  53--59.

 \bibitem[BBD82]{bbd} Beilinson A., Bernstein J., Deligne P., Faisceaux pervers// Asterisque 100 (1982), 5--171.


\bibitem[BoN93]{bokne} B\"okstedt M., Neeman A., Homotopy limits in triangulated categories// Comp. Math. 86 (1993), 209--234.

\bibitem[BoK89]{bondkaprserr} Bondal A.I., Kapranov M.M.,  Representable functors, Serre functors, and mutations//  Izvestiya Rossiiskoi Akademii Nauk. Seriya Matematicheskaya, 53(6), 1989, 1183--1205, transl. in  Izvestiya: Mathematics 35.3 (1990), 519--541.

 \bibitem[BVd03]{bvdb} Bondal A.I., Van den Bergh M.,  Generators and representability of functors in commutative and noncommutative geometry// Mosc. Math. J. 3(1), 2003, 1--36. 
	
	
	\bibitem[Bon10a]{bws} Bondarko M.V., Weight structures vs.  $t$-structures; weight filtrations,  spectral sequences, and complexes (for motives and in general)//  J. of K-theory, v. 6(3), 2010, 387--504,  see also \url{http://arxiv.org/abs/0704.4003}

\bibitem[Bon10b]{bger} Bondarko M.V., Motivically functorial coniveau spectral sequences; direct summands of  cohomology of function fields// Doc. Math., extra volume: Andrei Suslin's Sixtieth Birthday, 2010, 33--117; see also \url{http://arxiv.org/abs/0812.2672}

\bibitem[Bon11]{bzp} Bondarko M.V., $\mathbb{Z}[\frac{1}{p}]$-motivic resolution of singularities// Comp.\ Math., vol. 147(5), 2011,  1434--1446.


\bibitem[Bon14]{brelmot} Bondarko M.V., Weights for relative motives: relation with mixed complexes of sheaves// Int. Math. Res. Notes, 
vol. 2014(17), 2014, 4715--4767.  


 \bibitem[Bon15]{bkw} Bondarko M.V., On morphisms killing weights, weight complexes, and Eilenberg-Maclane (co)homology of spectra,  preprint, 2015, \url{http://arxiv.org/abs/1509.08453}


\bibitem[Bon16]{bcons} Bondarko M.V., On infinite effectivity of motivic spectra and the vanishing of their motives, preprint, 2016, \url{https://arxiv.org/abs/1602.04477}

\bibitem[Bon18a]{bgn} Bondarko M.V., Gersten weight structures for motivic homotopy categories; retracts of cohomology of function fields,  motivic dimensions, and  coniveau spectral sequences, preprint, 2018, \url{https://arxiv.org/abs/1803.01432}

\bibitem[Bon18b]{bwcp}  Bondarko M.V.,  On weight complexes,  pure functors, and detecting  weights, preprint, 2018, \url{https://arxiv.org/abs/1812.11952}

\bibitem[Bon19a]{bkwn}  Bondarko M.V.,  On morphisms killing weights and Hurewicz-type theorems, \url{https://arxiv.org/abs/1904.12853} 

\bibitem[Bon19b]{bvtr}  Bondarko M.V., 
From weight structures to (orthogonal) $t$-structures and back, preprint, 2019, \url{https://arxiv.org/abs/1907.03686}

\bibitem[Bon19c]{bpws}  Bondarko M.V., On perfectly generated weight structures and 
 adjacent $t$-structures, \url{https://arxiv.org/abs/1909.12819}

\bibitem[BoD17]{bondegl} Bondarko M.V., Deglise F., Dimensional  homotopy $t$-structures in motivic homotopy theory// Adv. in Math., vol. 311 (2017), 91--189. 


\bibitem[BoL16]{bkl}   Bondarko M.V., Luzgarev A.Ju., On relative $K$-motives,  weights for them, and negative $K$-groups, preprint, 2016, \url{http://arxiv.org/abs/1605.08435}

\bibitem[BoS14]{bscwh}  Bondarko M.V., Sosnilo V.A., Detecting the $c$-effectivity of motives, their weights, and dimension via Chow-weight (co)homology: a "mixed motivic decomposition of the diagonal", preprint, 2014, \url{http://arxiv.org/abs/1411.6354}

\bibitem[BoS15]{bsnull}  Bondarko M.V., Sosnilo V.A., A Nullstellensatz for triangulated categories//  Algebra i Analiz, v. 27(6), 2015, 41--56; transl. in St. Petersburg Math. J. 27(6), 2016,  889--898.  

\bibitem[BoS18a]{bsosn} Bondarko M.V., Sosnilo V.A., Non-commutative localizations of additive categories and weight structures; applications to birational motives//  J. Math. Jussie,  vol. 17(4), 2018,  785--821.


\bibitem[BoS18b]{bonspkar} Bondarko M.V., Sosnilo V.A., On constructing weight structures and extending them to idempotent extensions,  Homology, Homotopy and Appl., vol 20(1), 2018, 37--57. 


\bibitem[BoS19]{bsnew} Bondarko M.V., Sosnilo V.A., On purely generated $\al$-smashing weight structures and weight-exact localizations// J. of Algebra 535(2019),  407--455. 


 \bibitem[BoT17]{bontabu}  Bondarko M.V., Tabuada G., Picard groups, weight structures, and (noncommutative) mixed motives// Doc. Math. 
22 (2017),   45--66.

\bibitem[BoV19]{bvt}  Bondarko M.V.,  Vostokov S. V., On torsion theories, weight and $t$-structures in triangulated categories (Russian)//  Vestnik St.-Petersbg. Univ. Mat. Mekh. Astron., vol 6(64), iss. 1, 27--43, 2019; transl. in Vestnik St. Peters. Univers., Mathematics, 2019, vol. 52(1), 19--29. 


\bibitem[Cho06]{chor} Chorny B., A generalization of Quillen's small object argument// J. Pure and Appl. Alg., vol. 204 (2006), 568--583.

\bibitem[Chr98]{christ} Christensen J., Ideals in triangulated categories: phantoms, ghosts and skeleta// Adv. in Math. 136.2 (1998), 284--339.

\bibitem[ChI04]{prospect} Christensen J.D., Isaksen D. C., Duality and pro-spectra// Algebraic and Geometric Topology 4.2 (2004), 781--812.


\bibitem[FaI07]{tmodel} Fausk H., Isaksen D., t-model structures//  Homology, Homotopy and Appl. 9(1), 2007, 399--438. 

\bibitem[GiS96]{gs} Gillet H., Soul\'e C., Descent, motives and $K$-theory// J.  f. die reine und ang. Math. v. 478 (1996), 127--176.

\bibitem[HKM02]{hoshimi} Hoshino M., Kato Y., Miyachi J.-I.,  On t-structures and torsion theories induced by compact objects//  J. of Pure and Appl. Algebra, vol. 167(1), 2002,  15--35. 

\bibitem[Hov99]{hovey} Hovey M., Model categories,   Mathematical Surveys and Monographs, vol. 63, AMS, 2003.

\bibitem[HMV17]{humavit} H\"ugel L. A., Marks F., Vit\'oria J., Torsion pairs in silting theory // Pacific Journal of Math. 291.2 (2017), 257--278.

\bibitem[Isa02]{isalim} Isaksen D., Calculating limits and colimits in pro-categories// Fundamenta Math. 2.175 (2002), 175--194.
 
\bibitem[Jar15]{jardloc} Jardine R., Local homotopy theory, Springer, 2015. 

\bibitem[Kel94]{kellerema} Keller B.,  A remark on the generalized smashing conjecture// Manuscripta Math. 84.1 (1994), 193--198.

\bibitem[KeN13]{kellerw} Keller B., Nicolas P., Weight structures and simple dg modules for positive dg algebras //Int. Math. Res. Not. vol. 2013(5), 2013, 1028--1078.

\bibitem[Kelly12]{kellyth} Kelly S., Triangulated categories of motives in positive characteristic, Ph.D. thesis of Universit\'e Paris 13 and of the Australian National University,  
2012, \url{http://arxiv.org/abs/1305.5349}


\bibitem[KeS16]{kellyweighomol} Kelly S., Saito S., Weight homology of motives// Int. Math. Res. Notices, v. 2017(13), 2017, 3938--3984.

\bibitem[Kra00]{krause} Krause H., Smashing subcategories and the telescope conjecture  --- an algebraic approach// Invent. math. 139 (2000), 99--133.

\bibitem[Kra01]{krauwg} Krause H., On Neeman's well generated triangulated categories// Doc. Math. 6 (2001), 121--125.

\bibitem[Kra02]{kraucoh} Krause H., A Brown representability theorem via coherent functors// Topology 41(4), 2002, 853--861.


\bibitem[Kra15]{krauslender} Krause H., Deriving Auslander's formula// Doc. Math. 20 (2015),  669--688. 

\bibitem[Mar83]{marg} Margolis H.R., Spectra and the Steenrod Algebra: Modules over the Steenrod Algebra and the Stable Homotopy Category, Elsevier, North-Holland, Amsterdam-New York, 1983.


\bibitem[Mod10]{modoi} Modoi G.C.,  On perfectly generating projective classes in triangulated categories// Comm. Alg. 38(3), 2010, 995--1011.

\bibitem[Nee97]{neebrh} Neeman A., On a theorem of Brown and Adams// Topology 36 (1997), 619--645.

\bibitem[Nee01a]{neebook} Neeman A., Triangulated Categories, Annals of Mathematics Studies 148 (2001), Princeton University Press, viii+449 pp.

\bibitem[Nee01b]{neeshman} Neeman A., On the derived category of sheaves on a manifold// Doc. Math. 6 (2001), 483--488.

\bibitem[Nee08]{neefrosi} Neeman A., Brown representability follows from Rosick\'y's theorem// J. of Topology, 2.2 (2009), 262--276.

\bibitem[Nee18]{neesat} Neeman A., Triangulated categories with a single compact generator and a Brown representability theorem, preprint, 2018,  \url{https://arxiv.org/abs/1804.02240}

\bibitem[NiS09]{nisao}  Nicolas P., Saorin M., Parametrizing recollement data for triangulated categories//  J. of  Algebra  322 (2009), 1220--1250. 

\bibitem[NSZ19]{zvon}  Nicolas P., Saorin M., Zvonareva A., Silting theory in triangulated categories with coproducts, J. of Pure and Appl.  Algebra, vol. 223(6), 2019, 2273--2319.

\bibitem[PaS15]{parrasao} Parra C.E., Saorin M., Direct limits in the heart of a $t$-structure: the case of a torsion pair// J. of Pure and Applied Algebra 219.9 (2015), 4117--4143.

\bibitem[Pau12]{paucomp}  Pauksztello D., A note on compactly generated  co-t-structures// Comm. in Algebra, vol. 40(2), 2012, 386--394.

\bibitem[PoS16]{postov}  Pospisil D.,  \v{S}\v{t}ov\'\i\v{c}ek J., On compactly generated torsion pairs and the classification of co-t-structures for commutative Noetherian rings//   Trans. Amer. Math. Soc.  368 (2016), 6325--6361. 

\bibitem[Ros05]{rosibr}  Rosicky J., Generalized Brown representability in homotopy categories// Theory and applications of categories, vol. 14(19), 2005, 451--479.

\bibitem[Roq08]{roq} Rouquier R., Dimensions of triangulated categories //J. of K-theory 
1.02 (2008), 193--256.
 
\bibitem[Sch02]{schwemarg}  Schwede S.,  A uniqueness theorem for stable homotopy theory //Math. Zeitschrift 239.4 (2002), 803--828.

\bibitem[ScS03]{schwmod} Schwede S.,  Shipley B.,  Stable model categories are categories of modules// Topology 42(1), 2003, 103--153.

\bibitem[Sou04]{salorio} Souto M. J., On the cogeneration of t-structures// Archiv der Mathematik 83.2, 2004, 113--122.

\bibitem[Wil18]{wildcons} Wildeshaus J., Weights and conservativity,  Algebraic Geometry 5(6), 2018, 686--702. 


\bibitem[Wil17]{wildab} Wildeshaus J.,  Intermediate extension of Chow motives of Abelian type //Adv. in Math., vol. 305 (2017), 515--600.

\end{thebibliography}
\end{document}